\documentclass[10pt,reqno]{amsart}

\usepackage[utf8]{inputenc}
\usepackage{lmodern}
\usepackage[T1]{fontenc}
\usepackage{amsmath,amssymb}
\usepackage{hyperref}
\usepackage{mathrsfs}
\usepackage{amsfonts}
\usepackage{mathtools}
\usepackage{tikz}
\usepackage{lmodern}
\usepackage{bigints}
\usepackage{relsize}
\usepackage{bbm}
\usepackage{amsthm}
\usepackage{mathtools}
\mathtoolsset{showonlyrefs}
\usepackage{geometry}
\usepackage{scalerel,stackengine}
\stackMath
\newcommand\reallywidehat[1]{
\savestack{\tmpbox}{\stretchto{
  \scaleto{
    \scalerel*[\widthof{\ensuremath{#1}}]{\kern-.6pt\bigwedge\kern-.6pt}
    {\rule[-\textheight/2]{1ex}{\textheight}}WIDTH-LIMITED BIG WEDGE
  }{\textheight} 
}{0.5ex}}
\stackon[1pt]{#1}{\tmpbox}
}
\numberwithin{equation}{subsection}
\newcommand{\norm}[1]{\left\lVert#1\right\rVert}
\newtheorem{theorem}{Theorem}[section]
\newtheorem{corollary}[theorem]{Corollary}
\newtheorem{lemma}[theorem]{Lemma}
\newtheorem{proposition}[theorem]{Proposition}
\newtheorem{definition}[theorem]{Definition}
\newtheorem{notation}[theorem]{Notation}
\newtheorem{remark}[theorem]{Remark}
\theoremstyle{definition}
\DeclareMathOperator{\sppp}{Span}

\DeclareMathOperator{\I}{Im}
\DeclareMathOperator{\Raa}{Range}
\DeclareMathOperator{\Ree}{Re}

\DeclareMathOperator{\supp}{supp}

\DeclareMathOperator{\kerrr}{Ker}

\title{Dispersive analysis for one-dimensional charge transfer models}

\author[G. Chen]{Gong Chen$\textsuperscript{\textdagger}$}
\author[A. Moutinho]{Abdon Moutinho$^{\ast}$}

\begin{document}

\email{gc@math.gatech.edu}
\email{aneto8@gatech.edu}

\address{School of Mathematics, Georgia Institute of Technology, Atlanta, GA 30332, USA}
\thanks{$^\ast$Corresponding author; e-mail: aneto8@gatech.edu.}
\thanks{\textsuperscript{\textdagger}GC  was partially supported by NSF grant DMS-2350301 and by Simons foundation MP-TSM00002258.}

\date{\today}

\begin{abstract}
In this paper, we study one-dimensional linear Schr\"odinger equations with multiple moving potentials, known as transfer charge models. 
Focusing on the non-self-adjoint setting that arises in the study of solitons, we systematically develop the scattering theory and establish dispersive estimates under the assumption that the potentials move at significantly different velocities, even in the presence of unstable modes. In particular, we prove the existence of wave operators, asymptotic completeness, and pointwise decay of solutions, without requiring the absence of threshold resonances. 
Our analysis sets up the fundamental work for studying the nonlinear dynamics of multi-solitons, including asymptotic stability and collisions.
\end{abstract}
\maketitle

\tableofcontents

\section{Introduction}

\subsection{Background}

In this article, motivated by the study of dynamics of multi-solitons for the 1-dimensional nonlinear Schr\"odinger equations, fixed $m\in\mathbb{N}$, we consider  the following linear matrix Schr\"odinger equation with a time-dependent charge transfer Hamiltonian
\begin{equation}\label{p}\tag{CTM}
    \begin{aligned}
& i \partial_t \vec{\psi}+\left(\begin{array}{cc}
\partial_x^2 & 0 \\
0 & {-}\partial_x^2
\end{array}\right) \vec{\psi}+\sum_{j=1}^m V_j\left(t \right) \vec{\psi}=0 \text{,  $(t,x)\in\mathbb{R}\times\mathbb{R}$} \\
& \left.\vec{\psi}\right|_{t=0}=\vec{\psi}_0,
\end{aligned}
\end{equation}
where  the functions $V_j$ are matrix potentials of the form
$$
V_j(t)=\left(\begin{array}{cc}
U_j(x-v_jt-y_j) & -e^{i \theta_j(t, x)} W_j(x-v_jt-y_j) \\
e^{-i \theta_j(t, x)} W_j(x-v_jt-y_j) & -U_j(x-v_jt-y_j)
\end{array}\right),
$$
with $$\theta_j(t, x)=\left(\left|v_j\right|^2+\omega_j\right) t+2 x v_j+\gamma_j, \omega_j, \gamma_j \in \mathbb{R}, \omega_j > 0$$ and $v_j$ are distinct velocities (will be listed by their sizes).   We assume that the real-valued functions $U_j(x),\,W_j(x)$ exhibit rapid decay which is always ensured by the decay of solitons in the non-linear setting.

Other than our main motivation being the study of multi-solitons, the charge transfer model, in particular, its scalar formulation,  has numerous physical applications, primarily in studying the motion of a particle, often an electron, under the influence of attractive or repulsive forces exerted by atoms or ions; see \cite{ Yajima1} for example.  Although we study the matrix version of the Schr\"odinger equation, one can directly observe that our main results can be extended to the scalar charge transfer models after replacing the matrix distorted Fourier transforms by the scalar ones. 

\par Focusing on one dimension, in this paper, we will classify all the solutions of the models \eqref{p}. Moreover, we will derive the $L^2$ bound and  dispersive estimates for scattering solutions to \eqref{p}. More precisely, we will first establish the asymptotic completeness. After the asymptotic completeness is obtained from the analysis of the scattering part, we obtain the standard dispersive decay and local improved decay for this component.  These estimates are crucial step-tones to study the dynamics of multi-solitons. As an important by-product of our analysis, the asymptotic structure of the scattering part gives an effective notation of linear profile which plays a crucial role in the resonance analysis and long-range scattering, see for example \cite{GCnls,CP2,GPR,KaPu,collotger,nls3soliton}

Due to the physical importance of charge transfer models, in particular, the scalar version,  the study  of its asymptotic completeness and scattering theory has a long history. Without trying to be exhaustive,  we refer to  \cite{Cai,Dispesti,SoWu,Geocharge,Zielinski1} for various approaches under different assumptions and more historical references. One point to note is that most of the earlier results only cover the self-adjoint setting except that, in \cite{Dispesti}, the asymptotic completeness was obtained based on the dispersive decay. 

Moving to the dispersive decay and localized decay estimates, it is well known that to establish these estimates for the perturbed Schr\"odinger operators both in the scalar and matrix forms is the cornerstone to  study the long-time dynamics of solutions for dispersive models with potentials including scattering or modified scattering of small solutions, and asymptotic stability of a special class of solutions such as solitons. See for example \cite{Busper1, BusSule, schlag2,KriegerSchlag, Mizumachi,collotger,nls3soliton, perelmanasym,GCnls,CP1,CP2,GPR}.  In particular, when studying the asymptotic stability of large solitons of the nonlinear Schr\"odinger equation, the linearization around a solition  naturally leads us to  the following matrix non-self adjoint operators
\begin{equation}\label{Hww1}
\mathcal{H}_{\omega_{\ell}}\coloneqq {-}\partial^{2}_{x}
    \begin{bmatrix}
    1  & 0\\
    0 & {-}1
    \end{bmatrix}+\begin{bmatrix}
    \omega_{\ell}  & 0\\
    0 & {-}\omega_{\ell}
    \end{bmatrix}+ 
    \begin{bmatrix}
    U_{\ell}(x) & W_{\ell}(x)\\
    {-}W_{\ell}(x) & {-}U_{\ell}(x) 
    \end{bmatrix},
\end{equation}and then one has to study the decay properties of the corresponding Schr\"odinger flow. The study of the spectral properties of these matrix operators. We refer to \cite{Busper1, BusSule, schlag2,KriegerSchlag,collotger,nls3soliton,WSSurvey} for details. 

\par Considering the dynamics of solutions around a multi-soliton and performing a linearization around it, then \eqref{p} will appear, see \cite{nlsstates,perelmanasym}.
Considering the case where the dimension $d=3,$ there exist many references in the literature on the decay estimates for solutions of \eqref{p}. One of the most celebrated is the article \cite{Dispesti} of Rodnianski, Soffer, and Schlag where the first pointwise decays from $L^2\bigcap L^1$ to $L^2+L^\infty$ were obtained for solutions of \eqref{p} when $d=3.$  
The $L^2$ components later on were removed by Cai \cite{Cai}. Later, Strichartz estimates were obtained in \cite{Stritchargewave} by Chen and \cite{sofferYaoDeng} by Deng-Soffer-Yao. However, in dimension $d=1$, the key difference from  the case that $d=3$ is that the semigroup operator $e^{{-}it\mathcal{H}_{\omega_{\ell}}}$ has a slower decay.  More precisely, for example, if the operator $\mathcal{H}_{\omega}$ does not have an embedded eigenvalue on its essential spectrum, then the estimate
\begin{equation*}
\norm{e^{{-}it\mathcal{H}_{\omega_{\ell}}}P_{c}f}_{L^{\infty}_{x}(\mathbb{R})}\leq \frac{C}{t^{\frac{1}{2}}}\norm{P_{e}f}_{L^{1}_{x}(\mathbb{R})}, 
\end{equation*} holds, where $P_{e}$ is the projection on the essential spectrum, while the respective dispersive estimate of the same semigroup when $d=3$ written in \cite{Dispesti} has a decay rate of order $O(t^{{-}\frac{3}{2}}),$ which is integrable on any subset of $\mathbb{R}_{\geq 1}.$ Therefore we cannot repeat the bootstrap argument from \cite{Dispesti} due to the slow decay of the operators $e^{{-}it\mathcal{H}_{\omega_{\ell}}}$ when $d=1.$  More technically, certain Kato-smoothing estimates in \cite{Dispesti} can not be applied in $d=1$ case either. Overall, the analysis of interactions among potentials in dimension $1$ by decay and smoothing estimates is much more complicated than in higher dimensions. 
In dimension $d=1$, by Perelman, \cite{perelmanasym},  under the assumptions that $m=2$ and $|v_2-v_1|\gg 1$,  some pointwse decays and  the asymptotic completeness were obtained. 

In this paper, assuming that the velocities are well separated, i.e. $|v_\ell-v_{\ell-1}|\geq K,|y_\ell-y_{\ell-1}|\geq L,\,\ell=2,\ldots,m$ for some $K>0$ and $L>0$ large enough depending on $m$, we prove the standard pointwise decay and the local improved decay for the one-dimensional charge transfer model.  We also extend the dispersive estimates and the asymptotic completeness result obtained by Perelman \cite{perelmanasym} to charge transfer models with more than two potentials and for the case including unstable modes, see Theorems \ref{princ} and \ref{Decesti} in Section \ref{nn}.

More precisely, under the assumptions that  the discrete spectrum $\sigma_{d}(\mathcal{H}_{\omega_{\ell}})$ is a finite subset of $\mathbb{R}\cup i\mathbb{R}$,  and that all the functions $U_{\ell}(x),\,W_{\ell}(x)$ decay rapidly, we first introduce a notion of `free evolution', see Definition \ref{s0def}. Then in Theorem \ref{tcont} and Theorem \ref{tdis}, we obtain the existence of wave operators for our problem. Then after establishing the asymptotic completeness in Theorem \ref{princ}, we are going to verify that  any solution $\overrightarrow{\psi}(t)$ of \eqref{p} which is asymptotically orthogonal to the corresponding moving discrete modes of all operators $\mathcal{H}_{\omega_{\ell}}$ satisfies for all $t\geq 0$ {\small
\begin{align}\label{pp1}
   \norm{\overrightarrow{\psi}(t)}_{L^{2}_{x}(\mathbb{R})}\leq & C\norm{\overrightarrow{\psi}(0)}_{L^{2}_{x}(\mathbb{R})},\\ \label{pp2}
   \norm{\overrightarrow{\psi}(t)}_{L^{\infty}_{x}(\mathbb{R})}\leq & \frac{C }{t^{\frac{1}{2}}} \max_{\ell}\norm{(1+\vert x\vert )\chi_{\left\{\frac{y_{\ell}+y_{\ell+1}}{2},\frac{y_{\ell}+y_{\ell-1}}{2}\right\}}(x)\overrightarrow{\psi}(0,x+y_{\ell})}_{L^{2}_{x}(\mathbb{R})},
\end{align}}
where $\chi$ is a smooth bump function adapted to the corresponding interval. 
For more information see Theorems \ref{princ} and \ref{Decesti} on Subsection $1.5.$ In a more general setting, the asymptotic completeness from Theorem \ref{princ} implies that any   solution of \eqref{p} converges in the $L^{2}$ norm when $t$ approaches ${+}\infty$ to 
{\footnotesize\begin{align}\label{explicit0}
&\sum_{\ell}\sum_k\left[
    \begin{bmatrix}
    e^{i\left(\frac{v_{\ell}x}{2}-t\frac{v_{\ell}^2}{4}+\gamma_{\ell}\right)}  & 0\\
    0 & e^{{-}i\left(\frac{v_{\ell}x}{2}-t\frac{v_{\ell}^2}{4}+\gamma_{\ell}\right)}
    \end{bmatrix}e^{{-}i\lambda_{\ell,k} t}\overrightarrow{z}_{\omega_{\ell,k}}\left(x-v_{\ell}t-y_{\ell}\right)\right]\\
    &+\sum_{\ell}
    \begin{bmatrix}
    e^{i\left(\frac{v_{\ell}x}{2}-t\frac{v_{\ell}^2}{4}+\gamma_{\ell}\right)}  & 0\\
    0 & e^{{-}i\left(\frac{v_{\ell}x}{2}-t\frac{v_{\ell}^2}{4}+\gamma_{\ell}\right)}
    \end{bmatrix} \left(t\overrightarrow{z}^0_{\omega_{\ell}}\left(x-v_{\ell}t-y_{\ell}\right)+\overrightarrow{z}^1_{\omega_{\ell}}\left(x-v_{\ell}t-y_{\ell}\right)\right)+\mathcal{T}(\overrightarrow{\phi})(t,x),
\end{align}}
such that each function $\overrightarrow{z}_{\omega_{\ell,k}},\,\overrightarrow{z}^0_{\omega_{\ell}},\,\overrightarrow{z}^1_{\omega_{\ell}}\in L^{2}_{x}(\mathbb{R},\mathbb{C}^{2})$,   $\overrightarrow{z}_{\omega_{\ell,k}}\in \kerrr \left[\mathcal{H}_{\omega_{\ell}}-\lambda_{\ell,k}\mathrm{Id}\right]$ for all the complex values $\lambda_{\ell,k}\in \sigma_{d}\left(\mathcal{H}_{\omega_{\ell}}\right),$  $\overrightarrow{z}^0_{\omega_{\ell}}\in \kerrr \mathcal{H}_{\omega_{\ell}}$  ,  $\overrightarrow{z}^1_{\omega_{\ell}}\in \kerrr \mathcal{H}_{\omega_{\ell}}^2$  with $\mathcal{H}_{\omega_{\ell}}\overrightarrow{z}_{\omega_{\ell}}^1=i\overrightarrow{z}_{\omega_{\ell}}^0$,  and
$\mathcal{T}(\overrightarrow{\phi})(t,x)$ is asymptotically orthogonal to the eigenfunctions of all operators $\mathcal{H}_{\omega_{\ell}}.$ The error in the convergence of $\overrightarrow{\psi}(t)$ to \eqref{explicit0} is of order $O(e^{{-}\beta t}) $ for a constant $\beta>0$ depending on $\min_{\ell}v_{\ell}-v_{\ell+1}>0.$ In particular, $\mathcal{T}(\overrightarrow{\phi})(t,x)$ satisfies the decay estimates \eqref{pp1} and \eqref{pp2}, for more information in the properties of see also Theorems \ref{tcont} and \ref{Decesti} in Subsection $1.6.$

\par  

We emphasize that extending the result of \cite{perelmanasym} from two potentials to the general case is \emph{far from trivial}.  To achieve our goals, we prove the asymptotic completeness for one-dimensional charge transfer models, as stated in Theorem \ref{princ}, using a manner which \emph{differs significantly} from Perelman's approach in \cite{perelmanasym}. In particular, our proof does not rely on the residue calculus or the Cauchy integral
formulas which are crucial in \cite{perelmanasym}. Instead, we use the asymptotic completeness for one potential and carefully study the interaction of distorted planar waves given by different potentials. Moreover, our approach applies to models where the operators $\mathcal{H}_{\omega_{\ell}}$ may exhibit \emph{non zero discrete modes} including unstable ones. For further details, see Section \ref{Asc}. 
We also would like to point out that the proofs of asymptotic completeness and the standard pointwise decay in this paper \emph{do not} need the generic assumption, i.e. (H4), in Theorem \ref{Decesti}, which was imposed in \cite{perelmanasym}. Since the scattering theory with the appearance of threshold resonances in the general setting, in particular, the construction of distorted Fourier bases did not appear in other references to the best of our knowledge\footnote{For the special case of the cubic NLS, see \cite{nls3soliton}.}, we provided details on the distorted Fourier bases in general non-generic cases in Section \ref{sec:scatteringone}.
Finally, we should note that many of our linear decay estimates in this paper are sharped for the purpose of nonlinear applications.

\par To conclude the general introduction, we point out that understanding the decay estimates of solutions of \eqref{p} is crucial to the study of the asymptotic stability of multi-solitons for Schr\"odinger models. This approach was taken in the papers \cite{Dispesti} and \cite{nlsstates} by Rodnianski, Soffer, and Schlag in $d=3$ and in \cite{Perelman4} by Perelman to prove the asymptotic stability of multi-solitons for the nonlinear Schr\"odinger equation $d\geq 3$.
For $d=1$, in the paper \cite{perelmanasym}, Perelman proved the asymptotic stability of two fast solitons for a large set of one-dimensional nonlinear Schr\"odinger equations. Also see \cite{CJlinear,CJnonlinear,GCwave,GCwavecmp} in the settings of Klein-Gordon equations and wave equations respectively. With the general estimates we develop in this paper, the asymptotic stability of a general multi-soliton with well-separated speeds and positions for various $1d$ nonlinear Schr\"odinger equations will be reported in a later paper. These estimates will also be used to study the collision of fast-speed solitons. We expect that the study of the scattering solutions results in a good notation of the linear profile asymptotically, which will be crucial to study the long-time behavior in the long-range scattering setting.




\subsection{Organization and notations}
Before moving on to the main results,  we briefly outline the structure of the article and introduce some of the notation used throughout.
\subsubsection{Organization}

\par In \S \ref{subsub:assumption}, we state the hypotheses we imposed for  statements of our main results. In \S \ref{subsubsec:scattering},  we define the scattering space of solutions of \eqref{p}. 
\par In \S\ref{subsubsec:scatnotation}, we introduce the generalized eigenvectors $\mathcal{F}_\omega(x,k),\,\mathcal{G}_\omega(x,k)
\in L^{\infty}_{x}(\mathbb{R},\mathbb{C}^{2})$ of $\mathcal{H}_{\omega}$ with eigenvalue $k^{2}+\omega$ and explain some of their properties. In particular, these two generalized eigenvectors are used to find the orthogonal projection of any function $f\in L^{2}_{x}(\mathbb{R},\mathbb{C}^{2})$ into the range of the essential spectrum of $\mathcal{H}_{\omega}.$ Furthermore, the functions $\mathcal{F}_\omega(x,k),\,\mathcal{G}_\omega(x,k)$ will be important in the proof of the decay estimates in the main Theorem \ref{Decesti} for all solutions of \eqref{p}. For more information about $\mathcal{F}_\omega(x,k),\,\mathcal{G}_\omega(x,k),
$ see Section \ref{sec:scatteringone},  Section $6$ of \cite{KriegerSchlag}, see also \cite{Busper1}.
\par The main results  of this article, Theorems \ref{princ} and \ref{Decesti}, are written in Subsection \ref{Theoprincip}. Before introducing them, we consider in \S \ref{nn} the propositions about the properties of the solutions $\overrightarrow{\phi}$ of \eqref{p} belonging to the scattering space and the discrete space. More precisely, Theorem \ref{tcont} of \S \ref{nn} describes all the solutions of \eqref{p} belonging to the scattering space, while Theorem \ref{tdis} of the same subsection describes all the solutions of \eqref{p} converging to the Galiean transforms of the range of the discrete spectrum projection of $\mathcal{H}_{\omega_{\ell}}$ for any $\ell\in\{1,2,\,...,\,m\}$ when $t$ goes to ${+}\infty.$ In Subsection \ref{sub:sketch}, we illustrate the key ideas to establish the asymptotic completeness.  
\par In Section \ref{sec:scatteringone}, we explain the main information about the scattering theory for the operators $\mathcal{H}_{\omega}$ defined in \eqref{Hww1}. They will be used in the proof of Theorem \ref{princ}. References used for Section \ref{sec:scatteringone} were taken from \cite{Busper1}, \cite{KriegerSchlag}, and \cite{collotger}. Since there is no standard reference for the construction of distorted Fourier transforms in the general cases with threshold resonances, we provide details in this section. 

\par In Section \ref{sec:hardy}, we explain the properties of the Hardy spaces $H^{2}(\mathbb{C}_{\pm}),$ that will be crucial ingredients in the proof of  Theorem \ref{princ}. 
\par
Theorem \ref{princ} will follow from Theorem \ref{TT} in Section \ref{sec:S0}. 
To achieve the goal, it requires the lemmas and propositions in this section
which are proved using the theory of Hardy spaces $H^{2}(\mathbb{C}_{\pm})$ from Section \ref{sec:hardy} and the knowledge of the scattering theory from Section  \ref{sec:scatteringone}.

\par The proofs of the main Theorems \ref{princ} and \ref{Decesti} are given in Section \ref{Asc} and Section \ref{Dispsection} respectively. In Section \ref{asyinfinity}, we prove Theorem \ref{tcont} and Theorem \ref{tdis} about the existence of solutions belonging to the scattering space and discrete space. Some auxiliary details are provided in  Appendix \ref{app}  and  Appendix \ref{sec:appb}.

\subsubsection{Notations}
Throughout this article, in various places, we use $\Diamond$ to denote dummy variables.

As usual, “$A := B$” or “$B =: A$” is the definition of $A$ by means of the expression
$B$.  

We use $\langle \Diamond \rangle:=\sqrt{1+\Diamond^2},$ $p=\begin{bmatrix}
    1 & 0\\
    0 & 0
\end{bmatrix},$ and $q=\begin{bmatrix}
    0 & 0\\
    0 & 1
\end{bmatrix}.$

$\chi_A$ for some set $A$ is always denoted as a smooth indicator function adapted to the set $A$. 

Throughout, we use $u_t=\partial_t u:=\frac{\partial}{\partial_t}u$ and $u_x=\partial_x u:=\frac{\partial}{\partial_x} u$.

For non-negative $X$, $Y$, we write $X\lesssim Y$ if $X \leq C$ ,
and we use the notation $X\ll Y$ to indicate that the implicit constant should be regarded as small.
Furthermore, for nonnegative $X$ and arbitrary $Y$ , we use the shorthand notation $Y = \mathcal{O}(X)$ if
$|Y|\leq C X$. 

\noindent {\it Inner products.}
In terms of the $L^2$ inner product of complex-valued functions, we use
\begin{equation}\label{eq:L2inner}
    \langle f, g\rangle = \int_\mathbb{R} f\overline{g}\, dx.
\end{equation} 
Given two pairs of complex-valued vector functions $\vec{f} = (f_1, f_2)$ and $\vec{g} = (g_1, g_2)$, their  inner product is given by
\begin{equation}\label{eq:L2L2inner}
    \langle \vec{f}, \vec{g}\rangle:=  \int_\mathbb{R} \bigl( f_1 \overline{g_1} + f_2 \overline{g_2} \bigr) \, d x.
\end{equation}

\subsection{Acknowledgement}
 We would like to thank Jonas L\"uhrmann and Wilhelm Schlag for valuable comments and feedback.

\subsection{Main results}
In this subsection, we introduce assumptions on potentials, then state the main results in this paper after introducing basic notations from the scattering theory.

\subsubsection{Assumptions on potentials}\label{subsub:assumption}
Fix a natural number $m$. First, for each $\ell\in\{1,2,\,...,\,m\},$ we consider $U_{\ell},\,W_{\ell}:\mathbb{R}\to\mathbb{R}$ be fast decaying functions and, for $\omega_\ell>0,$ the following operators
\begin{align}\label{eq:H0V}
\mathcal{H}_{0,\omega_\ell}:=\left(\begin{array}{cc}
-\partial_{x x}+\omega_\ell & 0 \\
0 & \partial_{x x}-\omega_\ell
\end{array}\right), V_{\ell}(x)=\left(\begin{array}{cc}
U_{\ell} & -W_{\ell} \\
W_{\ell} & -U_{\ell}
\end{array}\right), \mathcal{H}_{\ell}:=\mathcal{H}_{0,\omega_{\ell}}+V_{\ell}, 
\end{align}
such that  $V_{\ell}$  well as all its derivatives are exponentially decay: for some $0<\gamma<1$
\begin{equation}\label{decV}
\left\|V^{(k)}_{\ell}(x)\right\| \leq C_{k,\ell} e^{-\gamma|x|} \quad \forall k \geq 0,    
\end{equation}
and $V_{\ell}(x)=V_{\ell}({-}x)$ for all $x\in\mathbb{R},$ and $\ell\in\{1,2,\,...,\,N\}.$ 
Using
\begin{equation}\label{eq:sigmas}
\sigma_{3}=
    \begin{bmatrix}
     1     & 0\\
     0 & {-}1
    \end{bmatrix},\, \sigma_{2}=
    \begin{bmatrix}
     0     & 1\\
     {-}1 & 0
    \end{bmatrix},\,
\sigma_{1}= \begin{bmatrix}
     0     & 1\\
     1 & 0
    \end{bmatrix},
\end{equation}    
 it is not difficult to verify the following identities

$$
\sigma_3 \mathcal{H}^{*}_{\ell} \sigma_3=\mathcal{H}_{\ell}, \sigma_1 \mathcal{H}_{\ell} \sigma_1=-\mathcal{H}_{\ell}.
$$
\par Next, we consider the following hypotheses for our main results on the asymptotic completeness.
\begin{itemize}
    \item [(H1)] There is no embedded eigenvalue in the essential spectrum of each operator $\mathcal{H}_{\ell}.$
    \item [(H2)] Each operator $\mathcal{H}_{\ell}$ has  $2N_{\ell}$  non-zero simple eigenvalues on the real line with absolute value less than $\omega_\ell$, and $2M_{\ell}$  non-zero simple eigenvalues on the imaginary line for some $N_{\ell},\,M_{\ell}\in\mathbb{N}.$
    \item [(H3)] $\ker \mathcal{H}_{\ell}^{n}=\ker\mathcal{H}_{\ell}^{2}$ for all $\ell,$ $\omega>0$ and every $n\geq 2.$ 
   
\end{itemize}

The three conditions above will be   standing assumptions throughout this paper 
and we will not mention it further.

We also record the notations
\begin{equation}\label{Pd}\tag{Discrete spectrum space}
P_{d,\omega_{\ell}}=
\text{ projection onto the discrete spectrum of $\mathcal{H}_{\ell},$}
\end{equation}
\begin{equation}\label{Pe}\tag{Essential spectrum space}
P_{e,\omega_{\ell}}=
\text{ projection onto the essential spectrum of $\mathcal{H}_{\ell}.$}
\end{equation}

\subsubsection{Scattering space}\label{subsubsec:scattering}
 Next, we consider the scattering space which is going to be useful in the statement of the main results.

We first introduce the indispensable tool in the study of the charge transfer model, Galilei transformations. 
\begin{definition}[Galilei Transformation]
Considering $g:\mathbb{R}\to\mathbb{C}^{2},$ the function $\mathfrak{g}_{\omega_{k},v_{k},y_{k}}(g)(t,x)$ is defined by
\begin{equation}\label{eq:Gali}
\mathfrak{g}_{\omega_{k},v_{k},y_{k},\gamma_{k}}(g)(t,x)=e^{i\sigma_{3}(\frac{v_{k}x}{2}-\frac{tv^{2}}{4}+\gamma_{k}+\omega_{k}t)}g\left(x-v_{k}t-y_{k}\right).
\end{equation}
\end{definition} 
\begin{definition}[Scattering space]\label{def:scatter}
A solution $\overrightarrow{\psi}(t,x)$ of \eqref{p} is in the scattering space if it satisfies for any $\ell\in\{1,2\,...,\,m\}$ and for any function $$h_{\ell}(x)\in \kerrr \mathcal{H}_{\ell}^{2}\bigcup\left(\bigcup_{\lambda_{\ell,k}\in \sigma_d(\mathcal{H}_{\ell}),\lambda_{\ell,k}\neq0} \kerrr[\mathcal{H}_{\ell}-\lambda_{\ell,k}\mathrm{Id}]\right ),$$ the following property
\begin{equation}\label{asyorth}
    \lim_{t\to{+}\infty}\left\langle \overrightarrow{\psi}(t,x),\sigma_{3}\mathfrak{g}_{\omega_{\ell},v_{\ell},y_{\ell},\gamma_{\ell}}(h_{\ell})(t,x)\right\rangle =0.
\end{equation}
Clearly, the scattering space defined above forms a subspace.  One can find the map $P_{c}:\mathbb{R}_{\geq 0}\times L^{2}_{x}(\mathbb{R},\mathbb{C}^{2})\to L^{2}_{x}(\mathbb{R},\mathbb{C}^{2})$ to be the unique projection satisfying
\begin{align*}
P_{c}(t,f(x))=\overrightarrow{\psi}(t) \text{, for a $\overrightarrow{\psi}(t)$ satisfying \eqref{asyorth},}\\
\left\langle f(x)-P_{c}(t,f(x)),\sigma_{3} \vec{\phi}(t)  \right\rangle =0 \text{, for any $\vec{\phi}(t)$ satisfying \eqref{asyorth}.}
\end{align*}

\end{definition}
\begin{remark}
Indeed, as shown by the main results of this paper,  Theorem \ref{princ} and Theorem \ref{tcont}, will imply the existence of a unique projection $P_{c}(t)$ satisfying the definition, see Theorem \ref{princ} and \eqref{eq:Pc}. Furthermore, to simplify the reading of the manuscript, we are going to use $P_{c}(t)f(t)$ to denote $P_{c}(t,f(t,x))$ for any function $f(t,x)\in L^{2}_{x}(\mathbb{R},\mathbb{C}^{2}).$ 


\end{remark}

\subsubsection{Basic notations from  the scattering theory}\label{subsubsec:scatnotation}
Before formulating our main results, we need some notations from the scattering theory. First, we recall the operator $\mathcal{H}_{\omega}$ which is defined on $H^{2}(\mathbb{R},\mathbb{C}^{2})$ by
\begin{align}\label{H}
\mathcal{H}_{\omega}\coloneqq {-}\left[\partial^{2}_{x}-\omega\right]
    \sigma_{3}+  V(x)
,
\end{align}
where
\begin{equation*}
   V(x)=\begin{bmatrix}
    U(x) & W(x)\\
    {-}W(x) & {-}U(x) 
    \end{bmatrix},\, \sigma_{3}=
    \begin{bmatrix}
     1     & 0\\
     0 & {-}1
    \end{bmatrix},
\end{equation*}
and $V(x)$ satisfies the condition \eqref{decV}.
Under the assumptions from \eqref{eq:H0V}, for any real number $k\neq0$, there exist bounded functions $\mathcal{G}_\omega(x,{-}k),\,\mathcal{F}_\omega(x,k)$ such that
\begin{align}\label{BB}
    \mathcal{H}_{\omega}\mathcal{G}_\omega(x,{-}k)=(k^{2}+\omega)\mathcal{G}_\omega(x,{-}k),\, \mathcal{H}_{\omega}\mathcal{F}_\omega(x,k)=(k^{2}+\omega)\mathcal{F}_\omega(x,k),
\end{align}
and $\mathcal{G}_\omega,\,\mathcal{F}_\omega$ have the following asymptotics: for some $\gamma>0$
\begin{align}\label{asy1}
    \mathcal{G}_\omega(x,{-}k)=&\overline{s(k)}\left[e^{ikx}
    \begin{bmatrix}
        1\\
        0
    \end{bmatrix}+O\left(\frac{e^{\gamma x}}{(1+\vert k \vert)}\right)\right] \text{, as $x\to {-}\infty,$}\\ \label{asy2}
    \mathcal{G}_\omega(x,{-}k)=&e^{ikx}
    \begin{bmatrix}
        1\\
        0
    \end{bmatrix}+\overline{r(k)}e^{{-}ikx}
    \begin{bmatrix}
        1\\
        0
    \end{bmatrix}+O\left(\frac{e^{{-}\gamma x}}{(1+\vert k \vert)}\right)
    \text{, as $x\to {+}\infty,$}\\ \label{asy3}
    \mathcal{F}_\omega(x,k)=& s(k)\left[e^{ikx}
    \begin{bmatrix}
        1\\
        0
    \end{bmatrix}+O\left(\frac{e^{{-}\gamma x}}{(1+\vert k \vert)}\right)\right] \text{, as $x\to {+}\infty,$}\\ \label{asy4}
    \mathcal{F}_\omega(x,k)=&e^{ikx}
    \begin{bmatrix}
        1\\
        0
    \end{bmatrix}+r(k)e^{{-}ikx}
    \begin{bmatrix}
        1\\
        0
    \end{bmatrix}+O\left(\frac{e^{\gamma x}}{(1+\vert k \vert)}\right)
    \text{, as $x\to {-}\infty,$}
\end{align}
such that the functions $s,\,r$ are smooth and satisfy for $k\in\mathbb{C}$ with $\vert\I k\vert\leq \delta$ for a constant $\delta>0$ 
\begin{align}\nonumber
s({-}k)=&\overline{s(k)},\,r({-}k)=\overline{r(k)},\, 
\vert s(k)\vert^{2}+\vert r(k) \vert^{2}=1,\,s(k)\overline{r}(k)+r(k)\overline{s}(k)=0,\\
 \nonumber
s(k)=&1+O\left(\frac{1}{1+\vert k\vert}\right),\,
r(k)=O\left(\frac{1}{1+\vert k\vert}\right) \text{, and} \\
\label{asyreftr}
    \frac{d^{m}}{dk^{m}}s(k)=&\frac{d^{m}}{dk^{m}}(1)+O\left(\vert k \vert^{{-}m+1}\right),\, \frac{d^{m}}{dk^{m}}r(k)=O\left(\vert k \vert^{{-}m-1}\right) \text{, when $k\in\mathbb{R},$ as $k\to{+}\infty,\,m\in\mathbb{N}\cup\{0\}$.}
\end{align}
See for example \cite{Busper1} or \cite{KriegerSchlag} for detailed proofs of the claims above. Also, see Lemma \ref{2.1} for the existence of these solutions when threshold resonances appear.
Moreover, using the asymptotic behavior in \eqref{asy1}, \eqref{asy2} and \eqref{asy3}, we obtain that
\begin{equation}\label{forp1}
    \mathcal{F}_\omega(x,k)=\frac{1}{\overline{s(k)}}\mathcal{G}_\omega(x,{-}k)-\frac{\overline{r(k)}}{\overline{s(k)}}\mathcal{G}_\omega(x,k).
\end{equation}
See also \cite{Busper1} and \cite{KriegerSchlag} for more information about the scattering theory of the operators $\mathcal{H}_{\omega}.$
\begin{definition}\label{def00}
For $\vec{u}(k)\in\mathcal{S} \subset L^{2}(\mathbb{R},\mathbb{C}^{2}),$ we define the following linear operators 
\begin{align}\label{DisF}
    \hat{F}_{\omega}\left(\overrightarrow{u}\right)(x)\coloneqq &\frac{1}{\sqrt{2\pi}}\int_{\mathbb{R}}\frac{1}{s(k)}\left[\mathcal{F}_\omega(x,k)\,\,\sigma_{1} \mathcal{F}_\omega(x,k)\right]\overrightarrow{u}(k)\,dk\in L^{2}(\mathbb{R},\mathbb{C}^{2}),\\ \label{DisG}
    \hat{G}_{\omega}\left(\overrightarrow{u}\right)(x)\coloneqq &\frac{1}{\sqrt{2\pi}}\int_{\mathbb{R}}\frac{1}{s({-}k)}\left[\mathcal{G}_\omega(x,{-}k)\,\,\sigma_{1} \mathcal{G}_\omega(x,{-}k)\right]\overrightarrow{u}(k)\,dk \in L^{2}(\mathbb{R},\mathbb{C}^{2}),\\ \nonumber
    F_{0}(\overrightarrow{u})(x)=&\frac{1}{\sqrt{2\pi}}\int_{\mathbb{R}}e^{ikx}\overrightarrow{u}(k)\,dk \in L^{2}(\mathbb{R},\mathbb{C}^{2}).
\end{align}
\end{definition}

\subsubsection{Solutions in the scattering space and discrete space}\label{nn}
First, for $\sigma_{\ell}=(\omega_{\ell},v_{\ell},y_{\ell},\gamma_{\ell}),$ we consider

\begin{align}\label{H2}
    V^{\sigma_{\ell}}_{\ell}(t)\coloneqq & \left(\begin{array}{cc}
U_\ell(x-v_\ell t-y_{\ell}) & -e^{2i \theta_\ell(t, x)} W_\ell(x-v_\ell t-y_{\ell}) \\
e^{-2i \theta_\ell(t, x)} W_\ell(x-v_\ell t-y_{\ell}) & -U_\ell(x-v_\ell t-y_{\ell})
\end{array}\right),
\end{align}

such that
\begin{equation}\label{theta}
    \theta_{\ell}(t,x)=\frac{v_{\ell}x}{2}-\frac{v_{\ell}^{2}t}{4}+\omega_{\ell}t+\gamma_{\ell}.
\end{equation}
Using notations from \eqref{eq:H0V}, \eqref{eq:sigmas}, assuming (H1), (H2) and (H3), and given $v_{1}>v_{2}>...>v_{m},$ and $y_{1}>y_{2}>...>y_{m},
$ we consider the initial value problem associated with the following charge transfer model
\begin{align}\label{ldpe} \tag{IVP}
   i\partial_{t}\overrightarrow{\psi}(t)+\sigma_{3}\partial^{2}_{x}\overrightarrow{\psi}(t)+\sum_{\ell=1}^mV^{\sigma_{\ell}}_{\ell}(t)\overrightarrow{\psi}(t)=0,\\ \nonumber
   \overrightarrow{\psi}(0)\in L^{2}_{x}\left(\mathbb{R},\mathbb{C}^{2}\right).
\end{align}
We first construct 
the following approximate solution 
of \eqref{ldpe}.\footnote{In \cite{perelmanasym}, an approximate solution was constructed for the case $m=2$. In this case, because of special symmetries of two potentials, the construction is much simpler. }
\begin{definition}\label{s0def}
Given $v_{1}>v_{2}>...>v_{m},\,\delta y_{\ell}=y_{\ell-1}-y_{\ell}\gg 1,$ for  any given
\begin{equation*}
\vec{\phi}(k)= 
    \begin{bmatrix}
     \phi_{1}(k)\\
     \phi_{2}(k)
    \end{bmatrix}\in L^{2} (\mathbb{R},\mathbb{C}^{2}),
\end{equation*}
using notations from Definition \ref{def00}, we define the following formula
\begin{align}\label{eq:Sphi}
   \mathcal{S}(\vec{\phi})(t,x):= & \sum_{\ell=1}^{m}e^{i\left(\frac{v_{\ell}x}{2}-\frac{v_{\ell}^{2}t}{4}+\omega_{\ell}t+\gamma_{\ell}\right)\sigma_{3}}\hat{G}_{\omega_{\ell}}\left(
   e^{{-}it(k^{2}+\omega_{\ell})\sigma_{3}}e^{{-}i\gamma_{\ell}\sigma_{3}} \begin{bmatrix}
       e^{iy_{\ell}k}\phi_{1,\ell}\left(k+\frac{v_{\ell}}{2}\right)\\
       e^{iy_{\ell}k}\phi_{2,\ell}\left(k-\frac{v_{\ell}}{2}\right)
    \end{bmatrix}\right)(x-y_{\ell}-v_{\ell}t)\\
    & {-}\frac{1}{\sqrt{2 \pi}}\int_{\mathbb{R}}e^{{-}it k^{2}\sigma_{3}}
    \begin{bmatrix}
       \varphi_{1}(k)\\
       \varphi_{2}(k)
    \end{bmatrix} e^{ikx}\,dk,\nonumber
\end{align}
where the sequence $\{\overrightarrow{\phi_{\ell}}\}_{\ell=1}^m$ and $\vec{\varphi}$ are constructed recursively from $\vec{\phi}(k)$ via the following conditions
\begin{enumerate}
    \item [a)] $\begin{bmatrix}
        \phi_{1,1}(k)\\
        \phi_{2,1}(k)
    \end{bmatrix}=\vec{\phi}(k); $
    \item [b)] for each $\ell\geq 1,$
    \begin{equation*}
      e^{{-}i\gamma_{\ell+1}\sigma_{3}} \begin{bmatrix}
           \phi_{1,\ell+1}(k)\\
           \phi_{2,\ell+1}(k)
       \end{bmatrix}
        =e^{{-}i\gamma_{\ell}\sigma_{3}}
        \begin{bmatrix}
        \frac{\phi_{1,\ell}\left(k\right)-r_{\omega_{\ell}}\left(k-\frac{v_{\ell}}{2}\right)e^{{-}i2y_{\ell}(k-\frac{v_{\ell+1}}{2})+iy_{\ell}(v_{\ell}-v_{1+\ell})}\phi_{1,\ell}\left({-}k+v_{\ell}\right)}{s_{\omega_{\ell}}\left(k-\frac{v_{\ell}}{2}\right)}\\
        \frac{\phi_{2,\ell}\left(k\right)-r_{\omega_{\ell}}\left(k+\frac{v_{\ell}}{2}\right)e^{{-}2iy_{\ell}(k+\frac{v_{\ell+1}}{2})+iy_{\ell}(v_{\ell+1}-v_{\ell})}\phi_{2,\ell}\left({-}k-v_{\ell}\right)}{s_{\omega_{\ell}}\left(k+\frac{v_{\ell}}{2}\right)}
    \end{bmatrix};
    \end{equation*}
    \item [c)] and
    \begin{equation*}
        \begin{bmatrix}
          \varphi_{1}(k)\\
          \varphi_{2}(k)
        \end{bmatrix}=\sum_{\ell=1}^{m-1}\begin{bmatrix}
           \phi_{1,\ell}(k)\\
           \phi_{2,\ell}(k)
       \end{bmatrix}.
    \end{equation*}

\end{enumerate} 
 For the convenience of notations, we use $S(t)$ to denote
\begin{equation}\label{eq:St}
    \mathcal{S}(t)\vec{\phi}:= \mathcal{S}(\vec{\phi})(t,x).
\end{equation}

\end{definition}
 
\begin{remark}\label{transition}
Note that using the asymptotics \eqref{asy1}, \eqref{asy2}, \eqref{asy3}, \eqref{asy4}, it is not difficult to verify for any $\ell>1$ that by condition b) above,
 \begin{align*}
   e^{\frac{v_{\ell}x}{2}\sigma_{3}}\hat{G}_{\omega_{\ell}}\left(\begin{bmatrix}
       e^{iy_{\ell}k}\phi_{1,\ell}(k+\frac{v_{\ell}}{2})\\
       e^{iy_{\ell}k}\phi_{2,\ell}(k-\frac{v_{\ell}}{2})
   \end{bmatrix}\right)(x-y_{\ell})=e^{\frac{v_{\ell}x}{2}\sigma_{3}}\hat{F}_{\omega_{\ell}}\left(\begin{bmatrix}
       e^{iy_{\ell}k}\phi_{1,\ell-1}(k+\frac{v_{\ell}}{2})\\
       e^{iy_{\ell}k}\phi_{2,\ell-1}(k-\frac{v_{\ell}}{2})
   \end{bmatrix}\right)(x-y_{\ell}).
 \end{align*}

 \end{remark}
\begin{remark}\label{rem:moti}
The motivation to obtain the formula \eqref{eq:Sphi} follows from the construction of an approximate solution of \eqref{ldpe}. Naturally, a good candidate for an approximate solution is
\begin{align}\label{appfor}
\overrightarrow{\psi}_{app}(t)&=e^{it\sigma_{3}\partial^{2}_{x}}\left(\begin{bmatrix}
    h_{1}(x)\\
    h_{2}(x)
\end{bmatrix}\right)\\
&+\sum_{\ell=1}^{m} e^{i\left(\frac{v_{\ell}x}{2}-\frac{v_{\ell}^{2}t}{4}+\omega_{\ell}t+\gamma_{\ell}\right)\sigma_{3}}\hat{G}_{\omega_{\ell}}\left(e^{{-}it(k^{2}+\omega_{\ell})\sigma_{3}}\begin{bmatrix}
        f_{1,\ell}(k)\\
        f_{2,\ell}(k)
    \end{bmatrix}\right)(x-y_{\ell}-v_{\ell}t),
\end{align}
for some choices of 
\begin{equation}\label{eq:candi}
    \begin{bmatrix}
        f_{1,\ell}(k)\\
        f_{2,\ell}(k)
    \end{bmatrix},\,\begin{bmatrix}
        h_{1}(x)\\
        h_{2}(x)
    \end{bmatrix} 
\end{equation}
since each term of \eqref{appfor} is a solution of one of the following linear partial differential equations
\begin{equation*}
    i\partial_{t}\overrightarrow{\psi}+\sigma_{3}\partial^{2}_{x}\overrightarrow{\psi}=0\, \text { or }\, i\partial_{t}\overrightarrow{\psi}+\sigma_{3}\partial^{2}_{x}\overrightarrow{\psi}+V^{\sigma_{\ell}}_{\ell}(t)\overrightarrow{\psi}=0.
\end{equation*}
In particular, we have the following identity{\footnotesize
\begin{multline*}
  i\partial_{t}\overrightarrow{\psi}_{app}(t)+\sigma_{3}\partial^{2}_{x}\overrightarrow{\psi}_{app}(t)+\sum_{\ell}V^{\sigma_{\ell}}_{\ell}(t)\overrightarrow{\psi}_{app}(t)\\
=\sum_{\ell}V^{\sigma_{\ell}}_{\ell}(t)\left[e^{it\sigma_{3}\partial^{2}_{x}}\left(\begin{bmatrix}
    h_{1}(x)\\
    h_{2}(x)
\end{bmatrix}\right)+\sum_{n\neq \ell}e^{i\left(\frac{v_{n}x}{2}-\frac{v_{n}^{2}t}{4}+\omega_{n}t+\gamma_{n}\right)\sigma_{3}}\hat{G}_{\omega_{n}}\left(e^{{-}it(k^{2}+\omega_{n})\sigma_{3}}\begin{bmatrix}
        f_{1,n}(k)\\
        f_{2,n}(k)
    \end{bmatrix}\right)(x-y_{n}-v_{n}t)\right].  
\end{multline*}}
Notice that using the asymptotics \eqref{asy1}, \eqref{asy2}, \eqref{asy3}, \eqref{asy4} the right-hand side of the equation will consist of a system of linear combinations of planar waves and collections of terms which decays exponentially in terms of the distance among different potentials.  To ensure the exponential decay in time on the right-hand side, we have to cancel out those planar waves. 
%
%
This cancellation condition can be ensured 
if the functions in \eqref{eq:candi}
satisfy the condition $b)$ and $c)$ of Definition \eqref{s0def} respectively. Also see Lemma \ref{localS0}. Under these conditions,  it follows that{\footnotesize
\begin{equation*}
   \norm{i\partial_{t}\overrightarrow{\psi}_{app}(t)+\sigma_{3}\partial^{2}_{x}\overrightarrow{\psi}_{app}(t)+\sum_{\ell}V^{\sigma_{\ell}}_{\ell}(t)\overrightarrow{\psi}_{app}(t)}_{H^{1}_{x}(\mathbb{R})}\leq C e^{{-}\beta (\min_{\ell}y_{\ell}-y_{\ell+1}+(v_{\ell}-v_{\ell+1})t)}\max_{n} \norm{\begin{bmatrix}
       f_{1,n}(k)\\
       f_{2,n}(k)
   \end{bmatrix}}_{L^{2}_{k}(\mathbb{R})},
\end{equation*}}
which indeed confirms that $\overrightarrow{\psi}_{app}$ is a good approximate solution.
\end{remark}
The notation $\mathcal{S}(\vec{\phi})(t,x)$ will be a crucial object in this paper. As indicated in Remark \ref{rem:moti}, this will play the role of an approximate solution to \eqref{ldpe}. In particular, the function $\mathcal{S}(\vec{\phi})(t,x)$ can be imagined as the `free' evolution for our problem. We will then construct a solution to \eqref{ldpe} that will asymptotically behave like $\mathcal{S}(\vec{\phi})(t,x)$.  In fact, we can find solutions with more general prescribed asymptotic behavior. 

\par In Section \ref{asyinfinity}, we will prove the following two statements concerning the asymptotic behavior of solutions of \eqref{ldpe}. 


\begin{theorem}[Solutions in the scattering space]\label{tcont}
Let $\overrightarrow{\phi_{0}}\in L^{2}_{k}(\mathbb{R},\mathbb{C})$ be any element in the domain of $\mathcal{S}(0).$ There exist positive constants $C>1,\,L>1\,K,\beta>0$ such that if  $$\min_{\ell}y_{\ell}-y_{\ell+1}>L\quad\text{and}\quad\min_{\ell}v_{\ell}-v_{\ell+1}>C,$$ then there is a unique solution $\vec{\psi}(t)$ of \eqref{ldpe}
\begin{equation}\label{eq:defcalT}
    \vec{\psi}(t,x)=:\mathcal{T}(\vec{\phi}_0)(t,x)\quad\text{or}\quad  \vec{\psi}(t)=:\mathcal{T}(t)\left(\overrightarrow{\phi_{0}}\right).
\end{equation}
satisfying for all $t\geq 0$
\begin{equation}\label{h1linft}
 \norm{\vec{\psi}(t)-\mathcal{S}(\overrightarrow{\phi_{0}})(t,x)}_{H^{1}_{x}(\mathbb{R})}\leq K e^{{-}\beta [(\min_{\ell}y_{\ell}-y_{\ell+1})+(\min_{\ell}v_{\ell}-v_{\ell+1}) t]} \norm{\mathcal{S}(\overrightarrow{\phi_{0}})(0)}_{L^{2}_{x}(\mathbb{R})} 
\end{equation}
where the parameters $K,\,T,\,\beta,\,C$ depend only on the potentials $V_{\omega_{1}},\,...,\,V_{\omega_{m}}.$ In particular, 
$\vec{\psi}(t)$ is in the scattering space in the sense of Definition \ref{def:scatter} and satisfies \eqref{asyorth}.
 
Furthermore, if the inequality $\max_{\ell}\norm{k^{2}(\phi_{1,\ell},\phi_{2,\ell})}_{L^{2}_{k}(\mathbb{R})}<{+}\infty$ is true, then there exists a constant $K_{2}>1$ satisfying for all $t\geq 0$ 
\begin{multline}\label{tcontest}
   \max_{j\in\{1,2\},\ell} \norm{(x-y_{\ell}-v_{\ell}t)^{j}\left( \vec{\psi}(t,x)-\mathcal{S}(\overrightarrow{\phi_{0}})(t,x)\right)}_{H^{1}_{x}(\mathbb{R})}\\
   \leq  K_{2} \max_{\ell}(y_{\ell}-y_{\ell+1}+(v_{\ell}-v_{\ell+1})t)^{2}e^{{-}\beta \min_{\ell}[y_{\ell}-y_{\ell+1}+(v_{\ell}-v_{\ell+1}) t]} \max_{j\in\{1,2\},\ell}\norm{\langle k \rangle^{j}(\phi_{1,\ell}(k),\phi_{2,\ell}(k))}_{L^{2}_{k}(\mathbb{R})}.
    \end{multline} 
\end{theorem}
The proof of Theorem \ref{tcont} is written in Subsection \ref{solscont}.

\begin{remark}
As can be seen from the  results above, \eqref{eq:Sphi} serves as a good approximation of the linear flow of \eqref{ldpe}. The approximation actually holds in more general settings. 
Indeed, for any set of speeds $\{v_{1},\,...,\,v_{m}\}$ satisfying $\min v_{\ell}-v_{\ell+1}>0,$ there exists a  solution of \eqref{ldpe} $ \vec{\psi}(t,x)=\mathcal{T}(t,x)(\overrightarrow{\phi_{0}})$ satisfying  \text{for all $t\geq 0,$} 
\begin{equation*}
 \norm{ \vec{\psi}(t,x)-\mathcal{S}(\overrightarrow{\phi_{0}})(t,x)}_{L^{2}_{x}(\mathbb{R})}\leq K e^{{-}\beta[(\min_{\ell}y_{\ell}-y_{\ell+1})+ (\min_{\ell}v_{\ell}-v_{\ell+1})t]} \max_{\ell}\norm{(\phi_{1,\ell},\phi_{2,\ell})}_{L^{2}_{x}(\mathbb{R})}.   
\end{equation*}
See Subsection \ref{solscont} for the proof of the estimate above.
\end{remark}

\begin{remark}
    We also comment that since \eqref{eq:Sphi} is explicit and it is a good approximation of the linear flow of the charge transfer model \eqref{ldpe} from the estimates above , it is a natural replacement of the exact linear flow to \eqref{ldpe} so that we can set up a notation of the profile   which plays a crucial role in the resonance analysis and long-range scattering, see for example \cite{GCnls,CP2,GPR,KaPu,collotger,nls3soliton}.
\end{remark}
\begin{theorem}[Solutions in the discrete space]\label{tdis}
Given any $1\leq k\leq M$, let $\mathfrak{v}_{\omega_{k},\lambda}\in L^{2}_{x}\left(\mathbb{R},\mathbb{C}^{2}\right)$ be any element in $\kerrr [\mathcal{H}_{\omega
_k}-\lambda \mathrm{Id}]$
for some eigenvalue 
$\lambda $ of $\mathcal{H}_{\omega_k}$. There exist constants $K>0,\,L,\,C\gg 1,\,\beta>0$ depending only on the potentials $V_{\omega_{1}},\,...,\,V_{\omega_{M}}$ such that if $$\min_{\ell}y_{\ell}-y_{\ell+1}>L\quad\text{and}\quad\min_{\ell}v_{\ell}-v_{\ell+1}>C,$$ then using the matrix Galilei transformation, there is a unique solution
\begin{align}\label{eq:Gv}
\mathfrak{G}_{\omega_{k},v_{k},y_{k},\gamma_k}(\mathfrak{v}_{\omega_{k},\lambda})(t,x)&\coloneq e^{i\sigma_{3}(\frac{v_{k}x}{2}+\gamma_{k})}e^{{-}i((\lambda-\omega_{k}) t+\frac{\sigma_{3}v_{k}^{2}t}{4})}\mathfrak{v}_{\omega_{k},\lambda}(x-y_{k}-v_{k}t)+r(t,x)\\
&=:\mathfrak{g}_{\omega_{k},v_{k},y_{k},\gamma_k}(\mathfrak{v}_{\omega_{k},\lambda})(t,x)+r(t,x)
\end{align} of \eqref{ldpe} satisfying
\begin{equation}
 \norm{r(t,x)}_{L^{2}_{x}(\mathbb{R})}\leq K e^{-\beta(\min_{\ell}y_{\ell}-y_{\ell+1}){-}\beta(\min_{\ell}v_{\ell}-v_{\ell+1}) t} \text{, for all $t\geq 0.$}   
\end{equation}
Moreover, if $\mathfrak{z}_{\omega_{k}}\in \kerrr \mathcal{H}_{\omega_{k}}^{2}$ and $\mathcal{H}_{\omega_{k}}(\mathfrak{z}_{\omega_{k}})=i\mathfrak{v}_{\omega_{k},0},$ then there exists a unique solution
\begin{equation}\label{eq:Gz}
\mathfrak{G}_{\omega_{k},v_{k},y_{k},\gamma_k}(\mathfrak{z}_{\omega_{k}})(t,x)\coloneq\mathfrak{g}_{\omega_{k},v_{k},y_{k},\gamma_k}(\mathfrak{z}_{\omega_{k}})(t,x)+t\mathfrak{g}_{\omega_{k},v_{k},y_{k},\gamma_k}(\mathfrak{v}_{\omega_{k},0})(t,x)(t,x)+r(t,x)
\end{equation}
of \eqref{ldpe} satisfying
$
 \norm{r(t,x)}_{L^{2}_{x}(\mathbb{R})}\leq K e^{{-}\beta (\min_{\ell}y_{\ell}-y_{\ell+1})-\beta(\min_{\ell}v_{\ell}-v_{\ell+1}) t} \text{, for all $t\geq 0.$}   
$
\end{theorem} 
The proof of Theorem \ref{tdis} is written in Subsection \ref{soldisc}.
\begin{remark}
   We comment that, in all the results above, one can have any $C,\,L>1$ satisfying $C>\alpha_{\min}>>1$ and $L>L_{\min}>>1$. The $L$ and $C$ do not necessarily need to be dependent on each other if $\alpha_{\min}$ and $L_{\min}$ are large enough.
\end{remark}
In particular, combining the results of Theorem \ref{tcont} and Theorem \ref{tdis}, we obtain a complete theory for the existence of wave operators in our setting.
\subsubsection{Asymptotic completeness and dispersive estimates}\label{Theoprincip}
The first main result of this article is the following theorem.
\begin{theorem}[Asymptotic completeness]\label{princ}
There exist $C>0,\,L>0$ depending on the potentials $V_{\ell}$ such that if$$\min_{\ell}y_{\ell}-y_{\ell+1}>L\quad\text{and}\quad\min_{\ell}v_{\ell}-v_{\ell+1}>C,$$ then using notations from Theorem \ref{tcont} and Theorem \ref{tdis}, every solution of \eqref{ldpe} has a unique representation given by
\begin{equation}\label{princ11}
   \vec{\psi}(t,x)=\mathcal{T}
\left(\overrightarrow{\phi_{0}}\right)(t,x)+\sum_{\ell=1}^{m}\sum_{\lambda_{\ell,k}\in \sigma_d(\mathcal{H}_{\omega_\ell})}\mathfrak{G}_{\omega_{\ell},v_{\ell},y_{\ell}}(\mathfrak{v}_{\omega_{\ell},\lambda_{\ell,k}})(t,x)+\sum_{\ell=1}^{m}\mathfrak{G}_{\omega_{\ell},v_{\ell},y_{\ell}}(\mathfrak{z}_{\omega_{\ell}})(t,x)  
\end{equation}
for some $\phi_{0}$ in the domain of $\mathcal{S}(0)$, $\mathfrak{v}_{\omega_{k},\lambda_{\ell,k}}\in\kerrr [\mathcal{H}_{\omega
_\ell}-\lambda_{\ell,k} \mathrm{Id}]$, and $\mathfrak{z}_{\omega_{\ell}}\in \kerrr \mathcal{H}_{\omega_{\ell}}^{2}$.

Furthermore, if $\vec{\psi}(0,x)\in H^{2}_{x}(\mathbb{R},\mathbb{C})$ and $\vec{\psi}(0,x)x\in L^{2}_{x}(\mathbb{R}),$ then $ \vec{\psi}(t,x)\in H^{2}_{x}(\mathbb{R},\mathbb{C})$ and
\begin{align}
    \max_{\ell}\norm{\begin{bmatrix}
       \langle k\rangle^{n} \phi_{1,\ell}\left(k\right)\\
        \langle k\rangle^{n}\phi_{2,\ell}\left(k\right)    \end{bmatrix}}_{L^{2}_{k}(\mathbb{R})}\leq &C_{v}\norm{\mathcal{T}(\overrightarrow{\phi_{0}})(0,x)}_{H^{n}_{x}(\mathbb{R})}\leq K_{v}\norm{\vec{\psi}(0,x)}_{H^{n}_{x}(\mathbb{R})} \text{, for $n\in\{0,1,2\},$} \label{phi2kkdecay}\\ \label{F0l1}
        \max_{\ell}\norm{\begin{bmatrix}
            F_{0}\left(\phi_{1,\ell}\right)\\
             F_{0}\left(\phi_{2,\ell}\right)
        \end{bmatrix}}_{L^{1}_{x}(\mathbb{R})}\leq & C_{v}\left[\max_{\ell}\norm{[1+\vert x-y_{\ell}\vert]\chi_{\left\{\frac{y_{\ell+1}+y_{\ell}}{2},\frac{y_{\ell-1}+y_{\ell}}{2}\right\}}(x)\mathcal{T}(\overrightarrow{\phi_{0}})(0,x)}_{L^{2}_{x}(\mathbb{R})}\right]\\&{+}C_{v}(y_{1}-y_{m})e^{{-}\beta \min_{\ell} (y_{\ell}-y_{\ell+1})}\norm{\mathcal{T}(\overrightarrow{\phi_{0}})(0,x)}_{H^{1}_{x}(\mathbb{R})},\\
\label{F0l2}
        \max_{\ell}\norm{\begin{bmatrix}
            F_{0}\left(\partial_{k}\phi_{1,\ell}\right)\\
             F_{0}\left(\partial_{k}\phi_{2,\ell}\right)
        \end{bmatrix}}_{L^{1}_{x}(\mathbb{R})}\leq & C_{v}\max_{\ell}(y_{\ell}-y_{\ell+1})\norm{\left[1+\vert x-y_{\ell}\vert\right]\chi_{\left\{\frac{y_{\ell+1}+y_{\ell}}{2},\frac{y_{\ell-1}+y_{\ell}}{2}\right\}}(x)\mathcal{T}(\overrightarrow{\phi_{0}})(0,x)}_{L^{2}_{x}(\mathbb{R})}\\ \nonumber
&{+}C_{v}\norm{\vert x-y_{\ell}\vert^{2}\chi_{\left\{\frac{y_{\ell+1}+y_{\ell}}{2},\frac{y_{\ell-1}+y_{\ell}}{2}\right\}}(x)\mathcal{T}(\overrightarrow{\phi_{0}})(0,x)}_{L^{2}_{x}(\mathbb{R})}\\ \nonumber
&{+}C_{v}(y_{1}-y_{m})^{2}e^{{-}\beta \min_{\ell} (y_{\ell}-y_{\ell+1})}\norm{\mathcal{T}(\overrightarrow{\phi_{0}})(0,x)}_{H^{2}_{x}(\mathbb{R})},
\end{align} and the constant $C_{v}>0$ depends only on the set of speeds $v_{1},\,...,\,v_{m}.$ In particular, for any $t\geq 0,$ using the notation from Definition \ref{def:scatter}, the following identity holds
\begin{equation}\label{eq:Pc}
    P_{c}(t) \vec{\psi}(t,x)=\mathcal{T}(t)(\overrightarrow{\phi_{0}}).
\end{equation}
\end{theorem}
\begin{remark}
In the statement of Theorem \ref{princ}, different from the results on asymptotic completeness in \cite{perelmanasym} by Perelman, we can allow the points ${-}\omega_{\ell}$ and $\omega_{\ell}$ to be resonances of the operator $\mathcal{H}_{\ell}$ for each $\ell.$ 
\end{remark}
\begin{remark}
    Estimates \eqref{F0l1} and \eqref{F0l2} are consequences of Theorem \ref{TT} of Section $4.$ The proof of estimate \ref{phi2kkdecay} for $n=2$ follows from Corollary \ref{coo} and Theorem \ref{tcont}. The proof of estimate \ref{phi2kkdecay} for $n=1$ follows from interpolation; see Lemma \ref{h1coercc} in Section $6.$
\end{remark}
\begin{remark}\label{eq:decomposition}
    In particular, when $t=0,$ using notations from \eqref{eq:defcalT}, we have
$$
    L^{2}(\mathbb{R},\mathbb{C}^{2})=\Raa \mathcal{T}(0)\oplus \bigoplus_{\ell=1}^{m} \Raa\mathfrak{G}_{\omega_{\ell},v_{\ell},y_{\ell},\gamma_\ell}(0),  
$$
and for all $t\geq 0$
\begin{equation*}
    \Raa \mathcal{T}(t)=\left(\sigma_{3} \bigoplus_{\ell=1}^{m} \Raa\mathfrak{G}_{\omega_{\ell},v_{\ell},y_{\ell},\gamma_\ell}(t)\right)^{\perp}.
\end{equation*}
\end{remark}

As a consequence of Theorem \ref{princ} above,  we can extend most of the decay estimates of the solutions of the following linear Schr\"odinger equation with stationary potential
\begin{equation*}
    i\partial_{t}\overrightarrow{\psi}(t)-\mathcal{H}_{\omega}\overrightarrow{\psi}(t)=0
\end{equation*}
to the solutions of the following linear Schr\"odinger equation with multi-potentials. 
\begin{definition}
The evolution operator $\mathcal{U}(t,\tau):L^{2}_{x}(\mathbb{R},\mathbb{C}^{2})\to L^{2}_{x}(\mathbb{R},\mathbb{C}^{2})$ is defined for any $\overrightarrow{{\psi}_{0}}\in L^{2}_{x}(\mathbb{R},\mathbb{C}^{2})$ as $$\vec{{\psi}}(t)=:\mathcal{U}(t,\tau)\overrightarrow{{\psi}_{0}}$$being the unique solution of the following linear partial differential equation
\begin{equation*}
    \begin{cases}
       i \partial_t \vec{\psi}+\sigma_{3}\partial^{2}_{x}\vec{\psi}+\sum_{\ell=1}^mV^{\sigma_{\ell}}_{\ell}(t) \vec{\psi}=0 \text{, for all $t\geq \tau,$}\\
\psi(\tau,x)=\overrightarrow{{\psi}_{0}}(x)
    \end{cases}
\end{equation*}
where $V^{\sigma_{\ell}}_{\ell}(t)$ is defined as \eqref{H2}.
\end{definition}
The next main theorem is the following.  
\begin{theorem}[$L^2$ estimate and decay estimates]\label{Decesti}
If $\min y_{\ell}-y_{\ell+1}>L$ and $\min v_{\ell}-v_{\ell+1}>C$ with $L$ and $C$ large enough,  then the projection $P_{c}(t)$ given by \eqref{eq:Pc} commutes with the flow, in other words, satisfies the following identity for any $\overrightarrow{{\psi}_{0}}\in L^{2}_{x}(\mathbb{R},\mathbb{C}^{2})$
\begin{equation*}
    P_{c}(t)\mathcal{U}(t,\tau)\overrightarrow{{\psi}_{0}}= \mathcal{U}(t,\tau)P_{c}(\tau)\overrightarrow{\psi_{0}}.
\end{equation*}
One can find  $K>1,\,\beta>0$ such that
the following estimates are true
 {\footnotesize \begin{align}\label{l22}\tag{$L^{2}$ Estimate}
  \norm{\mathcal{U}(t,\tau)P_{c}(\tau)\overrightarrow{{\psi}_{0}}}_{L^{2}_{x}(\mathbb{R})}\leq & K \norm{P_{c}(\tau)\overrightarrow{{\psi}_{0}}}_{L^{2}_{x}(\mathbb{R})},\\ \nonumber
   \norm{\mathcal{U}(t,\tau)P_{c}(\tau)\overrightarrow{{\psi}_{0}}}_{L^{\infty}_{x}(\mathbb{R})}\leq & K \frac{\max_{\ell} }{(1+t-\tau)^{\frac{1}{2}}}\norm{\vert x-y_{\ell}-v_{\ell}\tau\vert\chi_{\left[\frac{y_{\ell+1}+y_{\ell}+(v_{\ell+1}+v_{\ell})\tau}{2},\frac{y_{\ell-1}+y_{\ell}+(v_{\ell-1}+v_{\ell})\tau}{2}\right]}(x)P_{c}(\tau)\overrightarrow{\psi_{0}}\left(x\right)}_{L^{2}_{x}(\mathbb{R})}\\ \label{decaystrit}\tag{$L^{\infty}$ Decay}
   &{+} \frac{K}{(1+t-\tau)^{\frac{1}{2}}}\norm{P_{c}(\tau)\overrightarrow{\psi_{0}}}_{H^{1}_{x}(\mathbb{R})}.  
   \end{align}}

  Furthermore, 
  if for all $\ell\in\{1,2,\,...,\,m\}$ the following hypothesis holds
 \begin{itemize}
     \item [(H4)] $\omega_{\ell}$ and ${-}\omega_{\ell}$ are not  resonances  nor eigenvalues of $\mathcal{H}_{\omega_{\ell}}.$\footnote{Recall that one says that $\pm \omega_\ell$ is a  resonance of $\mathcal{H}_{\omega_{\ell}}$ provided that $\mathcal{H}_{\omega_{\ell}}f=\pm \omega_\ell f$ has a solution $f\in L^\infty\, \text{but} \,f\notin L^2$.}
 \end{itemize}
Then 
the following weighted estimate is true
{\footnotesize
   \begin{align}\nonumber
\norm{\frac{\mathcal{U}(t,\tau)P_{c}(\tau)\overrightarrow{\psi_{0}}}{(1+\vert x-y_{\ell}-v_{\ell}t\vert)}}_{L^{\infty}_{x}(\mathbb{R})}\leq &   \frac{K(y_{1}-y_{m}+\tau)}{(1+(t-\tau))^{\frac{1}{2}}(y_{1}-y_{m}+t)}\Bigg[\norm{P_{c}(\tau)\overrightarrow{\psi_{0}}}_{H^{1}_{x}(\mathbb{R})}{+}\max_{\ell}\norm{P_{c}(\tau)\overrightarrow{\psi_{0}}\left(x\right)}_{L^{1}_{x}(\mathbb{R})}\Bigg]\\ \nonumber
&{+}\frac{K\max_{\ell}\norm{\chi_{\left[\frac{y_{\ell+1}+y_{\ell}+(v_{\ell+1}+v_{\ell})\tau}{2},\frac{y_{\ell-1}+y_{\ell}+(v_{\ell-1}+v_{\ell})\tau}{2}\right]}(x)\vert x-y_{\ell}-v_{\ell}\tau\vert P_{c}(\tau)\overrightarrow{\psi_{0}}(x)}_{L^{1}_{x}(\mathbb{R})}}{(1+(t-\tau))^{\frac{3}{2}}} \\ \label{weight}\tag{Weighted decay}
&{+}\frac{K(y_{1}-y_{m}+(v_{1}-v_{m})\tau)e^{{-\beta \min_{\ell}(y_{\ell}-y_{\ell+1})}+(v_{\ell}-v_{\ell+1}) \tau}}{(1+(t-\tau))^{\frac{3}{2}}}\norm{p_{c}(\tau)\overrightarrow{\psi_{0}}}_{H^{2}}. 
 \end{align}}

\end{theorem}
\begin{remark}
   Note that if we do not need the denominator on the right-hand side of \eqref{decaystrit} to be non-singular when $t=\tau$, then the $H^1$ norm will not be required. For more details, see Theorem \ref{Decesti1} from Section \ref{Dispsection}. 
\end{remark}

The following proposition is a direct consequence of Theorem \ref{Decesti}.
\begin{corollary}
Let $F(t,x)\in L^{1}\left((0,{+}\infty),L^{2}_{x}(\mathbb{R})\right)$ and $\overrightarrow{u}(t,x)$ be a   solution of the following Cauchy Problem
\begin{equation*}
    \begin{cases}
        i\partial_{t}\overrightarrow{\psi}(t,x) -\sigma_{3}\partial^{2}_{x}\overrightarrow{\psi}(t,x)+\sum_{\ell=1}^mV^{\sigma_{\ell}}_{\ell}(t)\overrightarrow{\psi}(t,x)=F(t,x),\\
        \overrightarrow{\psi}(0,x)=\overrightarrow{\psi_{0}}\in L^{2}_{x}(\mathbb{R},\mathbb{C}^{2}).
    \end{cases}
\end{equation*}
There exists a constant $C>1,$ not depending on $\overrightarrow{\psi_{0}}$ and $F,$ such that if $P_{c}(t)\overrightarrow{u}(t,x)=\overrightarrow{u}(t,x)$ for all $t\in (0,{+}\infty),$  then 
\begin{equation*}
    \norm{\overrightarrow{\psi}(t)}_{L^{2}_{x}(\mathbb{R})}\leq C\left[\norm{\overrightarrow{\psi_{0}}}_{L^{2}_{x}(\mathbb{R})}+\int_{0}^{t}\norm{P_{c}(\tau)F(\tau)}_{L^{2}_{x}(\mathbb{R})}\right],
\end{equation*}
and in particular,
\begin{equation*}
\norm{\overrightarrow{\psi}}_{L^{\infty}\left((0,{+}\infty),L^{2}_{x}(\mathbb{R})\right)}\leq C\left[\norm{\overrightarrow{\psi_{0}}}_{L^{2}_{x}(\mathbb{R})}+\norm{F}_{L^{1}\left((0,{+}\infty),L^{2}_{x}(\mathbb{R})\right)}\right].
\end{equation*}

\end{corollary}
\begin{theorem}\label{interpolation est.}
If the hypotheses $(H1),\,(H2),\,(H3),$ and $(H4)$ hold, then, for any $p\in (1,2),\,$ any $\alpha\in (0,1),$ and $p^{*}=\frac{p}{p-1},$ the following inequality holds for all $t\geq s.$
\begin{multline}\label{interpweder}
\max_{\ell}\norm{\frac{\partial_{x}P_{c}(t)\mathcal{U}(t,s)(\vec{f})}{\langle x-v_{\ell}t-y_{\ell} \rangle^{1+\frac{p^{*}-2}{2p^{*}}+\alpha}}}_{L^{2}_{x}}\\
\begin{aligned}
\leq & \frac{C_{p,\alpha}\max_{\ell}\norm{(1+\vert x-y_{\ell}-v_{\ell}s\vert)\chi_{\ell}(s,x)\langle \partial_{x}\rangle \vec{f}(x)}_{L^{1}_{x}(\mathbb{R})}^{\frac{2-p}{p}}\norm{f}_{H^{1}}^{\frac{2(p-1)}{p}}}{(t-s)^{\frac{3}{2}(\frac{1}{p}-\frac{1}{p^{*}})}} \\
 &{+}C_{p,\alpha}\frac{(s+y_{1}-y_{m})}{(t-s)^{\frac{3}{2}(\frac{1}{p}-\frac{1}{p^{*}})}}\norm{\vec{f}}_{W^{1,1}_{x}(\mathbb{R})}^{\frac{2-p}{p}}\norm{f}_{H^{1}}^{\frac{2(p-1)}{p}}\\
 &{+}C_{p,\alpha}\frac{(y_{1}-y_{m}+s)^{2}e^{{-}\beta\min_{\ell}((v_{\ell}-v_{\ell+1})s+y_{\ell}-y_{\ell+1})}}{(t-s)^{\frac{3}{2}(\frac{1}{p}-\frac{1}{p^{*}})}}\norm{\vec{f}}_{H^{2}_{x}(\mathbb{R})},
\end{aligned}
\end{multline}
where
\begin{equation*}
    \chi_{\ell}(s,x)=\chi_{\left[\frac{y_{\ell}+y_{\ell+1}+(v_{\ell}+v_{\ell+1})s}{2},\frac{y_{\ell}+y_{\ell-1}+(v_{\ell}+v_{\ell-1})s}{2}\right]}(x).
\end{equation*}
\end{theorem}

\subsection{Sketch of key ideas}\label{sub:sketch}
In this subsection, we explain the main ideas behind the proof of the main result, Theorem \ref{princ}. After establishing Theorem \ref{princ}, then the dispersive decay will follow from Propositions $7.1$ and $8.1$ of \cite{KriegerSchlag} after a careful analysis of the mapping properties of $\mathcal{S}(t)$, for more details, see Section \ref{Dispsection}. 

\subsubsection{Asymptotic completeness for one single potential}
\par First, we use the asymptotic completeness of the operator $\mathcal{H}_{\omega_{\ell}}$ which is based on results from  \cite{KriegerSchlag}, \cite{Busper1} and Section \ref{sec:scatteringone}.  From the asymptotic completeness of $\mathcal{H}_{\omega_{\ell}}$,  any element $\vec{f}\in L^{2}_{x}(\mathbb{R},\mathbb{C}^{2})$ has a unique representation of the form   \begin{equation*}
    \vec{f}(x)=\hat{F}_{\omega_{\ell}}\left(\vec{\phi_{\ell}}(k)\right)(x)+\vec{z}_{d,\ell}(x),
\end{equation*}
such that $\vec{z}_{d,\ell}\in \Raa P_{d,\omega_{\ell}}$ and $\vec{\phi_{\ell}}\in L^{2}_{k}(\mathbb{R})$ belongs to the domain of $\hat{F}_{\omega_{\ell}}.$

Applying the decomposition above to our setting, near the center of $V_\ell$ (after a Galilei transform), in particular, for any $\vec{f}(x)$ and $\vec{g}(x)\in L^{2}_{x}(\mathbb{R}),$ there exist $\vec{\phi}_{\ell}\in L^{2}_{k}(\mathbb{R},\mathbb{C}^{2})$ and $\vec{v}_{n,\ell}\in \Raa P_{d,\omega_{\ell}}$ satisfying
\begin{equation}\label{eq:onepAC}
    \vec{f}(x)\chi_{\left[\frac{y_{\ell}+y_{\ell+1}}{2},\frac{y_{\ell}+y_{\ell-1}}{2}\right]}(x)+\vec{g}(x)\chi_{\left({-}\infty,\frac{y_{\ell}+y_{\ell+1}}{2}\right]\cup \left[\frac{y_{\ell}+y_{\ell-1}}{2},{+}\infty\right)}(x)=\hat{F}_{\omega_{\ell}}\left(\vec{\phi_{\ell}}(k)\right)(x)+\vec{z}_{d,\ell}(x).
\end{equation}

\subsubsection{Reduction to a linear system}
\par To achieve the decomposition  \eqref{princ11}, we next note that since the functions $\mathcal{T}(\vec{\phi})(t,x)$ and $\mathfrak{G}_{\omega_{k},v_{k},y_{k},\gamma_k}(\mathfrak{v}_{\omega_{k},\lambda})(t,x)$ and $\mathfrak{G}_{\omega_{k},v_{k},y_{k},\gamma_k}(\mathfrak{z}_{\omega_{k},\lambda})(t,x)$ defined in Theorem \ref{tcont} and  Theorem \ref{tdis} satisfy the Schr\"odinger flow \eqref{p} for all $t\geq 0,$ we can restrict the proof of identity \eqref{princ11} to  $t=0.$ In particular, when $t=0,$ we will deduce using Lemma \ref{localS0} and Theorem  \ref{tcont} that
\begin{align}\label{eqqq11}
    \mathcal{T}(\vec{\phi})(0,x)=\sum_{\ell}\chi_{\left[\frac{y_{\ell}+y_{\ell+1}}{2},\frac{y_{\ell}+y_{\ell-1}}{2}\right]}(x)e^{i\left(\frac{v_{\ell}x}{2}+\gamma_{\ell}\right)\sigma_{3}}\hat{G}_{\omega_{\ell}}\left(e^{{-}i\gamma_{\ell}\sigma_{3}}e^{iy_{\ell}k}\begin{bmatrix}
        \phi_{1,\ell}\left(k+\frac{v_{\ell}}{2}\right)\\
        \phi_{2,\ell}\left(k-\frac{v_{\ell}}{2}\right)
    \end{bmatrix}\right)(x-y_{\ell})\\ \nonumber
    {+}O\left(e^{{-}\beta \min_{\ell}(y_{\ell}-y_{\ell+1})}\norm{\mathcal{T}(0)(\vec{\phi})}_{L^{2}_{x}(\mathbb{R})}\right).
\end{align}
Consequently, using the standard Hilbert space theory, we will verify that the existence and uniqueness of $\vec{\phi}\in L^{2}_{k}(\mathbb{R},\mathbb{C}^{2})$ and $v_{d,\ell}\in L^{2}_{x}(\mathbb{R},\mathbb{C}^{2})$ satisfying the equation \eqref{princ11} for any $\vec{f}\in L^{2}_{x}(\mathbb{R},\mathbb{C}^{2})$ is equivalent to the uniqueness and existence of functions $\vec{\phi}\in L^{2}_{k}(\mathbb{R},\mathbb{C}^{2})$ and $\vec{z}_{\omega_{\ell}}\in \Raa P_{d,\omega_{\ell}} $ solving the following equation 
\begin{align}\label{Pooo1}
    \vec{f}(x)=&\sum_{\ell=1}^{m}\chi_{\left[\frac{y_{\ell}+y_{\ell+1}}{2},\frac{y_{\ell}+y_{\ell-1}}{2}\right]}(x)e^{i\left(\frac{v_{\ell}x}{2}+\gamma_{\ell}\right)\sigma_{3}}\hat{G}_{\omega_{\ell}}\left(e^{{-}i\gamma_{\ell}\sigma_{3}}e^{iy_{\ell}k}\begin{bmatrix}
        \phi_{1,\ell}\left(k+\frac{v_{\ell}}{2}\right)\\
        \phi_{2,\ell}\left(k-\frac{v_{\ell}}{2}\right)
    \end{bmatrix}\right)(x-y_{\ell})\\
    \\&{+}\sum_{\ell=1}^{m}
    e^{i\sigma_{3}\left(\frac{v_{\ell}x}{2}+\gamma_{\ell}\right)}\vec{z}_{\omega_{\ell}}(x-y_{\ell}).
    \end{align}
By Agmon's estimate, for any $\ell\in\{1,2,\,...,\,m\},$ $\Raa P_{d,\omega_{\ell}}$ is finite-dimensional for any $\ell$ and any element $\vec{z}(x)$ of $\Raa P_{d,\omega_{\ell}}$ has exponential decay. Therefore, to show  the existence and uniqueness of functions $\vec{\phi}\in L^{2}_{k}(\mathbb{R},\mathbb{C}^{2})$ and $\vec{z}_{\omega_{\ell}}\in\Raa P_{d,\omega_{\ell}}$ satisfying \eqref{Pooo1} is equivalent to the existence and uniqueness of functions $\vec{\phi}\in L^{2}_{k}(\mathbb{R},\mathbb{C}^{2})$ and $\vec{z}_{\omega_{\ell}}\in\Raa P_{d,\omega_{\ell}}$ satisfying
\begin{align}\label{Pooo2}
    \vec{f}(x)=&\sum_{\ell=1}^{m}\chi_{\left[\frac{y_{\ell}+y_{\ell+1}}{2},\frac{y_{\ell}+y_{\ell-1}}{2}\right]}(x)e^{i\left(\frac{v_{\ell}x}{2}+\gamma_{\ell}\right)\sigma_{3}}\hat{G}_{\omega_{\ell}}\left(e^{{-}i\gamma_{\ell}\sigma_{3}}e^{iy_{\ell}k}\begin{bmatrix}
        \phi_{1,\ell}\left(k+\frac{v_{\ell}}{2}\right)\\
        \phi_{2,\ell}\left(k-\frac{v_{\ell}}{2}\right)
    \end{bmatrix}\right)(x-y_{\ell})\\ \nonumber
    \\&{+}\sum_{\ell=1}^{m}
    e^{i\left(\frac{v_{\ell}x}{2}+\gamma_{\ell}\right)\sigma_{3}}\chi_{\left[\frac{y_{\ell}+y_{\ell+1}}{2},\frac{y_{\ell}+y_{\ell-1}}{2}\right]}(x)\vec{z}_{\omega_{\ell}}(x-y_{\ell}).
    \end{align}
Furthermore, after a change of variable and localization, one can reduce that the functions $\vec{\phi}$ and $\vec{z}_{\omega_{\ell}}$ being solutions of \eqref{Pooo2} is equivalent to the existence of functions $h_{\ell,\pm}\in L^{2}_{x}(\mathbb{R},\mathbb{C}^{2})$ satisfying the following equations
\begin{align}\nonumber
   &\hat{G}_{\omega_{\ell}}\left(e^{{-}i\gamma_{\ell}\sigma_{3}}e^{iy_{\ell}k}\begin{bmatrix}
        \phi_{1,\ell}\left(k+\frac{v_{\ell}}{2}\right)\\
        \phi_{2,\ell}\left(k-\frac{v_{\ell}}{2}\right)
    \end{bmatrix}\right)(x)\\
    &\quad\quad \quad \quad\quad \quad=\vec{z}_{\omega_{\ell}}(x)+e^{{-}i\left(\frac{v_{\ell}(x+y_{\ell})}{2}+\gamma_{\ell}\right)\sigma_{3}}\chi_{\left[\frac{y_{\ell+1}-y_{\ell}}{2},\frac{y_{\ell-1}-y_{\ell}}{2}\right]}f(x+y_{\ell})\\
    \label{eqlast2}
   &\quad\quad \quad \quad\quad \quad {+}\chi_{\left[\frac{y_{\ell-1}-y_{\ell}}{2},{+}\infty\right)}h_{\ell,+}(x)+ \chi_{\left({-}\infty,\frac{y_{\ell+1}-y_{\ell}}{2}\right]}(x)h_{\ell,-}(x),
\end{align}
for any $\ell\in\{1,2,\,...,\,m\}.$ 
\par From the propertis of distorted Fourier transforms, for any $\ell\in\{1,2,\,...,\,m\},$ the following identity holds 
\begin{equation*}
    \int_{\mathbb{R}}\overline{\mathcal{F}_{\omega_{\ell}}^{t}(x,\lambda)}\vec{z}_{\omega_{\ell}}(x)\,dx=\int_{\mathbb{R}}\overline{\mathcal{G}_{\omega_{\ell}}^{t}(x,\lambda)}\vec{z}_{\omega_{\ell}}(x)\,dx=0,
\end{equation*}
for all $\vec{z}_{\omega_{\ell}}\in \Raa P_{d,\omega_{\ell}}.$ By inversion formulas, one also has
\begin{equation*}
    \int_{\mathbb{R}}\sigma_{3}\overline{\mathcal{F}_{\omega_{\ell}}^{t}(x,\lambda)}\sigma_{3}\hat{G}_{\omega_{\ell}}\left(\begin{bmatrix}
        \phi_{1}\\
        \phi_{2}
    \end{bmatrix}\right)(x)\,dx=\begin{bmatrix}
        \phi_{1}(k)\\
        \phi_{2}(k)
    \end{bmatrix},\, \int_{\mathbb{R}}\sigma_{3}\overline{\mathcal{G}_{\omega_{\ell}}^{t}(x,\lambda)}\sigma_{3}\hat{F}_{\omega_{\ell}}\left(\begin{bmatrix}
        \psi_{1}\\
        \psi_{2}
    \end{bmatrix}\right)(x)\,dx=\begin{bmatrix}
        \psi_{1}(k)\\
        \psi_{2}(k)
    \end{bmatrix},
\end{equation*}
for any $\vec{\phi}\in L^{2}_{k}(\mathbb{R},\mathbb{C}^{2})$ and $\vec{\psi}\in L^{2}_{k}(\mathbb{R},\mathbb{C}^{2})$ belonging to the domain of each function $\hat{G}_{\omega_{\ell}}$ and $\hat{F}_{\omega_{\ell}}.$  Keeping in mind that
\begin{equation*}
    \hat{G}_{\omega_{\ell}}\left(e^{{-}i\gamma_{\ell}}e^{iy_{\ell}k}
    \begin{bmatrix}
        \phi_{1,\ell}(k+\frac{v_{\ell}}{2})\\
        \phi_{2,\ell}(k-\frac{v_{\ell}}{2})
    \end{bmatrix}
    \right)(x)=\hat{F}_{\omega_{\ell}}\left(e^{{-}i\gamma_{\ell}}e^{iy_{\ell}k}
    \begin{bmatrix}
        \phi_{1,\ell-1}(k+\frac{v_{\ell}}{2})\\
        \phi_{2,\ell-1}(k-\frac{v_{\ell}}{2})
    \end{bmatrix}
    \right)(x),
\end{equation*}
it can be obtained after applying each operator 
\begin{equation*}
    F^{*}_{\omega_{\ell}}\left(\sigma_{3}\vec{f}(x)\right)(\lambda)=\int_{\mathbb{R}}\overline{\mathcal{F}^{t}(x,\lambda)}\sigma_{3}\vec{f}(x)\,dx,\,\,G^{*}_{\omega_{\ell}}\left(\sigma_{3}\vec{f}(x)\right)(\lambda)=\int_{\mathbb{R}}\overline{\mathcal{G}^{t}(x,\lambda)}\sigma_{3}\vec{f}(x)\,dx,
\end{equation*}
on each equation \eqref{eqlast2} that all the identities of \eqref{eqlast2} are true for all $\ell\in\{1,\,...,m\}$ when the functions $h_{\ell,\pm}\in L^{2}_{x}(\mathbb{R},\mathbb{C}^{2})$ satisfy a linear equation of the following form
\begin{equation}\label{TA}
    L\left(
    [h_{\ell,\epsilon}]_{\epsilon\in\{+,-\},\ell\in\{1,2,\,...,m\}}\right)(k)=A(\vec{f})(k),
\end{equation}
where $L:L^{2}_{x}(\mathbb{R},\mathbb{C}^{2m-2})\to L^{2}_{k}(\mathbb{R},\mathbb{C}^{2m-2})$ and $A:L^{2}_{x}(\mathbb{R},\mathbb{C}^{2})\to L^{2}_{k}(\mathbb{R},\mathbb{C}^{2m-2})$ are linear bounded operators. If we could find a solution to \eqref{TA}, then one can apply \eqref{eq:onepAC} to recover $\vec{z}_{\omega_\ell}$ and $\vec{\phi}_\ell$ from \eqref{eqlast2}. 

\subsubsection{Invertibility of the linear system via Hardy spaces}
\par To study the operator $L$, firstly  using the asymptotic behavior in \eqref{asy1}, \eqref{asy2}, \eqref{asy3} and \eqref{asy4}, one has
\begin{align*}
    G^{*}_{\omega_{\ell}}\left(\chi_{[\frac{y_{\ell-1}-y_{\ell}}{2},{+}\infty)}(x)h_{\ell,+}(x)\right)= & \int_{\frac{y_{\ell-1}-y_{\ell}}{2}}^{{+}\infty}e^{{-}ikx}h_{\ell,+}(x)\,dx+r(k)\int_{\frac{y_{\ell-1}-y_{\ell}}{2}}^{{+}\infty}e^{ikx}h_{\ell,+}(x)\,dx\\
    =& e^{{-}ik\frac{(y_{\ell-1}-y_{\ell})}{2}}\int_{0}^{{+}\infty}e^{{-}ikx}h_{\ell,+}\left(x+\frac{y_{\ell-1}-y_{\ell}}{2}\right)\,dx\\&{+}r(k)e^{\frac{ik(y_{\ell-1}-y_{\ell})}{2}}\int_{0}^{{+}\infty}e^{ikx}h_{\ell,+}\left(x+\frac{y_{\ell-1}-y_{\ell}}{2}\right)\,dx\\&{+}O\left(e^{{-}\beta\vert y_{\ell-1}-y_{\ell}\vert}\norm{\chi_{[\frac{y_{\ell-1}-y_{\ell}}{2},{+}\infty)}(x)h_{\ell,+}(x)}_{L^{2}_{x}(\mathbb{R})}\right).
\end{align*}
One can deduce using the decay estimates \eqref{asyreftr} and Lemma \ref{+-interact} that the orthogonal projection of the function $e^{ik\frac{(y_{\ell-1}-y_{\ell})}{2}}G^{*}_{\omega_{\ell}}\left(\chi_{[\frac{y_{\ell-1}-y_{\ell}}{2},{+}\infty)}(x)h_{\ell,+}(x)\right)$ onto the Hardy space $H^{2}(\mathbb{C}_{-},\mathbb{C}^{2})$ 
is equal to
\begin{equation*}
    \int_{0}^{{+}\infty}e^{{-}ikx}h_{\ell,+}\left(x+\frac{y_{\ell-1}-y_{\ell}}{2}\right)\,dx+O\left(e^{{-}\beta\vert y_{\ell-1}-y_{\ell}\vert}\norm{\chi_{[\frac{y_{\ell-1}-y_{\ell}}{2},{+}\infty)}(x)h_{\ell,+}(x)}_{L^{2}_{x}(\mathbb{R})}\right),
\end{equation*}
and the orthogonal projection of the function $e^{ik\frac{(y_{\ell-1}-y_{\ell})}{2}}G^{*}_{\omega_{\ell}}\left(\chi_{[\frac{y_{\ell-1}-y_{\ell}}{2},{+}\infty)}(x)h_{\ell,+}(x)\right)$ onto the Hardy space $H^{2}(\mathbb{C}_{+},\mathbb{C}^{2})$\footnote{Note that $G^*_\omega$ represents that an incoming wave from $x=\infty$ interacts with the potential then splits into the reflection and transmission parts as outgoing waves.  One can see from the computations above that the Hardy space $H^{2}(\mathbb{C}_{+},\mathbb{C}^{2})$ detects the reflection part of the outgoing wave in the positive half-line. On the other hand, the Hardy space $H^{2}(\mathbb{C}_{-},\mathbb{C}^{2})$ detects the incoming wave in the positive half-line. Similarly, switching to the negative half-line, one can similarly note that  the Hardy space $H^{2}(\mathbb{C}_{+},\mathbb{C}^{2})$ detects the transmission  part of the outgoing wave. }  is equal to
{\small \begin{equation*}
    r(k)e^{\frac{ik(y_{\ell-1}-y_{\ell})}{2}}\int_{0}^{{+}\infty}e^{ikx}h_{\ell,+}\left(x+\frac{y_{\ell-1}-y_{\ell}}{2}\right)\,dx+O\left(e^{{-}\beta\vert y_{\ell-1}-y_{\ell}\vert}\norm{\chi_{[\frac{y_{\ell-1}-y_{\ell}}{2},{+}\infty)}(x)h_{\ell,+}(x)}_{L^{2}_{x}(\mathbb{R})}\right),
\end{equation*}}
both of them have large precision when $\vert y_{\ell-1}-y_{\ell} \vert\gg 1.$
Similarly, we can find  the orthogonal decomposition of
\begin{equation*}
F^{*}_{\omega_{\ell}}\left(\chi_{[\frac{y_{\ell-1}-y_{\ell}}{2},{+}\infty)}h_{\ell,+}(x)\right),\,F^{*}_{\omega_{\ell}}\left(\chi_{({-}\infty,\frac{y_{\ell+1}-y_{\ell}}{2}]}h_{\ell,+}(x)\right),\, G^{*}_{\omega_{\ell}}\left(\chi_{({-}\infty,\frac{y_{\ell+1}-y_{\ell}}{2}]}h_{\ell,+}(x)\right)  
\end{equation*}
onto $H^{2}(\mathbb{C}_{+},\mathbb{C}^{2})\oplus H^{2}(\mathbb{C}_{-},\mathbb{C}^{2}).$

\par Consequently, using the orthogonal decomposition $L^{2}_{x}(\mathbb{R},\mathbb{C}^{2})=H^{2}(\mathbb{C}_{+},\mathbb{C}^{2})\oplus H^{2}(\mathbb{C}_{-},\mathbb{C}^{2}),$ 
we will show that if $\min y_{\ell}-y_{\ell+1}>0$ and $\min v_{\ell}-v_{\ell+1}>0$ are sufficiently large, the operator $L$ has a bounded inverse, which implies the existence of functions $h_{\ell,\pm}\in L^{2}_{x}(\mathbb{R},\mathbb{C}^{2})$ satisfying \eqref{TA} for any $f\in L^{2}_{x}(\mathbb{R},\mathbb{C}^{2}).$ More precisely, 
we can deduce that the operator $L$ is isomorphic to a perturbation of the identity operator 
\begin{equation*}
    \mathrm{Id}:H^{2}(\mathbb{C}_{+},\mathbb{C}^{2m-2})\oplus H^{2}(\mathbb{C}_{-},\mathbb{C}^{2m-2})\to H^{2}(\mathbb{C}_{+},\mathbb{C}^{2m-2})\oplus H^{2}(\mathbb{C}_{-},\mathbb{C}^{2m-2})
\end{equation*}
of the form $\mathrm{Id}+N_{\vec{v},\vec{y}}$ such that the bounded operator $N_{\vec{v},\vec{y}}:L^{2}_{x}(\mathbb{R},\mathbb{C}^{2m-2})\to L^{2}_{x}(\mathbb{R},\mathbb{C}^{2m-2})$ satisfies uniformly
\begin{equation*}
    \norm{N_{\vec{v},\vec{y}}^{m}}\leq \frac{C}{\min_{\ell}v_{\ell}-v_{\ell+1}} 
\end{equation*}
when $\min_{\ell} v_{\ell}-v_{\ell+1}>0,$ and $\min_{\ell} y_{\ell}-y_{\ell+1}>0$ is large enough. 
\par
As a consequence, for any $f\in L^{2}_{x}(\mathbb{R},\mathbb{C}^{2}),$ there exist unique $\vec{\phi}\in L^{2}_{k}(\mathbb{R},\mathbb{C}^{2})$ and functions $\vec{z}_{\omega_{\ell}}\in \Raa P_{d,\omega_{\ell}}$ satisfying \eqref{Pooo2}, which implies the asymptotic completeness of the solutions of \eqref{p} through the identity \eqref{princ11} being true for all $t\geq 0.$ 

\subsubsection{Coercivity estimates}
\par The proof of estimates \eqref{phi2kkdecay}, \eqref{F0l1} and \eqref{F0l2} follows using the coercive estimate 
\begin{equation*}
    \max_{\ell}\norm{\begin{bmatrix}
        \phi_{1,\ell}(k)\\
        \phi_{2,\ell}(k)
    \end{bmatrix}}_{L^{2}_{k}(\mathbb{R})}\leq C\norm{\mathcal{S}(0)(\vec{\phi})}\leq 2C\norm{\mathcal{T}(0)(\vec{\phi})}
\end{equation*}
proved in Theorem \ref{TT} of Section $4,$ the triangular inequality and the propositions of Appendix \ref{sec:appb}.
\section{Scattering theory for linear Schr\"odinger equations with one potential}\label{sec:scatteringone}
In this section, we study the scattering theory for the standard matrix Schr\"odinger operator $\mathcal{H}_\omega$ defined in \eqref{H} with appearance of threshold resonances at $\pm\omega$. After constructing distorted Fourier bases associated to $\mathcal{H}_\omega$, we study the properties of mappings based on these distorted bases.

\subsection{Construction of distorted Fourier bases}
Since in the construction of distorted Fourier bases, the generic condition is imposed in Krieger-Schlag \cite{KriegerSchlag}, here we provide some details for the non-generic case in the general setting. It should be noted that in \cite{nls3soliton}, a special case given by cubic NLS was systematically studied. The whole construction is the same as in the generic setting. We just need to justify the definition of distorted Fourier bases, which is stated in Lemma 6.3 in \cite{KriegerSchlag}. We will follow the notations in \cite{KriegerSchlag}. 
\begin{lemma}\label{2.1}
Let $\underline{e}=\begin{bmatrix}
        1\\
        0
    \end{bmatrix}$. Then for all $k \in \mathbb{R}$, 
\begin{equation}\label{eq:Fomegadef}
\mathcal{F}_\omega(x,k):=2i\frac{k}{\sqrt{\omega}}F_1\left(\sqrt{\omega} x,\frac{k}{\sqrt{\omega}}\right)D\left(\frac{k}{\sqrt{\omega}}\right)^{-1}\underline{e}
\end{equation}
\begin{equation}\label{eq:Gomegadef}
\mathcal{G}_\omega(x,k):=2i\frac{k}{\sqrt{\omega}}G_2\left(\sqrt{\omega} x,\frac{k}{\sqrt{\omega}}\right)D\left(\frac{k}{\sqrt{\omega}}\right)^{-1}\underline{e}
\end{equation} are bounded solutions of $\mathcal{H}_\omega f=(\omega+k^2)f$ satisfying \eqref{asy1}, \eqref{asy2}, \eqref{asy3} and \eqref{asy4}.
\end{lemma}
\begin{proof}
With the proof of Lemma 6.3 in \cite{KriegerSchlag}, we just need to justify when $\pm\omega$ are resonances but not eigenvalues, \eqref{eq:Fomegadef} and \eqref{eq:Gomegadef} are well-defined functions.

Using the constructions of $A\left(\frac{k}{\sqrt{\omega}}\right),\, B\left(\frac{k}{\sqrt{\omega}}\right)$ in Lemma 5.14  and $D\left(\frac{k}{\sqrt{\omega}}\right)$ in (5.33) from  \cite{KriegerSchlag}, we know that for $k\neq 0$, 
\begin{equation}
    \det D\left(\frac{k}{\sqrt{\omega}}\right)=-4i\frac{k}{\sqrt{\omega}}\mu(\omega)\det A\left(\frac{k}{\sqrt{\omega}}\right)
\end{equation}by (5.26) from Lemma 5.14 in \cite{KriegerSchlag}. It follow that
\begin{equation}
    2i\frac{k}{\sqrt{\omega}}\left(\det D\left(\frac{k}{\sqrt{\omega}}\right)\right)^{-1}=-\frac{1}{2\mu(\omega)}\left(\det A\left(\frac{k}{\sqrt{\omega(\omega)}}\right)\right)^{-1}
\end{equation}when $\left(\det A\left(\frac{k}{\sqrt{\omega}}\right)\right)^{-1}$ is well-defined. By Lemma 5.17 from \cite{KriegerSchlag}, $E=k^2+\omega$ is an eigenvalue of $\mathcal{H}_\omega$ if and only if $\det A\left(\frac{k}{\sqrt{\omega}}\right)=0$. From Corollary 5.21 in \cite{KriegerSchlag}, $D\left(\frac{k}{\sqrt{\omega}}\right)$ is invertible if $E=k^2+\omega$ is not a resonance nor eigenvalue. We will justify that these properties hold as $k\rightarrow 0$. In particular, we will show that under our assumptions,  $\det A\left(\frac{k}{\sqrt{\omega}}\right) \neq 0$ as $k\rightarrow 0$.
\par 

To simplify notations, in later arguments, we set $\lambda=\frac{k}{\sqrt{\omega}}$. We recall that Corollary $5.25$ of \cite{KriegerSchlag} implies that the functions $A$ and $B$ satisfy 
\begin{equation}\label{trtrtr}
    \lambda p A\left(\lambda\right),\, \lambda p B\left(\lambda\right),\, q A\left(\lambda\right),\,qB\left(\lambda\right) \text{ are smooth for all $\lambda\in\mathbb{R},$}
\end{equation}
where
\begin{equation*}
p=\begin{bmatrix}
    1 & 0\\
    0 & 0
\end{bmatrix},\, q=\begin{bmatrix}
    0 & 0\\
    0 & 1
\end{bmatrix}.
\end{equation*}
\par Next, we are going to prove that in all the possible cases that $\lim_{\lambda\to 0}\left\vert \det\left(A(\lambda)\right)\right\vert\neq 0,$ which will imply the result of Lemma \ref{2.1}. In each case, we will use the following equation which can be deduced directly from the identity $(5.28)$ of Lemma $5.16.$ from \cite{KriegerSchlag}.
\begin{align}\label{ppppp1}
   2i\lambda\langle pv,v \rangle= &2i\lambda\left\langle A^{*}(\lambda)pA(\lambda)v,v \right\rangle+2\mu\left\langle  A^{*}(\lambda)q B(\lambda)v,v \right\rangle{-}2\mu\left\langle B^{*}(\lambda)qA(\lambda)v,v \right\rangle\\&{-}2i\lambda\left\langle B^{*}(\lambda)pB(\lambda)v,v \right\rangle,
\end{align}
for any $v\in \mathbb{C}^{2}.$
\\
\textbf{(Case 1: $A(\lambda)=A_{1}\lambda+O(\lambda^{2})$ near $\lambda=0.$)}
If $A_{1}$ is invertible, then there exists $v\neq 0\in\mathbb{C}^{2}$ satisfying $qA_{1}v=0.$ Consequently, the identity \eqref{ppppp1} implies for any $\lambda$ sufficiently close to $0$ that
\begin{equation*}
    2i\lambda\norm{pv}^{2}={-}2i\lambda\norm{pB(\lambda)v}^{2}+O(\lambda^{2}),
\end{equation*}
from which we can deduce that $pv=0,$ and $\lim_{\lambda\to 0} \norm{pB(\lambda)v}=0,$ and so there exists an $\alpha(\lambda)$ continuous on $\lambda=0$ satisfying
\begin{equation*}
   B(\lambda)\begin{bmatrix}
        0\\
        1
    \end{bmatrix}=\alpha(\lambda) \begin{bmatrix}
        0\\
        1
    \end{bmatrix}+O(\lambda).
\end{equation*}
Therefore, since in Case $1,$ $\lim_{\lambda\to 0}A(\lambda)=0$ and the estimate above holds for all $\lambda\in\mathbb{R},$ we can use the argument from the first part of the proof of   \cite[Lemma $5.17$]{KriegerSchlag} to obtain that $\omega$ is an eigenvalue of $\mathcal{H}_{\omega},$ which is a contradiction. The proof when $\det A_{1}=0$ is analogous.  
\\
\textbf{(Case 2: $A(\lambda)$ and $B(\lambda)$ are smooth on $\mathbb{R}.$)}
Using \eqref{trtrtr}, it is enough to consider the case where $A(\lambda)=A_{0}+\lambda A_{1}+O(\lambda^{2})$ and $B(\lambda)=B_{0}+\lambda B_{1}+O(\lambda^{2})$ near $\lambda=0.$

If $\det (A_{0})\neq 0,$ then $\lim_{\lambda\to 0}\det A(\lambda)\neq 0,$ and so there is nothing else needed to be done.   Otherwise, if $\det (A_{0})=0,$ there exists $v\neq 0\in\mathbb{C}^{2}$ satisfying $A_{0}v=0.$ Therefore,  $(5.29)$ of Lemma $5.16$ from \cite{KriegerSchlag} implies for $\mu(\lambda)=\sqrt{\lambda^{2}+2\omega}$ that
\begin{equation}\label{A0=0}
    0=\lim_{\lambda\to 0}A^{t}(\lambda)\begin{bmatrix}
        0 & 0\\
        0 & {-}2\mu(\lambda)
    \end{bmatrix}B(\lambda)v=A_{0}^{t}\begin{bmatrix}
        0 & 0\\
        0 & {-}2\mu(0)
    \end{bmatrix}B_{0}v.
\end{equation}
If $B_{0}v=0,$ then we can deduce from equation \eqref{ppppp1} and identity $A_{0}v=0$ that $\norm{pv}=0,$ and so $v=\alpha\begin{bmatrix}
    0\\
    1
\end{bmatrix}$ for an $\alpha\neq 0.$ However, Lemmas $5.2$ and $5.14$ of \cite{KriegerSchlag} imply that $B_{0}\begin{bmatrix}
    0\\
    1
\end{bmatrix}=A_{0}\begin{bmatrix}
    0\\
    1
\end{bmatrix}=0$ is not possible to happen.  
\par Now we assume that  $B_{0}v\neq 0$ and $A_{0}v=0$ for some $v\neq 0 \in\mathbb{C}^{2}.$ If 
\begin{equation*}
B_{0}v\in \sppp\left\{\begin{bmatrix}
    1\\
    0
\end{bmatrix}\right\},    
\end{equation*}
 then $qB_{0}v=qA_{0}v=0,$ from which we obtain from \eqref{ppppp1} that $\norm{pB_{0}v}=0$ which contradicts the hypothesis that $B_{0}v\neq 0.$
 \par The last remaining situation of Case $2$ is when $B_{0}v\neq 0$ and 
 \begin{equation*}
    B_{0}v\notin \sppp\left\{\begin{bmatrix}
    1\\
    0
\end{bmatrix}\right\}.
 \end{equation*}
In particular, \eqref{A0=0} implies that there exist $a_{11},\,a_{22}\in\mathbb{C}$ satisfying
\begin{equation}\label{ae1}
    A_{0}=\begin{bmatrix}
        a_{11} & a_{12}\\
        0 & 0
    \end{bmatrix}.
\end{equation}

In particular, $q A_{0}=\lim_{\lambda\to 0}qA(\lambda)=0.$ Consequently, Lemma $5.10$ and the identity $(5.26)$ from the article \cite{KriegerSchlag} imply that there exist non-zero smooth functions $F(x,0),\, G(x,0)\in C^{\infty}(\mathbb{R},\mathbb{C}^{2\times 2})$ such that $F(x,0)=\alpha G(x,0)$ for some constant $\alpha\in \mathbb{C},$ and there exists $\beta>0$ satisfying
\begin{align*}
\mathcal{H}_{\omega}F(x,0)=\omega F(x,0),\, \mathcal{H}_{\omega}G(x,0)=\omega G(x,0),\text{ for a $x\in\mathbb{R},$}\\
F(x,0)\begin{bmatrix}
    0\\
    1
\end{bmatrix}=e^{{-}\sqrt{2\omega}x}+O(e^{{-}(\sqrt{2\omega}+\beta)x}) \text{, when $x\to{+}\infty,$}\\
G(x,0)\begin{bmatrix}
    0\\
    1
\end{bmatrix}=e^{\sqrt{2\omega}x}+O(e^{(\sqrt{2\omega}+\beta)x}) \text{, when $x\to{-}\infty.$}
\end{align*}
Therefore, function $F(x,0)\begin{bmatrix}
    0\\
    1
\end{bmatrix}\in L^{2}_{x}(\mathbb{R})$  is an eigenfunction associated with  $\omega,$ which contradicts the hypothesis of $\omega$ not being an eigenvalue. For more details, see Lemmas $5.10$ and $5.14$ from Section $5$ of \cite{KriegerSchlag}.
\\
\textbf{(Case 3: $A(\lambda)$ or $B(\lambda)$ is not continuous on $\lambda=0.$)}
Assume that $\lim \inf_{\lambda\to 0}\vert\det A(\lambda)\vert=0.$ The condition \eqref{trtrtr} implies that when $\lambda$ is close to $0,$ $A(\lambda)$ has for constants $c_{1},\,c_{2},\,a_{11},\,a_{12},\,b_{11},\,b_{12},\,d_{11}$ and $d_{12}\in\mathbb{C}$ the following asymptotic:
\begin{equation*}
    A(\lambda)=\begin{bmatrix}
        \frac{a_{11}}{\lambda}+b_{11}+d_{11}\lambda & \frac{a_{12}}{\lambda}+b_{12}+d_{12}\lambda\\
         c_{1}a_{11}+c_{1}\lambda b_{11}+c_{2}\lambda a_{11} & c_{1}a_{12}+c_{1}\lambda b_{12}+c_{2}\lambda a_{12} 
    \end{bmatrix}+O(\lambda^{2}).
\end{equation*}
\par Furthermore, using the identity \eqref{ppppp1}, we can verify that $B(\lambda)$ should have the following asymptotic behavior near $\lambda=0$:
\begin{equation*}
    B(\lambda)=\begin{bmatrix}
        \pm\frac{a_{11}}{\lambda} & \pm\frac{a_{12}}{\lambda}\\
         0 & 0
    \end{bmatrix}+O(1).
\end{equation*}
Consequently, if $a_{11}=a_{12}=0,$ then $A(\lambda)$ and $B(\lambda)$ would be continuous on $\lambda=0,$ which would contradict the assumption of Case $3.$
\par From the asymptotic of $A(\lambda),$ we can verify that the vector
\begin{equation*}
    v_{\lambda}=\begin{bmatrix}
        a_{12}+b_{12}\lambda+d_{12}\lambda^{2}\\
        {-}a_{11}-b_{11}\lambda-d_{11}\lambda^{2}
    \end{bmatrix}
\end{equation*}
satisfies $A(\lambda)v_{\lambda}=O(\lambda^{2})$ when $\lambda$ is close to $0,$
and $\lim_{\lambda\to 0}v_{\lambda}=\begin{bmatrix}
    a_{12}\\
    {-}a_{11}
\end{bmatrix}\neq 0.$
\par Consequently, we can verify from \eqref{ppppp1} and properties \eqref{trtrtr} that
\begin{equation}\label{lastcase3}
   2i\lambda\norm{p v_{\lambda}}^{2}={-}2i\lambda\norm{pB(\lambda)v_{\lambda}}^{2}+O(\lambda^{2}), 
\end{equation}
when $\lambda$ is close to $0.$ Therefore, we deduce that $a_{12}=0$ and $a_{11}\neq 0.$ Therefore, we can conclude the existence of a $\beta\neq 0\in \mathbb{C}$ such that 
\begin{equation*}
    \lim_{\lambda\rightarrow 0}A(\lambda)v_{\lambda}=0,\,\lim_{\lambda\rightarrow 0}v_{\lambda}={-}a_{11}\begin{bmatrix}
        0\\
        1
    \end{bmatrix},\,\lim_{\lambda\to 0} B(\lambda)v_{\large}=\beta \begin{bmatrix}
        0\\
        1
    \end{bmatrix},
\end{equation*}
the last identity follows from \eqref{lastcase3}.
\par In conclusion, similar to the previous cases, we can repeat the argument from the first paragraph of the proof of Lemma $5.16$ of \cite{KriegerSchlag} to obtain that $\omega$ is an eigenvalue of $\mathcal{H}_{\omega},$ which contradicts our hypothesis.\\ 
\textbf{(Conclusion of the proof of Lemma \ref{2.1}.)}
\par Assume that $\omega$ is a resonance but not an eigenvalue. From discussions above, we know that  $\det A\left(\frac{k}{\sqrt{\omega}}\right) \neq 0,\,\lim\inf_{k\to 0}\left\vert \det A\left(\frac{k}{\sqrt{\omega}}\right)  \right\vert\neq 0,$   but $\left|\left(\det D \left(\frac{k}{\sqrt{\omega}}\right)\right)^{-1}\right| \rightarrow \infty$ as $k\rightarrow 0$.\\
Using a limiting argument as $k\rightarrow0$, the identity
\begin{equation}
    2i\frac{k}{\sqrt{\omega}}\left(\det D\left(\frac{k}{\sqrt{\omega}}\right)\right)^{-1}=-\frac{1}{2\mu(\omega)}\left(\det A\left(\frac{k}{\sqrt{\omega}}\right)\right)^{-1}
\end{equation} holds for any $k$. 
Actually, it also implies that $ 2i\frac{k}{\sqrt{\omega}}\left(\det D\left(\frac{k}{\sqrt{\omega}}\right)\right)^{-1}$ is a smooth function.  It follows that 
\eqref{eq:Fomegadef} and \eqref{eq:Gomegadef} are well-defined functions for all $k$.
\end{proof}
We can also summarize our computations above as characterizations of eigenvalues and resonances all the way up to the endpoint of the essential spectrum. This is an extension of Proposition 2.2.1 and Proposition 2.2.2 in \cite{Busper1},  and Lemma 5.17 and Corollary 5.12 in \cite{KriegerSchlag}.
\begin{proposition}
Using the notations of $D\Big(\frac{k}{\sqrt{\omega}}\Big)$ and $A\Big(\frac{k}{\sqrt{\omega}}\Big)$ which are defined \cite[(5.33)]{KriegerSchlag} and \cite[(5.26)]{KriegerSchlag} respectively, then  
\begin{itemize}
    \item $k^2+\omega$ is an eigenvalue of $H_\omega$ if and only if $\det A\Big(\frac{k}{\sqrt{\omega}}\Big)=0$ for $k\in\mathbb{R}$. 
    \item $\omega$ is a resonance of $H_\omega$ if and only if $\det D(0)=0$.
\end{itemize}
\end{proposition}

From Lemma \ref{2.1}, we can define distorted Fourier bases as
\begin{equation}\label{eq:distortbase}
   e_{+}(x,k)=\begin{cases}
    \mathcal{F}_{\omega}(x,k) \text{, if $k\geq 0$}\\
    \mathcal{G}_{\omega}(x,{-}k) \text{, if $k\leq 0$}
   \end{cases},\,\, e_{-}(x,k)=\sigma_{1}e_{+}(x,k).
\end{equation}

Then one can derive standard properties following the procedures in Krieger-Schlag \cite{KriegerSchlag} and Bulsaev-Perelman \cite{Busper1}.

Repeating the proof of Proposition $6.9$ from  \cite{KriegerSchlag}, or Lemma $12$ and Remark $1$ from \cite{ErSch}, we can verify the following construction formula of $P_{e,\omega}$.
\begin{lemma}\label{lem:PeCo}
For any $f\in L^{2}_{x}(\mathbb{R},\mathbb{C}^{2})$, using the distorted bases \eqref{eq:distortbase} above, one has
\begin{align*}
P_{e,\omega}(f)(x)=&\frac{1}{2\pi}\int_{-\infty}^{{+}\infty}e_+(x,k)\left\langle f,\sigma_{3}e_{+}(\cdot,k) \right\rangle\,dk
{-}\frac{1}{2\pi}\int_{-\infty}^{{+}\infty}e_{-}(x,k)\left\langle f,\sigma_{3}e_{-}(\cdot,k) \right\rangle\,dk
\end{align*}
\end{lemma}
\par Next, we consider the subspace $P_{d,\omega}\subset L^{2}(\mathbb{R},\mathbb{C}^{2})$ to be the range of the discrete spectrum projection of $\mathcal{H}_{\omega}.$ 
In particular, repeating the argument of Section $6$ of \cite{KriegerSchlag} and Lemma $9.4$ of the same paper,  we have the following proposition.
\begin{lemma}\label{Asy1sol1}
Let $H_{e,\omega}=\Raa P_{e,\omega}$ the range of the projection onto the essential spectrum of $\mathcal{H}_{\omega}.$ Then, $H_{e,\omega}$ is a closed subspace of $L^{2}(\mathbb{R},\mathbb{C}^{2})$ and
\begin{align*}
L^{2}(\mathbb{R},\mathbb{C}^{2})=&H_{e,\omega} +\Raa P_{d,\omega},\\
H_{e,\omega}=&\left(\sigma_{3} \Raa P_{d,\omega}\right)^{\perp}.
\end{align*}
\end{lemma}

We also have the following inversion formula in the frequency side.

\begin{lemma}\label{lem:pq1}
We  claim that the following identity holds in the distributional sense 
\begin{equation}\label{pq1}
    \frac{1}{2\pi}\int_{\mathbb{R}}\left\langle \mathcal{F}_{\omega}(x,k),\sigma_{3}\mathcal{G}_{\omega}(x,\ell)\right\rangle \,dx=s(k)\delta(k-\ell).
\end{equation}
\end{lemma}
\begin{remark}
 Actually \eqref{pq1} was claimed in \cite{perelmanasym}  without a proof.
Here we provide a proof for the sake of completeness.  
\end{remark}
\begin{proof}[Proof of Lemma \ref{lem:pq1}]
 It suffices to show that for any $g$ complex-valued Schwartz function,
\begin{equation}\label{eq:goalgp}
    \frac{1}{2\pi}\int_{\mathbb{R}}\left\langle \mathcal{F}_{\omega}(x,k),\sigma_{3}\mathcal{G}_{\omega}(x,\ell)\overline{}\right\rangle g(\ell)\,d\ell=s(k)g(k).
\end{equation}
Note that $\mathcal{H}_\omega \mathcal{F}_\omega(x,k)=(k^2+\omega)\mathcal{F}_\omega(x,k)$ and $\mathcal{H}^*_\omega \sigma_3\mathcal{G}_\omega(x,\ell)=(\ell^2+\omega)\sigma_3\mathcal{G}_\omega(x,\ell)$. Therefore, in the distribution sense,
\begin{align}
  & (k^2-\ell^2)  \left\langle \mathcal{F}_{\omega}(x,k),\sigma_{3}\mathcal{G}_{\omega}(x,\ell)\right\rangle  =  \left\langle \mathcal{H}_\omega\mathcal{F}_{\omega}(x,k),\sigma_{3}\mathcal{G}_{\omega}(x,\ell)\right\rangle -\left\langle\mathcal{F}_{\omega}(x,k),\mathcal{H}_\omega^*\sigma_{3}\mathcal{G}_{\omega}(x,\ell)\right\rangle =0.
\end{align}
It follows that if $k^2-\ell^2\neq 0$, in the distribution sense, 
\begin{equation}
    \left\langle \mathcal{F}_{\omega}(x,k),\sigma_{3}\mathcal{G}_{\omega}(x,\ell)\right\rangle =0
\end{equation}
Therefore, for any fixed $k$ and given any $0<\eta \ll 1$ small,  if $|k^2-\ell^2|\geq \eta $, one has
\begin{equation}
    \frac{1}{2\pi}\int_{|k^2-\ell^2|\geq \eta}\left\langle \mathcal{F}_{\omega}(x,k),\sigma_{3}\mathcal{G}_{\omega}(x,\ell)\right\rangle g(\ell)\,d\ell=0
\end{equation}
It follows that
\begin{align}
   & \frac{1}{2\pi}\int_{\mathbb{R}}\left\langle \mathcal{F}_{\omega}(x,k),\sigma_{3}\mathcal{G}_{\omega}(x,\ell)\right\rangle g(\ell)\,d\ell=\frac{1}{2\pi}\int_{|k^2-\ell^2|<\eta}\left\langle \mathcal{F}_{\omega}(x,k),\sigma_{3}\mathcal{G}_{\omega}(x,\ell))\right\rangle g(\ell)\,d\ell\label{eq:etaR}.
\end{align}
Invoking asymptotics \eqref{asy1}, \eqref{asy2}, \eqref{asy3}, \eqref{asy4} and ignoring localized terms from these asymptotics (which only contribute terms of order $\eta^{1/2}$ after integrating on $\ell$), one has
{\footnotesize \begin{align}
    &\frac{1}{2\pi}\int_{|k^2-\ell^2|<\delta}\left\langle \mathcal{F}_{\omega}(x,k),\sigma_{3}\mathcal{G}_{\omega}(x,\ell)\right\rangle g(\ell)\,d\ell\\
    & = \frac{1}{2\pi}\int_{\mathbb{R}} \int_{|k^2-\ell^2|<\delta}g(\ell)\left( s(k)e^{ikx}\mathrm{1}_{x\geq 0}+ (e^{ikx}+r(k)e^{-ikx})\mathrm{1}_{x\leq 0}\Big)\Big((e^{-i\ell x}+r(\ell)e^{i\ell x})\mathrm{1}_{x\geq 0}+ s(\ell)e^{-i\ell x}\mathrm{1}_{x\leq 0}\right)\,dx d\ell\\
    &+\mathcal{O}(\eta^{1/2})\\
    &= \frac{1}{2\pi} \int_{\mathbb{R}} \int_{|k^2-\ell^2|<\delta}g(\ell)  \left( (s(k)e^{ikx-i\ell x}+s(k)r(\ell)e^{ikx+i\ell x})\mathrm{1}_{x\geq0}+(s(\ell)e^{ikx-i\ell x}+s(\ell)r(k)e^{-ikx-i\ell x})\mathrm{1}_{x\leq0}\right)\,dx d\ell\\
     &+\mathcal{O}(\eta^{1/2})\\
    &= \frac{1}{2\pi} \int_{\mathbb{R}} \int_{|k^2-\ell^2|<\delta}g(\ell)  \left( (s(k)e^{ikx-i\ell x}+s(k)r(\ell)e^{ikx+i\ell x}+s(\ell)e^{-ikx+i\ell x}+s(\ell)r(k)e^{ikx+i\ell x}))\mathrm{1}_{x\geq0}\right)\,dx d\ell \\
     &+\mathcal{O}(\eta^{1/2}).
\end{align}}
Using the Fourier transform of $\mathrm{1}_{x\geq0}$, one has
\begin{align}
   & \frac{1}{2\pi} \int_{\mathbb{R}}   \left( (s(k)e^{ikx-i\ell x}+s(k)r(\ell)e^{ikx+i\ell x}+s(\ell)e^{-ikx+i\ell x}+s(\ell)r(k)e^{ikx+i\ell x}))\mathrm{1}_{x\geq0}\right)\,dx d\ell\\
   & =s(k)\delta(k-\ell)+\frac{i}{2\pi}\frac{s(k)-s(\ell)}{k-\ell}+(s(k)r(\ell)+s(\ell)r(k))\delta(k+\ell)+\frac{i}{2\pi} \frac{(s(k)r(\ell)+s(\ell)r(k))}{k+\ell}.
\end{align}
Now note that $\frac{s(k)-s(\ell)}{k-\ell}$ is a smooth function in $\ell$.  One also notices that $$(s(k)r(\ell)+s(\ell)r(k))\delta(k+\ell)=0$$ since $$s(k)r(-k)+s(-k)r(k)=s(k)\bar r(k)+\bar s(k)r(k)=0$$ which also implies that $\frac{(s(k)r(\ell)+s(\ell)r(k))}{k+\ell}$ is a smooth function. 

Therefore, integrating in $\ell$, from computations above, one has
\begin{align}\label{eq:onlyeta}
    \frac{1}{2\pi} \int_{|k^2-\ell^2|<\delta}\left\langle \mathcal{F}_{\omega}(x,k),\sigma_{3}\mathcal{G}_{\omega}(x,\ell)\right\rangle g(\ell)\,d\ell=s(k)g(k)+\mathcal{O}(\eta^{1/2}).
\end{align}
For any fixed  $\eta$, from \eqref{eq:etaR} and \eqref{eq:onlyeta}, we get
\begin{align}
    \frac{1}{2\pi} \int_{\mathbb{R}}\left\langle \mathcal{F}_{\omega}(x,k),\sigma_{3}\mathcal{G}_{\omega}(x,\ell)\right\rangle g(\ell)\,d\ell=s(k)g(k)+\mathcal{O}(\eta^{1/2}).
\end{align}Since the above identity holds for any $\eta>0$, the desired identity \eqref{eq:goalgp} holds.
\end{proof}

\subsection{Distorted Fourier transforms from distorted Fourier bases}
Recall that in Definition \ref{def00}, we introduce the following Fourier transforms transformations: $\vec{u}(k)\in\mathcal{S} \subset L^{2}(\mathbb{R},\mathbb{C}^{2}),$ we define the following linear operators 
\begin{align}
\label{DisF1}
    \hat{F}_{\omega}\left(\overrightarrow{u}\right)(x)\coloneqq &\frac{1}{\sqrt{2\pi}}\int_{\mathbb{R}}\frac{1}{s(k)}\left[\mathcal{F}_\omega(x,k)\,\,\sigma_{1} \mathcal{F}_\omega(x,k)\right]\overrightarrow{u}(k)\,dk\in L^{2}(\mathbb{R},\mathbb{C}^{2}),\\ \label{DisG1}
    \hat{G}_{\omega}\left(\overrightarrow{u}\right)(x)\coloneqq &\frac{1}{\sqrt{2\pi}}\int_{\mathbb{R}}\frac{1}{s({-}k)}\left[\mathcal{G}_\omega(x,{-}k)\,\,\sigma_{1} \mathcal{G}_\omega(x,{-}k)\right]\overrightarrow{u}(k)\,dk \in L^{2}(\mathbb{R},\mathbb{C}^{2}).\\ \nonumber
\end{align}

\par To invert the operators \eqref{DisF1} and \eqref{DisG1}, we consider the following operators. 
\begin{definition}
For any $\overrightarrow{u}\in L^{2}(\mathbb{R},\mathbb{C}^{2}),$ we define the bounded operators $F^{*}_{\omega},\,G^{*}_{\omega}:L^{2}(\mathbb{R},\mathbb{C}^{2})\to L^{2}(\mathbb{R},\mathbb{C}^{2})$ by
\begin{align*}
    F^{*}_{\omega}(\overrightarrow{u})(k)\coloneqq & \frac{1}{\sqrt{2\pi}}\int_{\mathbb{R}} \begin{bmatrix}
        \mathcal{F}^{t}_{\omega}(x,{-}k)\\
        \left(\sigma_{1}\mathcal{F}_{\omega}(x,{-}k)\right)^{t}
    \end{bmatrix}\overrightarrow{u}(x)\,dx\\
    G^{*}_{\omega}(\overrightarrow{u})(k)\coloneqq & \frac{1}{\sqrt{2\pi}}\int_{\mathbb{R}} \begin{bmatrix}
        \mathcal{G}^{t}_{\omega}(x,k)\\
        \left(\sigma_{1}\mathcal{G}_{\omega}(x,k)\right)^{t}
    \end{bmatrix}\overrightarrow{u}(x)\,dx,\\
\end{align*}
where for any vector $\overrightarrow{h}\in \mathbb{C}^{2}$ the expression $\overrightarrow{h}^{t}$ means its transport or the representation of $\overrightarrow{h}$ as a row vector of $\mathbb{C}^{2}.$  
\end{definition}
We also record the free Fourier transforms:

\begin{definition}
    For any $\overrightarrow{u}\in L_x^{2}(\mathbb{R},\mathbb{C}^{2}),$ we define the linear operator
\begin{align*}
  F_0^*\left(\overrightarrow{u}\right) (k)=&\frac{1}{\sqrt{2\pi}}\int_{\mathbb{R}} e^{ikx}\overrightarrow{u}(x)\,dx
\end{align*}for any $k\in\mathbb{R}$.
For any $\overrightarrow{u}\in L_k^{2}(\mathbb{R},\mathbb{C}^{2}),$ we define the adjoint linear operator
\begin{align*}
  F_{0}\left(\overrightarrow{u}\right)(x)=&\frac{1}{\sqrt{2\pi}}\int_{\mathbb{R}} e^{{-}ikx}\overrightarrow{u}(k)\,dk,
\end{align*}    
for any $x\in\mathbb{R}.$
\end{definition}
As a consequence of Lemma \ref{Asy1sol1}, we have the following corollary; see also Proposition $6.9$ from \cite{KriegerSchlag}.
\begin{corollary}\label{Asy1sol2}
 If $\psi\in P_{d,\omega},$ then 
$
     F_{\omega}^{*}(\sigma_{3}\psi)=G_{\omega}^{*}(\sigma_{3}\psi)=0.
$
\end{corollary}

From Lemma \ref{lem:pq1}, we get important inversions in the frequency side.
\begin{lemma}\label{leper}
We have the following identities in the frequency side
\begin{align*}
    \sigma_{3}F^{*}_{\omega}\sigma_{3}\hat{G}_{\omega}=&\mathrm{Id},\\
    \sigma_{3}G^{*}_{\omega}\sigma_{3}\hat{F}_{\omega}=&\mathrm{Id}.
\end{align*}
\end{lemma}
\begin{proof}
Using the identities $\mathcal{H}_{\omega}\mathcal{F}_{\omega}=(\omega+k^{2})\mathcal{F}_{\omega},\,\mathcal{H}_{\omega}\sigma_{1}\mathcal{G}_{\omega}={-}(\omega+k^{2})\sigma_1\mathcal{G}_{\omega}$ and $\sigma_{3}\mathcal{H}_{\omega}\sigma_{3}=\mathcal{H}_{\omega}^{*},$ and the asymptotics \eqref{asy1}, \eqref{asy2}, \eqref{asy3}, \eqref{asy4}, we can verify from integration by parts the following equation
\begin{equation*}
    \frac{1}{2\pi}\int_{\mathbb{R}}\left\langle \mathcal{F}_{\omega}(x,k),\sigma_{3}\sigma_{1}\mathcal{G}_{\omega}(x,\ell) \right\rangle \,dx=0 \text{, for all $\ell,\,k\in \mathbb{R}.$}
\end{equation*}
Consequently, using the equation \eqref{pq1} and the definition of the operators $F^*_{\omega}$ and $\hat{G}_{\omega},$ we can deduce the identity 
\begin{equation*}
\sigma_{3}F^{*}_{\omega}\sigma_{3}\hat{G}_{\omega}\left(
\begin{bmatrix}
g(k)\\
f(k)
\end{bmatrix}\right)=\begin{bmatrix}
g(k)\\
f(k)
\end{bmatrix},    
\end{equation*}
for any element $\begin{bmatrix}
    g(k)\\
    f(k)
\end{bmatrix}\in L^{2}_{k}(\mathbb{R},\mathbb{C})$ belonging to the domain of $\hat{G}_{\omega}.$
\end{proof}

Next, we consider the inversion formula in the physical side.
\begin{lemma}\label{FGID}
We have the following identities
\begin{align}\label{id0101}
   \hat{F}_{\omega}\sigma_{3}G^{*}_{\omega} \sigma_{3}=&P_{e,\omega},\\
   \hat{G}_{\omega} \label{id0202}\sigma_{3}F^{*}_{\omega} \sigma_{3}=&P_{e,\omega}.
\end{align}
\end{lemma}
\begin{proof}
First, using Lemma \ref{lem:PeCo},
we can verify for any $f\in L^{2}_{x}(\mathbb{R},\mathbb{C}^{2})$ that 
\begin{align*}
P_{e,\omega}(f)(x)=&\frac{1}{2\pi}\int_{0}^{{+}\infty}\mathcal{F}_{\omega}(x,k)\left\langle f,\sigma_{3}e_{+}(\cdot,k) \right\rangle\,dk+\frac{1}{2\pi}\int_{{-}\infty}^{0}\mathcal{G}_{\omega}(x,{-}k)\left\langle f,\sigma_{3}e_{+}(\cdot,k) \right\rangle\,dk \\
&{-}\frac{1}{2\pi}\int_{0}^{{+}\infty}\sigma_{1}\mathcal{F}_{\omega}(x,k)\left\langle f,\sigma_{3}e_{-}(\cdot,k) \right\rangle\,dk-\frac{1}{2\pi}\int_{{-}\infty}^{0}\sigma_{1}\mathcal{G}_{\omega}(x,{-}k)\left\langle f,\sigma_{3}e_{-}(\cdot,k) \right\rangle\,dk .
\end{align*}
Furthermore, Plancherel's Theorem and the asymptotics \eqref{asy1}, \eqref{asy2}, \eqref{asy3}, \eqref{asy4} imply that there exists a constant $C>1$ satisfying for any $f\in L^{2}(\mathbb{R},\mathbb{C}^{2})$
\begin{equation*}
    \max_{\pm}\norm{\left\langle f,\sigma_{3}e_{\pm}(\cdot,k)\right\rangle}_{L^{2}_{k}(\mathbb{R})}\leq C\norm{f}_{L^{2}_{x}(\mathbb{R})}.
\end{equation*}
\par Consequently, considering
\begin{equation*}
\begin{bmatrix}
    g_{1}(k)\\
    g_{2}(k)
\end{bmatrix}=
\begin{bmatrix}
s(\vert k\vert)\left\langle f,\sigma_{3}e_{+}(\cdot,k) \right\rangle \\
{-}s(\vert k\vert)\left\langle f,\sigma_{3}e_{-}(\cdot,k) \right\rangle
\end{bmatrix}\in L^{2}(\mathbb{R},\mathbb{C}^{2}),    
\end{equation*}
the following identity holds
\begin{equation*}
 P_{e,\omega}(f)(x)=\hat{F}_{\omega}\left(
 \begin{bmatrix}
   g_{1}(k)1_{[0,{+}\infty)}(k)\\
   g_{2}(k)1_{[0,{+}\infty)}(k)
 \end{bmatrix}
 \right)(x)+ \hat{G}_{\omega}\left(
 \begin{bmatrix}
g_{1}(k)1_{({-}\infty,0]}(k)\\
g_{2}(k)1_{({-}\infty,0]}(k)
 \end{bmatrix}
 \right)(x).  
\end{equation*}
In particular, using Remark \ref{re-infty}, we can verify the following equation.
\begin{equation*}
    P_{e,\omega}(f)(x)=\hat{G}_{\omega}\left(\frac{1}{s(k)}\begin{bmatrix}
   g_{1}(k)1_{[0,{+}\infty)}(k)\\
   g_{2}(k)1_{[0,{+}\infty)}(k)
 \end{bmatrix}-\frac{r(k)}{s(k)}\begin{bmatrix}
   g_{1}({-}k)1_{[0,{+}\infty)}({-}k)\\
   g_{2}({-}k)1_{[0,{+}\infty)}({-}k)
 \end{bmatrix}+\begin{bmatrix}
g_{1}(k)1_{({-}\infty,0]}(k)\\
g_{2}(k)1_{({-}\infty,0]}(k)
 \end{bmatrix}\right)(x).
\end{equation*}
Therefore, since $\frac{s(k)}{s(\pm k)}\in L^{\infty}_{k}(\mathbb{R}),$ we can verify from the definition of $(g_{1}(\lambda),g_{2}(\lambda))$ that $P_{c,\omega}(f)$ is an element of $\Raa \hat{G}_{\omega}=\Raa \hat{F}_{\omega}$ for all $f\in L^{2}(\mathbb{R},\mathbb{C}^{2}).$
\par In conclusion, since Lemma \ref{leper} implies that
\begin{equation*}
    \hat{F}_{\omega}\sigma_{3}G^{*}_{\omega}\sigma_{3}\left(\hat{F}_{\omega}\left(\begin{bmatrix}
        g(k)\\
        f(k)
    \end{bmatrix}\right)\right)=\hat{F}_{\omega}\left(\begin{bmatrix}
        g(k)\\
        f(k)
    \end{bmatrix}\right)
\end{equation*}
for any $\begin{bmatrix}
    g(k)\\
    f(k)
\end{bmatrix}\in L^{2}(\mathbb{R},\mathbb{C}^{2})$ belonging to the domain of $\hat{F}_{\omega},$ we can deduce from Lemma \ref{Asy1sol1} and Corollary \ref{Asy1sol2} that \eqref{id0101} is true. The proof of \eqref{id0202} is completely analogous.
\end{proof}

Next, we study the mapping properties of distorted Fourier transforms.

It is well-known from \cite{perelmanasym} that there exists $c_{\omega}>0$ satisfying
\begin{equation}\label{coerchatF}
    \norm{\hat{F}_{\omega}(\overrightarrow{u})(x)}_{L^{2}_{x}(\mathbb{R})}\geq c_{\omega}\norm{\overrightarrow{u}(k)}_{L^{2}_{k}(\mathbb{R})} \text{, for all $\overrightarrow{u}$ in the domain of $\hat{F}_{\omega},$}
\end{equation}
see also Lemma \ref{leper} below.
Furthermore, using the asymptotic behavior of \eqref{asy1}, \eqref{asy2}, \eqref{asy3} and \eqref{asy4}, we can verify from the Plancherel theorem the following proposition.
\begin{lemma}\label{kphi}
If $n\in\mathbb{N},$ there exists $C_{n,\omega}>1$ satisfying for all $\overrightarrow{u}$ in the domain of $\hat{F}_\omega$
the estimate
\begin{equation}\label{ldecay}
    \norm{\frac{d^{n}}{dx^{n}}\hat{F}_{\omega}(\overrightarrow{u})(x)}_{L^{2}_{x}(\mathbb{R})}\leq C_{n,\omega}\left[\norm{\overrightarrow{u}(k)(1+\vert k\vert)^{n}}_{L^{2}_{k}(\mathbb{R})}+\norm{\hat{F}_{\omega}(\overrightarrow{u})(x)}_{L^{2}_{x}(\mathbb{R})}\right].
\end{equation}
\end{lemma}

\begin{remark}\label{re-infty}
 From the formula \eqref{forp1}, we can verify that
 \begin{align}
   \hat{F}_{\omega}\left(\overrightarrow{u}\right)(x)=&\int_{\mathbb{R}}\frac{\left[\mathcal{G}_\omega(x,{-}k)\,\,\sigma_{1} \mathcal{G}_\omega(x,{-}k)\right]}{\overline{s(k)}s(k)}\overrightarrow{u}(k)-\left[\mathcal{G}_\omega(x,k)\,\,\sigma_{1} \mathcal{G}_\omega(x,k)\right]\frac{\overline{r(k)}}{\overline{s(k)} s(k)}\overrightarrow{u}(k)\,dk\\
   &=\int_{\mathbb{R}}\frac{\left[\mathcal{G}_\omega(x,{-}k)\,\,\sigma_{1} \mathcal{G}_\omega(x,{-}k)\right]}{s({-}k)}\left[\frac{\overrightarrow{u}(k)}{s(k)}-\frac{r(k)}{s(k)}\overrightarrow{u}({-}k)\right]\,dk\\
   &=\hat{G}_{\omega}\left[\frac{\overrightarrow{u}(k)}{s(k)}-\frac{r(k)}{s(k)}\overrightarrow{u}({-}k)\right](x).
\end{align}
\end{remark}

\begin{remark}\label{weightremark}
From the asymptotic behaviors \eqref{asy1} and \eqref{asy2}, we can verify that there exists $C>1$ satisfying
 \begin{equation*}
  \norm{\frac{d}{dk}G^{*}_{\omega}(\overrightarrow{u})(k)}_{L^{2}_{k}(\mathbb{R})}\leq C \norm{\vert x\vert\overrightarrow{u}(x)}_{L^{2}_{x}(\mathbb{R})},\, \norm{\frac{d^{2}}{dk^{2}}G^{*}_{\omega}(\overrightarrow{u})(k)}_{L^{2}_{k}(\mathbb{R})}\leq C \norm{\vert x\vert^{2}\overrightarrow{u}(x)}_{L^{2}_{x}(\mathbb{R})}.
 \end{equation*}
Moreover, since the identities \eqref{asy1} is true for all $x\geq 0$ and \eqref{asy2} is true for all $x\leq 0,$ we can also verify using integration by parts and  the Plancherel theorem for any $n\in\mathbb{N}$ that
\begin{equation*}
\norm{k^{n}G^{*}_{\omega}(\overrightarrow{u})(k)}_{L^{2}_{k}(\mathbb{R})}\leq C \norm{\overrightarrow{u}(x)}_{H^{n}_{x}(\mathbb{R})},\, \norm{k^{n}F^{*}_{\omega}(\overrightarrow{u})(k)}_{L^{2}_{k}(\mathbb{R})}\leq C \norm{\overrightarrow{u}(x)}_{H^{n}_{x}(\mathbb{R})}.
 \end{equation*}

\end{remark}
\begin{remark}\label{f*r}
Using the asymptotics \eqref{asy1} and \ref{asy3}, we can verify from the Cauchy-
Schwarz inequality, for any $N\gg 1$ that there exists a number $\beta>0$ depending only on $\omega$ satisfying for all $N>0$ large enough 
   \begin{align*}
F_{\omega}^{*}\left(\overrightarrow{u}(x)\chi_{\{x\geq N\}}\right) (k)=&\int_{N}^{{+}\infty}\left(\overrightarrow{u}(x)e^{{-}ikx}+r({-}k)\overrightarrow{u}(x)e^{ikx}\right)\,dx+O\left(\frac{\norm{\overrightarrow{u}(x)\chi_{\{x\geq N\}}}_{L^{2}}e^{{-}\beta N}}{(1+\vert k\vert)}\right),\\
F^{*}_{\omega}\left(\overrightarrow{u}(x)\chi_{\{x< {-}N\}}\right)(k)=&s(k)\int_{{-}\infty}^{{-}N}\overrightarrow{u}(x)e^{{-}ikx}\,dx+O\left(\frac{\norm{\overrightarrow{u}(x)}_{L^{2}}e^{{-}\beta N}}{(1+\vert k\vert)}\right)
   \end{align*}
for a universal constant  in the error terms.
\end{remark}

Moreover, by comparing the distorted transforms with the flat transform, we obtain the following.
\begin{lemma}\label{appFourier}
 For any $n\in\mathbb{N}$ and $m\in\mathbb{N}_{\geq 1},$ there exist parameters  $\gamma_{n,m}>0,\, K>1$ such that if $N\gg 1$ and $1\leq j\leq m$ then
\begin{align}\label{asrr}
     & \max_{0\leq \ell\leq n}\int_{N}^{{+}\infty} \left\vert (1+\vert x\vert)^{\ell}\frac{d^{j}}{dx^{j}}\left[\hat{F}_\omega\left(\overrightarrow{u}(k)\right)(x)-\frac{1}{\sqrt{2\pi}}\int_{\mathbb{R}}e^{ik x}\overrightarrow{u}(k)\,dk\right]\right\vert^{2}\,dx  \\
     & \quad \quad \quad \quad\leq Ke^{{-}\gamma_{n,m} N}\norm{\overrightarrow{u}(k)(1+\vert k\vert)^{m-1}}_{L^{2}_{x}(\mathbb{R})}^{2},\\ \nonumber
    & \max_{0\leq \ell\leq n} \int_{{-}\infty}^{{-}N} \left\vert(1+\vert x\vert)^{\ell}\frac{d^{j}}{dx^{j}}\left[ \hat{G}_\omega\left(\overrightarrow{u}(k)\right)(x)-\frac{1}{\sqrt{2\pi}}\int_{\mathbb{R}}e^{ik x}\overrightarrow{u}(k)\,dk\right]\right\vert^{2}\,dx
     \\
    & \quad \quad \quad \quad \leq  Ke^{{-}\gamma_{n,m}N}\norm{\overrightarrow{u}(k)(1+\vert k\vert)^{m-1}}_{L^{2}_{x}(\mathbb{R})}^{2},
\end{align}
 for any $\vec{u}$ in the domain of $\hat{F}_\omega,\,\hat{G}_\omega$ and any $0\leq \ell\leq n .$
\end{lemma}
\begin{proof}
   Because of the asymptotic behaviors \eqref{asy1} and \ref{asy3}, we can verify, using the Cauchy–Schwarz inequality and for any $N\gg 1$ that
   \begin{align*}
      \int_{N}^{{+}\infty} \left\vert \hat{F}_\omega\left(\overrightarrow{u}\right)(x)-\frac{1}{\sqrt{2\pi}}\int_{\mathbb{R}}e^{ik x}\overrightarrow{u}(k)\,dk\right\vert^{2}\,dx\leq &C\int_{N}^{{+}\infty}e^{{-}\gamma x}\left\vert\int_{\mathbb{R}}\frac{\vert\overrightarrow{u}(k) \vert}{1+\vert k\vert}\,dk\right\vert^{2}\,dx\\
     \leq & Ke^{{-}\gamma N}\norm{\overrightarrow{u}}_{L^{2}}^{2},
   \end{align*}
for a universal constant $K>1.$ 
\par Next, it is well-known that
\begin{align}\label{psed1}
\frac{\partial^{\ell+n}}{\partial x^{\ell}\partial k^{n}}\frac{e^{{-}ikx}\mathcal{G}(x,{-}k)}{\overline{s(k)}}=&\frac{\partial^{\ell+n}}{\partial x^{\ell}\partial k^{n}}\begin{bmatrix}
        1\\
        0
    \end{bmatrix}+O\left(\frac{e^{\gamma x}}{(1+\vert k \vert)^{1+n}}\right) \text{, for all $x<{-}1,$}\\
    \label{psed2}\frac{\partial^{\ell+n}}{\partial x^{\ell} \partial k^{n}}\frac{e^{{-}ikx}\mathcal{F}(x,k)}{s(k)}=& \frac{\partial^{\ell+n}}{\partial x^{\ell} \partial k^{n}}
    \begin{bmatrix}
        1\\
        0
    \end{bmatrix}+O\left(\frac{e^{{-}\gamma x}}{(1+\vert k \vert)^{1+n}}\right) \text{, for all $x >1,$}
\end{align}
see \cite{Busper1} or Sections $4$ and $5$ of \cite{collotger} for example.
Consequently, we can verify when $x> 1$ that
\begin{equation*}
    \frac{d^{j}}{dx^{j}}\left[ \hat{F}_\omega (\overrightarrow{u}(k))(x)-\frac{1}{\sqrt{2\pi}}\int_{\mathbb{R}}e^{ikx}\overrightarrow{u}(k)\,dk\right]=\int_{\mathbb{R}}\frac{d^{j}}{dx^{j}}\left[e^{ikx}c(k,x)\overrightarrow{u}(k)\,dk\right]\,dk,
\end{equation*}
such that $c(k,x)$ is a function satisfying
$
    \left\vert \frac{\partial^{\ell+n}c(k,x)}{\partial k^{\ell}\partial x^{n}} \right\vert=O\left(\frac{e^{{-}\gamma x}}{(1+\vert k\vert)^{1+\ell}}\right)
$ for any $\ell,\,n\in\mathbb{N}$. Then we  can deduce using the theory of pseudodifferential operators and the chain rule of derivative for any $j\geq 1$ that
\begin{equation*}
    \norm{\mathrm{1}_{\{x\geq N\}}(x)e^{\gamma x}\hat{F}_\omega(\overrightarrow{u}(k))(x)-\frac{\mathrm{1}_{\{x\geq N\}}(x)e^{\gamma x}}{\sqrt{2\pi}}\int_{\mathbb{R}}e^{ikx}\overrightarrow{u}(k)\,dk}_{H^{j}_{x}(\mathbb{R})}\leq C_{j}\norm{(1+\vert k\vert)^{j-1}\overrightarrow{u}(k)}_{L^{2}_{k}(\mathbb{R})}, 
\end{equation*}
for a constant $C_{j}>1$ depending only on $j.$
Consequently, the proof of the first estimate for any $j\geq 1$ follows from the Cauchy-Schwarz.
\par The proof of the second estimate is completely analogous. 
\end{proof}
\begin{lemma}\label{sobolevdecayofG}
 Let $\overrightarrow{u}(k)=\begin{bmatrix}
     u_{1}(k)\\
     u_{2}(k)
 \end{bmatrix}\in L^{2}_{k}(\mathbb{R},\mathbb{C}^2)$ be an element of the domain of $\hat{G}_\omega.$ For any $v\in\mathbb{R}$ and any $\epsilon\in (0,1),$ there exists a parameter $C_{\epsilon,v}>1$ satisfying
 \begin{multline*}
\norm{\frac{d^{2\ell}}{dx^{2\ell}}\left[\sigma_{3}\hat{G}_{\omega}(\overrightarrow{u}(k))(x)\right]-({-}1)^{\ell}\hat{G}_{\omega}\left(\begin{bmatrix}
(k+v)^{2\ell}u_{1}(k)\\
(k-v)^{2\ell}u_{2}(k)
\end{bmatrix}\right)(x)}_{L^{2}_{x}}\\\leq C_{\epsilon,v}\left(\norm{[\hat{G}_{\omega}(\overrightarrow{u}(k))(x)}_{L^{2}_{x}(\mathbb{R})}+\max_{0\leq n\leq \ell-2}\norm{\begin{bmatrix}
(k+v)^{2n}u_{1}(k)\\
(k-v)^{2n}u_{2}(k)
\end{bmatrix}}_{L^{2}_{k}(\mathbb{R})}\right)+\epsilon   \norm{\begin{bmatrix}
(k+v)^{2\ell}u_{1}(k)\\
(k-v)^{2\ell}u_{2}(k)
\end{bmatrix}}_{L^{2}_{k}(\mathbb{R})}.
 \end{multline*}
\end{lemma}
In particular, Lemma \ref{sobolevdecayofG} implies the following corollary. 
\begin{corollary}\label{h2Gesti}
 There exists a constant $c>0$ such that if $\overrightarrow{u}\in L^{2}_{k}(\mathbb{R},\mathbb{C}^{2})$ belongs to the domain of $\hat{G}_{\omega},$ then the following inequality holds.
 \begin{equation*}
     \norm{\hat{G}_{\omega}(\overrightarrow{u}(k))(x)}_{H^{2}_{x}(\mathbb{R})}\geq c\norm{(1+k^{2})\overrightarrow{u}(k)}_{L^{2}_{k}(\mathbb{R})}.
 \end{equation*}
\end{corollary}
\begin{proof}[Proof of Lemma \ref{sobolevdecayofG}.]
 First, recalling $\mathcal{G}_{\omega}$ satisfies $\mathcal{H}_{\omega}\mathcal{G}_{\omega}(x,k)=(k^{2}+\omega)\mathcal{G}_{\omega}(x,k)$ and the asymptotic behavior \eqref{asy1}, \eqref{asy2}, it is not difficult to verify the following estimates
 \begin{align}\label{identit1dx2}
   \sigma_{3} \partial^{2}_{x}\mathcal{G}_{\omega}(x,k)=&{-}k^{2}\mathcal{G}_{\omega}(x,k)+\begin{bmatrix}
        U_{\omega}(x) & {-}W_{\omega}(x)\\
        W_{\omega}(x) & {-}U_{\omega}(x)
    \end{bmatrix}\mathcal{G}_{\omega}(x,k),\\ \nonumber
    \norm{\frac{d^{\ell}}{dx^{\ell}}\hat{G}_{\omega}(\overrightarrow{u}(k))(x)}_{L^{2}_{x}(\mathbb{R})}\leq &C_{\ell}\norm{(1+\vert k\vert^{\ell})\overrightarrow{u}(k)}_{L^{2}_{k}(\mathbb{R})} \text{, for a constant $C_{\ell}>1$ for any $\ell\in\mathbb{N}_{\geq 1}.$}
 \end{align}
We recall that the functions $U_{\omega}(x)$ and $W_{\omega}(x)$ are smooth and exhibit the following rapid decay
\begin{equation*}
    \left\vert\frac{d^{\ell}U_{\omega}(x)}{dx^{\ell}} \right\vert +  \left\vert\frac{d^{\ell}W_{\omega}(x)}{dx^{\ell}} \right\vert\leq C_{\ell}e^{{-}\alpha \vert x\vert } \text{, for a constant $\alpha>0$ and $C_{\ell}>1$ depending only on $\ell,$}
\end{equation*}
and it is well-known that
\begin{equation}\label{coerGG}
    \norm{\hat{G}_{\omega}(\overrightarrow{u})(x)}_{L^{2}_{x}(\mathbb{R})}\geq c_{\omega} \norm{
    \overrightarrow{u}(k)}_{L^{2}_{k}(\mathbb{R})} \text{, for a constant $c_{\omega}>0.$}
\end{equation}
Consequently,
\begin{align*}
    \sigma_{3}\frac{d^{2}}{dx^{2}}\hat{G}_{\omega}(\overrightarrow{u}(k))(x)+\hat{G}_{\omega}\left(
    \begin{bmatrix}
(k+v)^{2}u_{1}(k)\\
(k-v)^{2}u_{2}(k)
    \end{bmatrix}
    \right)(x)=&\hat{G}_{\omega}\left(
    \begin{bmatrix}
(2kv+v^{2})u_{1}(k)\\
({-}2kv+v^{2})u_{2}(k)
    \end{bmatrix}
    \right)(x)\\&{+}\begin{bmatrix}
        U_{\omega}(x) & {-}W_{\omega}(x)\\
        W_{\omega}(x) & {-}U_{\omega}(x)
    \end{bmatrix}\hat{G}_{\omega}(\overrightarrow{u}(k))(x),
\end{align*}
and Young's inequality implies that
\begin{align}\label{eqaxx1}
    \norm{\hat{G}_{\omega}(2kv\overrightarrow{u}(k))(x)}_{L^{2}_{x}(\mathbb{R})}^{2}\leq &Cv^{2}\norm{(1+\vert k\vert)\overrightarrow{u}(k)}_{L^{2}_{k}(\mathbb{R})}^{2}\\ \nonumber
    \leq & Cv^{2}  \left[\frac{1}{2\epsilon}+1\right]\norm{\overrightarrow{u}(k)}_{L^{2}_{k}(\mathbb{R})}^{2}{+}\frac{Cv^{2}\epsilon}{2}\norm{k^{2}\overrightarrow{u}(k)}_{L^{2}_{k}(\mathbb{R})}^{2}.
\end{align}
  Furthermore, we can deduce the following estimates using Young's inequality. 
  \begin{align}\label{eqaxxx2}
      \norm{k^{2}\overrightarrow{u}(k)}_{L^{2}_{k}(\mathbb{R})}\leq & \norm{\begin{bmatrix}
          (k+\frac{v}{2})^{2}u_{1}(k)\\
          (k-\frac{v}{2})^{2}u_{2}(k)
      \end{bmatrix}}_{L^{2}_{k}(\mathbb{R})}+2v\norm{\begin{bmatrix}
          ku_{1}(k)\\
          ku_{2}(k)
      \end{bmatrix}}_{L^{2}_{k}(\mathbb{R})}+\frac{v^{2}}{4}\norm{\vec{u}(k)}_{L^{2}_{k}(\mathbb{R})}\\ \nonumber
      \leq & \norm{\begin{bmatrix}
          (k+\frac{v}{2})^{2}u_{1}(k)\\
          (k-\frac{v}{2})^{2}u_{2}(k)
      \end{bmatrix}}_{L^{2}_{k}(\mathbb{R})}+2v\norm{\begin{bmatrix}
          (k+\frac{v}{2})u_{1}(k)\\
          (k-\frac{v}{2})u_{2}(k)
      \end{bmatrix}}_{L^{2}_{k}(\mathbb{R})}+\frac{5v^{2}}{4}\norm{\vec{u}(k)}_{L^{2}_{k}(\mathbb{R})}\\ \nonumber
     \leq & (1+\sqrt{2}\epsilon )\norm{\begin{bmatrix}
          (k+\frac{v}{2})^{2}u_{1}(k)\\
          (k-\frac{v}{2})^{2}u_{2}(k)
      \end{bmatrix}}_{L^{2}_{k}(\mathbb{R})}+(\frac{5v^{2}}{4}+\frac{\sqrt{2}v^{2}}{\epsilon})\norm{\vec{u}(k)}_{L^{2}_{k}(\mathbb{R})}.
  \end{align}
    Consequently, we can verify from estimates \eqref{eqaxx1} and \eqref{eqaxxx2} that the statement of Lemma \ref{sobolevdecayofG} is true when $\ell=1.$
\par Next, assuming \ref{sobolevdecayofG} is true for $\ell-1\geq 1,$ we can deduce using \eqref{coerGG} and the chain rule that
\begin{equation*}
    \norm{\frac{d^{2(\ell-1)}}{dx^{2(\ell-1)}}\left[V_{\omega}(x)\hat{G}_{\omega}(\overrightarrow{u})(x)\right]}_{L^{2}_{x}(\mathbb{R})}\leq C_{\ell,v}\left[\norm{\hat{G}_{\omega}(\overrightarrow{u})(x)}_{L^{2}_{x}(\mathbb{R})}+\norm{\begin{bmatrix}
(k+v)^{2(\ell-1)}u_{1}(k)\\
(k-v)^{2(\ell-1)}u_{2}(k)
    \end{bmatrix}}_{L^{2}_{k}(\mathbb{R})}\right],
\end{equation*}
for a parameter $C_{\ell,v}>1$ depending only on $\ell$ and $v.$
Moreover, if $\ell$ is larger than $1,$ Holder's inequality and Young's inequality imply for any $\epsilon\in (0,1)$ that
\begin{align*}
    \norm{ \begin{bmatrix}
(k+v)^{2(\ell-1)}u_{1}(k)\\
(k-v)^{2(\ell-1)}u_{2}(k)
    \end{bmatrix}}_{L^{2}_{k}(\mathbb{R})}^{2}\leq & \norm{\begin{bmatrix}
        (k+v)^{2\ell}u_{1}(k)\\
        (k-v)^{2\ell}u_{2}(k)
    \end{bmatrix}}_{L^{2}_{k}(\mathbb{R})}^{2-\frac{1}{\ell}}\norm{\begin{bmatrix}
u_{1}(k)\\
u_{2}(k)
    \end{bmatrix}}_{L^{2}_{k}(\mathbb{R})}^{\frac{1}{\ell}}\\
    \leq &
    \frac{(2\ell-1)\epsilon}{2\ell}\norm{\begin{bmatrix}
(k+v)^{2\ell}u_{1}(k)\\
(k-v)^{2\ell}u_{2}(k)
    \end{bmatrix}}_{L^{2}_{x}(\mathbb{R})}^{2}+\frac{1}{2\ell \epsilon^{2\ell-1}}\norm{
    \begin{bmatrix}
u_{1}(k)\\
u_{2}(k)
    \end{bmatrix}}_{L^{2}_{x}(\mathbb{R})}^{2},
\end{align*}
and  so we deduce the lemma by differentiating identity \eqref{identit1dx2}
 with respect to $x$  for $2(\ell-1)$ times.
 \end{proof}

\subsection{Schr\"odinger flows}
We consider the following linear partial differential equation
\begin{equation}\label{LSS}
    i \overrightarrow{\psi}_{t}-\mathcal{H}_{\omega}\overrightarrow{\psi}=0
\end{equation}where $\mathcal{H}_\omega$ is given by \eqref{H}.
In particular, since $\mathcal{F}_\omega(x,k),\, \mathcal{G}_\omega(x,{-}k)$ are solutions of \eqref{BB} in $L^{\infty},$ using notations from Definition \ref{def00}, it is standard to verify that
\begin{align}\label{fo1}
    \hat{F}_\omega(e^{{-}it(\Diamond^{2}+\omega)\sigma_{3}}\overrightarrow{u})(x)=&\frac{1}{\sqrt{2\pi}}\int_{\mathbb{R}}\left[\mathcal{F}_\omega(x,k)\,\,\sigma_{1} \mathcal{F}_\omega(x,k)\right]\frac{e^{{-}it(k^{2}+\omega)\sigma_{3}}\overrightarrow{u}(k)}{s(k)}\,dk, \\ \label{fo2}
    \hat{G}_\omega(e^{{-}it(\Diamond^{2}+\omega)\sigma_{3}}\overrightarrow{u})(x)=&\frac{1}{\sqrt{2\pi}}\int_{\mathbb{R}}\left[\mathcal{G}_\omega(x,{-}k)\,\,\sigma_{1} \mathcal{G}_\omega(x,{-}k)\right]\frac{e^{{-}it(k^{2}+\omega)\sigma_{3}}\overrightarrow{u}(k)}{s({-}k)}\,dk
\end{align}
are solutions of \eqref{LSS}, where $\Diamond$ denotes the dummy variable.

Using notations above, we have following  estimates for Schr\"odinger flows.
\begin{lemma}\label{lem:decayonepotential}
First of all, we have the $L^2$ bounds:
\begin{equation}
   \left\Vert \hat{F}_\omega(e^{{-}it(\Diamond^{2}+\omega)\sigma_{3}}\overrightarrow{u})(\cdot) \right\Vert_{L^2} \lesssim  \left\Vert \hat{F}_\omega(\overrightarrow{u})(\cdot) \right\Vert_{L^2}
\end{equation}
\begin{equation}
   \left\Vert \hat{G}_\omega(e^{{-}it(\Diamond^{2}+\omega)\sigma_{3}}\overrightarrow{u})(\cdot) \right\Vert_{L^2} \lesssim \left\Vert \hat{G}_\omega(\overrightarrow{u})(\cdot) \right\Vert_{L^2}.
\end{equation}

One has the following standard decay estimates:
\begin{equation}
   \left\Vert \hat{F}_\omega(e^{{-}it(\Diamond^{2}+\omega)\sigma_{3}}\overrightarrow{u})(\cdot) \right\Vert_{L^\infty} \lesssim \frac{1}{|t|^{\frac{1}{2}}} \left\Vert \hat{F}_\omega(\overrightarrow{u})(\cdot) \right\Vert_{L^1}
\end{equation}
\begin{equation}
   \left\Vert \hat{G}_\omega(e^{{-}it(\Diamond^{2}+\omega)\sigma_{3}}\overrightarrow{u})(\cdot) \right\Vert_{L^\infty} \lesssim \frac{1}{|t|^{\frac{1}{2}}} \left\Vert \hat{G}_\omega(\overrightarrow{u})(\cdot) \right\Vert_{L^1}.
\end{equation}
If the non-resonance condition (H4) from Theorem \ref{Decesti} holds, then one has local improved decay estimates 
\begin{equation}
   \left\Vert \langle x \rangle^{-1}\hat{F}_\omega(e^{{-}it(\Diamond^{2}+\omega)\sigma_{3}}\overrightarrow{u})(\cdot) \right\Vert_{L^\infty} \lesssim \frac{1}{|t|^{\frac{3}{2}}} \left\Vert \langle x \rangle \hat{F}_\omega(\overrightarrow{u})(\cdot) \right\Vert_{L^1}
\end{equation}
\begin{equation}
   \left\Vert \langle x \rangle^{-1} \hat{G}_\omega(e^{{-}it(\Diamond^{2}+\omega)\sigma_{3}}\overrightarrow{u})(\cdot) \right\Vert_{L^\infty} \lesssim \frac{1}{|t|^{\frac{3}{2}}} \left\Vert \langle x \rangle \hat{G}_\omega(\overrightarrow{u})(\cdot) \right\Vert_{L^1}.
\end{equation}
\end{lemma}
\begin{proof}
    Since $ \hat{F}_\omega(e^{{-}it(\Diamond^{2}+\omega)\sigma_{3}}\overrightarrow{u})(x)$ and $ \hat{G}_\omega(e^{{-}it(\Diamond^{2}+\omega)\sigma_{3}}\overrightarrow{u})(x)$ both solve the Schr\"odinger equation \eqref{LSS}, all these estimates follow from 
    Theorem 1.4 in  Li \cite{Lidecay},  and Lemma 6.11, Proposition 7.1, Proposition 8.1 in Krieger-Schlag \cite{KriegerSchlag} provided that 
\begin{equation}
    P_{e,\omega}\left(\hat{F}_\omega(e^{{-}it(\Diamond^{2}+\omega)\sigma_{3}}\overrightarrow{u})(x)\right)=\hat{F}_\omega(e^{{-}it(\Diamond^{2}+\omega)\sigma_{3}}\overrightarrow{u})(x)
\end{equation}  and
\begin{equation}
    P_{e,\omega}\left(\hat{G}_\omega(e^{{-}it(\Diamond^{2}+\omega)\sigma_{3}}\overrightarrow{u})(x)\right)=\hat{G}_\omega(e^{{-}it(\Diamond^{2}+\omega)\sigma_{3}}\overrightarrow{u})(x)
\end{equation}
which are direct consequences of Corollary \ref{Asy1sol2}.

Actually, with the construction of the distorted bases, Lemma \ref{2.1}, the standard pointwise decay can also be obtained as in   Goldberg-Schlag. \cite{GoSch}
\end{proof}
\begin{remark}
Using distorted bases and distorted Fourier transforms introduced in the last subsection, one can actually  compute pointwise decays like    Lemma 4.1 in Li-L\"uhrmann \cite{nls3soliton} with RHS in the $L^2$ based Sobolev spaces. These estimates should be more useful in the long-range scattering setting but we will not pursue them here.
\end{remark}

Finally, we record the following elementary observation.

\begin{lemma}\label{galil}
The following identity holds for any $\phi \in\mathscr{S}(\mathbb{R})$ and any $x,\,y,\,v\in\mathbb{R}$
\begin{equation*}
    \frac{1}{\sqrt{2\pi}}\int_{\mathbb{R}}e^{{-}itk^{2}}e^{ikx}\phi(k)\,dk=\frac{e^{\frac{ivx}{2}-i\frac{v^{2}t}{4}}}{\sqrt{2\pi}}\int_{\mathbb{R}}e^{i k (x-y-vt)}e^{{-}itk^{2}}\phi\left(k+\frac{v}{2}\right)e^{i y k}\,dk.
\end{equation*}
\end{lemma}
\begin{proof}
It follows directly from the change of variables $k \rightarrow k+\frac{v}{2}.$     
\end{proof}

\begin{remark}
 Using the scalar Galilei transforms, 
 \begin{equation*}
 \mathfrak{g}_{v,y,\gamma}(f)(t,x)=e^{i(\frac{vx}{2}-\frac{tv^{2}}{4}+\gamma)}f\left(x-vt-y\right),
\end{equation*} the identity from Lemma \ref{galil} above can be written as
\begin{equation}
 e^{it\partial_x^2}\check{\phi}=   \mathfrak{g}_{v,y,\gamma}^{-1}\left(e^{it\partial_x^2} \left(\mathfrak{g}_{v,y,\gamma}(\check{\phi})(0,\Diamond)\right)\right)(t,x)
\end{equation}which is the invariance of the Schr\"odinger flow under Galilei transforms.
\end{remark}

\section{Properties of the Hardy spaces $H^{2}(\mathbb{C}_{\pm})$}\label{sec:hardy}
\begin{definition}
Let $H^{2}\left(\mathbb{C}_{\pm}\right)$ be the Hardy space in the upper/lower half-plane. 
\par More precisely, $H^{2}\left(\mathbb{C}_{+}\right)$ is the set of all analytic functions $g$ in $\mathbb{C}_{+}$ satisfying
\begin{equation*}
\sup_{y>0}\norm{g(x+iy)}_{L^{2}_{x}(\mathbb{R})}<{+}\infty.
\end{equation*}
 Similarly, we can define $H^{2}\left(\mathbb{C}_{-}\right)$ as the set of analytic functions $g$ in $\mathbb{C}_{-}$ satisfying
\begin{equation*}
\sup_{y<0}\norm{g(x+iy)}_{L^{2}_{x}(\mathbb{R})}<{+}\infty.
\end{equation*}
\end{definition}

\begin{theorem}\label{h2rep}[Representation of $H^{2}(\mathbb{C}_{\pm}).$]
The function $f$ is in $H^{2}\left(\mathbb{C}_{+}\right)$ and the function $g$ is in $H^{2}\left(\mathbb{C}_{-}\right)$ if, and only if, there exist $f_{1}\in L^{2}\left([0,{+}\infty),dx\right)$ and $g_{1}\in L^{2}\left(({-}\infty,0],dx\right)$ satisfying
\begin{align*}
    f(x)=&\frac{1}{\sqrt{2\pi}}\int_{0}^{{+}\infty}f_{1}(k)e^{ikx}\,dk\\
    g(x)=&\frac{1}{\sqrt{2\pi}}\int_{{-}\infty}^{0}g_{1}(k)e^{ikx}\,dk,
\end{align*}
almost everywhere on $\mathbb{C}_{+}$ and $\mathbb{C}_{-}$ respectively.
\end{theorem}    

\begin{remark}\label{rrr}
 In particular, the theorem above implies the existence of  the orthogonal projections $P_{+},\,P_{-}$ satisfying
 \begin{equation*}
     P_{+}\left(L^{2}(\mathbb{R})\right)=H^{2}\left(\mathbb{C}_{+}\right),\, P_{-}\left(L^{2}(\mathbb{R})\right)=H^{2}\left(\mathbb{C}_{-}\right),\,P_{+}\oplus P_{-}=P_{+}+P_{-}=\mathrm{Id}.
 \end{equation*}
   Moreover, we have that $P_{+},\,P_{-}$ can be defined by the following maps
    \begin{align*}
        P_{+}(f)(x)=&\frac{1}{2\pi i}\lim_{\epsilon\to{+}0}\int_{\mathbb{R}}\frac{f(k)}{k-x-i\epsilon}\,dk\\
        P_{-}(f)(x)=&{-}\frac{1}{2\pi i}\lim_{\epsilon\to{+}0}\int_{\mathbb{R}}\frac{f(k)}{k-x+i\epsilon}\,dk.
    \end{align*}
In terms of Fourier transforms, they are explicitly given by
\begin{align}
     P_{+}(f)(x)&=\frac{1}{2\pi}\int_{{-}\infty}^{0}e^{{-}ikx}\left[\int_{\mathbb{R}}e^{iky}f(y)\,dy\right]dk,\label{eq:FdefPp}\\
      P_{-}(f)(x)&=\frac{1}{2\pi}\int_{0}^{{+}\infty}e^{{-}ikx}\left[\int_{\mathbb{R}}e^{iky}f(y)\,dy\right]dk\label{eq:FdefPm}.
\end{align}

\end{remark}
\par Furthermore, we will need the following proposition to prove the main result of this article.

\begin{lemma}\label{+-interact}
Let $f_{\pm}$ be in $H^{2}(\mathbb{C}_{\pm}),$
and let $s(k),\,r(k)$ be two analytic functions on the strip $$\vert\I k \vert \leq \delta$$ satisfying \eqref{asyreftr}. If $y_{0}>1,\,h_{0}\in\mathbb{R},$ then
\begin{align}\label{p-1e}
\norm{ P_{\mp}\left(e^{\pm iy_{0}k}\left[s(k+h_{0})f_{\pm}(k)\right]\right)(x)}_{L^{2}(\mathbb{R},dx)}\leq &Ce^{{-}\frac{99}{100}\delta y_{0}}\norm{f_\pm}_{H^{2}(\mathbb{C}_\pm)},\\ \label{p-2e}
     \norm{ P_{\mp}\left(e^{\pm iy_{0}k}\left[r(k+h_{0})f_{\pm}(k)\right]\right)(x)}_{L^{2}(\mathbb{R},dx)}\leq &Ce^{{-}\frac{99}{100}\delta y_{0}}\norm{f_\pm}_{H^{2}(\mathbb{C}_\pm)}.
\end{align}
\end{lemma}

\begin{remark}\label{r+-}
Considering $r(k)$ a smooth function satisfying \eqref{asyreftr}, it is not difficult to verify that the function
\begin{equation*}
 r_{\epsilon}(k)=\int_{\mathbb{R}}e^{{-}\pi y^{2}}r(k-\epsilon y)\,dy,   
\end{equation*}
is analytic and bounded on the strip $\vert \I{k}\vert \leq \delta.$ 
 Moreover, for any $\delta_{0} \in (0,1),$ there exists $\epsilon_{0}>0,\,\delta_{1}\in (0,1)$ such that the following estimates hold for any $\epsilon\leq \epsilon_{0}$.
\begin{align*}
  \left\vert r_{\epsilon}(k)-r(k) \right\vert\leq  \delta_{0},\,\,\,
  \left\vert \frac{d^{\ell}r_{\epsilon}(k)}{dk^{\ell}} \right\vert\leq  \frac{C_{\epsilon_{0},\ell}}{1+\vert k\vert^{\ell}} \text{, for any $k\in\mathbb{R}$ and $\ell\in\mathbb{N},$}\\
  \left\vert r_{\epsilon}(k+ih)\right\vert\leq\frac{C_{\epsilon_{0},\ell}}{1+\vert k\vert} \text{, for all $k\in \mathbb{R}$ and $h\in (0,\delta_{1}),$}
\end{align*}
where $C_{\epsilon_{0},\ell}>0$ is a parameter depending only on $\epsilon_{0}$ and $\ell.$
\end{remark}
\begin{proof}[Proof of Lemma \ref{+-interact}.]
First, because of an argument of symmetry, it is enough to prove the three inequalities considering only $f_{+}\in H^{2}\left(\mathbb{C}_{+}\right)$. We also recall that $f_{+}(x)\in H^{2}(\mathbb{C}_{+})$ is equivalent to $f_{+}({-}x)\in H^{2}(\mathbb{C}_{-}).$ Therefore, we consider $f=f_{+}$ in the following paragraphs of this proof. 
\par Next, from Theorem \ref{h2rep}, we can assume that $f\in L^{\infty}$ by density. For simplicity, we consider \eqref{p-1e} with $h_0=0$. From Remark \ref{rrr}, we have that
 \begin{equation*}
P_{-}\left(e^{iy_{0}k}\left[s(k)f(k)-f(k)\right]\right)(x)= {-}\frac{1}{2\pi i}\lim_{\epsilon\to{+}0}\int_{\mathbb{R}}\frac{e^{iy_{0}k}f(k)}{(k-x+i\epsilon)}(s(k)-1)\,dk.
 \end{equation*}
\par Next, using that $s(k)=1+O\left(\vert  k \vert^{-1}\right),$ we observe that if $0<\delta_{1}\leq \delta,$ then
\begin{equation*}
\left\vert\int_{0}^{\delta_{1}}\frac{e^{iy_{0}(N+ik)}f(N+ki)}{(N+ik-x+i\epsilon)}\left(s(N+ik)-1\right)\,dk\right\vert\leq C\frac{\norm{f}_{L^{\infty}}}{\left\vert(N-x) N \right\vert}\delta_{1}
\end{equation*}
\par Consequently, for any $\epsilon \in(0,\delta),$ we have that
\begin{equation*}
    \lim_{N\to {+}\infty }\left\vert\int_{0}^{\delta_{1}}\frac{e^{iy_{0}(N+ik)}f(N+ki)}{(N+ik-x+i\epsilon)}\left(s(N+ik)-1\right)\,dk\right\vert=0.
\end{equation*}
\par Therefore,  the Cauchy-Goursat integral theorem implies for any $\epsilon>0$ that
\begin{multline*}
{-}\frac{1}{2\pi i}\int_{\mathbb{R}}\frac{e^{iy_{0}k}f(k)}{(k-x+i\epsilon)}(s(k)-1)\,dk\\=
{-}\frac{1}{2\pi i}\int_{\mathbb{R}}\frac{e^{iy_{0}(k+i\delta_{1})}f(k+i\delta_{1})}{(k-x+i(\epsilon+\delta_{1}))}(s(k+i\delta_{1})-1)\,dk=O\left(C(\delta_{1})e^{{-}y_{0}\delta_{1}}\norm{f}_{L^{2}(\mathbb{R})}\right),
\end{multline*}
because $\norm{f(\Diamond+ih)}_{L^{2}(\mathbb{R})}\leq \norm{f}_{L^{2}(\mathbb{R})}$ for any $h\geq 0,$ see Theorem \ref{h2rep}.
\par Next, since $f\in H^{2}(\mathbb{C}_{+})$ and $y_{0}>0,$ we have that $e^{iy_{0}k}f(k)\in H^{2}(\mathbb{C}_{+}),$ so
\begin{equation}\label{id11}
P_{-}\left(e^{iy_{0}k}f(k)\right)=0.
\end{equation}
\par In conclusion, taking $\epsilon$ converging to $0$ and using identity \eqref{id11}, we obtain the estimate \eqref{p-1e} with $h_0$. The general case of \eqref{p-1e} follows in the same manner. The proof of the estimate \eqref{p-2e}  is completely analogous, but in this case, we use $r(k+h) = O \left(\vert k\vert ^{{-}1}\right)$ for any fixed $h\in\mathbb{R}.$ 
\end{proof}
\section{Properties of $\mathcal{S}(0)$}\label{sec:S0}
In this section, we study the properties of the map $\mathcal{S}(t)$ defined in Definition \ref{s0def}. 
The map $\mathcal{S}(\vec{\phi})(t,x)$ can be identified as a small perturbation of $e^{i\frac{\sigma_{3}v_{\ell}x}{2}}\hat{F}_{\omega_{\ell}}\left(\vec{\phi}_{\ell}\right)(x-y_{\ell})$ when $x$ belongs to a large interval containing $y_{\ell}.$ In particular, when $t$ approaches ${+}\infty,$ the function $\mathcal{S}(\vec{\phi})(t,x)$ behaves as a finite sum of  functions
\begin{equation*}
    \chi_{P_{\ell}(t)}(x)e^{i\left(\frac{v_{\ell}x}{2}-\frac{v_{\ell}^{2}t}{4}+\omega_\ell t+\gamma_{\ell}\right)\sigma_{3}}\hat{G}_{\omega_{\ell}}\left(
   e^{{-}it(k^{2}+\omega_\ell)\sigma_{3}}e^{{-}i\gamma_{\ell}\sigma_{3}} \begin{bmatrix}
       e^{iy_{\ell}k}\phi_{1,\ell}\left(k+\frac{v_{\ell}}{2}\right)\\
       e^{iy_{\ell}k}\phi_{2,\ell}\left(k-\frac{v_{\ell}}{2}\right)
    \end{bmatrix}\right)(x-y_{\ell}-v_{\ell}t),
\end{equation*}
for $y_{0}={+}\infty,\,v_{0}=0,$ and $y_{m+1}={-}\infty,\,v_{m+1}=0,$ let
\begin{equation*}
    P_{\ell}(t)=\left\{x\in\mathbb{R}\big\vert\,\, \frac{y_{\ell}+y_{\ell+1}+v_{\ell}t+v_{\ell+1}t}{2}<x<\frac{y_{\ell}+y_{\ell-1}+v_{\ell}t+v_{\ell-1}t}{2}\right\}
\end{equation*}
plus a remainder $r(t,x)$ satisfying $\lim_{t\rightarrow\infty}\norm{r(t)}_{L^{2}_{x}(\mathbb{R})}=0.$ More precisely, we have the following estimate.

\begin{lemma}\label{localS0}
From  conditions $a)$, $b)$, $c)$,  and the formula \eqref{eq:Sphi} from Definition \ref{s0def},  the following estimate holds 
for any $\ell\in\{1,2,\,...,m\}$:
\begin{multline*}
    \chi_{P_{\ell}(t)}(x)\mathcal{S}(\vec{\phi})(t,x)\\=\chi_{P_{\ell}(t)}(x)e^{i\left(\frac{v_{\ell}x}{2}-\frac{v_{\ell}^{2}t}{4}+\omega_{\ell}t+\gamma_{\ell}\right)\sigma_{3}}\hat{G}_{\omega_{\ell}}\left(
e^{{-}it(k^{2}+\omega_{\ell})\sigma_{3}}e^{{-}i\gamma_{\ell}\sigma_{3}} \begin{bmatrix}
       e^{iy_{\ell}k}\phi_{1,\ell}\left(k+\frac{v_{\ell}}{2}\right)\\
       e^{iy_{\ell}k}\phi_{2,\ell}\left(k-\frac{v_{\ell}}{2}\right)
    \end{bmatrix}\right)(x-y_{\ell}-v_{\ell}t)\\{+}O\left(e^{{-}\beta\min_{\ell}[(y_{\ell}-y_{\ell+1})+(v_{\ell}-v_{\ell+1})t]}\max_{\ell}\norm{\begin{bmatrix}
        \phi_{1,\ell}(k)\\
        \phi_{2,\ell}(k)
\end{bmatrix}}_{L^{2}_{k}(\mathbb{R})}\right).
\end{multline*}

\end{lemma}
\begin{proof}
Using the asymptotic behaviors in \eqref{asy1}, \eqref{asy2}, \eqref{asy3} and \eqref{asy4} of $\mathcal{F}_\omega(x,k)$ and $\mathcal{G}_\omega(x,k),$ we can verify that condition $b)$ and Remark \ref{transition} imply the following estimate
when $j\neq \ell$ for a constant $\beta> 0$
\begin{multline*}
\chi_{P_{\ell}(t)}(x)e^{i\left(\frac{v_{j}x}{2}-\frac{v_{j}^{2}t}{4}+\omega_{j}t+\gamma_{\ell}\right)\sigma_{3}}\hat{G}_{\omega_{j}}\left(
   e^{{-}it(k^{2}+\omega_{j})\sigma_{3}}e^{{-}i\gamma_{j}\sigma_{3}} \begin{bmatrix}
       e^{iy_{j}k}\phi_{1,j}\left(k+\frac{v_{j}}{2}\right)\\
       e^{iy_{j}k}\phi_{2,j}\left(k-\frac{v_{j}}{2}\right)
    \end{bmatrix}\right)(x-y_{j}-v_{j}t)\\=
    \begin{cases}
    \chi_{P_{\ell}(t)}(x)e^{i\left(\frac{v_{j}x}{2}-\frac{v_{j}^{2}t}{4}\right)\sigma_{3}}F_{0}\left(
e^{{-}itk^{2}\sigma_{3}} \begin{bmatrix}
    e^{iy_{j}k}\phi_{1,j}\left(k+\frac{v_{j}}{2}\right)\\
    e^{iy_{j}k}\phi_{2,j}\left(k-\frac{v_{j}}{2}\right)
    \end{bmatrix}\right)(x-y_{j}-v_{j}t) \text{, if $j<\ell,$}\\
    \chi_{P_{\ell}(t)}(x)e^{i\left(\frac{v_{j}x}{2}-\frac{v_{j}^{2}t}{4}\right)\sigma_{3}}F_{0}\left(
e^{{-}itk^{2}\sigma_{3}} \begin{bmatrix}
    e^{iy_{j}k}\phi_{1,j-1}\left(k+\frac{v_{j}}{2}\right)\\
    e^{iy_{j}k}\phi_{2,j-1}\left(k-\frac{v_{j}}{2}\right)
    \end{bmatrix}\right)(x-y_{j}-v_{j}t) \text{, if $j>\ell,$}
    \end{cases}
    \\{+}O\left(e^{{-}\beta\min_{\ell}[(y_{\ell}-y_{\ell+1})+(v_{\ell}-v_{\ell+1})t]}\max_{\ell}\norm{\begin{bmatrix}
        \phi_{1,\ell}(k)\\
        \phi_{2,\ell}(k)
\end{bmatrix}}_{L^{2}_{k}(\mathbb{R})}\right),
\end{multline*}
for any $\ell\in\{1,2,...\,,m\}$ and any $t\geq 0$. 

Note that Lemma \ref{galil} of  Section \ref{sec:scatteringone} implies the following identity
\begin{equation*}
e^{i\left(\frac{v_{j}x}{2}-\frac{v_{j}^{2}t}{4}\right)\sigma_{3}}F_{0}\left(
e^{{-}itk^{2}\sigma_{3}} \begin{bmatrix}
e^{iy_{j}k}\phi_{1}\left(k+\frac{v_{j}}{2}\right)\\
    e^{iy_{j}k}\phi_{2}\left(k-\frac{v_{j}}{2}\right)
    \end{bmatrix}\right)(x-y_{j}-v_{j}t)=F_{0}\left(e^{{-}it k^{2}\sigma_{3}}\begin{bmatrix}
    \phi_{1}\left(k\right)\\
    \phi_{2}\left(k\right)
    \end{bmatrix}\right)(x), 
\end{equation*}
for any $j\in\{1,\,2,\,...,\,m-1\},$ and any $\vec{\phi}=(\phi_{1},\phi_{2})\in L^{2}_{k}(\mathbb{R},\mathbb{C}^{2}).$ Consequently, we have the following estimate 
for any $t\geq 0,$ and $\ell\neq j$
\begin{multline*}
\chi_{P_{\ell}(t)}(x)e^{i\left(\frac{v_{j}x}{2}-\frac{v_{j}^{2}t}{4}+\omega_{j}t+\gamma_{\ell}\right)\sigma_{3}}\hat{G}_{\omega_{j}}\left(
   e^{{-}it(k^{2}+\omega_{j})\sigma_{3}}e^{{-}i\gamma_{j}\sigma_{3}} \begin{bmatrix}
       e^{iy_{j}k}\phi_{1,j}\left(k+\frac{v_{j}}{2}\right)\\
       e^{iy_{j}k}\phi_{2,j}\left(k-\frac{v_{j}}{2}\right)
    \end{bmatrix}\right)(x-y_{j}-v_{j}t)\\=
    \begin{cases}
    \chi_{P_{\ell}(t)}(x)F_{0}\left(
e^{{-}itk^{2}\sigma_{3}} \begin{bmatrix}
    \phi_{1,j}\left(k\right)\\
   \phi_{2,j}\left(k\right)
    \end{bmatrix}\right)(x) \text{, if $j<\ell,$}\\
    \chi_{P_{\ell}(t)}(x)F_{0}\left(
e^{{-}itk^{2}\sigma_{3}} \begin{bmatrix}
    \phi_{1,j-1}\left(k\right)\\
   \phi_{2,j-1}\left(k\right)
    \end{bmatrix}\right)(x) \text{, if $j>\ell,$}
    \end{cases}
    \\{+}O\left(e^{{-}\beta\min_{\ell}[(y_{\ell}-y_{\ell+1})+(v_{\ell}-v_{\ell+1})t]}\max_{\ell}\norm{\begin{bmatrix}
        \phi_{1,\ell}(k)\\
        \phi_{2,\ell}(k)
\end{bmatrix}}_{L^{2}_{k}(\mathbb{R})}\right).
\end{multline*}
Using condition $c)$ and combining the computations above, the desired result follows.
\end{proof}

\par Clearly, $\mathcal{S}(\vec{\phi})(t,x)$ is not an exact solution to \eqref{ldpe}  but it serves as a good approximation.  The proof of the existence of the solution $\mathcal{T}(\vec{\phi})(t,x)$ of \eqref{ldpe} having the same asymptotic behavior as $\mathcal{S}(\vec{\phi})(t,x)$ will be given later in Section \ref{asyinfinity}. The proof will follow using the Lemmas \ref{galil}, \ref{appFourier} and a fixed-point argument. 
\par
Before we study the properties of $\mathcal{S}(t),$ we analyze the operator $\mathcal{S}(0):L^{2}_{x}(\mathbb{R},\mathbb{C}^{2})\to L^{2}_{x}(\mathbb{R},\mathbb{C}^{2})$. In \cite{perelmanasym}, when $v_{1}-v_{2}>0$ is large enough, Perelman proved for $m=2$ as $y_{1}-y_{2}$ goes to ${+}\infty$ that $\mathcal{S}(0)$ is a homeomorphism mapping through the subspace generated by the union of the essential spectrum space of each $\mathcal{H}_{\omega_{j}}.$ 
We are going to verify later in this section that this property can be extended for any number $m\geq 2$ of potentials. 
\par The first main result of this section is the following.
\begin{theorem}\label{TT}
Let $m\in\mathbb{N},$ there exist $K(m)>1$ and a parameter $c(m)>0$ such that if, for any $\ell\geq 1,$ $v_{\ell+1}-v_{\ell}>K(m),\,y_{\ell+1}-y_{\ell}>K(m),\,$ and $\vec{\phi}=\begin{bmatrix}
    \phi_{1}\\
    \phi_{2}
\end{bmatrix}\in L^{2}(\mathbb{R},\mathbb{C}^{2})$ is any element in the domain of $\mathcal{S}(0),$ then
\begin{equation}\label{Scoerc}
    \norm{\mathcal{S}(0)(\vec{\phi})}_{L^{2}_{x}}\geq c(m)\left[\sum_{\ell=1}^{m}\norm{\begin{bmatrix}
        \phi_{1,\ell}(k)\\
        \phi_{2,\ell}(k)
    \end{bmatrix}}_{L^{2}_{k}(\mathbb{R})}\right].
\end{equation}
Moreover, there exists a constant $c_2(m)>0$ satisfying
for all $\ell\in\{1,2,\,...,m\}$

\begin{align}\label{derivdecay}
   \norm{\mathcal{S}(0)(\vec{\phi})}_{L^{2}_{x}(\mathbb{R})}&+\norm{x\chi_{P_{\ell}(0)}(x+y_{\ell})\mathcal{S}(\vec{\phi})(0,x+y_{\ell})}_{L^{2}_{x}(\mathbb{R})}\\
   &\geq c_2(m)\left[\sum_{\ell=1}^{m}\norm{F_{0}(x)\left(e^{iy_{\ell}k}\begin{bmatrix}
        \phi_{1,\ell}(k)\\
        \phi_{2,\ell}(k)
    \end{bmatrix}\right)}_{L^{1}_{x}(\mathbb{R})}\right],\nonumber  
\end{align}
and
\begin{align}\label{derivdecay2}
    &\max_{j\in\{0,1\},\ell}\vert y_{\ell-1}-y_{\ell+1}\vert^{j}\norm{(1+\vert x\vert)^{2-j}\chi_{P_{\ell}}(x+y_{\ell})\mathcal{S}(\phi)(0,x+y_{\ell})}_{L^{2}_{x}} \\
    &\quad\quad \quad \quad \quad\quad \geq c_2(m)\sum_{\ell=1}^{m}\norm{F_{0}(x)\left(\frac{d}{dk}\left[e^{iy_{\ell}k}\begin{bmatrix}
        \phi_{1,\ell}(k)\\
        \phi_{2,\ell}(k)
    \end{bmatrix}\right]\right)}_{L^{1}_{x}(\mathbb{R})}. 
\end{align}
Furthermore, if $\max_{\ell}\norm{
\begin{bmatrix}
    k^{2}\phi_{1,\ell}(k)\\
    k^{2}\phi_{2,\ell}(k)
\end{bmatrix}}_{L^{2}_{k}(\mathbb{R})}<{+}\infty,$ then for some constant $c_3(m)$
\begin{equation}\label{Scoerc2}
    \norm{\mathcal{S}(0)(\vec{\phi})}_{H^{2}_{x}}\geq c_3(m)\left[\sum_{\ell=1}^{m}\norm{\begin{bmatrix}
        (1+k^{2})\phi_{1,\ell}(k)\\
        (1+k^{2})\phi_{2,\ell}(k)
    \end{bmatrix}}_{L^{2}_{k}(\mathbb{R})}\right].
\end{equation}
Furthermore, if the hypothesis $(H4)$ defined in Theorem \ref{Decesti} holds, then the following improved inequalities hold for some constant $c_4(m)$
\begin{align}\label{fl1teo2}
      \norm{\mathcal{S}(0)(\vec{\phi})}_{L^{1}_{x}(\mathbb{R})} \geq & c_4(m)\left[\sum_{\ell=1}^{m}\norm{F_{0}(x)\left(e^{iy_{\ell}k}\begin{bmatrix}
        \phi_{1,\ell}(k)\\
        \phi_{2,\ell}(k)
    \end{bmatrix}\right)}_{L^{1}_{x}(\mathbb{R})}\right],
\end{align}
    and
 \begin{multline}\label{fl1teo1}  
    \max_{\ell}\left[\vert y_{\ell-1}-y_{\ell+1}\vert\max\left(\norm{\mathcal{S}(0)(\vec{\phi})}_{L^{2}_{x}(\mathbb{R})},\norm{\mathcal{S}(0)(\vec{\phi})}_{L^{1}_{x}(\mathbb{R})}\right)+\norm{\vert x-y_{\ell}\vert\chi_{P_{\ell}(0)}(x)\mathcal{S}(\vec{\phi})(0,x)}_{L^{1}_{x}(\mathbb{R})}\right]\\ \geq  c_4(m)\sum_{\ell=1}^{m}\norm{F_{0}(x)\left(\frac{d}{dk}\left[e^{iy_{\ell}k}\begin{bmatrix}
        \phi_{1,\ell}(k)\\
        \phi_{2,\ell}(k)
    \end{bmatrix}\right]\right)}_{L^{1}_{x}(\mathbb{R})}
\end{multline}
\end{theorem}

\begin{corollary}\label{ccc}
If all the hypotheses of Theorem \ref{TT} are true, then there exists $C>0$ satisfying
\begin{equation*}
    \norm{\mathcal{S}(t)\vec{\phi}}_{L^{2}_{x}(\mathbb{R})}\leq C\norm{\mathcal{S}(0)\vec{\phi}}_{L^{2}_{x}(\mathbb{R})},
\end{equation*}
for any $t\geq 0.$ 
\end{corollary}
\begin{proof}[Proof of Corollary \ref{ccc}.]
First, Theorem \ref{TT} implies that there exists a constant $c>0$ satisfying 
\begin{equation*}
    \norm{\mathcal{S}(0)(\vec{\phi})}_{L^{2}_{x}(\mathbb{R})}\geq c\max_{\ell}\norm{\begin{bmatrix}
        \phi_{1,\ell}(k)\\
        \phi_{2,\ell}(k)
    \end{bmatrix}}_{L^{2}_{k}(\mathbb{R})}.
\end{equation*}
Consequently, we can verify using Lemma \ref{localS0} that there exists a constant $C>1$ satisfying 
\begin{multline*}
    \norm{\mathcal{S}(0)(\vec{\phi})(x)-\sum_{\ell}\chi_{P_{\ell}(0)}e^{i\left(\frac{v_{\ell}x}{2}-\frac{v_{\ell}^{2}t}{4}+\gamma_{\ell}\right)}\hat{G}_{\omega_{\ell}}\left(e^{{-}i\gamma_{\ell}\sigma_{3}}
    \begin{bmatrix}
    e^{iy_{\ell}k}\phi_{1,\ell}\left(k+\frac{v_{\ell}}{2}\right)\\
    e^{iy_{\ell}k}\phi_{1,\ell}\left(k-\frac{v_{\ell}}{2}\right)
    \end{bmatrix}\right)(x-y_{\ell})}_{L^{2}_{x}(\mathbb{R})}\\
    \leq Ce^{{-}\beta\min_{\ell}(y_{\ell}-y_{\ell+1})}\max_{\ell}\norm{\begin{bmatrix}
        \phi_{1,\ell}(k)\\
        \phi_{2,\ell}(k)
    \end{bmatrix}}_{L^{2}_{x}(\mathbb{R})}\leq \frac{C}{c}e^{{-}\beta\min_{\ell}(y_{\ell}-y_{\ell+1})}\norm{\mathcal{S}(0)(\vec{\phi})(x)}_{L^{2}_{x}(\mathbb{R})},
\end{multline*}
when $\min_{\ell}y_{\ell}-y_{\ell+1}$ is large enough. Therefore, by applying Minkowski’s inequality on the estimate above that there exists a constant $K>0$ depending on $C$ and $c$ satisfying
\begin{equation}\label{coerrrr}
    \norm{\mathcal{S}(0)(\vec{\phi})(x)}_{L^{2}_{x}(\mathbb{R})}\geq K \max_{\ell}\norm{\chi_{P_{\ell}(0)}\hat{G}_{\omega_{\ell}}\left(e^{{-}i\gamma_{\ell}\sigma_{3}}
    \begin{bmatrix}
    e^{iy_{\ell}k}\phi_{1,\ell}\left(k+\frac{v_{\ell}}{2}\right)\\
    e^{iy_{\ell}k}\phi_{1,\ell}\left(k-\frac{v_{\ell}}{2}\right)
    \end{bmatrix}\right)(x-y_{\ell})}_{L^{2}_{x}(\mathbb{R})}.
\end{equation}
\par Next, using the asymptotic behavior of \eqref{asy1}, \eqref{asy2}, \eqref{asy3}, \eqref{asy4}, we can verify that
\begin{multline*}
    \norm{\left[1-\chi_{P_{\ell}(0)}(x)\right]\hat{G}_{\omega_{\ell}}\left(e^{{-}i\gamma_{\ell}\sigma_{3}}
    \begin{bmatrix}
    e^{iy_{\ell}k}\phi_{1,\ell}\left(k+\frac{v_{\ell}}{2}\right)\\
    e^{iy_{\ell}k}\phi_{1,\ell}\left(k-\frac{v_{\ell}}{2}\right)
    \end{bmatrix}\right)(x-y_{\ell})}_{L^{2}_{x}(\mathbb{R})}\\
    \leq 2\max\left(\norm{\begin{bmatrix}
        \phi_{1,\ell}(k)\\
        \phi_{2,\ell}(k)
    \end{bmatrix}}_{L^{2}_{k}(\mathbb{R})},\norm{\begin{bmatrix}
        \phi_{1,\ell-1}(k)\\
        \phi_{2,\ell-1}(k)
    \end{bmatrix}}_{L^{2}_{k}(\mathbb{R})}\right)\leq\frac{2}{c}\norm{\mathcal{S}(0)(\vec{\phi})}_{L^{2}_{x}(\mathbb{R})}.
\end{multline*}
In conclusion, if $\min_{\ell} y_{\ell}-y_{\ell+1}>0$ is large enough, then for some constant $C_{2}>1$ the following estimate holds
\begin{align*}
 \norm{\mathcal{S}(t)(\vec{\phi})(x)}_{L^{2}_{x}(\mathbb{R})}\leq & 2m\max_{\ell}\norm{\hat{G}_{\omega_{\ell}}\left(
   e^{{-}it(k^{2}+\omega_\ell)\sigma_{3}}e^{{-}i\gamma_{\ell}\sigma_{3}} \begin{bmatrix}
       e^{iy_{\ell}k}\phi_{1,\ell}\left(k+\frac{v_{\ell}}{2}\right)\\
       e^{iy_{\ell}k}\phi_{2,\ell}\left(k-\frac{v_{\ell}}{2}\right)
    \end{bmatrix}\right)(x)}_{L^{2}_{x}(\mathbb{R})}\\
    \leq & 2mC_{2} \norm{\hat{G}_{\omega_{\ell}}\left(e^{{-}i\gamma_{\ell}\sigma_{3}} \begin{bmatrix}
       e^{iy_{\ell}k}\phi_{1,\ell}\left(k+\frac{v_{\ell}}{2}\right)\\
       e^{iy_{\ell}k}\phi_{2,\ell}\left(k-\frac{v_{\ell}}{2}\right)
    \end{bmatrix}\right)(x)}_{L^{2}_{x}(\mathbb{R})}\\
    \leq &\left[\frac{2mC_{2}}{K}+\frac{2mC_{2}}{c}\right]\norm{\mathcal{S}(0)(\vec{\phi})}_{L^{2}_{x}(\mathbb{R})},    
\end{align*}
which implies Corollary \ref{ccc}.
\end{proof}
\begin{remark}\label{re000}
 Furthermore, from the proof of Theorem \ref{TT}, we can verify that \eqref{Scoerc} also holds with $\mathcal{S}(t)(\overrightarrow{\phi})$ in place of $\mathcal{S}(0)(\overrightarrow{\phi})$ with the same constant $c(m).$    
\end{remark}

\begin{remark}
The map $\mathcal{S}(0)$ is also a closed map and a homeomorphism from its domain to a subspace of $L^{2}(\mathbb{R},\mathbb{C}^{2}).$ It also has a bounded linear inverse map $B.$ 
\end{remark}

\subsection{Applications of $P_{-}$ and $P_{+}$ }
\begin{definition}\label{dd}
For any $v\in\mathbb{R}$ and any function  $f(x)$ domain $\mathbb{R}$ with  the variable in the physical side, we define
\begin{equation*}
    \tau_{v}\left(f\right)(x)\coloneqq f(x+v),
\end{equation*}
for all $x\in \mathbb{R}.$ Moreover, for any function $g=\begin{bmatrix}
g_{1}(k)\\
g_{2}(k)
\end{bmatrix}:\mathbb{R}\to\mathbb{C}^{2},$ with the variable in the frequency side, we define
\begin{equation*}
\tau_{\sigma_{3}v}(g)(k)\coloneqq \begin{bmatrix}
    g_{1}(k+v)\\
    g_{2}(k-v)
\end{bmatrix}, 
\end{equation*}
for any $k\in \mathbb{R}.$
\end{definition}
 \par Next, we consider the following elementary proposition.
\begin{lemma}\label{trivial}
 All the functions $\begin{bmatrix}
     \psi_{1}(k)\\
     \psi_{2}(k)
 \end{bmatrix}\in L^{2}(\mathbb{R},\mathbb{C}^{2})$ satisfy for any $v,\,y\in\mathbb{R}$ the following identities
 \begin{align*}
P_{\pm}\left(\begin{bmatrix}
    e^{iyk}\psi_{1}(k+v)\\
e^{iyk}\psi_{2}(k-v)
    \end{bmatrix}
    \right)=&\begin{bmatrix}
    P_{\pm}(e^{iy\Diamond}\psi_{1}(\Diamond))(k+v)e^{{-}iyv}\\
P_{\pm}(e^{iy\Diamond}\psi_{2}(\Diamond))(k-v)e^{iyv}
\end{bmatrix}=e^{{-}i\sigma_{3}yv}P_{\pm}\left(\begin{bmatrix}
    e^{iyk}e^{iy v}\psi_{1}(k+v)\\
e^{iyk}e^{{-}iy v}\psi_{2}(k-v)
    \end{bmatrix}\right).
 \end{align*}
\end{lemma}
\begin{proof}
    Elementary computations.
\end{proof}
Motivated by the approach in \cite{perelmanasym}, we consider the following maps
\begin{align*}
    B_{1}(h)\coloneqq  & \sigma_{3} F^{*}_{\omega_{1}}\left[\sigma_{3}\chi_{\{x\geq \frac{y_{2}-y_{1}}{2}\}}(x)\tau_{y_{1}}\left(e^{{-}i\frac{v_{1}x}{2}\sigma_{3}}h(x)\right)\right] (k),\\
B_{2n}(h) \coloneqq  & \sigma_{3} G^{*}_{\omega_{n+1}}\left[\sigma_{3}\chi_{\{\frac{y_{n+2}-y_{n+1}}{2}\leq x\leq \frac{y_{n}-y_{n+1}}{2}\}}(x)\tau_{y_{n+1}}\left(e^{{-}i\frac{v_{n+1}x}{2}\sigma_{3}}h(x)\right)\right](k)\text{ if $n+1<m,$}
    \\
    B_{2n+1}(h) \coloneqq  & \sigma_{3} F^{*}_{\omega_{n+1}}\left[\sigma_{3}\chi_{\{\frac{y_{n+2}-y_{n+1}}{2}\leq x\leq \frac{y_{n}-y_{n+1}}{2}\}}(x)\tau_{y_{n+1}}\left(e^{{-}i\frac{v_{n+1}x}{2}\sigma_{3}}h(x)\right)\right](k) \text{ if $n+1<m,$}\\
    B_{2m-2}(h)\coloneqq  & \sigma_{3} G^{*}_{\omega_{m}}\left[\sigma_{3}\chi_{\{x\leq \frac{y_{m-1}-y_{m}}{2}\}}(x)\tau_{y_{m}}\left(e^{{-}i\frac{v_{m}x}{2}\sigma_{3}}h(x)\right)\right](k)
\end{align*}
to be applied on $\mathcal{S}(0)(\vec{\phi}).$

The motivation to apply the maps above to $\mathcal{S}(0)$ is the following: first, we recall that Lemma \ref{leper} gives a formula for the inverses of $\hat{G}_{\omega}$ and $\hat{F}_{\omega}$. As shown in Lemma \ref{localS0}, $\mathcal{S}(0)(\vec{\phi})$ behaves as $e^{i(\frac{v_{\ell}x}{2}+\gamma_{\ell})\sigma_{3}}\hat{G}_{\omega_{\ell}}(e^{iy_{\ell}k-i\gamma_{\ell}\sigma_{3}}\vec{\phi}_{\ell})(x-y_{\ell})$ when $x\in(\frac{y_{\ell+1}+y_{\ell}}{2},\frac{y_{\ell-1}+y_{\ell}}{2})$. So for each of this interval,  we will use the appropriate maps from $\left(B_{\ell}\right)_{\ell\in\{1,2\,...,\,2m-2\}}$ to invert $\mathcal{S}(0)$.  Putting all these maps together, one can study the inverse map of $\mathcal{S}(0).$ 
%

\begin{notation}\label{congg}
In the next paragraphs, with the function $\mathcal{S}(0)(\vec{\phi})$ defined in Definition \ref{s0def}, we define  $f(k)\cong g(k)$ if there exist constants $\beta>0$ and $C>0$ satisfying
\begin{equation*}
    \norm{f(k)-g(k)}_{L^{2}_{k}(\mathbb{R})}\leq C e^{{-}\beta \min_{\ell}(y_{\ell}-y_{\ell+1})}\left[\max_{1\leq \ell\leq m-1 }\norm{\begin{bmatrix}
        \phi_{1,\ell}(k)\\
        \phi_{2,\ell}(k)
    \end{bmatrix}}_{L^{2}_{k}(\mathbb{R})} \right].
\end{equation*}
\par Furthermore, we say that 
$f(k)\cong_{1} g(k)$ if there exists $\beta>0$ and $C>1$ 
\begin{equation*}
    \norm{F_{0}\left[f(k)-g(k)\right]}_{L^{1}_{x}(\mathbb{R})}\leq C e^{{-}\beta \min_{\ell}(y_{\ell}-y_{\ell+1})}\left[\max_{1\leq \ell\leq m-1 }\norm{\begin{bmatrix}
        \phi_{1,\ell}(k)\\
        \phi_{2,\ell}(k)
    \end{bmatrix}}_{L^{2}_{k}(\mathbb{R})} \right],
\end{equation*}
and that $f(k)\cong_{2} g(k)$ if
\begin{equation*}
    \norm{(1+\vert x\vert)F_{0}\left[f(k)-g(k)\right](x)}_{L^{1}_{x}(\mathbb{R})}\leq C e^{{-}\beta \min_{\ell}(y_{\ell}-y_{\ell+1})}\left[\max_{1\leq \ell\leq m-1 }\norm{\begin{bmatrix}
        \phi_{1,\ell}(k)\\
        \phi_{2,\ell}(k)
    \end{bmatrix}}_{L^{2}_{k}(\mathbb{R})} \right].
\end{equation*}

\end{notation}
\par With notations above, applying  the maps $B_{\ell}$ to $\mathcal{S}(0)(\vec{\phi})$, we obtain the following results. 

\begin{lemma}\label{T}
The following estimates are true for any $1<n< m-1$
\begin{align}\label{b1}
B_{1}\left(\mathcal{S}(0)(\vec{\phi})\right)(k)\cong_{2}  & e^{i\frac{k(y_{1}-y_{2})}{2}}P_{-}\left(e^{i\frac{\Diamond(y_{1}+y_{2})}{2}}
 \begin{bmatrix}
    \phi_{1}\left(\Diamond+\frac{v_{1}}{2}\right)\\
    \phi_{2}\left(\Diamond-\frac{v_{1}}{2}\right)
 \end{bmatrix}\right) (k)\nonumber  
\\ & {-}r_{\omega_{1}}({-}k)e^{{-}i\frac{k(y_{1}-y_{2})}{2}}\left[P_{+}\left(e^{i\frac{\Diamond (y_{1}+y_{2})}{2}}\begin{bmatrix}
    \phi_{1}\left(\Diamond+\frac{v_{1}}{2}\right)\\
    \phi_{2}\left(\Diamond-\frac{v_{1}}{2}\right)
 \end{bmatrix}\right)\right]({-}k),\\ \label{b2n}
  B_{2n}\left(\mathcal{S}(0)(\vec{\phi})\right)(k)\cong_{2} & e^{{-}i\frac{k(y_{n}-y_{n+1})}{2}}P_{+}\left(\begin{bmatrix}
    e^{i\frac{(y_{n}+y_{n+1})\Diamond}{2}}\phi_{1,n}(\Diamond+\frac{v_{n+1}}{2})\\ 
    e^{i\frac{(y_{n}+y_{n+1})\Diamond}{2}}\phi_{2,n}(\Diamond-\frac{v_{n+1}}{2})
    \end{bmatrix}\right)(k)\\ \nonumber
    &{-}r_{\omega_{n+1}}(k)e^{i\frac{k(y_{n}-y_{n+1})}{2}}P_{-}\left(\begin{bmatrix}
    e^{i\frac{(y_{n}+y_{n+1})\Diamond}{2}}\phi_{1,n}(\Diamond+\frac{v_{n+1}}{2})\\
    e^{i\frac{(y_{n}+y_{n+1})\Diamond}{2}}\phi_{2,n}(\Diamond-\frac{v_{n+1}}{2})
    \end{bmatrix}\right)({-}k)\\ \nonumber
    &{-}s_{\omega_{n+1}}(k)e^{{-}i\frac{k(y_{n+2}-y_{n+1})}{2}}P_{+}\left(\begin{bmatrix}
e^{i\frac{(y_{n+1}+y_{n+2})\Diamond}{2}}\phi_{1,n+1}(\Diamond+\frac{v_{n+1}}{2})\\
e^{i\frac{(y_{n+1}+y_{n+2})\Diamond}{2}}\phi_{2,n+1}(\Diamond-\frac{v_{n+1}}{2})
\end{bmatrix}\right)(k),\\ \label{b2n+1}
 B_{2n+1}\left(\mathcal{S}(0)(\vec{\phi})\right)(k)\cong_{2} & e^{{-}i\frac{k(y_{n+2}-y_{n+1})}{2}}P_{-}\left(\begin{bmatrix}
e^{i\frac{(y_{n+1}+y_{n+2})\Diamond}{2}}\phi_{1,n+1}(\Diamond+\frac{v_{n+1}}{2})\\
e^{i\frac{(y_{n+1}+y_{n+2})\Diamond}{2}}\phi_{2,n+1}(\Diamond-\frac{v_{n+1}}{2})
\end{bmatrix}
\right)(k)\\ \nonumber
    &{-}r_{\omega_{n+1}}({-}k)e^{i\frac{k(y_{n+2}-y_{n+1})}{2}}P_{+}\left(\begin{bmatrix}
e^{i\frac{(y_{n+1}+y_{n+2})\Diamond}{2}}\phi_{1,n+1}(\Diamond+\frac{v_{n+1}}{2})\\
e^{i\frac{(y_{n+1}+y_{n+2})\Diamond}{2}}\phi_{2,n+1}(\Diamond-\frac{v_{n+1}}{2})
\end{bmatrix}\right)({-}k)\\ \nonumber
    &{-}s_{\omega_{n+1}}({-}k)e^{{-}i\frac{k(y_{n}-y_{n+1})}{2}}P_{-}\left(
    \begin{bmatrix}
e^{i\frac{k(y_{n}+y_{n+1})k}{2}}\phi_{1,n}(k+\frac{v_{n+1}}{2})\\
e^{i\frac{(y_{n}+y_{n+1})k}{2}}\phi_{2,n}(k-\frac{v_{n+1}}{2})
\end{bmatrix}\right),\\ \label{b2m-2}
 B_{2m-2}\left(\mathcal{S}(0)(\vec{\phi})\right)(k)\cong_{2} & e^{{-}ik\frac{(y_{m-1}-y_{m})}{2}}P_{+}\left(\begin{bmatrix}
      e^{i\frac{(y_{m-1}+y_{m})\Diamond}{2}}\phi_{1,m-1}(\Diamond+\frac{v_{m}}{2})\\
        e^{i\frac{(y_{m-1}+y_{m})\Diamond}{2}}\phi_{2,m-1}(\Diamond-\frac{v_{m}}{2})
  \end{bmatrix}\right)(k)\\ \nonumber
  &{-}r_{\omega_{m}}(k)e^{ik\frac{(y_{m-1}-y_{m})}{2}}\left[P_{-}\left(\begin{bmatrix}
      e^{i\frac{(y_{m-1}+y_{m})\Diamond}{2}}\phi_{1,m-1}(\Diamond+\frac{v_{m}}{2})\\
        e^{i\frac{(y_{m-1}+y_{m})\Diamond}{2}}\phi_{2,m-1}(\Diamond-\frac{v_{m}}{2})
  \end{bmatrix}\right)\right]({-}k).
\end{align}
\end{lemma}
 Lemma \ref{T} will be crucial for the proof of inequalities \eqref{Scoerc} and \eqref{derivdecay} of Theorem \ref{TT}.

\begin{corollary}\label{coo}
 If all the hypotheses of Theorem \ref{tcont} are true and $\mathcal{T}(0)(\vec{\phi})\in H^{2\ell}_{x}(\mathbb{R})$ for an $\ell\geq 1,$ then there exists $C_{\ell}\geq 1$ satisfying
 \begin{equation*}
    \max_{1 \leq n\leq m} \norm{\begin{bmatrix}
(k+v_{n})^{2\ell}\phi_{1,n}(k)\\
(k-v_{n})^{2\ell}\phi_{2,n}(k)
\end{bmatrix}}_{L^{2}_{k}(\mathbb{R})}\leq C_{\ell}\norm{\mathcal{S}(0)(\vec{\phi})}_{H^{2\ell}_{x}(\mathbb{R})}\leq 2C_{\ell}\norm{\mathcal{T}(0)(\vec{\phi})}_{H^{2\ell}_{x}(\mathbb{R})}
 \end{equation*}
where $\phi_{1,n},\,\phi_{2,n}$ are defined in Definition \ref{s0def}.
\end{corollary}
\begin{proof}[Proof of Corollary \ref{coo}.]\label{estimk2psi}
 Using Lemma \ref{appFourier} for $j\in\{0,1,2\}$ and Young's inequality, as in Lemma \ref{localS0}, we can verify that
 \begin{multline*}
     \chi_{\left\{\frac{y_{n+2}-y_{n+1}}{2}\leq x\leq \frac{y_{n}-y_{n+1}}{2}\right\}}\frac{d^{2}}{dx^{2}}e^{{-}\frac{v_{n+1}x\sigma_{3}}{2}}\mathcal{S}(\vec{\phi})(0,x+y_{n+1})\\=\chi_{\left\{\frac{y_{n+2}-y_{n+1}}{2}\leq x\leq \frac{y_{n}-y_{n+1}}{2}\right\}}(x)\frac{d^{2}}{dx^{2}}\hat{G}_{\omega_{n+1}}\left(e^{iy_{\ell}k}\begin{bmatrix}
        \phi_{1,n+1}\left(k+\frac{v_{n+1}}{2}\right)\\
\phi_{2,n+1}\left(k-\frac{v_{n+1}}{2}\right)
     \end{bmatrix}\right)(x)\\+ O\left(e^{{-}\beta \min_{\ell}(y_{\ell}-y_{\ell+1})}\max_{\ell}\norm{(1+\vert k\vert)\begin{bmatrix}
      \phi_{1,\ell}(k+\frac{v_{\ell}}{2})\\
\phi_{2,\ell}\left(k-\frac{v_{\ell}}{2}\right)
     \end{bmatrix}}_{L^{2}}
     \right).
 \end{multline*}
 Therefore, using Lemma \ref{sobolevdecayofG}, and Lemma \ref{T}, we can conclude that 
 \begin{multline*}
     \chi_{\left\{\frac{y_{n+2}-y_{n+1}}{2}\leq x\leq \frac{y_{n}-y_{n+1}}{2}\right\}}\frac{d^{2}}{dx^{2}}e^{{-}\frac{v_{n+1}x\sigma_{3}}{2}}\mathcal{S}(\vec{\phi})(0,x+y_{n+1})\\={-}\chi_{\left\{\frac{y_{n+2}-y_{n+1}}{2}\leq x\leq \frac{y_{n}-y_{n+1}}{2}\right\}}(x)\hat{G}_{\omega_{n+1}}\left(e^{iy_{\ell}k}\begin{bmatrix}
        \left(k+\frac{v_{n+1}}{2}\right)^{2}\phi_{1,n+1}\left(k+\frac{v_{n+1}}{2}\right)\\
\left(k-\frac{v_{n+1}}{2}\right)^{2}\phi_{2,n+1}\left(k-\frac{v_{n+1}}{2}\right)
     \end{bmatrix}\right)(x)\\
      +O\left(\epsilon \norm{ \begin{bmatrix}
        \left(k+\frac{v_{n+1}}{2}\right)^{2}\phi_{1,n+1}\left(k+\frac{v_{n+1}}{2}\right)\\
\left(k-\frac{v_{n+1}}{2}\right)^{2}\phi_{2,n+1}\left(k-\frac{v_{n+1}}{2}\right)
     \end{bmatrix}}_{L^{2}_{k}(\mathbb{R})}+C_{\epsilon,v}\norm{ \mathcal{S}(0)(\vec{\phi})}_{L^{2}_{x}(\mathbb{R})}\right)
     \\+ O\left(e^{{-}\beta \min_{\ell}(y_{\ell}-y_{\ell+1})}\max_{\ell}\norm{(1+\vert k\vert)\begin{bmatrix}
      \phi_{1,\ell}(k+\frac{v_{\ell}}{2})\\
\phi_{2,\ell}\left(k-\frac{v_{\ell}}{2}\right)
     \end{bmatrix}}_{L^{2}_{k}(\mathbb{R})}
     \right).
 \end{multline*}
 
Since all operators $F^{*}_{\omega_{\ell}}$ and $G^{*}_{\omega_{\ell}}$ are bounded on $L^{2}(\mathbb{R},\mathbb{C}^{2}),$ we can verify choosing an $\epsilon>0$ small and repeating the proofs of Lemma \ref{T} and of estimate \eqref{Scoerc} of Theorem \ref{TT} that   
\begin{equation*}
    \max_{n}\norm{\begin{bmatrix}
         \left(k+\frac{v_{n+1}}{2}\right)^{2}\phi_{1,n+1}\left(k+\frac{v_{n+1}}{2}\right)\\
\left(k-\frac{v_{n+1}}{2}\right)^{2}\phi_{2,n+1}\left(k-\frac{v_{n+1}}{2}\right)
    \end{bmatrix}}_{L^{2}_{k}(\mathbb{R})}\leq C_{v}\norm{\mathcal{S}(\vec{\phi})(0,x)}_{H^{2}_{x}(\mathbb{R})},
\end{equation*}
for a positive parameter $C_{v}>1$ depending only on $v_{1},\,v_{2},\,...,\,v_{m}.$
\end{proof}
\begin{remark}\label{polyweightre}
In particular, using the asymptotic behavior \eqref{asy1}, \eqref{asy2}, and Remark \ref{weightremark}, we can verify that all the remainders $r_{1}(k),\, r_{2}(k),\,r_{3}(k)$ and $r_{4}(k)$ from the estimates \eqref{b1}, \eqref{b2est1}, \eqref{b2n+1} and \eqref{b2m-2} satisfy for a constant $\beta>0$
\begin{equation*}
    \max_{n\in\{0,1,2\}}\max_{\ell\in\{1,2,3,4\}}\norm{\frac{d^{n}r_{\ell}(k)}{dk^{n}}}_{L^{2}_{k}(\mathbb{R})}\leq C e^{{-}\beta \min_{\ell}(y_{\ell}-y_{\ell+1})}\left[\max_{1\leq \ell\leq m-1 }\norm{\begin{bmatrix}
        \phi_{1,\ell}(k)\\
        \phi_{2,\ell}(k)
    \end{bmatrix}}_{L^{2}_{k}(\mathbb{R})} \right].
\end{equation*}
Moreover, it is not difficult to verify using Remark \ref{weightremark} that for any $j\in\{1,\,2,\,...,\,2m-2\}$ that there exists $K_{m}>1$ satisfying
\begin{multline*}
   \max_{j}\norm{\frac{d^{2}B_{j}(h)(k)}{dk^{2}}}_{L^{2}_{k}(\mathbb{R})}\\
    \leq K_{m}\max_{n}\norm{e^{{-}i\frac{v_{\ell}(x+y_{n+1})\sigma_{3}}{2}}\chi_{\{\frac{y_{n+2}-y_{n+1}}{2}\leq x\leq \frac{y_{n}-y_{n+1}}{2}\}}(x)x^{2}h(x+y_{n+1})}_{L^{2}_{x}(\mathbb{R})}.
\end{multline*}
\end{remark}
\begin{remark}\label{Tremark}Consider the following functions 
\begin{align*}
    \psi_{1}(k)=&P_{-}\left(e^{i\frac{k(y_{1}+y_{2})}{2}}
 \begin{bmatrix}
    \phi_{1}\left(k+\frac{v_{1}}{2}\right)\\
    \phi_{2}\left(k-\frac{v_{1}}{2}\right)
 \end{bmatrix}\right),\\
 \psi_{2n}(k)=&P_{+}\left(\begin{bmatrix}
e^{i\frac{(y_{n}+y_{n+1})k}{2}}\phi_{1,n}(k+\frac{v_{n+1}}{2})\\
e^{i\frac{(y_{n}+y_{n+1})k}{2}}\phi_{2,n}(k-\frac{v_{n+1}}{2})
\end{bmatrix}\right),\\
\psi_{2n+1}=&P_{-}\left(\begin{bmatrix}
e^{i\frac{(y_{n+1}+y_{n+2})k}{2}}\phi_{1,n+1}(k+\frac{v_{n+1}}{2})\\
e^{i\frac{(y_{n+1}+y_{n+2})k}{2}}\phi_{2,n+1}(k-\frac{v_{n+1}}{2})
\end{bmatrix}\right),\\
\psi_{2m-2}=& P_{+}\left(\begin{bmatrix}
e^{i\frac{(y_{m-1}+y_{m})k}{2}}\phi_{1,m-1}(k+\frac{v_{m}}{2})\\
e^{i\frac{(y_{m-1}+y_{m})k}{2}}\phi_{2,m-1}(k-\frac{v_{m}}{2})
\end{bmatrix}\right).
\end{align*}
Using Definition \ref{dd} and Lemma \ref{trivial}, we can verify that
\begin{align*}
 P_{-}\left(e^{i\frac{(y_{\ell}+y_{\ell+1})k}{2}}\begin{bmatrix}
     \phi_{1,\ell}\left(k+\frac{v_{n}}{2}\right)\\
     \phi_{2,\ell}\left(k-\frac{v_{n}}{2}\right)
 \end{bmatrix}\right)= & P_{-}\left(\begin{bmatrix}
e^{i\frac{(y_{\ell}+y_{\ell+1})(v_{\ell}-v_{n})}{4}}\left[e^{i\frac{(y_{\ell}+y_{\ell+1})\Diamond}{2}}\phi_{1,\ell}\left(\Diamond+\frac{v_{\ell}}{2}\right)\right]\left(k+\frac{v_{n}}{2}-\frac{v_{\ell}}{2}\right)\\
e^{{-}i\frac{(y_{\ell}+y_{\ell+1})(v_{\ell}-v_{n})}{4}} \left[e^{i\frac{(y_{\ell}+y_{\ell+1})\Diamond}{2}}\phi_{2,\ell}\left(\Diamond-\frac{v_{\ell}}{2}\right)\right]\left(k+\frac{v_{\ell}}{2}-\frac{v_{n}}{2}\right)
 \end{bmatrix}\right)\\
 = & e^{i\sigma_{3}\frac{(y_{\ell}+y_{\ell+1})(v_{\ell}-v_{n})}{4}} \tau_{\sigma_{3}(\frac{v_{n}-v_{\ell}}{2})}P_{-}\left(e^{i\frac{(y_{\ell}+y_{\ell+1})\Diamond}{2}}
     \begin{bmatrix}
     \phi_{1,\ell}\left(\Diamond+\frac{v_{\ell}}{2}\right)\\
     \phi_{2,\ell}\left(\Diamond-\frac{v_{\ell}}{2}\right)
 \end{bmatrix}
 \right)(k)\\
 =&
 e^{i\sigma_{3}\frac{(y_{\ell}+y_{\ell+1})(v_{\ell}-v_{n})}{4}} \tau_{\sigma_{3}(\frac{v_{n}-v_{\ell}}{2})}\psi_{2\ell-1}(k)
\end{align*}
for any $n\in\{1,\,...,\,m\}.$
Similarly, using the definition of $\psi_{2\ell}$ above, we can verify that
\begin{equation*}
   P_{+}\left(e^{i\frac{(y_{\ell}+y_{\ell+1})k}{2}}\begin{bmatrix}
     \phi_{1,\ell}\left(k+\frac{v_{n}}{2}\right)\\
     \phi_{2,\ell}\left(k-\frac{v_{n}}{2}\right)
 \end{bmatrix}\right)=e^{i\sigma_{3}\frac{(y_{\ell}+y_{\ell+1})(v_{\ell+1}-v_{n})}{4}} \tau_{\sigma_{3}(\frac{v_{n}-v_{\ell+1}}{2})}\psi_{2\ell}(k)
\end{equation*}

\par Therefore, estimates \eqref{b1}, \eqref{b2n}, \eqref{b2n+1} and \eqref{b2m-2} from Lemma \ref{T} imply the following
equivalence relations{\footnotesize
\begin{align*}
e^{{-}ik\frac{(y_{1}-y_{2})}{2}}B_{1}\left(\mathcal{S}(0)(\vec{\phi})\right)(k)\cong_{1} & \psi_{1}(k)-r_{\omega_{1}}({-}k)e^{{-}ik(y_{1}-y_{2})}e^{\frac{i(v_{2}-v_{1})(y_{1}+y_{2})}{4}\sigma_{3}}\left[\tau_{\sigma_{3}\frac{(v_{1}-v_{2})}{2}}\psi_{2}\right]({-}k)\\
e^{ik\frac{(y_{n}-y_{n+1})}{2}}B_{2n}\left(\mathcal{S}(0)(\vec{\phi})\right)(k)\cong_{1} & \psi_{2n}(k)\\&{-}r_{\omega_{n+1}}(k)e^{ik(y_{n}-y_{n+1})}e^{i\frac{(v_{n}-v_{n+1})(y_{n}+y_{n+1})\sigma_{3}}{4}}\left[\tau_{\sigma_{3}\frac{(v_{n+1}-v_{n})}{2}}\psi_{2n-1}\right]({-}k)\\&{-}s_{\omega_{n+1}}(k)e^{i\frac{k(y_{n}-y_{n+2})}{2}}e^{i\frac{(y_{n+1}+y_{n+2})(v_{n+2}-v_{n+1})}{4}\sigma_{3}}\left[\tau_{\sigma_{3}\frac{(v_{n+1}-v_{n+2})}{2}}\psi_{2n+2}\right](k),\\
e^{ik\frac{(y_{n+2}-y_{n+1})}{2}}B_{2n+1}\left(\mathcal{S}(0)(\vec{\phi})\right)(k)\cong_{1} & \psi_{2n+1}(k)\\
{-}& r_{\omega_{n+1}}({-}k)e^{ik(y_{n+2}-y_{n+1})}e^{i\frac{(v_{n+2}-v_{n+1})(y_{n+1}+y_{n+2})\sigma_{3}}{4}}\left[\tau_{\sigma_{3}\frac{(v_{n+1}-v_{n+2})}{2}}\psi_{2n+2}\right]({-}k)\\
&{-}s_{\omega_{n+1}}({-}k)e^{i\frac{k(y_{n+2}-y_{n})}{2}}e^{i\frac{(v_{n}-v_{n+1})(y_{n}+y_{n+1})\sigma_{3}}{4}}\left[\tau_{\sigma_{3}\frac{(v_{n+1}-v_{n})}{2}}\psi_{2n-1}\right](k),\\
e^{ik\frac{(y_{m-1}-y_{m})}{2}}B_{2m-2}\left(\mathcal{S}(0)(\vec{\phi})\right)(k)\cong_{1} & \psi_{2m-2}(k)\\&{-}r_{\omega_{m}}(k)e^{ik(y_{m-1}-y_{m})}e^{i\frac{(v_{m-1}-v_{m})(y_{m-1}+y_{m})\sigma_{3}}{4}}\left[\tau_{\sigma_{3}\frac{(v_{m}-v_{m-1})}{2}}\psi_{2m-3}\right]({-}k),
\end{align*}}

\end{remark}

\begin{proof}[Proof of Lemma \ref{T}.]
\textbf{Step 1.(Proof of estimates of \eqref{b1}, \eqref{b2m-2}.)} First, using the asymptotics \eqref{asy3},  \eqref{asy4}, Lemmas \ref{appFourier} and \ref{LB3} of Appendix  \ref{sec:appb},  we can verify that
\begin{align*}
    B_{1}\left(\mathcal{S}(0)(\vec{\phi})\right)(k)\cong_{2} & \sigma_{3} F^{*}_{\omega_{1}}\left[ \sigma_{3} \chi_{\{x\geq \frac{y_{2}-y_{1}}{2}\}}(x)\hat{G}_{\omega_{1}}\left(\begin{bmatrix}
e^{iy_{1}\Diamond}\phi_{1}\left(\Diamond+\frac{v_{1}}{2}\right)\\
        e^{iy_{1}\Diamond}\phi_{2}\left(\Diamond-\frac{v_{1}}{2}\right)
    \end{bmatrix}\right)(x)\right](k).
\end{align*}
Consequently, using  Lemma \ref{leper}, we deduce first that
\begin{equation*}
    B_{1}\left(\mathcal{S}(0)(\vec{\phi})\right)(k)\cong_{2} \begin{bmatrix}
        e^{iy_{1}k}\phi_{1}\left(k+\frac{v_{1}}{2}\right)\\
e^{iy_{1}k}\phi_{2}\left(k-\frac{v_{1}}{2}\right)
    \end{bmatrix}
    -\sigma_{3} F^{*}_{\omega_{1}}\left[ \sigma_{3} \chi_{\{x< \frac{y_{2}-y_{1}}{2}\}}(x)\hat{G}_{\omega_{1}}\left(\begin{bmatrix}
        e^{iy_{1}\Diamond}\phi_{1}\left(\Diamond+\frac{v_{1}}{2}\right)\\
e^{iy_{1}\Diamond}\phi_{2}\left(\Diamond-\frac{v_{1}}{2}\right)
    \end{bmatrix}(x)\right)\right](k),
\end{equation*}
and so
\begin{align*}
    B_{1}\left(\mathcal{S}(0)(\vec{\phi})\right)(k)\cong_{2}& \begin{bmatrix}
    e^{iy_{1}k}\phi_{1}\left(k+\frac{v_{1}}{2}\right)\\
         e^{iy_{1}k}\phi_{2}\left(k-\frac{v_{1}}{2}\right)
    \end{bmatrix}
    -\sigma_{3} F^{*}_{0}\left[ \sigma_{3} \chi_{\{x< \frac{y_{2}-y_{1}}{2}\}}(x)F_{0}\left(\begin{bmatrix}
        e^{iy_{1}\Diamond}\phi_{1}\left(\Diamond+\frac{v_{1}}{2}\right)\\
        e^{iy_{1}\Diamond}\phi_{2}\left(\Diamond-\frac{v_{1}}{2}\right)
    \end{bmatrix}\right)(x)\right](k)\\
   & -\sigma_{3} r_{1}({-}k)F^*_{0}\left[\sigma_{3} \chi_{\{x< \frac{y_{2}-y_{1}}{2}\}}(x)F_{0}\left(\begin{bmatrix}
        e^{iy_{1}\Diamond}\phi_{1}\left(\Diamond+\frac{v_{1}}{2}\right)\\
        e^{iy_{1}\Diamond}\phi_{2}\left(\Diamond-\frac{v_{1}}{2}\right)
    \end{bmatrix}\right)(x)\right ](-k),   
\end{align*}
see the asymptotics from \eqref{asy1} and \eqref{asy4}, and Lemma \ref{appFourier}.
In conclusion, we obtain that
\begin{multline}\label{eqinv1}
B_{1}\left(\mathcal{S}(0)(\vec{\phi})\right)(k) \cong_{2} \sigma_{3}F^{*}_{0}\left[\sigma_{3}\chi_{\{x\geq \frac{y_{2}-y_{1}}{2}\}}(x)F_{0}\left(\begin{bmatrix}
    e^{iy_{1}\Diamond}\phi_{1}\left(\Diamond+\frac{v_{1}}{2}\right)\\
        e^{iy_{1}\Diamond}\phi_{2}\left(\Diamond-\frac{v_{1}}{2}\right)
    \end{bmatrix}\right)(x)\right](k)
    \\{-}\sigma_{3} r_{1}({-}k)F^*_{0}\left[\sigma_{3} \chi_{\{x< \frac{y_{2}-y_{1}}{2}\}}(x)F_{0}\left(\begin{bmatrix}
        e^{iy_{1}\Diamond}\phi_{1}\left(\Diamond+\frac{v_{1}}{2}\right)\\
        e^{iy_{1}\Diamond}\phi_{2}\left(\Diamond-\frac{v_{1}}{2}\right)(x)
    \end{bmatrix}\right)(x)\right](-k).
\end{multline}
Note that explicitly, the first term on the right-hand side above can be written as
\begin{align}
&\sigma_{3}F^{*}_{0}\left[\sigma_{3}\chi_{\{x\geq \frac{y_{2}-y_{1}}{2}\}}(x)F_{0}\left(\begin{bmatrix}
e^{iy_{1}\Diamond}\phi_{1}\left(\Diamond+\frac{v_{1}}{2}\right)\\
    e^{iy_{1}\Diamond}\phi_{2}\left(\Diamond-\frac{v_{1}}{2}\right)
    \end{bmatrix}\right)(x)\right](k)\\
    &=\begin{bmatrix}
  \frac{1}{2\pi} \int_{\frac{y_2-y_1}{2}}^\infty e^{-ixk} \int e^{i\eta x}e^{iy_{1}\eta}\phi_{1}\left(\eta+\frac{v_{1}}{2}\right)\,d\eta dx\\ \frac{1}{2\pi} \int^\infty_{\frac{y_2-y_1}{2}} e^{-ixk} \int e^{i\eta x}
        e^{iy_{1}\eta}\phi_{2}\left(\eta-\frac{v_{1}}{2}\right)\,d\eta dx
    \end{bmatrix}\\ &=e^{i\frac{k(y_{1}-y_{2})}{2}}\begin{bmatrix}
  \frac{1}{2\pi} \int_{z\geq 0} e^{-izk} \int e^{i\eta z}e^{i\frac{\eta (y_{1}+y_{2})}{2}}\phi_{1}\left(\eta+\frac{v_{1}}{2}\right)\,d\eta dz\\ \frac{1}{2\pi}  \int_{z\geq 0} e^{-izk} \int e^{i\eta z}e^{i\frac{\eta (y_{1}+y_{2})}{2}}\phi_{2}\left(\eta-\frac{v_{1}}{2}\right)\,d\eta dz
    \end{bmatrix}\\
    &=e^{i\frac{k(y_{1}-y_{2})}{2}}\begin{bmatrix}
  P_-\left[e^{i\frac{\Diamond (y_{1}+y_{2})}{2}}\phi_{1}\left(\Diamond+\frac{v_{1}}{2}\right)\right](k)\\  P_-\left[e^{i\frac{\Diamond (y_{1}+y_{2})}{2}}\phi_{2}\left(\Diamond-\frac{v_{1}}{2}\right)\right](k)
    \end{bmatrix}
\end{align}where in the last line, we applied \eqref{eq:FdefPm}. The same argument can be applied the to the second term on the right-hand side of \eqref{eqinv1}. Then the first desired estimate \eqref{b1} follows.

\par Next for $B_{2m-2}\left(\mathcal{S}(0)(\vec{\phi})\right)$ ,  from Remark \ref{transition}, we first note that 
\begin{align*}
e^{i\frac{v_{m}x}{2}\sigma_{3}}\hat{G}_{\omega_{m}}\left(\begin{bmatrix}
       e^{iy_{m}k}\phi_{1,m}(k+\frac{v_{m}}{2})\\
       e^{iy_{m}k}\phi_{2,m}(k+\frac{v_{m}}{2})
   \end{bmatrix}\right)(x-y_{m})=e^{i\frac{v_{m}x}{2}\sigma_{3}}\hat{F}_{\omega_{m}}\left(\begin{bmatrix}
       e^{iy_{m}k}\phi_{1,m-1}(k+\frac{v_{m}}{2})\\
       e^{iy_{m}k}\phi_{2,m-1}(k+\frac{v_{m}}{2})
   \end{bmatrix}\right)(x-y_{m}).
 \end{align*}
Therefore, using Remark \ref{asrr} and Lemma \ref{LB3} of Appendix \ref{sec:appb},  we can deduce that 
\begin{equation*}
    B_{2m-2}\left(\mathcal{S}(0)(\vec{\phi})\right)\cong_{2} \sigma_{3}G^{*}_{\omega_{m}}\left[\sigma_{3}\chi_{\{x\leq \frac{y_{m-1}-y_{m}}{2}\}}(x)\hat{F}_{\omega_{m}}\left(\begin{bmatrix}
       e^{iy_{m}\Diamond}\phi_{1,m-1}(\Diamond+\frac{v_{m}}{2})\\
       e^{iy_{m}\Diamond}\phi_{2,m-1}(\Diamond+\frac{v_{m}}{2})
   \end{bmatrix}\right)(x)\right](k).
\end{equation*}
Consequently, similarly to the proof of \eqref{eqinv1}, using the asymptotics \eqref{asy1}, \eqref{asy2}, \eqref{asy3}, \eqref{asy4}, we can verify that 
\begin{align}\label{bfinal}
&B_{2m-2}\left(\mathcal{S}(0)(\vec{\phi})\right)(k)  \\
& \quad \quad \cong_{2}  \sigma_{3}F^{*}_{0}\left[\sigma_{3}\chi_{\left\{x\leq  \frac{y_{m-1}-y_{m}}{2}\right\}}(x)F_{0}\left(\begin{bmatrix}
    e^{iy_{m}\Diamond}\phi_{1,m-1}(\Diamond+\frac{v_{m}}{2})\\
       e^{iy_{m}\Diamond}\phi_{2,m-1}(\Diamond+\frac{v_{m}}{2})
    \end{bmatrix}\right)(x)\right](k)\\ \nonumber &\quad \quad {-}\sigma_{3} r_{\omega_{m}}(k)F_{0}^*\left[\sigma_{3} \chi_{\left\{x> \frac{y_{m-1}-y_{m}}{2}\right\}}(x)F_{0}\left(\begin{bmatrix}
        e^{iy_{m}\Diamond}\phi_{1,m-1}(\Diamond+\frac{v_{m}}{2})\\
       e^{iy_{m}\Diamond}\phi_{2,m-1}(\Diamond+\frac{v_{m}}{2})
    \end{bmatrix}\right)(x)\right](-k).
\end{align}

\par Next, after applying a change of variables on the estimates \eqref{eqinv1} and \eqref{bfinal}, we deduce that 
\begin{align}\label{B3esti}
    & B_{1}\left(\mathcal{S}(0)(\vec{\phi})\right)(k) \\
    &\quad \quad \cong_{2}  e^{i\frac{k(y_{1}-y_{2})}{2}}F^{*}_{0}\left[\chi_{\{x\geq 0\}}F_{0}\left(\begin{bmatrix}e^{iy_{1}\Diamond}\phi_{1}\left(\Diamond+\frac{v_{1}}{2}\right)\\
        e^{i y_{1}\Diamond}\phi_{1}\left(\Diamond-\frac{v_{1}}{2}\right)
    \end{bmatrix}\right)\left(x-\frac{y_{1}-y_{2}}{2}\right)\right](k)
\\ \nonumber &\quad \quad{-}r_{\omega_{1}}({-}k)e^{{-}i\frac{k(y_{1}-y_{2})}{2}}F^*_{0}\left[ \chi_{\{x< 0\}}(x)F_{0}\left(\begin{bmatrix}
e^{iy_{1}\Diamond}\phi_{1}\left(\Diamond+\frac{v_{1}}{2}\right)\\
        e^{i y_{1}\Diamond}\phi_{1}\left(\Diamond-\frac{v_{1}}{2}\right)
    \end{bmatrix}\right)\left(x-\frac{y_{1}-y_{2}}{2}\right)\right](-k),
\end{align}
and
\begin{align} \nonumber
&B_{2m-2}\left(\mathcal{S}(0)(\vec{\phi})\right)(k) \\
&\quad \quad \cong_{2} e^{{-}i\frac{k(y_{m-1}-y_{m})}{2}}F^{*}_{0}\left[\chi_{\{x\leq 0\}}F_{0}\left(\begin{bmatrix}
        e^{iy_{m}\Diamond}\phi_{1,m-1}(\Diamond+\frac{v_{m}}{2})\\
        e^{iy_{m}\Diamond}\phi_{1,m-1}(\Diamond-\frac{v_{m}}{2})
    \end{bmatrix}\right)\left(x+\frac{y_{m-1}-y_{m}}{2}\right)\right](k)
\\ \nonumber
&\quad \quad  {-} r_{\omega_{m}}(k)e^{i\frac{k(y_{m-1}-y_{m})}{2}}F^*_{0}\left[ \chi_{\{x>0\}}(x)F_{0}\left(\begin{bmatrix}
            e^{iy_{m}\Diamond}\phi_{1,m-1}(\Diamond+\frac{v_{m}}{2})\\
        e^{iy_{m}\Diamond}\phi_{1,m-1}(\Diamond-\frac{v_{m}}{2})
    \end{bmatrix}\right)\left(x+\frac{y_{m-1}-y_{m}}{2}\right)\right](-k),
\end{align}
from which we obtain the estimates \eqref{b1} and \eqref{b2m-2}.
\\
\textbf{Step 2.(Proof of estimates \eqref{b2n}, \eqref{b2n+1}.)}
First, since $\sigma_{3}G^{*}_{\omega_{n+1}}\sigma_{3}\hat{F}_{\omega_{n+1}}=\mathrm{Id}$ from Lemma \ref{leper}, we can deduce from Lemma \ref{appFourier} and Lemma \ref{LB3} that
\begin{align*}
    B_{2n}\left(\mathcal{S}(0)(\vec{\phi})\right)(k)\cong_{2} & \begin{bmatrix}
        e^{iy_{n+1}k}\phi_{1,n}(k+\frac{v_{n+1}}{2})\\
        e^{iy_{n+1}k}\phi_{2,n}(k-\frac{v_{n+1}}{2})
    \end{bmatrix}\\&{-}\sigma_{3}G^{*}_{\omega_{n+1}}\left(\sigma_{3}\chi_{\left\{x>\frac{y_{n}-y_{n+1}}{2}\right\}} F_{0}\begin{bmatrix}
        e^{iy_{n+1}\Diamond}\phi_{1,n}(\Diamond+\frac{v_{n+1}}{2})\\
        e^{iy_{n+1}\Diamond}\phi_{2,n}(\Diamond-\frac{v_{n+1}}{2})
    \end{bmatrix}(x)\right)(k)
\\&{-}\sigma_{3}G^{*}_{\omega_{n+1}}\left(\sigma_{3}\chi_{\left \{x<\frac{y_{n+2}-y_{n+1}}{2}\right \}} F_{0}\begin{bmatrix}
    e^{iy_{n+1}\Diamond}\phi_{1,n+1}(\Diamond+\frac{v_{n+1}}{2})\\
    e^{iy_{n+1}\Diamond}\phi_{2,n+1}(\Diamond-\frac{v_{n+1}}{2})
\end{bmatrix}(x)\right)(k).
\end{align*}
\par Next, using the asymptotics of $G^{*}_{\omega_{n+1}}$ in \eqref{asy1} and \eqref{asy2}, we obtain from the estimate above that
\begin{align}\label{b2est1}
    B_{2n}\left(\mathcal{S}(0)(\vec{\phi})\right)(k)\cong_{2} & \begin{bmatrix}
        e^{iy_{n+1}k}\phi_{1,n}(k+\frac{v_{n+1}}{2})\\
        e^{iy_{n+1}k}\phi_{2,n}(k-\frac{v_{n+1}}{2})
    \end{bmatrix}\\ \nonumber
    &{-}\sigma_{3}F^{*}_{0}\left(\sigma_{3}\chi_{\left\{x>\frac{y_{n}-y_{n+1}}{2}\right\}} F_{0}(x)\begin{bmatrix}
        e^{iy_{n+1}\Diamond}\phi_{1,n}(\Diamond+\frac{v_{n+1}}{2})\\
        e^{iy_{n+1}\Diamond}\phi_{2,n}(\Diamond-\frac{v_{n+1}}{2})
    \end{bmatrix}(x)\right)(k)\\ \nonumber
  &{-}\sigma_{3}r_{\omega_{n+1}}(k)F_{0}\left(\sigma_{3}\chi_{\left\{x>\frac{y_{n}-y_{n+1}}{2}\right\}} F_{0}\begin{bmatrix}
        e^{iy_{n+1}\Diamond}\phi_{1,n}(\Diamond+\frac{v_{n+1}}{2})\\
        e^{iy_{n+1}\Diamond}\phi_{2,n}(\Diamond-\frac{v_{n+1}}{2})
    \end{bmatrix}(x)\right)(k)\\  
&{-}\sigma_{3}s_{\omega_{n+1}}(k)F^{*}_{0}\left(\sigma_{3}\chi_{\left \{x<\frac{y_{n+2}-y_{n+1}}{2}\right \}} F_{0}\begin{bmatrix}
e^{iy_{n+1}\Diamond}\phi_{1,n+1}(\Diamond+\frac{v_{n+1}}{2})\\
e^{iy_{n+1}\Diamond}\phi_{2,n+1}(\Diamond-\frac{v_{n+1}}{2})
\end{bmatrix}(x)\right)(k).
\end{align}


\par After applying the same argument as for \eqref{b1}, in conclusion, we deduce the following estimates
\begin{align*}
    B_{2n}\left(\mathcal{S}(0)(\vec{\phi})\right)\cong_{2} & e^{{-}i\frac{k(y_{n}-y_{n+1})}{2}}P_{+}\left(\begin{bmatrix}
    e^{i\frac{(y_{n}+y_{n+1})\Diamond}{2}}\phi_{1,n}(\Diamond+\frac{v_{n+1}}{2})\\ 
    e^{i\frac{(y_{n}+y_{n+1})\Diamond}{2}}\phi_{2,n}(\Diamond-\frac{v_{n+1}}{2})
    \end{bmatrix}\right)(k)\\ \nonumber
    &{-}r_{\omega_{n+1}}(k)e^{i\frac{k(y_{n}-y_{n+1})}{2}}P_{-}\left(\begin{bmatrix}
    e^{i\frac{(y_{n}+y_{n+1})\Diamond}{2}}\phi_{1,n}(\Diamond+\frac{v_{n+1}}{2})\\
    e^{i\frac{(y_{n}+y_{n+1})\Diamond}{2}}\phi_{2,n}(\Diamond-\frac{v_{n+1}}{2})
    \end{bmatrix}\right)({-}k)\\ \nonumber
    &{-}s_{\omega_{n+1}}(k)e^{{-}i\frac{k(y_{n+2}-y_{n+1})}{2}}P_{+}\left(\begin{bmatrix}
e^{i\frac{(y_{n+1}+y_{n+2})\Diamond}{2}}\phi_{1,n+1}(\Diamond+\frac{v_{n+1}}{2})\\
e^{i\frac{(y_{n+1}+y_{n+2})\Diamond}{2}}\phi_{2,n+1}(\Diamond-\frac{v_{n+1}}{2})
\end{bmatrix}\right)(k).
\end{align*}
\par Similarly, we can verify the following estimate
\begin{align*} 
    B_{2n+1}\left(\mathcal{S}(0)(\vec{\phi})\right)\cong_{2} & e^{{-}i\frac{k(y_{n+2}-y_{n+1})}{2}}P_{-}\left(\begin{bmatrix}
e^{i\frac{(y_{n+1}+y_{n+2})\Diamond}{2}}\phi_{1,n+1}(\Diamond+\frac{v_{n+1}}{2})\\
e^{i\frac{(y_{n+1}+y_{n+2})\Diamond}{2}}\phi_{2,n+1}(\Diamond-\frac{v_{n+1}}{2})
\end{bmatrix}
\right)(k)\\ \nonumber
    &{-}r_{\omega_{n+1}}({-}k)e^{i\frac{k(y_{n+2}-y_{n+1})}{2}}P_{+}\left(\begin{bmatrix}
e^{i\frac{(y_{n+1}+y_{n+2})\Diamond}{2}}\phi_{1,n+1}(\Diamond+\frac{v_{n+1}}{2})\\
e^{i\frac{(y_{n+1}+y_{n+2})\Diamond}{2}}\phi_{2,n+1}(\Diamond-\frac{v_{n+1}}{2})
\end{bmatrix}\right)({-}k)\\ \nonumber
    &{-}s_{\omega_{n+1}}({-}k)e^{{-}i\frac{k(y_{n}-y_{n+1})}{2}}P_{-}\left(
    \begin{bmatrix}
e^{i\frac{(y_{n}+y_{n+1})\Diamond}{2}}\phi_{1,n}(\Diamond+\frac{v_{n+1}}{2})\\
e^{i\frac{(y_{n}+y_{n+1})\Diamond}{2}}\phi_{2,n}(\Diamond-\frac{v_{n+1}}{2})
\end{bmatrix}\right)(k),
\end{align*}
which are exactly \eqref{b2n} and \eqref{b2n+1}.
\end{proof}
\subsection{Proof of Theorem \ref{TT}}

First, we consider the following  subspace: for $n\in\mathbb{N}$
\begin{equation*}
    \mathcal{F}L_{2,n}\coloneqq \left\{f:\mathbb{R}\to\mathbb{C}^{2n}\vert\, \norm{f}_{\mathcal{F}L_{2,n}}:=\norm{(1+\vert x\vert)\hat{f}(x)}_{L^{1}_{x}(\mathbb{R})}<{+}\infty\right\}.
\end{equation*}
 The proof of Theorem \ref{TT} will follow using the estimates from Remark \ref{Tremark} and Lemma \ref{T}. More precisely, we consider the following map
\begin{equation*}
    T_{m}:L^{2}(\mathbb{R},\mathbb{C}^{4m-4})\cap \mathcal{F}L_{2,2m-2}\to L^{2}(\mathbb{R},\mathbb{C}^{4m-4})\cap\mathcal{F}L_{2,2m-2}
\end{equation*}
defined by $T_{m}\coloneqq(T_{1,m},T_{2,m},T_{3,m},...,T_{2m-2,m})$ such that each map 
\begin{equation*}
    T_{\ell,m}:L^{2}(\mathbb{R},\mathbb{C}^{2})\cap \mathcal{F}L_{2,1}\to L^{2}(\mathbb{R},\mathbb{C}^{2})\cap \mathcal{F}L_{2,1},\,\ell=1,\ldots, 2m-2
\end{equation*}
 satisfies for each element $\overrightarrow{\psi}=(\psi_{1},\psi_{2},...,\psi_{2m-2})\in L^{2}(\mathbb{R},\mathbb{C}^{4m-4})$ the following identities
\begin{align*}
T_{1,m}(\overrightarrow{\psi})=&  r_{\omega_{1}}({-}k)e^{{-}ik(y_{1}-y_{2})}e^{\frac{(v_{2}-v_{1})(y_{1}+y_{2})}{4}\sigma_{3}}\left[\tau_{\sigma_{3}\frac{(v_{1}-v_{2})}{2}}\psi_{2}\right]({-}k),\\
T_{2,m}(\overrightarrow{\psi})=&r_{\omega_{2}}(k)e^{ik(y_{1}-y_{2})}e^{i\frac{(v_{1}-v_{2})(y_{1}+y_{2})\sigma_{3}}{4}}\left[\tau_{\sigma_{3}\frac{(v_{2}-v_{1})}{2}}\psi_{1}\right]({-}k)\\&{+}s_{\omega_{2}}(k)e^{i\frac{k(y_{1}-y_{3})}{2}}e^{i\frac{(y_{2}+y_{3})(v_{3}-v_{2})}{4}\sigma_{3}}\left[\tau_{\sigma_{3}\frac{(v_{2}-v_{3})}{2}}\psi_{4}\right](k), ...,\\
T_{2n-1,m}(\overrightarrow{\psi})=& r_{\omega_{n}}({-}k)e^{ik(y_{n+1}-y_{n})}e^{i\frac{(v_{n+1}-v_{n})(y_{n}+y_{n+1})\sigma_{3}}{4}}\left[\tau_{\sigma_{3}\frac{(v_{n}-v_{n+1})}{2}}\psi_{2n}\right]({-}k)\\
&{+}s_{\omega_{n}}({-}k)e^{i\frac{k(y_{n+1}-y_{n-1})}{2}}e^{i\frac{(v_{n-1}-v_{n})(y_{n-1}+y_{n})\sigma_{3}}{4}}\left[\tau_{\sigma_{3}\frac{(v_{n}-v_{n-1})}{2}}\psi_{2n-3}\right](k),\\
T_{2n,m}(\overrightarrow{\psi})=&r_{\omega_{n+1}}(k)e^{ik(y_{n}-y_{n+1})}e^{i\frac{(v_{n}-v_{n+1})(y_{n}+y_{n+1})\sigma_{3}}{4}}\left[\tau_{\sigma_{3}\frac{(v_{n+1}-v_{n})}{2}}\psi_{2n-1}\right]({-}k)\\&{+}s_{\omega_{n+1}}(k)e^{i\frac{k(y_{n}-y_{n+2})}{2}}e^{i\frac{(y_{n+1}+y_{n+2})(v_{n+2}-v_{n+1})}{4}\sigma_{3}}\left[\tau_{\sigma_{3}\frac{(v_{n+1}-v_{n+2})}{2}}\psi_{2n+2}\right](k),...,\\
T_{2m-2,m}(\overrightarrow{\psi})=&r_{\omega_{m}}(k)e^{ik(y_{m-1}-y_{m})}e^{i\frac{(v_{m-1}-v_{m})(y_{m-1}+y_{m})\sigma_{3}}{4}}\left[\tau_{\sigma_{3}\frac{(v_{m}-v_{m-1})}{2}}\psi_{2m-3}\right]({-}k).
\end{align*}
We will verify by induction on $m$ that 
\begin{equation}\label{indT}
    \norm{(\mathrm{Id}-T_m)^{{-}1}}<{+}\infty, 
\end{equation}
from which  Theorem \ref{TT} will follow.
\par The estimate \eqref{indT} for $m=2$ was already proved in the paper \cite{perelmanasym}. In particular, one should notice that in the case  $m=2,$ the construction of the family of maps above is much simpler without invoking the reflection coefficients, for more details see the proof of Proposition $4.3.2$ from \cite{perelmanasym}. Therefore, it is enough to prove that \eqref{indT} is true for $m=3$ and the inductive step to conclude for all $m\geq 3$ that \eqref{indT} is true for all $m\in\mathbb{N}_{\geq 3}.$

\par Moreover, to efficiently simplify  our reasoning, for each $1\leq i\leq 2m-2,$ we consider the function
\begin{align}\label{psii2}
\overrightarrow{\psi_{+}}(k)\coloneqq \left(\begin{bmatrix}
    \psi_{1}(k)\\
    0
\end{bmatrix},...,\begin{bmatrix}
    \psi_{i}(k)\\
    0
\end{bmatrix},...,\begin{bmatrix}
    \psi_{2m-2}(k)\\
    0
\end{bmatrix}\right).
\end{align}
  Using  \eqref{psii2}, we consider the following abstract lemma.
\begin{lemma}\label{Tp}
 Let
 $\mathcal{L}=(\mathcal{L}_{1,m},...,\mathcal{L}_{2m-2,m}):L^{2}(\mathbb{R},\mathbb{C}^{2m-2})\cap \mathcal{F}L_{1,2m-2}\to L^{2}(\mathbb{R},\mathbb{C}^{2m-2})\cap \mathcal{F}L_{1,2m-2}$ be a linear transformation of the form
 \begin{align}\nonumber
\mathcal{L}_{1,m}(\overrightarrow{\psi})=&\psi_{1}(k)-  r_{1}({-}k)\psi_{2}\left({-}k+\frac{v_{1}-v_{2}}{2}\right),\\ \nonumber
\mathcal{L}_{2,m}(\overrightarrow{\psi})=&\psi_{2}(k)-r_{2}(k)\psi_{1}\left({-}k+\frac{v_{2}-v_{1}}{2}\right)-s_{2}(k)\psi_{4}\left(k+\frac{v_{2}-v_{3}}{2}\right), ...,\\ \nonumber
\mathcal{L}_{2n-1,m}(\overrightarrow{\psi})=&\psi_{2n-1}(k)- r_{n}({-}k)\psi_{2n}\left({-}k+\frac{v_{n}-v_{n+1}}{2}\right){-}s_{n}({-}k)\psi_{2n-3}\left(k+\frac{v_{n}-v_{n-1}}{2}\right),\\ \nonumber
\mathcal{L}_{2n,m}(\overrightarrow{\psi})=&\psi_{2n}(k)-r_{n+1}(k)\psi_{2n-1}\left({-}k+\frac{v_{n+1}-v_{n}}{2}\right){-}s_{n+1}(k)\psi_{2n+2}\left(k+\frac{v_{n+1}-v_{n+2}}{2}\right),...,\\ \label{lastmap}
\mathcal{L}_{2m-2,m}(\overrightarrow{\psi})=&\psi_{2m-2}(k)-r_{m}(k)\psi_{2m-3}\left({-}k+\frac{v_{m}-v_{m-1}}{2}\right),
\end{align}
where all the parameters $\gamma_{i,1},\,\,\gamma_{i,2},\theta_{i,1},\, \theta_{i,2}$ are constant real numbers, and each one of the complex functions $r_{n},\,s_{n}$ satisfying $\vert r_{n}(k)\vert\leq 1,\,\vert s_{n}(k) \vert\leq 1$ for all $k\in\mathbb{R},$ and
\begin{equation}\label{POPO}
    r_{n}(k)=O\left(\frac{1}{\vert k\vert +1}\right),\,s_{n}(k)=1+O\left(\frac{1}{\vert k\vert +1}\right),
\end{equation}
as $\vert k\vert $ goes to ${+}\infty.$ There exist parameters $K(m)>0,\,c(m)>0$ depending only on $m\in \mathbb{N}$ such that if $\min_{\ell} v_{\ell}-v_{\ell+1}>K(m),$ then
\begin{equation}\label{Tff}
    \norm{\mathcal{L}(\psi)}_{L^{2}}\geq c(m)\max_{\ell\in\{1,...,2m-2\}} \norm{\psi_{\ell}}_{L^{2}},
\end{equation}
for all $\overrightarrow{\psi}\in L^{2}(\mathbb{R},\mathbb{C}^{2m-2}).$ Furthermore, 
\begin{equation}\label{fl11}
\norm{\mathcal{L}(\psi)}_{\mathcal{F}L_{1}}\geq   c(m)\max_{\ell\in\{1,...,2m-2\}}\norm{\psi_{\ell}}_{\mathcal{F}L_{1}}.
\end{equation}

\end{lemma}

\begin{remark}\label{TTistrue}
Using Lemmas \ref{T} and \ref{Tp}, we are going to verify the existence of a constant $c>0$ satisfying \begin{equation*}
     \norm{\mathcal{S}(0)\left(\begin{bmatrix}
     \phi_{1}\\
     0
 \end{bmatrix}\right)(x)}_{L^{2}_{x}(\mathbb{R})}\geq c\norm{\begin{bmatrix}
     \phi_{1}\\
     0
 \end{bmatrix}}_{L^{2}_{k}(\mathbb{R})},
 \end{equation*} 
when $\min_{\ell} v_{\ell}-v_{\ell+1},\,$ and $\min y_{\ell}-y_{\ell+1}$ are both sufficiently large. See also Remark \ref{Tremark} of Lemma \ref{T}.
\end{remark}

\begin{remark}We see in the proof of Lemma \ref{T} that estimates \eqref{Scoerc} and \eqref{derivdecay} will follow from an application of Lemma \ref{Tp}. More precisely, applying Lemma \ref{Tp} applied onto the right-hand side of the estimates \eqref{b1}, \eqref{b2n}, \eqref{b2n+1}, and \eqref{b2m-2} implies that
    \begin{equation*}
        \max_{\ell}\norm{B_{\ell}\left(\mathcal{S}(0)(\vec{\phi})\right)}_{L^{2}_{k}(\mathbb{R})}\geq c\max_{1\leq \ell\leq m-1 }\norm{\begin{bmatrix}
            \phi_{1,\ell}(k)\\
            \phi_{2,\ell}(k)
        \end{bmatrix}}_{L^{2}_{k}(\mathbb{R})}-  C e^{{-}\beta \min_{\ell}(y_{\ell}-y_{\ell+1})}\left[\max_{1\leq \ell\leq m-1 }\norm{\begin{bmatrix}
        \phi_{1,\ell}(k)\\
        \phi_{2,\ell}(k)
\end{bmatrix}}_{L^{2}_{k}(\mathbb{R})} \right],
    \end{equation*}
for constants $c,\,C>0.$ In conclusion, one can verify that estimate \eqref{Scoerc} of Theorem \eqref{TT} follows from the estimate above and the uniform boundness of the operators $B_{\ell}:L^{2}_{x}(\mathbb{R},\mathbb{C}^{2})\to L^{2}_{k}(\mathbb{R},\mathbb{C}^{2}).$ The proof of estimate \eqref{derivdecay} is similar, the unique difference is that requires the following estimate
\begin{equation*}
    \max_{\langle x \rangle f(x)\in L^{2}_{x}(\mathbb{R})}\norm{B_{\ell}(f)}_{L^{2}_{k}(\mathbb{R})}\leq C\left[\norm{\vert x-y_{\ell}\vert\chi_{\left[\frac{y_{\ell}+y_{\ell+1}}{2},\frac{y_{\ell}+y_{\ell-1}}{2}\right]}f(x)}_{L^{2}_{x}(\mathbb{R})}+\norm{f}_{L^{2}_{x}(\mathbb{R})}\right]
\end{equation*}
to be true for any Schwartz function $f\in\mathscr{S}(\mathbb{R},\mathbb{C}^{2}),$ and any $\ell\in\{1,2\,...,\,m-1\}.$
\end{remark}

 \begin{proof}[Proof of Lemma \ref{Tp}.]
 \par Based on the maps defined on Lemma \ref{Tp}, we consider the following terms
 \begin{align*}
A_{2}(k)=&\mathcal{L}_{2,m}(\overrightarrow{\psi})\left({-}k+\frac{v_{1}-v_{2}}{2}\right),\\
A_{4}(k)=&\mathcal{L}_{4,m}(\overrightarrow{\psi})\left({-}k+\frac{v_{1}-v_{3}}{2}\right), ...,\\
A_{2n}(k)=&\mathcal{L}_{2n,m}(\overrightarrow{\psi})\left({-}k+\frac{v_{1}-v_{n+1}}{2}\right), ...,\\
A_{2m-2}(k)=&\mathcal{L}_{2m-2,m}(\overrightarrow{\psi})\left({-}k+\frac{v_{1}-v_{m}}{2}\right).
\end{align*}
Moreover, it is not difficult by induction for any $2<n< m-1$ that  
\begin{align}\label{Indprop}
 S_{n}(\psi)(k)&\coloneqq\mathcal{L}_{1,m}(\psi)(k)+r_{1}\left({-}k\right)A_{2}\left(k\right)+\sum_{\ell=2}^{n}r_{1}({-}k)\left[\Pi_{j=2}^{\ell} s_{j}\left({-}k+\frac{v_{1}-v_{j}}{2}\right)\right]A_{2\ell}\left(k\right)\\
 &=\psi_{1}(k)-r_{1}({-}k)\left[\Pi_{j=2}^{n+1} s_{j}\left({-}k+\frac{v_{1}-v_{j}}{2}\right)\right]\psi_{2n+2}\left({-}k+\frac{v_{1}-v_{n+2}}{2}\right)\\&{+}\mathrm{rest}(k)
\end{align}
where $\mathrm{rest}(k)$ denotes a collection of terms such that
$\mathrm{rest}(k)\in L^{2}_{k}(\mathbb{R})\cap \mathcal{F}L_{1}$ satisfying for some constant $C>1$ the following inequality
\begin{align*}
    \norm{\mathrm{rest}(k)}_{L^{2}_{k}(\mathbb{R})}\leq C \left[\min\frac{1}{v_{\ell}-v_{\ell+1}}\right]\max_{j} \norm{\psi_{j}}_{L^{2}},\,\norm{\mathrm{rest}(k)}_{\mathcal{F}L_{1}}\leq C \left[\min\frac{1}{v_{\ell}-v_{\ell+1}}\right]\max_{j} \norm{\psi_{j}}_{\mathcal{F}L^{1}} .
\end{align*}
\par Indeed, when $n=2,$ estimate \eqref{Indprop} follows directly from the definitions of $\mathcal{L}_{1,m},\,\mathcal{L}_{2,m},$ and applications of  Lemma \ref{lB1} and estimates
\begin{equation}\label{pep}
    \norm{ r_{\ell}(k)r_{j}(k+v_{\ell}-v_{j})}_{L^{\infty}_{k}(\mathbb{R})}=O\left(\frac{1}{1+\vert v_{\ell}-v_{j}\vert }\right) \text{, for all $j\neq \ell,$}
\end{equation}
which follow from \eqref{POPO}. 
If \eqref{Indprop} holds until $n=M<m-2,$ then, since
\begin{multline*}
    A_{2M+2}(k)=\psi_{2M+2}\left({-}k+\frac{v_{1}-v_{M+2}}{2}\right)-r_{M+2}\left({-}k+\frac{v_{1}-v_{M+2}}{2}\right)\psi_{2M+1}\left(k+v_{M+2}-\frac{(v_{M+1}+v_{1})}{2}\right)\\{-}s_{M+2}\left({-}k+\frac{v_{1}-v_{M+2}}{2}\right)\psi_{2M+4}\left({-}k+\frac{v_{1}-v_{M+3}}{2}\right),
\end{multline*}
we can verify using \eqref{pep} and Lemma \ref{lB1} that 
\begin{align*}
    S_{M}(\psi)(k)&+r_{1}({-}k)\left[\Pi_{j=2}^{M+1} s_{j}\left({-}k+\frac{v_{1}-v_{j}}{2}\right)\right]A_{2M+2}(k)\\
    &=\psi_{1}(k)-r_{1}({-}k)\left[\Pi_{j=2}^{M+2} s_{j}\left({-}k+\frac{v_{1}-v_{j}}{2}\right)\right]\psi_{2n+2}\left({-}k+\frac{v_{1}-v_{n+2}}{2}\right)
    \\&{+}\mathrm{rest}(k),
\end{align*}
which is equivalent to \eqref{Indprop} for $n=M+1,$ and so, \eqref{Indprop} is true for all $2<n\leq m-2.$
\par Next, using the estimate \eqref{Indprop} for $n=m-2,$ \eqref{pep}, Lemma \ref{lB1}, and the map $\mathcal{L}_{2m-2,m}(\overrightarrow{\psi})$ defined at \eqref{lastmap}, it is not difficult to verify that
\begin{align}
S_{m-2}(\overrightarrow{\psi})(k)+& r_{1}({-}k)\left[\Pi_{j=2}^{M+1} s_{j}\left({-}k+ \frac{v_{1}-v_{j}}{2}\right)\right]A_{2m-2}(k)\label{eq:Sm-2}\\
=&\psi_{1}(k)+\mathrm{rest}(k)
\end{align}
which also implies the existence of  constants $c>0,\,K>1$ satisfying
\begin{align}\label{Final1}
    \left[c\norm{\psi_{1}}_{L^{2}}-K\max_{\ell}\frac{1}{v_{\ell}-v_{\ell+1}}\max_{j} \norm{\psi_{j}}_{L^{2}}\right]\leq &\norm{\mathcal{L}(\overrightarrow{\psi})}_{L^{2}},\\
    \left[c\norm{\psi_{1}}_{\mathcal{F}L^{1}}-K\max_{\ell}\frac{1}{v_{\ell}-v_{\ell+1}}\max_{j} \norm{\psi_{j}}_{\mathcal{F}L^{1}}\right]\leq &\norm{\mathcal{L}(\overrightarrow{\psi})}_{\mathcal{F}L^{1}}.
\end{align}since \eqref{eq:Sm-2} is a special linear combination of coordinates of $\mathcal{L}$. 
\par Next, if Lemma \ref{Tp} holds until  $m-1\in\mathbb{N}_{\geq 2},$ we consider the following linear map \begin{equation*}
  \mathcal{L}_{0,m-1}: L^{2}(\mathbb{R},\mathbb{C}^{2m-4})\to L^{2}(\mathbb{R},\mathbb{C}^{2m-4})  
\end{equation*} defined by $\mathcal{L}_{0,m-1}=(\mathcal{L}_{0,3,m},\mathcal{L}_{4,m},\mathcal{L}_{5,m}, ...,\mathcal{L}_{2m-2,m})$ such that 
\begin{align*}
    \mathcal{L}_{0,3,m}(\overrightarrow{\psi})=\psi_{3}(k)- r_{2}({-}k)\psi_{4}\left({-}k+\frac{v_{2}-v_{3}}{2}\right),
\end{align*}
and the remaining coordinate maps of $\mathcal{L}_{0,m-1}$ are the same maps defined at \eqref{Tp}.\footnote{One can think of this map as setting $\psi_1=\psi_2=0$ in  \eqref{Tp}, and then  disregard the trivial map $\mathcal{L}_{2,m}$.} Because of the inductive hypotheses of Lemma \ref{Tp} being true for $m-1,$ if $\min (v_{\ell}-v_{\ell+1})>0$ is large enough, then there is a parameter $c(m-1)>0$ depending only on $m-1$ satisfying
\begin{equation}\label{cc}
    \norm{\mathcal{L}_{0,m-1}(\overrightarrow{\psi})}_{L^{2}}\geq c(m-1)\max_{\ell\in\{3,4,...,2m-2\}}\norm{\psi_{\ell}}_{L^{2}},\, \norm{\mathcal{L}_{0,m-1}(\overrightarrow{\psi})}_{\mathcal{F}L^{1}}\geq c(m-1)\max_{\ell\in\{3,4,...,2m-2\}}\norm{\psi_{\ell}}_{\mathcal{F}L^{1}}.
\end{equation}
Notice that 
\begin{align}
    \norm{\mathcal{L}_{3,m}(\vec{\psi})}_{L^{2}}&\geq  \norm{\mathcal{L}_{0,3,m}(\vec{\psi})}_{L^{2}}-c'\norm{\psi_1}_{L^{2}}\\
      \norm{\mathcal{L}_{3,m}(\vec{\psi})}_{\mathcal{F}L^{1}}&\geq  \norm{\mathcal{L}_{0,3,m}(\vec{\psi})}_{\mathcal{F}L^{1}}-c'\norm{\psi_1}_{\mathcal{F}L^{1}}
\end{align}for some constant $c'$.
It follows from \eqref{cc}  that
\begin{align}
    \norm{\mathcal{L}(\vec{\psi})}_{L^{2}}&\geq c(m-1)\max_{\ell\in\{3,4,...,2m-2\}}\norm{\psi_{\ell}}_{L^{2}}-c'\norm{\psi_1}_{L^{2}}\\
      \norm{\mathcal{L}(\vec{\psi})}_{\mathcal{F}L^{1}}&\geq   c(m-1)\max_{\ell\in\{3,4,...,2m-2\}}\norm{\psi_{\ell}}_{\mathcal{F}L^{1}}-c'\norm{\psi_1}_{\mathcal{F}L^{1}}
\end{align}
After multiplying  \eqref{Final1} by a large constant and adding the resulting estimates to two estimates above,  we can deduce that
\begin{align}
    \norm{\mathcal{L}(\overrightarrow{\psi})}_{L^2}&\geq c(m-1)\max_{\ell\in\{1,3,4,...,2m-2\}}\norm{\psi_{\ell}}_{L^{2}}-K\max_{\ell}\frac{1}{v_{\ell}-v_{\ell+1}}\max_{j} \norm{\psi_{j}}_{L^{2}}\label{eq:finalmiss21}\\
     \norm{\mathcal{L}(\overrightarrow{\psi})}_{\mathcal{F}L^{1}}&\geq c(m-1)\max_{\ell\in\{1,3,4,...,2m-2\}}\norm{\psi_{\ell}}_{\mathcal{F}L^{1}}-K\max_{\ell}\frac{1}{v_{\ell}-v_{\ell+1}}\max_{j} \norm{\psi_{j}}_{\mathcal{F}L^{1}}\label{eq:finalmiss22}
\end{align}
By the definition of $\mathcal{L}_{2,m}(\vec{\psi})$, one has
\begin{align}
      \norm{\mathcal{L}(\vec{\psi})}_{L^{2}}&\geq  \norm{\mathcal{L}_{2,m}(\vec{\psi})}_{L^{2}}\geq  \norm{\psi_2}_{L^{2}}-\Tilde{c}\Big(\norm{\psi_1}_{L^{2}}+\norm{\psi_3}_{L^{2}}\Big)\\
      \norm{\mathcal{L}(\vec{\psi})}_{\mathcal{F}L^{1}}&\geq \norm{\mathcal{L}_{2,m}(\vec{\psi})}_{\mathcal{F}L^{1}}\geq  \norm{\psi_2}_{\mathcal{F}L^{1}}-\Tilde{c}\Big(\norm{\psi_1}_{\mathcal{F}L^{1}}+\norm{\psi_3}_{\mathcal{F}L^{1}}\Big)
\end{align}for some constant $\Tilde{c}$. 
We can use add the first estimate above  to a large multiple of \eqref{eq:finalmiss21} and obtain the existence of a $c(m)>0$ satisfying
\begin{align*}
    \norm{\mathcal{L}(\overrightarrow{\psi})}_{L^{2}} \geq&c_{1}(m)\max_{\ell\in\{1,2,3,4,...,2m-2\}}\norm{\psi_{\ell}}_{L^{2}}-K\max_{\ell}\frac{1}{v_{\ell}-v_{\ell+1}}\max_{j} \norm{\psi_{j}}_{L^{2}}\\  \geq &c(m)\max_{\ell\in\{1,2,3,4,...,2m-2\}}\norm{\psi_{\ell}}_{L^{2}},
\end{align*}
if $\min_{\ell} v_{\ell}-v_{\ell+1}>0$ is large enough, so Lemma \ref{Tp} is true for all $m\in\mathbb{N}_{\geq 2}.$
Similarly, we can deduce that
\begin{equation*}
    \norm{\mathcal{L}(\overrightarrow{\psi})}_{\mathcal{F}L_{1}} \geq c(m)\max_{\ell\in\{1,2,3,4,...,2m-2\}}\norm{\psi_{\ell}}_{\mathcal{F}L^{1}},
\end{equation*}
which implies \eqref{fl11}.
\end{proof}
\begin{proof}[Proof of Theorem \ref{TT}]
\par First, in notation of Lemma \ref{T}, we consider the following function
\begin{equation}\label{Bss}
 B\left(\begin{bmatrix}
     \phi_{1}\\
     \phi_{2}
 \end{bmatrix}\right)(k):=\left(B_{1}\left(\mathcal{S}(0)\left(\begin{bmatrix}
     \phi_{1}\\
     \phi_{2}
 \end{bmatrix}\right)\right)(k),...,B_{2m-2}\left(\mathcal{S}(0)\left(\begin{bmatrix}
     \phi_{1}\\
     \phi_{2}
 \end{bmatrix}\right)\right)(k)\right).
\end{equation}
Applying Lemma \ref{T} to the functions above, we note that the  right-hand side of the resulting estimates enjoys the form of the linear map defined in Lemma \ref{Tp}.  So  Lemma \ref{Tp} implies that
\begin{align}\label{pp}
    \norm{B\left(\begin{bmatrix}
     \phi_{1}\\
     0
 \end{bmatrix}\right)(k)}_{L^{2}_{k}(\mathbb{R})}\geq & c(m)\left[\sum_{\ell=1}^{m}\norm{\begin{bmatrix}
        \phi_{1,\ell}(k)\\
        0
    \end{bmatrix}}_{L^{2}_{k}(\mathbb{R})}\right],\\ \label{negli1}
    B\left(\begin{bmatrix}
     \phi_{1}\\
     0
 \end{bmatrix}\right)(k)=&\begin{bmatrix}
     f_{1}(k)\\
     f_{2}(k)
 \end{bmatrix} \text{, such that $f_{2}\cong 0,$ see Notation \ref{congg}.}
\end{align}
In the formula above each function $\phi_{1,\ell}$ is defined in Definition \ref{s0def} at the item $b),$ see also the Remark \ref{Tremark} which explains briefly the reason that the estimate above is a consequence of Lemma \ref{Tp}. 
\par Therefore, since for any $1\leq \ell \leq 2m-2$ all the linear maps $B_{\ell}$ are bounded into $L^{2}_{x}(\mathbb{R},\mathbb{C}),$ the equation \eqref{Bss} and the estimate \eqref{pp} imply that
\begin{equation}\label{++}
    \norm{\mathcal{S}(0)\left(
    \begin{bmatrix}
    \phi_{1}\\
    0
    \end{bmatrix}
    \right)(x)}_{L^{2}_{x}(\mathbb{R})}\geq c(m)\left[\sum_{\ell=1}^{m}\norm{\begin{bmatrix}
        \phi_{1,\ell}(k)\\
        0
    \end{bmatrix}}_{L^{2}_{k}(\mathbb{R})}\right],
\end{equation}
when $\min_{\ell} v_{\ell}-v_{\ell+1}$ and $\min y_{\ell}-y_{\ell+1}$ are sufficiently large.
\par Similarly to the proofs of \eqref{++} and \eqref{negli1}, we can verify that
\begin{align}\label{--}
    \norm{\mathcal{S}(0)\left(
    \begin{bmatrix}
    0\\
    \phi_{2}
    \end{bmatrix}
    \right)(x)}_{L^{2}_{x}(\mathbb{R})}\geq &c(m)\left[\sum_{\ell=1}^{m}\norm{\begin{bmatrix}
        0\\
        \phi_{2,\ell}(k)
    \end{bmatrix}}_{L^{2}_{k}(\mathbb{R})}\right],\\ \label{negli2}
    B\left(\begin{bmatrix}
     0\\
     \phi_{2}
 \end{bmatrix}\right)(k)=&\begin{bmatrix}
     g_{1}(k)\\
     g_{2}(k)
 \end{bmatrix} \text{, such that $g_{1}\cong 0,$ see Notation \ref{congg}},
\end{align}
and all the functions $\phi_{2,\ell}$ are defined at the item $b)$ of Definition \ref{s0def}.
\par In conclusion, the estimate \eqref{Scoerc}
of Theorem \ref{T} proceeds from the estimates \eqref{--}, \eqref{++} and the linearity of the map $\mathcal{S}(0).$ 
\par Next, using Remarks \ref{weightremark}, \ref{polyweightre} and the decay estimates
\begin{equation*}
\left\vert\frac{d^{\ell}}{dk^{\ell}}s_{\omega_{n}}(k)\right\vert+\left\vert\frac{d^{\ell}}{dk^{\ell}}r_{\omega_{n}}(k)\right\vert=O\left(\frac{1}{(1+\vert k\vert)^{\ell+1}}\right),
\end{equation*}
we can derive the estimates \eqref{b1}, \eqref{b2n}, \eqref{b2n+1} and \eqref{b2m-2} twice on $k$ and obtain similarly the proof of \eqref{Scoerc} that the estimate \eqref{derivdecay} is true.
\par Similarly, using Lemmas \ref{T} and \ref{Tp}, we can verify that
\begin{equation*}
    \norm{S(0)\left(\begin{bmatrix}
        \phi_{1}\\
        \phi_{2}
    \end{bmatrix}\right)(x)}_{L^{2}_{x}(\mathbb{R})}+\norm{B\left(\begin{bmatrix}
        \phi_{1}\\
        \phi_{2}
    \end{bmatrix}\right)(k)}_{\mathcal{F}L_{1}}\geq \sum_{\ell}c(m)\norm{\begin{bmatrix}
        \phi_{1,\ell}(k)\\
        \phi_{2,\ell}(k)
    \end{bmatrix}}_{\mathcal{F}L_{1}},
\end{equation*}
from which we can conclude that inequality \eqref{derivdecay} as a consequence of the elementary estimate
\begin{equation*}
    \norm{F^{*}_{\omega_{\ell}}\left(\chi_{P_{\ell}}(x)h(x)\right)(k)}_{\mathcal{F}L_{1}}\leq c \norm{F^{*}_{\omega_{\ell}}(k)\left(\chi_{P_{\ell}}(x)h(x+y_{0})\right)}_{H^{1}_{k}(\mathbb{R})}\leq c_{2}\norm{(1+\vert x\vert)\chi_{P_{\ell}}(x)h(x+y_{0})}_{L^{2}_{x}(\mathbb{R})},
\end{equation*}
for any real number $y_{0}$ and $c_{1},\,c_{2}$ are positive constants. Moreover, if the hypothesis $(H4)$ is true, Lemma \ref{fordecay} implies that
\begin{equation*}
    \max_{\ell}\norm{F^{*}_{\omega_{\ell}}\left(\chi_{P_{\ell}}(x)h(x)\right)(k)}_{\mathcal{F}L_{1}}\leq C\norm{h}_{L^{1}_{x}(\mathbb{R})},
\end{equation*}
for some constant $C>1,$ from which we can deduce \eqref{fl1teo2}. 
\par Next, using all the estimates in Lemma \ref{T} and the decay rates \eqref{asyreftr} satisfied by all functions $r_{\omega_{\ell}}(k),\,s_{\omega_{\ell}}(k),$ we can verify that 
all the functions
\begin{align*}
    \psi_{1}(k)=&\frac{\partial}{\partial k}\left[ P_{-}\left(e^{i\frac{k(y_{1}+y_{2})}{2}}
 \begin{bmatrix}
    \phi_{1}\left(k+\frac{v_{1}}{2}\right)\\
    \phi_{2}\left(k-\frac{v_{1}}{2}\right)
 \end{bmatrix}\right)\right],\\
 \psi_{2n}(k)=&\frac{\partial}{\partial k}\left[P_{+}\left(\begin{bmatrix}
e^{i\frac{(y_{n}+y_{n+1})k}{2}}\phi_{1,n}(k+\frac{v_{n+1}}{2})\\
e^{i\frac{(y_{n}+y_{n+1})k}{2}}\phi_{2,n}(k-\frac{v_{n+1}}{2})
\end{bmatrix}\right)\right],\\
\psi_{2n+1}=&\frac{\partial}{\partial k}\left[P_{-}\left(\begin{bmatrix}
e^{i\frac{(y_{n+1}+y_{n+2})k}{2}}\phi_{1,n+1}(k+\frac{v_{n+1}}{2})\\
e^{i\frac{(y_{n+1}+y_{n+2})k}{2}}\phi_{2,n+1}(k-\frac{v_{n+1}}{2})
\end{bmatrix}\right)\right],\\
\psi_{2m-2}=&\frac{\partial}{\partial k}\left[ P_{+}\left(\begin{bmatrix}
e^{i\frac{(y_{m-1}+y_{m})k}{2}}\phi_{1,m-1}(k+\frac{v_{m}}{2})\\
e^{i\frac{(y_{m-1}+y_{m})k}{2}}\phi_{2,m-1}(k-\frac{v_{m}}{2})
\end{bmatrix}\right)\right]
\end{align*}
satisfy the conditions of Lemma  \ref{Tp}  for a set of functions $\mathcal{L}_{j,m}(\vec{\psi})$ satisfying for a constant $C>1$
\begin{align}\label{l000}
    \norm{\mathcal{L}_{j,m}(\vec{\psi})}_{\mathcal{F}L_{1}}\leq & C\Bigg[ \max_{\pm,\ell}\vert y_{\ell}-y_{\ell\pm 1}\vert \norm{\begin{bmatrix}
        \phi_{1,\ell}(k)\\
        \phi_{2,\ell}(k)
\end{bmatrix}}_{\mathcal{F}L_{1}}+\max_{\ell}\norm{\frac{d}{dk}B_{\ell}\left(\mathcal{S}(0)(\vec{\phi})\right)(k)}_{\mathcal{F}L_{1}}\\&{+}e^{{-}\beta\min_{\ell}(y_{\ell}-y_{\ell+1})}\max_{\ell}\norm{\begin{bmatrix}
        \phi_{1,\ell}(k)\\
        \phi_{2,\ell}(k)
    \end{bmatrix}}_{L^{2}_{k}(\mathbb{R})}\Bigg],
\end{align}        
for any $j\in\{1,2,\,...,\,2m-2\}.$ In conclusion, Lemma \ref{Tp}, estimates \eqref{fl11} and the following inequalities satisfied with some constants $c_{1},\,c_{2}>0$
\begin{align*}
    \norm{\frac{d}{dk}B_{\ell}\left(\mathcal{S}(0)(\vec{\phi})\right)(k)}_{\mathcal{F}L_{1}}\leq & c_{1}\norm{B_{\ell}\left(\mathcal{S}(0)(\vec{\phi})\right)(k)}_{H^{2}_{k}(\mathbb{R})}\\
    \leq & c_{2} \max_{\ell}\norm{(1+\vert x\vert )^{2}\chi_{\ell}(\tau,x+y_{\ell}+v_{\ell}\tau)\mathcal{S}(\vec{\phi})(0,x+y_{\ell}+v_{\ell}\tau)}_{L^{2}_{x}(\mathbb{R})}
\end{align*}
imply \eqref{derivdecay2}. \par Furthermore, if the hypothesis $(H4)$ is true, then Lemma \ref{fordecay}, \eqref{fl1teo2} and inequality \eqref{l000} imply that
\begin{align}\nonumber
  \norm{\mathcal{L}_{j,m}(\vec{\psi})}_{\mathcal{F}L_{1}}\leq &C\max_{\pm,\ell}\vert y_{\ell}-y_{\ell\pm 1}\vert\left[\norm{\mathcal{S}(0)\left(\begin{bmatrix}
      \phi_{1}\\
      \phi_{2}
  \end{bmatrix}\right)(x)}_{L^{1}_{x}(\mathbb{R})}+e^{{-}\beta \min_{\ell}(y_{\ell}-y_{\ell+1})}\norm{\mathcal{S}(0)\left(\begin{bmatrix}
      \phi_{1}\\
      \phi_{2}
  \end{bmatrix}\right)(x)}_{L^{2}_{x}(\mathbb{R})}\right]\\ \label{fofofo1}
  &{+}C\left[\max_{\ell}\norm{(1+\vert x\vert)\chi_{\ell}(\tau,x+y_{\ell}+v_{\ell}\tau)\mathcal{S}(\vec{\phi})(0,x+y_{\ell}+v_{\ell}\tau)}_{L^{1}_{x}(\mathbb{R})}\right].
\end{align}
\par Therefore, using Lemmas \ref{sobolevdecayofG}, \ref{appFourier}, along with \eqref{l000}, we can repeat the argument used to prove inequality \eqref{Scoerc} that the estimate \eqref{Scoerc2} is true for some constant $c(m)>0.$ For the case where hypothesis $(H4)$ is true, we can deduce \eqref{fl1teo1} from the same argument, but using estimate \eqref{fofofo1} instead of \eqref{l000}. 
\end{proof}

\section{Proof of asymptotic completeness}\label{Asc}
\par With the aim of proving asymptotic completeness, the demonstration will mainly be based on an application of the following results.
\begin{lemma}\label{dec}
 If $f\in L^{2}(\mathbb{R},\mathbb{C}^{2})$ is any function and $\min_{\ell}y_{\ell}-y_{\ell+1}>0$ is large enough, then there exist $\vec{\phi}\in L^{2}(\mathbb{R},\mathbb{C}^{2})$ belonging to the domain of $\mathcal{S}(0),$ and a finite set of functions $\overrightarrow{v_{d_{\ell}}}\in \Raa P_{d,\omega_{\ell}}$ satisfying
 \begin{align*}
     f(x)=&\mathcal{S}(0)(\vec{\phi})(x)+\sum_{\ell=1}^{m}e^{i\frac{\sigma_{3}v_{\ell}x}{2}}\overrightarrow{v_{d_{\ell}}}(x-y_{\ell}).\end{align*}

\end{lemma}
\begin{corollary}\label{c62}
  If $f\in L^{2}(\mathbb{R},\mathbb{C}^{2}),$ such that
  \begin{align}\label{idprincipal}
     f(x)=&\mathcal{S}(0)(\vec{\phi})(x)+\sum_{\ell=1}^{m}e^{i\frac{\sigma_{3}v_{\ell}x}{2}}\overrightarrow{v_{d_{\ell}}}(x-y_{\ell}),\end{align}
  there exists $C_{m}>1$ satisfying
\begin{equation}\label{poo}
    \max_{\ell}\left[\norm{(\phi_{1,\ell},\phi_{2,\ell})}_{L^{2}_{x}(\mathbb{R})}+\norm{\overrightarrow{v_{d_{\ell}}}}_{L^{2}_{x}(\mathbb{R})}\right]\leq C_{m}\norm{f}_{L^{2}_{x}(\mathbb{R})}.
\end{equation}
Furthermore, if $f(x)\in H^{1}_{x}(\mathbb{R}),$ then there exists $C_{m,n}$ satisfying
\begin{equation}\label{eq222}
    \max \norm{\left((k+\frac{v_{\ell}}{2})\phi_{1,\ell}(k),(k-\frac{v_{\ell}}{2})\phi_{2,\ell}(k)\right)}_{L^{2}_{x}(\mathbb{R})}\leq C_{m,2}\norm{f}_{H^{1}_{x}(\mathbb{R})}.
\end{equation}
\end{corollary}
\begin{proof}[Proof of Corollary \ref{c62}.]
First, using Lemma \ref{localS0} and Corollary \ref{Asy1sol2}, we can verify for any function $v_{d_{\ell}}\in \Raa P_{d,\omega_{\ell}}$ that there exists $\beta>0$ satisfying
\begin{equation*}
    \left\langle \sigma_{3}\mathcal{S}(0)(\vec{\phi})(x),e^{i\frac{\sigma_{3}v_{\ell}x}{2}}v_{d_{\ell}}(x-y_{\ell}) \right\rangle=O\left(\norm{v_{d_{\ell}}}_{L^{2}_{x}(\mathbb{R})}\max_{\ell}\norm{(\phi_{1,\ell},\phi_{2,\ell})}_{L^{2}_{x}(\mathbb{R})}e^{{-}\beta \min (y_{\ell}-y_{\ell+1})}\right),
\end{equation*}
when $\min y_{\ell}-y_{y_{\ell+1}}>0$ is large enough. In particular, Theorem \ref{TT} implies the existence of a constant $C>1$ satisfying\begin{equation}\label{prepre}
    \left\vert \left\langle \sigma_{3}\mathcal{S}(0)(\vec{\phi})(x),e^{i\frac{\sigma_{3}v_{\ell}x}{2}}v_{d_{\ell}}(x-y_{\ell}) \right\rangle\right\vert\leq  C\norm{v_{d_{\ell}}}_{L^{2}_{x}(\mathbb{R})}\max_{\ell}\norm{\mathcal{S}(0)(\vec{\phi})(x)}_{L^{2}_{x}(\mathbb{R})}e^{{-}\beta \min (y_{\ell}-y_{\ell+1})},
\end{equation}
for any $v_{d_{\ell}}\in \Raa P_{d,\omega_{\ell}}.$
Consequently,  $\mathcal{S}(0)(\vec{\phi})$ can not belong to the finite-dimensional Hilbert subspace$\mathfrak{g}_{\omega_{\ell},v_{\ell},\gamma_{\ell},y_{\ell}}(\Raa P_{d,\omega_{\ell}})$ of $L^{2}_{x}(\mathbb{R},\mathbb{C}^{2}),$ otherwise \eqref{prepre} would be false.
\par Furthermore, it is well-known for any basis $\{\overrightarrow{u_{d_{\ell},1}},...,\overrightarrow{u_{d_{\ell},n_{\ell}}}\}$ of the set $\Raa P_{d,\omega_{\ell}}$ that the matrix
\begin{equation*}
   A_{_{n_{\ell}\times n_{\ell}} }=\begin{bmatrix}
     \left\langle \overrightarrow{u_{d_{\ell},m}},\sigma_{3}\overrightarrow{u_{d_{\ell},n}} \right\rangle  
   \end{bmatrix}_{n_{\ell}\times n_{\ell}} 
\end{equation*}
is non-degenerate, see \cite{Busper1} or \cite{KriegerSchlag} for example. Therefore, we can verify that
identity \eqref{idprincipal} and estimate \eqref{prepre} imply the following estimate
\begin{equation*}
\max_{\ell}\norm{\overrightarrow{v_{d_{\ell}}}}_{L^{2}_{x}(\mathbb{R})}\leq C\left[\norm{f(x)}_{L^{2}_{x}(\mathbb{R})}+e^{{-}\beta \min_{\ell}(y_{\ell}-y_{\ell+1})}\norm{\mathcal{S}(0)(\vec{\phi})(x)}_{L^{2}_{x}(\mathbb{R})}\right],
\end{equation*}
from which we can deduce using the Minkowski inequality, Young's inequality and \eqref{idprincipal} the existence of a constant $c>0$ satisfying
\begin{equation*}
c\left[ \norm{\mathcal{S}(0)(\vec{\phi}(x))}_{L^{2}_{x}(\mathbb{R})}+\max_{\ell}\norm{\overrightarrow{v_{d_{\ell}}}(x)}_{L^{2}_{x}(\mathbb{R})}\right]\leq \norm{f(x)}_{L^{2}_{x}(\mathbb{R})},
\end{equation*}
when $\min_{\ell}y_{\ell}-y_{\ell+1}>0$ is large enough.
\par Next, since all elements of $\Raa P_{d,\omega_{\ell}}$ are Schwartz functions, we have that $f(x)\in H^{1}_{x}(\mathbb{R})$ implies that
$\mathcal{S}(0)(\vec{\phi})\in H^{1}_{x}(\mathbb{R}).$ Consequently, using Theorem \ref{tcont}, the fact that $\dim \Raa P_{d,\omega_{\ell}}<{+}\infty,$ and Lemma \ref{h1coercc} from Section \ref{Dispsection}, we can verify that inequality \eqref{eq222} holds when $f\in H^{1}_{x}(\mathbb{R}).$
\end{proof}

Before proceeding to the proof of Lemma \ref{dec}, we state the following proposition which is one of the main techniques that we use to prove Lemma \ref{dec} above.

\begin{lemma}\label{tcopp}
 Let $T:L^{2}(\mathbb{R},\mathbb{C}^{2m-2})\to L^{2}(\mathbb{R},\mathbb{C}^{2m-2})$ such that $T(g)\coloneqq (T_{1}(g),\,..., T_{2m-2}(g))$ is defined by
 \begin{align*}
     T_{1}(g)=&g_{1}\left({-}k-\frac{v_{2}}{2}\right)-r_{2}(k)g_{2}\left(k-\frac{v_{2}}{2}\right)-s_{2}(k)g_{3}\left({-}k-\frac{v_{2}}{2}\right),\\
     T_{2}(g)=&g_{2}\left({-}k-\frac{v_{2}}{2}\right)-r_{1}\left({-}k-\frac{v_{2}}{2}+\frac{v_{1}}{2}\right)g_{1}\left(k+\frac{v_{2}}{2}-v_{1}\right),\\
      T_{3}(g)=& g_{3}\left({-}k-\frac{v_{3}}{2}\right)-r_{3}\left(k\right)g_{4}\left(k-\frac{v_{3}}{2}\right)-s_{3}\left(k\right)g_{5}\left({-}k-\frac{v_{3}}{2}\right),\\
      T_{4}(g)=& g_{4}\left({-}k-\frac{v_{3}}{2}\right)-s_{2}\left({-}k-\frac{v_{3}}{2}+\frac{v_{2}}{2}\right)g_{2}\left({-}k-\frac{v_{3}}{2}\right)\\
      &{-}r_{2}\left({-}k+\frac{v_{2}}{2}-\frac{v_{3}}{2}\right)g_{3}\left(k+\frac{v_{3}}{2}-v_{2}\right),\, ...,\\
      T_{2\ell}(g)=& g_{2\ell}\left({-}k-\frac{v_{\ell+1}}{2}\right)-s_{\ell}\left({-}k-\frac{v_{\ell+1}}{2}+\frac{v_{\ell}}{2}\right)g_{2\ell-2}\left({-}k-\frac{v_{\ell+1}}{2}\right)\\
      &{-}r_{\ell}\left({-}k+\frac{v_{\ell}}{2}-\frac{v_{\ell+1}}{2}\right)g_{2\ell-1}\left(k+\frac{v_{\ell}}{2}-v_{\ell-1}\right),\\
      T_{2\ell+1}(g)=&g_{2\ell+1}\left({-}k-\frac{v_{\ell+2}}{2}\right)-s_{\ell+2}(k)g_{2\ell+3}\left({-}k-\frac{v_{\ell+2}}{2}\right)\\
      &{-}r_{\ell+2}(k)g_{2\ell+2}\left(k-\frac{v_{\ell+2}}{2}\right),\, ...,\\
      T_{2m-3}(g)=& g_{2m-3}\left({-}k-\frac{v_{m}}{2}\right)-r_{m}(k)g_{2m-2}\left(k-\frac{v_{m}}{2}\right),\\
      T_{2m-2}(g)=& g_{2m-2}\left({-}k-\frac{v_{m}}{2}\right)-s_{m-1}\left({-}k-\frac{v_{m}}{2}+\frac{v_{m-1}}{2}\right)g_{2m-4}\left({-}k-\frac{v_{m}}{2}\right)\\
      &{-}r_{m-1}\left({-}k+\frac{v_{m-1}}{2}-\frac{v_{m}}{2}\right)g_{2m-3}\left(k+\frac{v_{m}}{2}-v_{m-1}\right),
 \end{align*}
for any function $g=(g_{1},\,...,\,g_{2m-2})\in L^{2}(\mathbb{R},\mathbb{C}^{2m-2}),$ and all the functions $r_{\ell},\,s_{\ell}$ are in $L^{\infty}(\mathbb{R},\mathbb{C})$ satisfying estimates \eqref{pep}. There exist $c,\,K>0$ such that if $\min_{\ell}v_{\ell}-v_{\ell+1}>K,$ then $T$ is an isomorphism and 
\begin{equation*}
    \norm{T(g)(k)}_{L^{2}_{k}(\mathbb{R})}\geq c \norm{g(k)}_{L^{2}_{k}(\mathbb{R})},
\end{equation*}
for all $g\in L^{2}_{k}(\mathbb{R},\mathbb{C}^{2m-2}).$
\end{lemma}
\begin{proof}
See Appendix  \ref{app}.
\end{proof}

\begin{proof}[Proof of Lemma \ref{dec}.]
 \par First, given a function $f\in L^{2}(\mathbb{R},\mathbb{C}^{2}),$ we will construct a system of equations to associate $f$ with an element $\vec{\phi}$ in the domain of $\mathcal{S}(0).$ More precisely, for a fixed $f\in L^{2}(\mathbb{R},\mathbb{C}^{2}),$ we consider the following system of equations below having  unknowns: the functions $f_{\ell,\pm}\in L^{2}(\mathbb{R},\mathbb{C}^{2})$ for $\ell\in\{1,2\,...,\,m\},$ an element $\overrightarrow{\phi}$ belonging to the domain of $\mathcal{S}(0),$ and a finite set of functions $\overrightarrow{v_{d_{\ell}}}\in\Raa P_{d,\omega_{\ell}}$ 
 \begin{align}\label{lG}
& e^{i\frac{\sigma_{3}v_{\ell}x}{2}}\hat{G}_{\omega_{\ell}}\left(e^{iy_{\ell}k}\begin{bmatrix}
     \phi_{1,\ell}\left(k+\frac{v_{\ell}}{2}\right)\\
     \phi_{2,\ell}\left(k-\frac{v_{\ell}}{2}\right)
 \end{bmatrix}\right)(x-y_{\ell})  \\
 &= \chi_{\left\{\frac{(y_{\ell+1}+y_{\ell})}{2}<x\leq \frac{(y_{\ell-1}+y_{\ell})}{2}\right\}}(x)f(x)+ e^{i\frac{\sigma_{3}v_{\ell}x}{2}}\overrightarrow{v_{d_{\ell}}}(x-y_{\ell})\\ \nonumber
 & {+}\chi_{\left\{x\leq \frac{(y_{\ell+1}+y_{\ell})}{2}\right\}}(x)f_{\ell,-}(x)+\chi_{\left\{x> \frac{(y_{\ell}+y_{\ell-1})}{2}\right\}}(x)f_{\ell,+}(x),\\ \label{1G}
 & e^{i\frac{\sigma_{3}v_{1}x}{2}}\hat{G}_{\omega_{1}}\left(e^{iy_{1}k}\begin{bmatrix}
     \phi_{1,1}\left(k+\frac{v_{1}}{2}\right)\\
     \phi_{2,1}\left(k-\frac{v_{1}}{2}\right)
 \end{bmatrix}\right)(x-y_{1})\\    
 &=  \chi_{\left\{x>\frac{(y_{1}+y_{2})}{2}\right\}}(x)f(x)+\chi_{\left\{x\leq \frac{(y_{1}+y_{2})}{2}\right\}}(x)f_{1,-}(x)+e^{i\frac{\sigma_{3}v_{1}x}{2}}\overrightarrow{v_{d_{1}}}(x-y_{1}),\\ \label{mG}
& e^{i\frac{\sigma_{3}v_{m}x}{2}}\hat{G}_{\omega_{m}}\left(e^{iy_{m}k}\begin{bmatrix}
     \phi_{1,m}\left(k+\frac{v_{m}}{2}\right)\\
     \phi_{2,m}\left(k-\frac{v_{m}}{2}\right)
 \end{bmatrix}\right)(x-y_{m})\\
 &= \chi_{\left\{x\leq  \frac{(y_{m}+y_{m-1})}{2}\right\}}(x)f(x)+\chi_{\left\{x>\frac{(y_{m}+y_{m-1})}{2}\right\}}(x)f_{m,+}(x){+}e^{i\frac{\sigma_{3}v_{m}x}{2}}\overrightarrow{v_{d_{m}}}(x-y_{m}),
 \end{align}
for any $\ell\in\{2,..., m-1\}.$
Multiplying \eqref{lG}, \eqref{1G} and \eqref{mG}  by $\chi_{\ell}$,  $\chi_{1},$  and $\chi_{m}$ respectively, and summing all of these new equations,
%
one can verify that if the system of equations above is solvable, then the function $f$ satisfies the following identity
\begin{align*}
    f(x)=&\sum_{\ell}e^{i\frac{\sigma_{3}v_{\ell}x}{2}}\hat{G}_{\omega_{\ell}}\left(e^{iy_{\ell}k}\begin{bmatrix}
     \phi_{1,\ell}\left(k+\frac{v_{\ell}}{2}\right)\\
     \phi_{2,\ell}\left(k-\frac{v_{\ell}}{2}\right)
 \end{bmatrix}\right)(x-y_{\ell})\chi_{\left\{\frac{(y_{\ell+1}+y_{\ell})}{2}<x\leq \frac{(y_{\ell-1}+y_{\ell})}{2}\right\}}(x)\\&{-}\sum_{\ell}e^{i\frac{\sigma_{3}v_{\ell}x}{2}}\overrightarrow{v_{d_{\ell}}}(x-y_{\ell}) \chi_{\left\{\frac{(y_{\ell+1}+y_{\ell})}{2}<x\leq \frac{(y_{\ell-1}+y_{\ell})}{2}\right\}}(x).
\end{align*}
In particular, using Remark \ref{transition} and identities $F^{*}_{\omega_\ell}\sigma_{3}\hat{G}_{\omega_\ell}\sigma_{3}=\mathrm{Id},\, G^{*}_{\omega_\ell}\sigma_{3}\hat{F}_{\omega_\ell}\sigma_{3}=\mathrm{Id}$ from Lemma \ref{leper}, we verify that the finite system of equations above is equivalent to a simpler linear system. 
 \par Moreover, Corollary \ref{Asy1sol2} implies that for each $\ell\in\{2,... , m-1\}$
 \begin{equation*}
     F^{*}_{\omega_{\ell}}\left(\sigma_{3}\overrightarrow{v_{d_{\ell}}}(x)\right)= G^{*}_{\omega_{\ell}}\left(\sigma_{3}\overrightarrow{v_{d_{\ell}}}(x)\right)=0.
 \end{equation*}
  Consequently, using Lemma \ref{FGID} and Corollary \ref{Asy1sol2}, 
  we can deduce for any $\ell\in\{2,..,m-1\}$ that
 \begin{align}\label{l1}
  \begin{bmatrix}
     e^{iy_{\ell}k} \phi_{1,\ell}(k+\frac{v_{\ell}}{2})\\
      {-}e^{iy_{\ell}k}\phi_{2,\ell}(k-\frac{v_{\ell}}{2})
  \end{bmatrix}=&F_{\omega_{\ell}}^{*}\left(\sigma_{3}e^{{-}i\sigma_{3}\frac{v_{\ell}(x+y_{\ell})}{2}}\chi_{\{\frac{y_{\ell+1}-y_{\ell}}{2}< x\leq \frac{y_{\ell-1}-y_{\ell}}{2}\}}(x)f(x+y_{\ell})\right)(k)\\ \nonumber
&{+}F_{\omega_{\ell}}^{*}\left(\sigma_{3}e^{{-}i\sigma_{3}\frac{v_{\ell}(x+y_{\ell})}{2}}\chi_{\{x\leq \frac{y_{\ell+1}-y_{\ell}}{2}\}}(x)f_{\ell,-}(x+y_{\ell})\right)(k)\\ \nonumber
&{+}F_{\omega_{\ell}}^{*}\left(\sigma_{3}e^{{-}i\sigma_{3}\frac{v_{\ell}(x+y_{\ell})}{2}}\chi_{\{x> \frac{y_{\ell-1}-y_{\ell}}{2}\}}(x)f_{\ell,+}(x+y_{\ell})\right)(k),
 \end{align}
and, using Remark \ref{transition}, we can verify for any $\ell\in \{2,..,m-1\}$ that
\begin{align}\label{l2}
  \begin{bmatrix}
     e^{iy_{\ell}k} \phi_{1,\ell-1}(k+\frac{v_{\ell}}{2})\\
      {-}e^{iy_{\ell}k}\phi_{2,\ell-1}(k-\frac{v_{\ell}}{2})
  \end{bmatrix}=&G_{\omega_{\ell}}^{*}\left(\sigma_{3}e^{{-}i\sigma_{3}\frac{v_{\ell}(x+y_{\ell})}{2}}\chi_{\{\frac{y_{\ell+1}-y_{\ell}}{2}< x\leq \frac{y_{\ell-1}-y_{\ell}}{2}\}}(x)f(x+y_{\ell})\right)(k)\\ \nonumber
&{+}G_{\omega_{\ell}}^{*}\left(\sigma_{3}e^{{-}i\sigma_{3}\frac{v_{\ell}(x+y_{\ell})}{2}}\chi_{\{x\leq \frac{y_{\ell+1}-y_{\ell}}{2}\}}(x)f_{\ell,-}(x+y_{\ell})\right)(k)\\ \nonumber
&{+}G_{\omega_{\ell}}^{*}\left(\sigma_{3}e^{{-}i\sigma_{3}\frac{v_{\ell}(x+y_{\ell})}{2}}\chi_{\{x> \frac{y_{\ell-1}-y_{\ell}}{2}\}}(x)f_{\ell,+}(x+y_{\ell})\right)(k).
 \end{align}
Furthermore, we can verify that
\begin{align}\label{l3}
\begin{bmatrix}
e^{iy_{m}k}\phi_{1,m-1}\left(k+\frac{v_{m}}{2}\right)\\
{-}e^{iy_{m}k}\phi_{2,m-1}\left(k-\frac{v_{m}}{2}\right)
\end{bmatrix}=&G_{\omega_{m}}^{*}\left(\sigma_{3}e^{{-}i\sigma_{3}\frac{v_{m}(x+y_{m})}{2}}\chi_{\{ x\leq  \frac{y_{m-1}-y_{m}}{2}\}}(x)f(x+y_{m})\right)(k)\\ \nonumber
&{+}G_{\omega_{m}}^{*}\left(\sigma_{3}e^{{-}i\sigma_{3}\frac{v_{m}(x+y_{m})}{2}}\chi_{\{x> \frac{y_{m-1}-y_{m}}{2}\}}(x)f_{m,+}(x+y_{m})\right)(k),   
\end{align}
and
\begin{align}\label{l4}
\begin{bmatrix}
e^{iy_{1}k}\phi_{1,1}\left(k+\frac{v_{1}}{2}\right)\\
{-}e^{iy_{1}k}\phi_{2,1}\left(k-\frac{v_{1}}{2}\right)
\end{bmatrix}=&F_{\omega_{1}}^{*}\left(\sigma_{3}e^{{-}i\sigma_{3}\frac{v_{1}(x+y_{1})}{2}}\chi_{\{ x>  \frac{y_{2}-y_{1}}{2}\}}(x)f(x+y_{1})\right)(k)\\ \nonumber
&{+}F_{\omega_{1}}^{*}\left(\sigma_{3}e^{{-}i\sigma_{3}\frac{v_{1}(x+y_{1})}{2}}\chi_{\{x\leq  \frac{y_{2}-y_{1}}{2}\}}(x)f_{1,-}(x+y_{1})\right)(k).   \end{align}
\par Next, applying Lemma \ref{appFourier} to the equations \eqref{l1}, \eqref{l2}, \eqref{l3} and \eqref{l4}, we will transform the system of equations above into a non-degenerate linear system of equations such that the unknowns are given by the  following  transforms of $f_{j,\pm}$
\begin{align}\label{unknown}
\begin{bmatrix}
    g_{1,-,1}(k)\\
    g_{1,-,2}(k)
\end{bmatrix}=&\int_{{-}\infty}^{0}f_{1,-}\left(x+\frac{y_{1}+y_{2}}{2}\right)e^{ikx}\,dx,\\ \nonumber  \begin{bmatrix}
    g_{\ell,-,1}(k)\\
    g_{\ell,-,2}(k)
\end{bmatrix}=&\int_{{-}\infty}^{0}f_{\ell,-}\left(x+\frac{y_{\ell}+y_{\ell+1}}{2}\right)e^{ikx}\,dx \text{, when $2\leq \ell\leq m-1,$}\\ \nonumber
\begin{bmatrix}
    g_{\ell,+,1}(k)\\
    g_{\ell,+,2}(k)
\end{bmatrix}=&\int_{0}^{{+}\infty}f_{\ell,+}\left(x+\frac{y_{\ell}+y_{\ell-1}}{2}\right)e^{ikx}\,dx \text{, when $2\leq \ell\leq m-1,$}\\ \nonumber \begin{bmatrix}
    g_{m,+,1}(k)\\
    g_{m,+,2}(k)
\end{bmatrix}=&\int_{0}^{{+}\infty}f_{m,+}\left(x+\frac{y_{m}+y_{m-1}}{2}\right)e^{ikx}\,dx.
\end{align}
Clearly, by Theorem \ref{h2rep}, the first two functions  lie in $H^2(\mathbb{C}_{-})$, and the last two lie in $H^2(\mathbb{C}_{+})$.

\textbf{Step 1. (Estimates of \eqref{l1} and \eqref{l2}.)} 
 First, to simplify  notations, we consider
\begin{align}
  F_{1,\ell}(f)(k)=&F_{\omega_{\ell}}^{*}\left(\sigma_{3}e^{{-}i\sigma_{3}\frac{v_{\ell}(x+y_{\ell})}{2}}\chi_{\{\frac{y_{\ell+1}-y_{\ell}}{2}< x\leq \frac{y_{\ell-1}-y_{\ell}}{2}\}}(x)f(x+y_{\ell})\right)(k),\\
  G_{1,\ell}(f)(k)=&G_{\omega_{\ell}}^{*}\left(\sigma_{3}e^{{-}i\sigma_{3}\frac{v_{\ell}(x+y_{\ell})}{2}}\chi_{\{\frac{y_{\ell+1}-y_{\ell}}{2}< x\leq \frac{y_{\ell-1}-y_{\ell}}{2}\}}(x)f(x+y_{\ell})\right)(k).
\end{align}
\par Next, using Remark \ref{f*r} and identity \eqref{forp1}, we obtain from the equation \eqref{l1} the following estimate
\begin{multline*}
\begin{aligned}
    \begin{bmatrix}
     e^{iy_{\ell}k} \phi_{1,\ell}(k+\frac{v_{\ell}}{2})\\
      {-}e^{iy_{\ell}k}\phi_{2,\ell}(k-\frac{v_{\ell}}{2})
  \end{bmatrix}=& F_{1,\ell}(f)(k){+}\overline{s_{\omega_{\ell}}(k)}\sigma_{3}\left[\int_{\frac{y_{\ell-1}-y_{\ell}}{2}}^{{+}\infty}e^{{-}ikx}e^{{-}i\sigma_{3}\frac{v_{\ell}(x+y_{\ell})}{2}}f_{\ell,+}(x+y_{\ell})\,dx\right]
\\ &{+}\sigma_{3}\int_{{-}\infty}^{\frac{y_{\ell+1}-y_{\ell}}{2}}\left(e^{{-}ikx}+r_{\omega_{\ell}}({-}k)e^{ikx}\right)e^{{-}i\sigma_{3}\frac{v_{\ell}(x+y_{\ell})}{2}}f_{\ell,-}(x+y_{\ell})\,dx\\&{+}O\left(\frac{1}{(1+\vert k\vert )}\norm{f_{\ell,-}(x)\chi_{\{x\leq\frac{y_{\ell}+y_{\ell+1}}{2}\}}}_{L^{2}_{x}(\mathbb{R})}e^{\beta\frac{(y_{\ell+1}-y_{\ell})}{2}}\right)\\&{+}O\left(\frac{1}{(1+\vert k\vert )}\norm{f_{\ell,+}(x)\chi_{\{x>\frac{y_{\ell}+y_{\ell-1}}{2}\}}}_{L^{2}_{x}(\mathbb{R})}e^{{-}\beta\frac{(y_{\ell-1}-y_{\ell})}{2}}\right).
\end{aligned}
\end{multline*}
Therefore, using a change of variables in the estimate above, and the second and third identities of \eqref{unknown}, we deduce that
\begin{multline}
\label{estl1}
\begin{aligned}
    \begin{bmatrix}
     e^{iy_{\ell}k} \phi_{1,\ell}(k+\frac{v_{\ell}}{2})\\
{-}e^{iy_{\ell}k}\phi_{2,\ell}(k-\frac{v_{\ell}}{2})
  \end{bmatrix}=& F_{1,\ell}(f)(k){+}\overline{s_{\omega_{\ell}}(k)}e^{{-}\frac{ik(y_{\ell-1}-y_{\ell})}{2}}
  \begin{bmatrix}
   e^{{-}i\frac{v_{\ell}(y_{\ell-1}+y_{\ell})}{4}}g_{\ell,+,1}\left({-}k-\frac{v_{\ell}}{2}\right)\\
   {-}e^{i\frac{v_{\ell}(y_{\ell-1}+y_{\ell})}{4}}g_{\ell,+,2}\left({-}k+\frac{v_{\ell}}{2}\right)
  \end{bmatrix}
\\ &{+}e^{i\frac{(y_{\ell}-y_{\ell+1})k}{2}}\begin{bmatrix}
    e^{{-}i\frac{v_{\ell}(y_{\ell}+y_{\ell+1})}{4}}g_{\ell,-,1}\left({-}k-\frac{v_{\ell}}{2}\right)\\
    {-}e^{i\frac{v_{\ell}(y_{\ell}+y_{\ell+1})}{4}}g_{\ell,-,2}\left({-}k+\frac{v_{\ell}}{2}\right)
\end{bmatrix}\\
& +r_{\omega_{\ell}}({-}k)e^{i\frac{k(y_{\ell+1}-y_{\ell})}{2}}\begin{bmatrix}
    e^{{-}i\frac{v_{\ell}(y_{\ell}+y_{\ell+1})}{4}}g_{\ell,-,1}\left(k-\frac{v_{\ell}}{2}\right)\\
    {-}e^{i\frac{v_{\ell}(y_{\ell}+y_{\ell+1})}{4}}g_{\ell,-,2}\left(k+\frac{v_{\ell}}{2}\right)
\end{bmatrix}
\\&{+}O\left(\frac{1}{(1+\vert k\vert )}\norm{f_{\ell,-}(x)\chi_{\{x\leq\frac{y_{\ell}+y_{\ell+1}}{2}\}}}_{L^{2}_{x}(\mathbb{R})}e^{\beta\frac{(y_{\ell+1}-y_{\ell})}{2}}\right)\\&{+}O\left(\frac{1}{(1+\vert k\vert )}\norm{f_{\ell,+}(x)\chi_{\{x>\frac{y_{\ell}+y_{\ell-1}}{2}\}}}_{L^{2}_{x}(\mathbb{R})}e^{{-}\beta\frac{(y_{\ell-1}-y_{\ell})}{2}}\right).
\end{aligned}
\end{multline}
\par Similarly, when $\min\left(y_{\ell}-y_{\ell+1},y_{\ell-1}-y_{\ell}\right)$ is sufficiently large, we can verify the existence of a $\beta>0$ satisfying 
\begin{multline*}
\begin{aligned}
 \begin{bmatrix}
     e^{iy_{\ell}k}\phi_{1,\ell-1}\left(k+\frac{v_{\ell}}{2}\right)\\
     e^{iy_{\ell}k}\phi_{1,\ell-1}\left(k-\frac{v_{\ell}}{2}\right)
 \end{bmatrix}=& G_{1,\ell}(f)(k){+}s_{\omega_{\ell}}(k)\sigma_{3}\int_{{-}\infty}^{\frac{y_{\ell+1}-y_{\ell}}{2}}e^{{-}ikx}e^{{-}i\sigma_{3}\frac{v_{\ell}(x+y_{\ell})}{2}}f_{\ell,-}(x+y_{\ell})\,dx\\
 &{+}\sigma_{3}\int_{\frac{y_{\ell-1}-y_{\ell}}{2}}^{{+}\infty}e^{{-}i\sigma_{3}\frac{v_{\ell}(x+y_{\ell})}{2}}\left[e^{{-}ikx}+r_{\omega_{\ell}}(k)e^{ikx}\right]f_{\ell,+}(x+y_{\ell})\,dx\\
 &{+}O\left(\frac{1}{(1+\vert k\vert )}\norm{f_{\ell,-}(x)\chi_{\{x\leq \frac{y_{\ell}+y_{\ell+1}}{2}\}}}_{L^{2}_{x}(\mathbb{R})}e^{\beta\frac{(y_{\ell+1}-y_{\ell})}{2}}\right)\\
 &{+}O\left(\frac{1}{(1+\vert k\vert )}\norm{f_{\ell,+}(x)\chi_{\{x>\frac{y_{\ell}+y_{\ell-1}}{2}\}}}_{L^{2}_{x}(\mathbb{R})}e^{{-}\beta\frac{(y_{\ell-1}-y_{\ell})}{2}}\right).
\end{aligned}
\end{multline*}
Consequently, using the second and third identities of \eqref{unknown}, we obtain the following estimate 
\begin{multline}\label{l-1es}
\begin{aligned}
 \begin{bmatrix}
     e^{iy_{\ell}k}\phi_{1,\ell-1}\left(k+\frac{v_{\ell}}{2}\right)\\
     {-}e^{iy_{\ell}k}\phi_{2,\ell-1}\left(k-\frac{v_{\ell}}{2}\right)
 \end{bmatrix}=& G_{1,\ell}(f)(k){+}s_{\omega_{\ell}}(k)e^{i\frac{(y_{\ell}-y_{\ell+1})k}{2}}\begin{bmatrix}
     g_{\ell,-,1}\left({-}k-\frac{v_{\ell}}{2}\right)e^{{-}i\frac{v_{\ell}(y_{\ell}+y_{\ell+1})}{4}}\\
     {-}g_{\ell,-,2}\left({-}k+\frac{v_{\ell}}{2}\right)e^{i\frac{v_{\ell}(y_{\ell}+y_{\ell+1})}{4}}
 \end{bmatrix}\\
 &{+}e^{{-}i\frac{k(y_{\ell-1}-y_{\ell})}{2}}\begin{bmatrix}
     e^{{-}i\frac{v_{\ell}\left(y_{\ell-1}+y_{\ell}\right)}{4}}g_{\ell,+,1}\left({-}k-\frac{v_{\ell}}{2}\right)\\
    {-}e^{i\frac{v_{\ell}\left(y_{\ell-1}+y_{\ell}\right)}{4}}g_{\ell,+,2}\left({-}k+\frac{v_{\ell}}{2}\right)
 \end{bmatrix}
 \\&{+}r_{\omega_{\ell}}(k)e^{i\frac{k(y_{\ell-1}-y_{\ell})}{2}}
 \begin{bmatrix}
 e^{{-}i\frac{v_{\ell}(y_{\ell}+y_{\ell-1})}{4}}g_{\ell,+,1}\left (k-\frac{v_{\ell}}{2}\right)\\
 {-}e^{i\frac{v_{\ell}(y_{\ell}+y_{\ell-1})}{4}}g_{\ell,+,2}\left (k+\frac{v_{\ell}}{2}\right)
 \end{bmatrix}
 \\&{+}O\left(\frac{1}{(1+\vert k\vert )}\norm{f_{\ell,-}(x)\chi_{\{x\leq \frac{y_{\ell}+y_{\ell+1}}{2}\}}}_{L^{2}_{x}(\mathbb{R})}e^{\beta\frac{(y_{\ell+1}-y_{\ell})}{2}}\right)\\
 &{+}O\left(\frac{1}{(1+\vert k\vert )}\norm{f_{\ell,+}(x)\chi_{\{x>\frac{y_{\ell}+y_{\ell-1}}{2}\}}}_{L^{2}_{x}(\mathbb{R})}e^{{-}\beta\frac{(y_{\ell-1}-y_{\ell})}{2}}\right).
\end{aligned}
\end{multline}
\\
\textbf{Step 2. (Estimates of \eqref{l3} and \eqref{l4}.)} Similarly to the approach in the previous step, we consider the following functions
\begin{align}\label{G1m}
   G_{1,m}(f)(k)=& G_{\omega_{m}}^{*}\left(\sigma_{3}e^{{-}i\sigma_{3}\frac{v_{m}(x+y_{m})}{2}}\chi_{\{x\leq \frac{y_{m-1}-y_{m}}{2}\}}f(x+y_{m})\right)(k),\\ \label{F1}
   F_{1,1}(f)(k)=& F_{\omega_{1}}^{*}\left(\sigma_{3}e^{{-}i\sigma_{3}\frac{v_{1}(x+y_{1})}{2}}\chi_{\{x> \frac{y_{2}-y_{1}}{2}\}}f(x+y_{m})\right)(k).
\end{align}
From the identity \eqref{G1m}, it is not difficult to verify that the equation \eqref{l3} is equivalent to
\begin{multline}
\begin{aligned}
\begin{bmatrix}
    e^{iy_{m}k}\phi_{1,m-1}\left(k+\frac{v_{m}}{2}\right)\\
{-}e^{iy_{m}k}\phi_{2,m-1}\left(k-\frac{v_{m}}{2}\right)
\end{bmatrix}=&G_{1,m}(f)(k)\\ \nonumber
&{+}G_{\omega_{m}}^{*}(k)\sigma_{3}\left(e^{{-}i\sigma_{3}\frac{v_{m}(x+y_{m})}{2}}\chi_{\{x> \frac{y_{m-1}-y_{m}}{2}\}}(x)f_{m,+}(x+y_{m})\right).  
\end{aligned}
\end{multline}
\par Consequently, we deduce from Remark \ref{f*r} and identity \eqref{forp1} that
\begin{multline}
\begin{aligned}
\begin{bmatrix}
    e^{iy_{m}k}\phi_{1,m-1}\left(k+\frac{v_{m}}{2}\right)\\
{-}e^{iy_{m}k}\phi_{2,m-1}\left(k-\frac{v_{m}}{2}\right)
\end{bmatrix}=&G_{1,m}(f)(k)\\ \nonumber
&{+}r_{\omega_{m}}(k)\sigma_{3}\left[\int_{\frac{y_{m-1}-y_{m}}{2}}^{{+}\infty}e^{{-}i\sigma_{3}\frac{v_{m}(x+y_{m})}{2}}f_{m,{+}}(x+y_{m})e^{ikx}\,dx\right]\\&{+}\sigma_{3}\int_{\frac{y_{m-1}-y_{m}}{2}}^{{+}\infty}e^{{-}i\sigma_{3}\frac{v_{m}(x+y_{m})}{2}}f_{m,{+}}(x+y_{m})e^{{-}ikx}\,dx\\&{+}O\left(\frac{1}{(1+\vert k\vert)}\norm{f_{m,+}(x)\chi_{\{x> \frac{y_{m-1}+y_{m}}{2}\}}}_{L^{2}_{x}(\mathbb{R})}e^{{-}\beta \frac{y_{m-1}-y_{m}}{2} }\right).
\end{aligned}
\end{multline}
see also \eqref{asy2} for understanding the estimate above. Therefore, using the fourth identity at \eqref{unknown}, we conclude the following estimate
\begin{multline}\label{l-mest}
\begin{aligned}
\begin{bmatrix}
    e^{iy_{m}k}\phi_{1,m-1}\left(k+\frac{v_{m}}{2}\right)\\
{-}e^{iy_{m}k}\phi_{2,m-1}\left(k-\frac{v_{m}}{2}\right)
\end{bmatrix}=&G_{1,m}(f)(k)\\ 
&{+}r_{\omega_{m}}(k)
e^{i\frac{k(y_{m-1}-y_{m})}{2}}\begin{bmatrix}
    e^{{-}i\frac{v_{m}(y_{m-1}+y_{m})}{4}}g_{m,+,1}\left(k-\frac{v_{m}}{2}\right)\\
    {-}e^{i\frac{v_{m}(y_{m-1}+y_{m})}{4}}g_{m,+,2}\left(k+\frac{v_{m}}{2}\right)
\end{bmatrix}
\\&{+}e^{{-}i\frac{k(y_{m-1}-y_{m})}{2}}
\begin{bmatrix}
    e^{{-}i\frac{v_{m}(y_{m-1}+y_{m})}{4}}g_{m,+,1}\left({-}k-\frac{v_{m}}{2}\right)\\
    {-}e^{i\frac{v_{m}(y_{m-1}+y_{m})}{4}}g_{m,+,2}\left({-}k+\frac{v_{m}}{2}\right)
\end{bmatrix}
\\&{+}O\left(\frac{1}{(1+\vert k\vert)}\norm{f_{m,+}(x)\chi_{\{x> \frac{y_{m-1}+y_{m}}{2}\}}}_{L^{2}_{x}(\mathbb{R})}e^{{-}\beta \frac{y_{m-1}-y_{m}}{2} }\right),
\end{aligned}
\end{multline}
\par Next, considering equations \eqref{l4} and \eqref{F1}, we have the following identity
\begin{align}\nonumber
\begin{bmatrix}  
    e^{iy_{1}k}\phi_{1,1}\left(k+\frac{v_{1}}{2}\right)\\
{-}e^{iy_{1}k}\phi_{2,1}\left(k-\frac{v_{1}}{2}\right)
\end{bmatrix}=F_{1,1}(f)(k){+}F_{\omega_{1}}^{*}(k)\sigma_{3}\left(e^{{-}i\sigma_{3}\frac{v_{1}(x+y_{1})}{2}}\chi_{\{x\leq  \frac{y_{2}-y_{1}}{2}\}}(x)f_{1,-}(x+y_{1})\right). 
\end{align}
Consequently, similarly to the estimate \eqref{l-mest} and using Remark \ref{f*r}, we can verify that
\begin{align}\label{l-1est}
\begin{bmatrix}  
    e^{iy_{1}k}\phi_{1,1}\left(k+\frac{v_{1}}{2}\right)\\
{-}e^{iy_{1}k}\phi_{2,1}\left(k-\frac{v_{1}}{2}\right)
\end{bmatrix}=&F_{1,1}(f)(k){+}e^{{-}i\frac{k(y_{2}-y_{1})}{2}}
\begin{bmatrix}
e^{{-}i\frac{v_{1}(y_{1}+y_{2})}{4}}g_{1,-,1}\left({-}k-\frac{v_{1}}{2}\right)\\
{-}e^{i\frac{v_{1}(y_{1}+y_{2})}{4}}g_{1,-,2}\left({-}k+\frac{v_{1}}{2}\right)
\end{bmatrix}
\\ \nonumber &{+} r_{\omega_{1}}({-}k)e^{i\frac{k(y_{2}-y_{1})}{2}}
\begin{bmatrix}
    e^{{-}i\frac{v_{1}(y_{1}+y_{2})}{4}}g_{1,{-},1}\left(k-\frac{v_{1}}{2}\right)\\
    {-}e^{i\frac{v_{1}(y_{1}+y_{2})}{4}}g_{1,-,2}\left(k+\frac{v_{1}}{2}\right)
\end{bmatrix}. 
\end{align}
\\
\textbf{Step 3. (Applications of Remark \ref{transition} to \eqref{l1} and \eqref{l2}.)} 
First, using the notation from the previous step, we consider the functions
\begin{equation*}
    F_{1,\ell}(f)(k)=\begin{bmatrix}
        f_{1,\ell}(k)\\
        f_{2,\ell}(k)
    \end{bmatrix},\, G_{1,\ell}(f)(k)=\begin{bmatrix}
        g_{1,\ell}(k)\\
        g_{2,\ell}(k)
    \end{bmatrix},
\end{equation*}
they are going to be very useful for us to simplify our reasoning.
\par Moreover, using the identities \eqref{l1} and \eqref{l2}, we can verify for any $\ell\in\{2,..., \, m-1\}$ that
\begin{align*}
\begin{bmatrix}
e^{iy_{\ell+1}k}\phi_{1,\ell}\left(k+\frac{v_{\ell+1}}{2}\right)\\
{-}e^{iy_{\ell+1}k}\phi_{2,\ell}\left(k-\frac{v_{\ell+1}}{2}\right)
    \end{bmatrix}
\end{align*}
from \eqref{l-1es} is equivalent to
\begin{align}\label{estpt1}
 G_{1,\ell+1}(f)(k)&{+}s_{\omega_{\ell+1}}(k)e^{i\frac{(y_{\ell+1}-y_{\ell+2})k}{2}}
 \begin{bmatrix}
g_{\ell+1,-,1}\left({-}k-\frac{v_{\ell+1}}{2}\right)e^{{-}i\frac{v_{\ell+1}(y_{\ell+1}+y_{\ell+2})}{4}}\\
{-}g_{\ell+1,-,2}\left({-}k+\frac{v_{\ell+1}}{2}\right)e^{i\frac{v_{\ell+1}(y_{\ell+1}+y_{\ell+2})}{4}}
 \end{bmatrix}\\ \nonumber
&{+}r_{\omega_{\ell+1}}(k)e^{i\frac{k(y_{\ell}-y_{\ell+1})}{2}}
 \begin{bmatrix}
e^{{-}i\frac{v_{\ell+1}(y_{\ell+1}+y_{\ell})}{4}}g_{\ell+1,+,1}\left(k-\frac{v_{\ell+1}}{2}\right)\\
{-}e^{i\frac{v_{\ell+1}(y_{\ell+1}+y_{\ell})}{4}}g_{\ell+1,+,2}\left(k+\frac{v_{\ell+1}}{2}\right)
 \end{bmatrix}\\ \nonumber
&{+}e^{{-}i\frac{k(y_{\ell}-y_{\ell+1})}{2}}\begin{bmatrix}
e^{{-}i\frac{v_{\ell+1}\left(y_{\ell}+y_{\ell+1}\right)}{4}}g_{\ell+1,+,1}\left({-}k-\frac{v_{\ell+1}}{2}\right)\\
{-}e^{i\frac{v_{\ell+1}\left(y_{\ell}+y_{\ell+1}\right)}{4}}g_{\ell+1,+,2}\left({-}k+\frac{v_{\ell+1}}{2}\right)
 \end{bmatrix}\\ \nonumber
 &{+} O\left(\frac{1}{(1+\vert k\vert)}\norm{f_{\ell+1,{+}}(x)\chi_{\{x>\frac{y_{\ell+1}+y_{\ell}}{2}\}}}_{L^{2}_{x}(\mathbb{R})}e^{\beta \frac{(y_{\ell+1}-y_{\ell})}{2}}\right)\\ \nonumber
 &{+} O\left(\frac{1}{(1+\vert k\vert)}\norm{f_{\ell+1,{-}}(x)\chi_{\{x\leq \frac{y_{\ell+1}+y_{\ell+2}}{2}\}}}_{L^{2}_{x}(\mathbb{R})}e^{{-}\beta \frac{(y_{\ell+1}-y_{\ell+2})}{2}}\right).
\end{align}
Furthermore, from \eqref{estl1}, it is not difficult to verify the existence of constants $\theta_{0}(\ell),\,\theta_{1}(\ell),\,\theta_{2}(\ell)$ and $\theta_{3}(\ell)$ depending only on the parameters $y_{\ell-1},\,y_{\ell},\,y_{\ell+1},\,v_{\ell-1},\,v_{\ell}$ and $v_{\ell+1}$ such that the expression above is equivalent to
\begin{multline}
\label{estl1122}
\begin{aligned}
\begin{bmatrix}
e^{iy_{\ell+1}k} \phi_{1,\ell}(k+\frac{v_{\ell+1}}{2})\\
{-}e^{iy_{\ell+1}k}\phi_{2,\ell}(k-\frac{v_{\ell+1}}{2})
  \end{bmatrix}=&\begin{bmatrix}
    e^{i(y_{\ell+1}-y_{\ell})k}e^{i\theta_{0}(\ell)}f_{1,\ell}\left(k+\frac{v_{\ell+1}}{2}-\frac{v_{\ell}}{2}\right)\\
    e^{i(y_{\ell+1}-y_{\ell})k}e^{{-}i\theta_{0}(\ell)}f_{2,\ell}\left(k-\frac{v_{\ell+1}}{2}+\frac{v_{\ell}}{2}\right)
\end{bmatrix}\\ &{+}e^{i(y_{\ell+1}-\frac{(y_{\ell}+y_{\ell-1})}{2})k}
  \begin{bmatrix}
\overline{s_{\omega_{\ell}}\left(k+\frac{v_{\ell+1}}{2}-\frac{v_{\ell}}{2}\right)}e^{i\theta_{1}(\ell)}g_{\ell,+,1}\left({-}k-\frac{v_{\ell+1}}{2}\right)\\
{-}\overline{s_{\omega_{\ell}}\left(k-\frac{v_{\ell+1}}{2}+\frac{v_{\ell}}{2}\right)}e^{{-}i\theta_{1}(\ell)}g_{\ell,+,2}\left({-}k+\frac{v_{\ell+1}}{2}\right)
  \end{bmatrix}
\\ &{+}e^{i\frac{(y_{\ell+1}-y_{\ell})k}{2}}\begin{bmatrix}
e^{i\theta_{2}(\ell)}g_{\ell,-,1}\left({-}k-\frac{v_{\ell+1}}{2}\right)\\
{-}e^{{-}i\theta_{2}(\ell)}g_{\ell,-,2}\left({-}k+\frac{v_{\ell+1}}{2}\right)
\end{bmatrix}\\
&{+}e^{i\frac{3(y_{\ell+1}-y_{\ell})k}{2}}\begin{bmatrix}
r_{\omega_{\ell}}\left({-}k-\frac{v_{\ell+1}}{2}+\frac{v_{\ell}}{2}\right)e^{i\theta_{3}(\ell)}g_{\ell,-,1}\left(k+\frac{v_{\ell+1}}{2}-v_{\ell}\right)\\
{-}r_{\omega_{\ell}}\left({-}k+\frac{v_{\ell+1}}{2}-\frac{v_{\ell}}{2}\right)e^{{-}i\theta_{3}(\ell)}g_{\ell,-,2}\left(k-\frac{v_{\ell+1}}{2}+v_{\ell}\right)
\end{bmatrix}
\\&{+}O\left(\frac{1}{(1+\vert k\vert )}\norm{f_{\ell,-}(x)\chi_{\{x\leq\frac{y_{\ell}+y_{\ell+1}}{2}\}}}_{L^{2}_{x}(\mathbb{R})}e^{\beta\frac{(y_{\ell+1}-y_{\ell})}{2}}\right)\\&{+}O\left(\frac{1}{(1+\vert k\vert )}\norm{f_{\ell,+}(x)\chi_{\{x>\frac{y_{\ell}+y_{\ell-1}}{2}\}}}_{L^{2}_{x}(\mathbb{R})}e^{{-}\beta\frac{(y_{\ell-1}-y_{\ell})}{2}}\right).
\end{aligned}
\end{multline}
Moreover, due to Remark \ref{r+-}, we can restrict ourselves to the case where all the functions $r_{\ell}$ and $s_{\ell}$  satisfying \eqref{asyreftr}.
\par In conclusion, for any $\ell\in\{2,\, ...,m-1\,\},$ it is not difficult to verify using the equations \eqref{estpt1}, \eqref{estl1122}, and Lemma \ref{+-interact} that the following estimates are true
\begin{multline}\label{pl+1}
  P_{+}\left(e^{i\frac{k(y_{\ell}-y_{\ell+1})}{2}}g_{1,\ell+1}(k)-e^{i\frac{k(y_{\ell+1}-y_{\ell})}{2}}e^{i\theta_{0}(\ell)}f_{1,\ell}\left(k+\frac{v_{\ell+1}}{2}-\frac{v_{\ell}}{2}\right)\right)\\
 \begin{aligned}
 =&e^{i\theta_{2}(\ell)}g_{\ell,-,1}\left({-}k-\frac{v_{\ell+1}}{2}\right)\\&{-}r_{\omega_{\ell+1}}(k)e^{i(y_{\ell}-y_{\ell+1})k}e^{{-}i\frac{v_{\ell+1}(y_{\ell+1}+y_{\ell})}{4}}g_{\ell+1,+,1}\left(k-\frac{v_{\ell+1}}{2}\right)\\
 &{-}s_{\omega_{\ell+1}}(k)e^{i\frac{(y_{\ell}-y_{\ell+2})k}{2}}e^{{-}i\frac{v_{\ell+1}(y_{\ell+1}+y_{\ell+2})}{4}}g_{\ell+1,-,1}\left({-}k-\frac{v_{\ell+1}}{2}\right)\\
&{+}O\left(e^{{-}\beta (y_{\ell}-y_{\ell+1})}\max_{\pm}\norm{f_{\ell,\pm}(x)}_{L^{2}_{x}(\mathbb{R})}+e^{{-}\beta (y_{\ell-1}-y_{\ell})}\max_{\pm}\norm{f_{\ell,\pm}(x)}_{L^{2}_{x}(\mathbb{R})}\right)
\\&{+}O\left(e^{{-}\beta (y_{\ell}-y_{\ell+1})}\max_{\pm}\norm{f_{\ell+1,\pm}(x)}_{L^{2}_{x}(\mathbb{R})}+e^{{-}\beta (y_{\ell-1}-y_{\ell})}\max_{\pm}\norm{f_{\ell+1,\pm}(x)}_{L^{2}_{x}(\mathbb{R})}\right),
\end{aligned}
\end{multline}
and
\begin{multline}\label{Pl-11}
P_{-}\left(e^{i\frac{k(y_{\ell}-y_{\ell+1})}{2}}g_{1,\ell+1}(k)-e^{i\frac{k(y_{\ell+1}-y_{\ell})}{2}}e^{i\theta_{0}(\ell)}f_{1,\ell}\left(k+\frac{v_{\ell+1}}{2}-\frac{v_{\ell}}{2}\right)\right)\\
\begin{aligned}
=&{-}e^{{-}i\frac{v_{\ell+1}(y_{\ell+1}+y_{\ell})}{4}}g_{\ell+1,+,1}\left({-}k-\frac{v_{\ell+1}}{2}\right)\\
    &{+}s_{\omega_{\ell}}\left({-}k-\frac{v_{\ell+1}}{2}+\frac{v_{\ell}}{2}\right)e^{i\frac{(y_{\ell+1}-y_{\ell-1})k}{2}}e^{i\theta_{1}(\ell)}g_{\ell,+,1}\left({-}k-\frac{v_{\ell+1}}{2}\right)\\
    &{+}e^{i(y_{\ell+1}-y_{\ell})k}r_{\omega_{\ell}}\left({-}k+\frac{v_{\ell}}{2}-\frac{v_{\ell+1}}{2}\right)e^{i\theta_{3}(\ell)}g_{\ell,-,1}\left(k+\frac{v_{\ell+1}}{2}-v_{\ell}\right)\\
&{+}O\left(e^{{-}\beta (y_{\ell}-y_{\ell+1})}\max_{\pm}\norm{f_{\ell,\pm}(x)}_{L^{2}_{x}(\mathbb{R})}+e^{{-}\beta (y_{\ell-1}-y_{\ell})}\max_{\pm}\norm{f_{\ell,\pm}(x)}_{L^{2}_{x}(\mathbb{R})}\right)\\
&{+}O\left(e^{{-}\beta (y_{\ell}-y_{\ell+1})}\max_{\pm}\norm{f_{\ell+1,\pm}(x)}_{L^{2}_{x}(\mathbb{R})}+e^{{-}\beta (y_{\ell-1}-y_{\ell})}\max_{\pm}\norm{f_{\ell+1,\pm}(x)}_{L^{2}_{x}(\mathbb{R})}\right).
\end{aligned}    
\end{multline}
\par Similarly, we can verify using estimates \eqref{estpt1}, \eqref{estl1122}, and Lemma \ref{+-interact} that  
\begin{multline}\label{pl+2}
P_{+}\left(e^{i\frac{k(y_{\ell}-y_{\ell+1})}{2}}g_{2,\ell+1}(k)-e^{i\frac{k(y_{\ell+1}-y_{\ell})}{2}}e^{{-}i\theta_{0}(\ell)}f_{2,\ell}\left(k-\frac{v_{\ell+1}}{2}+\frac{v_{\ell}}{2}\right)\right)\\    
\begin{aligned}
=& s_{\omega_{\ell+1}}(k)e^{i\frac{k(y_{\ell}-y_{\ell+2})}{2}}e^{i\frac{v_{\ell+1}(y_{\ell+1}+y_{\ell+2})}{4}}g_{\ell+1,-,2}\left({-}k+\frac{v_{\ell+1}}{2}\right)\\
&{+}r_{\omega_{\ell+1}}(k)e^{ik(y_{\ell}-y_{\ell+1})}e^{i\frac{v_{\ell+1}(y_{\ell}+y_{\ell+1})}{4}}g_{\ell+1,+,2}\left(k+\frac{v_{\ell+1}}{2}\right)\\
&{-}e^{{-}i\theta_{2}(\ell)}g_{\ell,-,2}\left({-}k+\frac{v_{\ell+1}}{2}\right)\\
&{+}O\left(e^{{-}\beta (y_{\ell}-y_{\ell+1})}\max_{\pm}\norm{f_{\ell,\pm}(x)}_{L^{2}_{x}(\mathbb{R})}+e^{{-}\beta (y_{\ell-1}-y_{\ell})}\max_{\pm}\norm{f_{\ell,\pm}(x)}_{L^{2}_{x}(\mathbb{R})}\right)
\\&{+}O\left(e^{{-}\beta (y_{\ell}-y_{\ell+1})}\max_{\pm}\norm{f_{\ell+1,\pm}(x)}_{L^{2}_{x}(\mathbb{R})}+e^{{-}\beta (y_{\ell-1}-y_{\ell})}\max_{\pm}\norm{f_{\ell+1,\pm}(x)}_{L^{2}_{x}(\mathbb{R})}\right),
\end{aligned}
\end{multline}
and
\begin{multline}\label{pl-2}
P_{-}\left(e^{i\frac{k(y_{\ell}-y_{\ell+1})}{2}}g_{2,\ell+1}(k)-e^{i\frac{k(y_{\ell+1}-y_{\ell})}{2}}e^{{-}i\theta_{0}(\ell)}f_{2,\ell}\left(k-\frac{v_{\ell+1}}{2}+\frac{v_{\ell}}{2}\right)\right)\\
\begin{aligned}
=& {-}s_{\omega_{\ell}}\left({-}k+\frac{v_{\ell+1}}{2}-\frac{v_{\ell}}{2}\right)e^{i\frac{k(y_{\ell+1}-y_{\ell-1})}{2}}e^{{-}i\theta_{1}(\ell)}g_{\ell,+,2}\left({-}k+\frac{v_{\ell+1}}{2}\right)\\
&{-}r_{\omega_{\ell}}\left({-}k+\frac{v_{\ell+1}}{2}-\frac{v_{\ell}}{2}\right)e^{ik(y_{\ell+1}-y_{\ell})}e^{{-}i\theta_{3}(\ell)}g_{\ell,-,2}\left(k-\frac{v_{\ell+1}}{2}+v_{\ell}\right)\\
&{+}e^{i\frac{v_{\ell+1}(y_{\ell}+y_{\ell+1})}{4}}g_{\ell+1,+,2}\left({-}k+\frac{v_{\ell+1}}{2}\right)\\
&{+}O\left(e^{{-}\beta (y_{\ell}-y_{\ell+1})}\max_{\pm}\norm{f_{\ell,\pm}(x)}_{L^{2}_{x}(\mathbb{R})}+e^{{-}\beta (y_{\ell-1}-y_{\ell})}\max_{\pm}\norm{f_{\ell,\pm}(x)}_{L^{2}_{x}(\mathbb{R})}\right)\\
&{+}O\left(e^{{-}\beta (y_{\ell}-y_{\ell+1})}\max_{\pm}\norm{f_{\ell+1,\pm}(x)}_{L^{2}_{x}(\mathbb{R})}+e^{{-}\beta (y_{\ell-1}-y_{\ell})}\max_{\pm}\norm{f_{\ell+1,\pm}(x)}_{L^{2}_{x}(\mathbb{R})}\right).
\end{aligned}
\end{multline}
\\
\textbf{Step 4. (Applications of Remark \ref{transition} to \eqref{l4} and \eqref{l3}.)} Next, using the estimate \eqref{estl1122} for $\ell=m-1$ and \eqref{estpt1} for $\ell=1,$ we can verify that there exist parameters $\theta_{m,1},\,\theta_{m,2},\,\theta_{m,3}$ and $\theta_{m,4}$ depending only on $v_{m-1},\,v_{m},\,y_{m},\,y_{m-1}$ satisfying
\begin{multline}\label{Em}
 \begin{bmatrix}
e^{iy_{m}k}\phi_{1,m-1}\left(k+\frac{v_{m}}{2}\right)\\
{-}e^{iy_{m}k}\phi_{2,m-1}\left(k-\frac{v_{m}}{2}\right)
 \end{bmatrix}\\
 \begin{aligned}
 =& e^{i(y_{m}-y_{m-1})k}
 \begin{bmatrix}
e^{i\theta_{m,1}}f_{1,m-1}\left(k+\frac{v_{m}-v_{m-1}}{2}\right)\\
e^{{-}i\theta_{m,1}}f_{2,m-1}\left(k+\frac{v_{m-1}-v_{m}}{2}\right)
 \end{bmatrix}\\
&{+}e^{i(y_{m}-\frac{(y_{m-1}+y_{m-2})}{2})k}\begin{bmatrix}
s_{\omega_{m-1}}\left({-}k-\frac{v_{m}}{2}+\frac{v_{m-1}}{2}\right)e^{i\theta_{m,2}}g_{m-1,+,1}\left({-}k-\frac{v_{m}}{2}\right)\\
{-}s_{\omega_{m-1}}\left({-}k+\frac{v_{m}}{2}-\frac{v_{m-1}}{2}\right)e^{{-}i\theta_{m,2}}g_{m-1,+,2}\left({-}k+\frac{v_{m}}{2}\right)
 \end{bmatrix}\\
&{+}e^{i\frac{(y_{m}-y_{m-1})k}{2}}\begin{bmatrix}
e^{i\theta_{m,3}}g_{m-1,-,1}\left({-}k-\frac{v_{m}}{2}\right)\\
{-}e^{{-}i\theta_{m,3}}g_{m-1,-,2}\left({-}k+\frac{v_{m}}{2}\right)
 \end{bmatrix}\\
&{+}e^{i\frac{3(y_{m}-y_{m-1})k}{2}}
\begin{bmatrix}
r_{\omega_{m-1}}\left({-}k-\frac{v_{m}}{2}+\frac{v_{m-1}}{2}\right)e^{i\theta_{m,4}}g_{m-1,-,1}\left(k+\frac{v_{m}}{2}-v_{m-1}\right)\\
{-}r_{\omega_{m-1}}\left({-}k+\frac{v_{m}}{2}-\frac{v_{m-1}}{2}\right)e^{{-}i\theta_{m,4}}g_{m-1,-,2}\left(k-\frac{v_{m}}{2}+v_{m-1}\right)
\end{bmatrix}\\
&{+} O\left(\frac{1}{(1+\vert k\vert)}\norm{f_{m-1,{-}}(x)\chi_{\{x\leq \frac{y_{m-1}+y_{m}}{2}\}}}_{L^{2}_{x}(\mathbb{R})}e^{{-}\beta \frac{(y_{m-1}-y_{m})}{2}}\right)\\
&{+}O\left(\frac{1}{(1+\vert k\vert)}\norm{f_{m-1,{+}}(x)\chi_{\{x> \frac{y_{m-1}+y_{m-2}}{2}\}}}_{L^{2}_{x}(\mathbb{R})}e^{\beta \frac{(y_{m-1}-y_{m-2})}{2}}\right),
\end{aligned}
\end{multline}
and
\begin{align}\label{RR}
 \begin{bmatrix}
e^{iy_{2}k}\phi_{1,1}\left(k+\frac{v_{2}}{2}\right)\\
{-}e^{iy_{2}k}\phi_{2,1}\left(k-\frac{v_{2}}{2}\right)
 \end{bmatrix}=&G_{1,2}(f)(k){+}s_{\omega_{2}}(k)e^{i\frac{(y_{2}-y_{3})k}{2}}
 \begin{bmatrix}
g_{2,-,1}\left({-}k-\frac{v_{2}}{2}\right)e^{{-}i\frac{v_{2}(y_{2}+y_{3})}{4}}\\
{-}g_{2,-,2}\left({-}k+\frac{v_{2}}{2}\right)e^{i\frac{v_{2}(y_{2}+y_{3})}{4}}
 \end{bmatrix}\\ \nonumber
&{+}r_{\omega_{2}}(k)e^{i\frac{k(y_{1}-y_{2})}{2}}
 \begin{bmatrix}
e^{{-}i\frac{v_{2}(y_{2}+y_{1})}{4}}g_{2,+,1}\left(k-\frac{v_{2}}{2}\right)\\
{-}e^{i\frac{v_{2}(y_{2}+y_{1})}{4}}g_{2,+,2}\left(k+\frac{v_{2}}{2}\right)
 \end{bmatrix}\\ \nonumber
&{+}e^{{-}i\frac{k(y_{1}-y_{2})}{2}}\begin{bmatrix}
e^{{-}i\frac{v_{2}\left(y_{1}+y_{2}\right)}{4}}g_{2,+,1}\left({-}k-\frac{v_{2}}{2}\right)\\
{-}e^{i\frac{v_{2}\left(y_{1}+y_{2}\right)}{4}}g_{2,+,2}\left({-}k+\frac{v_{2}}{2}\right)
 \end{bmatrix}\\ \nonumber
 &{+} O\left(\frac{1}{(1+\vert k\vert)}\norm{f_{2,{+}}(x)\chi_{\{x>\frac{y_{2}+y_{1}}{2}\}}}_{L^{2}_{x}(\mathbb{R})}e^{\beta \frac{(y_{2}-y_{1})}{2}}\right)\\ \nonumber
 &{+} O\left(\frac{1}{(1+\vert k\vert)}\norm{f_{2,{-}}(x)\chi_{\{x\leq \frac{y_{2}+y_{3}}{2}\}}}_{L^{2}_{x}(\mathbb{R})}e^{{-}\beta \frac{(y_{2}-y_{3})}{2}}\right).
\end{align}
\par Next, using estimate \eqref{l-1est}, we can also verify the existence of real parameters $\theta_{0,1},\,\theta_{1,1},\,\theta_{2,1},\,\theta_{3,1}$ depending only on $v_{1},\, v_{2},\,y_{1}$ satisfying
\begin{multline}\label{l-11est}
\begin{aligned}
    \begin{bmatrix}
        e^{iy_{2}k}\phi_{1,1}\left(k+\frac{v_{2}}{2}\right)\\
        {-}e^{i y_{2} k}\phi_{2,1}\left(k-\frac{v_{2}}{2}\right)
    \end{bmatrix}=&
    e^{i(y_{2}-y_{1})k}\begin{bmatrix}
        e^{i\theta_{1,1}}f_{1,1}\left(k+\frac{v_{2}}{2}-\frac{v_{1}}{2}\right)\\
        e^{{-}i\theta_{1,1}}f_{2,1}\left(k-\frac{v_{2}}{2}+\frac{v_{1}}{2}\right)
    \end{bmatrix}\\
    &{+}e^{i\frac{k(y_{2}-y_{1})}{2}}
    \begin{bmatrix}
        e^{i\theta_{2,1}}g_{1,-,1}\left({-}k-\frac{v_{2}}{2}\right)\\
        {-}e^{{-}i\theta_{2,1}}g_{1,-,2}\left({-}k+\frac{v_{2}}{2}\right)
    \end{bmatrix}\\
    &{+}e^{i\frac{3k(y_{2}-y_{1})}{2}}\begin{bmatrix}
        r_{\omega_{1}}\left({-}k-\frac{v_{2}}{2}+\frac{v_{1}}{2}\right)e^{i\theta_{3,1}}g_{1,-,1}\left(k+\frac{v_{2}}{2}-v_{1}\right)\\
        {-}r_{\omega_{1}}\left({-}k+\frac{v_{2}}{2}-\frac{v_{1}}{2}\right)e^{{-}i\theta_{3,1}}g_{1,-,2}\left(k-\frac{v_{2}}{2}+v_{1}\right)
    \end{bmatrix}.
\end{aligned}
\end{multline}
\par Therefore, similarly to the approach of Step $4$, and the second estimate \eqref{RR} with \eqref{l-11est} that
\begin{multline}\label{l+1}
P_{+}\left(
e^{i\frac{(y_{1}-y_{2})k}{2}}g_{1,2}(k)-e^{\frac{i(y_{2}-y_{1})k}{2}}e^{i\theta_{1,1}}f_{1,1}\left(k+\frac{v_{2}}{2}-\frac{v_{1}}{2}\right)
\right)\\
\begin{aligned}
=&{-}s_{\omega_{2}}(k)e^{i\frac{(y_{1}-y_{3})k}{2}}e^{{-}i\frac{v_{2}(y_{2}+y_{3})}{4}}   g_{2,-,1}\left({-}k-\frac{v_{2}}{2}\right)\\
&{-}r_{\omega_{2}}(k)e^{i(y_{1}-y_{2})k}e^{{-}i\frac{v_{2}(y_{2}+y_{1})}{4}}g_{2,+,1}\left(k-\frac{v_{2}}{2}\right)\\
&{+}e^{i\theta_{2,1}}g_{1,-,1}\left({-}k-\frac{v_{2}}{2}\right)\\
 &{+} O\left(\norm{f_{2,{+}}(x)\chi_{\{x>\frac{y_{2}+y_{1}}{2}\}}}_{L^{2}_{x}(\mathbb{R})}e^{\beta \frac{(y_{2}-y_{1})}{2}}+\norm{f_{2,{-}}(x)\chi_{\{x\leq \frac{y_{2}+y_{3}}{2}\}}}_{L^{2}_{x}(\mathbb{R})}e^{{-}\beta \frac{(y_{2}-y_{3})}{2}}\right)\\
&{+}O\left(\norm{f_{1,-}(x)\chi_{\{x\leq \frac{y_{1}+y_{2}}{2}\}}}_{L^{2}_{x}(\mathbb{R})}e^{\beta\frac{(y_{2}-y_{1})}{2}}\right),
\end{aligned}
\end{multline}
and
\begin{multline}\label{l-1}
P_{-}\left(
e^{i\frac{(y_{1}-y_{2})k}{2}}g_{1,2}(k)-e^{\frac{i(y_{2}-y_{1})k}{2}}e^{i\theta_{1,1}}f_{1,1}\left(k+\frac{v_{2}}{2}-\frac{v_{1}}{2}\right)
\right)\\
\begin{aligned}
=&{-}e^{{-}i\frac{v_{2}(y_{1}+y_{2})}{4}}g_{2,+,1}\left({-}k-\frac{v_{2}}{2}\right)\\
 &{+}r_{\omega_{1}}\left({-}k-\frac{v_{2}}{2}+\frac{v_{1}}{2}\right)e^{i(y_{2}-y_{1})k}e^{i\theta_{3,1}}g_{1,-,1}\left(k+\frac{v_{2}}{2}-v_{1}\right)\\
 &{+} O\left(\norm{f_{2,{+}}(x)\chi_{\{x>\frac{y_{2}+y_{1}}{2}\}}}_{L^{2}_{x}(\mathbb{R})}e^{\beta \frac{(y_{2}-y_{1})}{2}}+\norm{f_{2,{-}}(x)\chi_{\{x\leq \frac{y_{2}+y_{3}}{2}\}}}_{L^{2}_{x}(\mathbb{R})}e^{{-}\beta \frac{(y_{2}-y_{3})}{2}}\right)\\
&{+}O\left(\norm{f_{1,-}(x)\chi_{\{x\leq \frac{y_{1}+y_{2}}{2}\}}}_{L^{2}_{x}(\mathbb{R})}e^{\beta\frac{(y_{2}-y_{1})}{2}}\right).
\end{aligned}
\end{multline}
\par Similarly, we can verify that
\begin{multline*}
P_{+}\left(e^{i\frac{(y_{1}-y_{2})k}{2}}g_{2,2}(k)-e^{i\frac{(y_{2}-y_{1})k}{2}}e^{{-}i\theta_{1,1}}f_{2,1}\left(k-\frac{v_{2}}{2}+\frac{v_{1}}{2}\right)\right)\\
\begin{aligned}
=&s_{\omega_{2}}(k)e^{i\frac{(y_{1}-y_{3})k}{2}}e^{i\frac{v_{2}(y_{2}+y_{3})}{4}}   g_{2,-,2}\left({-}k+\frac{v_{2}}{2}\right)\\
&{+}r_{\omega_{2}}(k)e^{i(y_{1}-y_{2})k}e^{i\frac{v_{2}(y_{2}+y_{1})}{4}}g_{2,+,2}\left(k+\frac{v_{2}}{2}\right)\\
&{-}e^{{-}i\theta_{2,1}}g_{1,-,2}\left({-}k+\frac{v_{2}}{2}\right)\\
 &{+} O\left(\norm{f_{2,{+}}(x)\chi_{\{x>\frac{y_{2}+y_{1}}{2}\}}}_{L^{2}_{x}(\mathbb{R})}e^{\beta \frac{(y_{2}-y_{1})}{2}}+\norm{f_{2,{-}}(x)\chi_{\{x\leq \frac{y_{2}+y_{3}}{2}\}}}_{L^{2}_{x}(\mathbb{R})}e^{{-}\beta \frac{(y_{2}-y_{3})}{2}}\right)\\
&{+}O\left(\norm{f_{1,-}(x)\chi_{\{x\leq \frac{y_{1}+y_{2}}{2}\}}}_{L^{2}_{x}(\mathbb{R})}e^{\beta\frac{(y_{2}-y_{1})}{2}}\right),
\end{aligned}
\end{multline*}
and
\begin{multline*}
P_{-}\left(
e^{i\frac{(y_{1}-y_{2})k}{2}}g_{2,2}(k)-e^{\frac{i(y_{2}-y_{1})k}{2}}e^{-i\theta_{1,1}}f_{2,1}\left(k-\frac{v_{2}}{2}+\frac{v_{1}}{2}\right)
\right)\\
\begin{aligned}
=&e^{i\frac{v_{2}(y_{1}+y_{2})}{4}}g_{2,+,2}\left({-}k+\frac{v_{2}}{2}\right)\\
&{-}r_{\omega_{1}}\left({-}k+\frac{v_{2}}{2}-\frac{v_{1}}{2}\right)e^{i(y_{2}-y_{1})k}e^{{-}i\theta_{3,1}}g_{1,-,1}\left(k-\frac{v_{2}}{2}+v_{1}\right)\\
 &{+} O\left(\norm{f_{2,{+}}(x)\chi_{\{x>\frac{y_{2}+y_{1}}{2}\}}}_{L^{2}_{x}(\mathbb{R})}e^{\beta \frac{(y_{2}-y_{1})}{2}}+\norm{f_{2,{-}}(x)\chi_{\{x\leq \frac{y_{2}+y_{3}}{2}\}}}_{L^{2}_{x}(\mathbb{R})}e^{{-}\beta \frac{(y_{2}-y_{3})}{2}}\right)\\
&{+}O\left(\norm{f_{1,-}(x)\chi_{\{x\leq \frac{y_{1}+y_{2}}{2}\}}}_{L^{2}_{x}(\mathbb{R})}e^{\beta\frac{(y_{2}-y_{1})}{2}}\right).
\end{aligned}
\end{multline*}
\par Next, similarly to the argument in the last step, we can deduce using estimate \eqref{Em} with \eqref{l-mest}
that
\begin{multline}\label{P+m}
P_{+}\left(e^{i\theta_{m,1}}e^{i\frac{(y_{m}-y_{m-1})k}{2}}f_{1,m-1}\left(k+\frac{v_{m}}{2}-\frac{v_{m-1}}{2}\right)-e^{{-}i\frac{(y_{m}-y_{m-1})k}{2}}g_{1,m}\left(k-\frac{v_{m}}{2}\right)\right)\\
\begin{aligned}
=&r_{\omega_{m}}(k) e^{i(y_{m-1}-y_{m})k}e^{{-}i\frac{v_{m}(y_{m-1}+y_{m})}{4}}g_{m,1,+}\left(k-\frac{v_{m}}{2}\right)\\
&{-}e^{i\theta_{m,3}}g_{m-1,-,1}\left({-}k-\frac{v_{m}}{2}\right)\\
&{+} O\left(\norm{f_{m-1,{+}}(x)\chi_{\{x>\frac{y_{m-1}+y_{m-2}}{2}\}}}_{L^{2}_{x}(\mathbb{R})}e^{\beta \frac{(y_{m-1}-y_{m-2})}{2}}\right)\\
&{+}O\left(\norm{f_{m-1,{-}}(x)\chi_{\{x\leq \frac{y_{m-1}+y_{m}}{2}\}}}_{L^{2}_{x}(\mathbb{R})}e^{{-}\beta \frac{(y_{m-1}-y_{m})}{2}}\right)\\
&{+}O\left(\norm{f_{m,+}(x)\chi_{\{x> \frac{y_{m}+y_{m-1}}{2}\}}}_{L^{2}_{x}(\mathbb{R})}e^{\beta\frac{(y_{m}-y_{m-1})}{2}}\right),
\end{aligned}    
\end{multline}
and 
\begin{multline}\label{P-m}
P_{-}\left(e^{i\theta_{m,1}}e^{i\frac{(y_{m}-y_{m-1})k}{2}}f_{1,m-1}\left(k+\frac{v_{m}}{2}-\frac{v_{m-1}}{2}\right)-e^{{-}i\frac{(y_{m}-y_{m-1})k}{2}}g_{1,m}\left(k-\frac{v_{m}}{2}\right)\right)\\
\begin{aligned}
 =& e^{{-}i\frac{v_{m}(y_{m-1}+y_{m})}{4}}g_{m,1,+}\left({-}k-\frac{v_{m}}{2}\right)\\
&{-}s_{\omega_{m-1}}\left({-}k-\frac{v_{m}}{2}+\frac{v_{m-1}}{2}\right)e^{i\frac{(y_{m}-y_{m-2})k}{2}}e^{i\theta_{m,2}}g_{m-1,+,1}\left({-}k-\frac{v_{m}}{2}\right)\\
&{-}r_{\omega_{m-1}}\left({-}k-\frac{v_{m}}{2}+\frac{v_{m-1}}{2}\right)e^{i(y_{m}-y_{m-1})k}e^{i\theta_{m,4}}g_{m-1,-,1}\left(k+\frac{v_{m}}{2}-v_{m-1}\right)\\
&{+} O\left(\norm{f_{m-1,{+}}(x)\chi_{\{x>\frac{y_{m-1}+y_{m-2}}{2}\}}}_{L^{2}_{x}(\mathbb{R})}e^{\beta \frac{(y_{m-1}-y_{m-2})}{2}}\right)\\
&{+}O\left(\norm{f_{m-1,{-}}(x)\chi_{\{x\leq \frac{y_{m-1}+y_{m}}{2}\}}}_{L^{2}_{x}(\mathbb{R})}e^{{-}\beta \frac{(y_{m-1}-y_{m})}{2}}\right)\\
&{+}O\left(\norm{f_{m,+}(x)\chi_{\{x> \frac{y_{m}+y_{m-1}}{2}\}}}_{L^{2}_{x}(\mathbb{R})}e^{\beta\frac{(y_{m}-y_{m-1})}{2}}\right).
\end{aligned}
\end{multline}
Furthermore, repeating the argument of the last step, we can verify based on the two estimates above that
\begin{multline}
P_{+}\left(e^{{-}i\theta_{m,1}}e^{i\frac{(y_{m}-y_{m-1})k}{2}}f_{2,m-1}\left(k-\frac{v_{m}}{2}+\frac{v_{m-1}}{2}\right)-e^{{-}i\frac{(y_{m}-y_{m-1})k}{2}}g_{2,m}\left(k+\frac{v_{m}}{2}\right)\right)\\
\begin{aligned}
=&{-}r_{\omega_{m}}(k) e^{i(y_{m-1}-y_{m})k}e^{i\frac{v_{m}(y_{m-1}+y_{m})}{4}}g_{m,2,+}\left(k+\frac{v_{m}}{2}\right)\\
&{+}e^{{-}i\theta_{m,3}}g_{m-1,-,2}\left({-}k+\frac{v_{m}}{2}\right)\\
&{+} O\left(\norm{f_{m-1,{+}}(x)\chi_{\{x>\frac{y_{m-1}+y_{m-2}}{2}\}}}_{L^{2}_{x}(\mathbb{R})}e^{\beta \frac{(y_{m-1}-y_{m-2})}{2}}\right)\\
&{+}O\left(\norm{f_{m-1,{-}}(x)\chi_{\{x\leq \frac{y_{m-1}+y_{m}}{2}\}}}_{L^{2}_{x}(\mathbb{R})}e^{{-}\beta \frac{(y_{m-1}-y_{m})}{2}}\right)\\
&{+}O\left(\norm{f_{m,+}(x)\chi_{\{x> \frac{y_{m}+y_{m-1}}{2}\}}}_{L^{2}_{x}(\mathbb{R})}e^{\beta\frac{(y_{m}-y_{m-1})}{2}}\right),
\end{aligned}
\end{multline}
and
\begin{multline}
P_{-}\left(e^{{-}i\theta_{m,1}}e^{i\frac{(y_{m}-y_{m-1})k}{2}}f_{1,m-1}\left(k-\frac{v_{m}}{2}+\frac{v_{m-1}}{2}\right)-e^{{-}i\frac{(y_{m}-y_{m-1})k}{2}}g_{1,m}\left(k+\frac{v_{m}}{2}\right)\right)\\
\begin{aligned}
 =&{-} e^{i\frac{v_{m}(y_{m-1}+y_{m})}{4}}g_{m,1,+}\left({-}k+\frac{v_{m}}{2}\right)\\
&{+}s_{\omega_{m-1}}\left({-}k+\frac{v_{m}}{2}-\frac{v_{m-1}}{2}\right)e^{i\frac{(y_{m}-y_{m-2})k}{2}}e^{{-}i\theta_{m,2}}g_{m-1,+,1}\left({-}k+\frac{v_{m}}{2}\right)\\
&{+}r_{\omega_{m-1}}\left({-}k+\frac{v_{m}}{2}-\frac{v_{m-1}}{2}\right)e^{i(y_{m}-y_{m-1})k}e^{{-}i\theta_{m,4}}g_{m-1,-,1}\left(k-\frac{v_{m}}{2}+v_{m-1}\right)\\
&{+} O\left(\norm{f_{m-1,{+}}(x)\chi_{\{x>\frac{y_{m-1}+y_{m-2}}{2}\}}}_{L^{2}_{x}(\mathbb{R})}e^{\beta \frac{(y_{m-1}-y_{m-2})}{2}}\right)\\
&{+}O\left(\norm{f_{m-1,{-}}(x)\chi_{\{x\leq \frac{y_{m-1}+y_{m}}{2}\}}}_{L^{2}_{x}(\mathbb{R})}e^{{-}\beta \frac{(y_{m-1}-y_{m})}{2}}\right)\\
&{+}O\left(\norm{f_{m,+}(x)\chi_{\{x> \frac{y_{m}+y_{m-1}}{2}\}}}_{L^{2}_{x}(\mathbb{R})}e^{\beta\frac{(y_{m}-y_{m-1})}{2}}\right).
\end{aligned}
\end{multline}\\
\textbf{Step 5. (Proof of the Existence of the Solution of \eqref{lG}, \eqref{1G}, \eqref{mG}.)}
From an argument of symmetry, for any $3\leq \ell\leq m-2 ,$ it is enough to consider the following linear system
\begin{align}\label{giant linear system 1}
    A_{1,-}(f)=& g_{1,{-},1}\left({-}k-\frac{v_{2}}{2}\right)-r_{2}(k)g_{2,+,1}\left(k-\frac{v_{2}}{2}\right)-s_{2}(k)g_{2,{-},1}\left({-}k-\frac{v_{2}}{2}\right),\\  \nonumber
A_{2,+}(f)=&g_{2,{+},1}\left({-}k-\frac{v_{2}}{2}\right)-r_{1}\left({-}k-\frac{v_{2}}{2}+\frac{v_{1}}{2}\right)g_{1,-,1}\left(k+\frac{v_{2}}{2}-v_{1}\right),\\ \label{firstofall}
A_{2,-}(f)=&g_{2,{-},1}\left({-}k-\frac{v_{3}}{2}\right)-r_{3}\left(k\right)g_{3,+,1}\left(k-\frac{v_{3}}{2}\right)-s_{3}(k)g_{3,{-},1}\left({-}k-\frac{v_{3}}{2}\right)
,\\ \label{giant ind1}
A_{3,+}\left(f\right)=& g_{3,+,1}\left({-}k-\frac{v_{3}}{2}\right)-s_{2}\left({-}k-\frac{v_{3}}{2}+\frac{v_{2}}{2}\right)g_{2,+,1}\left({-}k-\frac{v_{3}}{2}\right)
\\ \nonumber
&{-}r_{2}\left({-}k+\frac{v_{2}}{2}-\frac{v_{3}}{2}\right)g_{2,-,1}\left(k+\frac{v_{3}}{2}-v_{2}\right),\, ...,\\  \nonumber
A_{\ell,+}\left(f\right)(k)=&g_{\ell,+,1}\left({-}k-\frac{v_{\ell}}{2}\right)-s_{\ell-1}\left({-}k-\frac{v_{\ell}}{2}+\frac{v_{\ell-1}}{2}\right)g_{\ell-1,+,1}\left({-}k-\frac{v_{\ell}}{2}\right)\\ \nonumber
&{-}r_{\ell-1}\left({-}k+\frac{v_{\ell-1}}{2}-\frac{v_{\ell}}{2}\right)g_{\ell-1,-,1}\left(k+\frac{v_{\ell}}{2}-v_{\ell-1}\right),\\ \nonumber
A_{\ell,-}(f)=& g_{\ell,-,1}\left({-}k-\frac{v_{\ell+1}}{2}\right)-r_{\ell+1}\left(k\right)g_{\ell+1,+,1}\left(k-\frac{v_{\ell+1}}{2}\right)-s_{\ell+1}(k)g_{\ell+1,-,1}\left({-}k-\frac{v_{\ell+1}}{2}\right), \, ..., \\ \nonumber 
A_{m-2,-}(f)=& g_{m-2,-,1}\left({-}k-\frac{v_{m-1}}{2}\right)-r_{m-1}(k)g_{m-1,+,1}\left(k-\frac{v_{m-1}}{2}\right)\\ \nonumber &{-}s_{m-1}(k)g_{m-1,-,1}\left({-}k-\frac{v_{m-1}}{2}\right),\\  \nonumber
A_{m-1,+}(f)=&g_{m-1,+,1}\left({-}k-\frac{v_{m-1}}{2}\right)-s_{m-2}\left({-}k-\frac{v_{m-1}}{2}+\frac{v_{m-2}}{2}\right)g_{m-2,+,1}\left({-}k-\frac{v_{m-1}}{2}\right)\\
\nonumber
&{-}r_{m-2}\left({-}k+\frac{v_{m-2}-v_{m-1}}{2}\right)g_{m-2,-,1}\left(k+\frac{v_{m-1}}{2}-v_{m-2}\right),\\ \nonumber
A_{m-1,-}(f)=&g_{m-1,-,1}\left({-}k-\frac{v_{m}}{2}\right)-r_{m}(k)g_{m,+,1}\left(k-\frac{v_{m}}{2}\right),\\ \nonumber
A_{m,+}(f)=& g_{m,+,1}\left({-}k-\frac{v_{m}}{2}\right)-s_{m-1}\left({-}k-\frac{v_{m}}{2}+\frac{v_{m-1}}{2}\right)g_{m-1,+,1}\left({-}k-\frac{v_{m}}{2}\right)\\ \nonumber &{-}r_{m-1}\left({-}k+\frac{v_{m-1}}{2}-\frac{v_{m}}{2}\right)g_{m-1,-,1}\left(k+\frac{v_{m}}{2}-v_{m-1}\right),
\end{align}
such that all the functions $A_{n,\pm} (f)$ satisfy for a constant $K>1$ the following inequality
\begin{equation*}
    \norm{A_{n,\pm}(f)}_{L^{2}_{x}(\mathbb{R})}\leq K \norm{f}_{L^{2}_{x}(\mathbb{R})},
\end{equation*}
and all the functions $r_{n},\,s_{n}$ satisfy \eqref{POPO}, and it is enough to verify that the existence of a unique solution $\overrightarrow{g}=(g_{\ell,\pm,1})_{1\leq \ell \leq m}$ of the linear system above, and that
\begin{equation*}
    \norm{g_{\ell,\pm,1}}_{L^{2}_{k}(\mathbb{R},\mathbb{C})}\leq C\max_{n}\norm{A_{n,\pm}(f)}_{L^{2}_{k}(\mathbb{R},\mathbb{C})},
\end{equation*}
for a constant $C>1.$
\par Moreover, the linear system above is exactly a very small perturbation, when $\min y_{\ell}-y_{\ell+1}>1$ is sufficiently large, of the following estimates 
\eqref{l+1},
\eqref{l-1}, \eqref{pl+1}, \eqref{Pl-11}, \eqref{P-m}, \eqref{P+m} of the previous two steps.
\par In particular, since the right-hand side of the huge linear system above is exactly $T(g)$ where $T$ is a map satisfying all the hypotheses of Lemma \ref{tcopp},
we deduce from this lemma the existence of a unique solution $\left(g_{\ell,\pm,1}\right)_{1\leq \ell \leq m}$ of the system of equations above and the existence of a constant $K>1$
satisfying
\begin{equation*}
    \max_{n}\norm{g_{n,\pm,1}}_{L^{2}_{x}(\mathbb{R})}\leq K\norm{f}_{L^{2}_{x}(\mathbb{R})}.
\end{equation*}
\par Therefore, by an argument of symmetry, we can also verify the existence and uniqueness of functions $(g_{n,\pm,2})_{2 n\leq m},$ and these functions also satisfy 
\begin{equation*}
    \max_{n}\norm{g_{n,\pm,2}}_{L^{2}_{x}(\mathbb{R})}\leq K\norm{f}_{L^{2}_{x}(\mathbb{R})}.
\end{equation*}
\par Therefore, for any $f\in L^{2}_{x}(\mathbb{R},\mathbb{C}),$ there exists a unique solution $\phi$ in the domain of $\mathcal{S}(0),$ and a unique set of functions $(f_{n,\pm})_{1\leq n\leq m}$ satisfying \eqref{lG}, \eqref{1G}, \eqref{mG} when $\min v_{\ell}-v_{\ell+1}>0$ is sufficiently large. 
In particular, we can deduce from Step $1$ and Theorem \ref{TT} that there exists a $C_{m}>1$ satisfying
\begin{equation}\label{poitt}
   \max_{\ell} \norm{v_{d_{\ell}}(x)}_{L^{2}_{x}(\mathbb{R})}\leq C_{m}\left[\norm{f(x)}_{L^{2}_{x}(\mathbb{R})}+\norm{\mathcal{S}(0)(\vec{\phi})}_{L^{2}_{x}(\mathbb{R})}\right]\leq 2C_{m}\norm{f(x)}_{L^{2}_{x}(\mathbb{R})}.
\end{equation}\\
\textbf{Step 6. (Conclusion of the Proof of Lemma \ref{dec}.)}Next, we consider $D\coloneqq\min_{\ell}y_{\ell}-y_{\ell+1}>0.$ Using the unique solution of
\eqref{lG}, \eqref{1G}, \eqref{mG} that we obtained in the previous steps, we have the following identity
\begin{align}\label{almostperfect}
    f(x)=&\sum_{\ell=2}^{m-1}\chi_{\left\{\frac{y_{\ell+1}+y_{\ell}}{2}<x\leq \frac{y_{\ell}+y_{\ell-1}}{2}\right\}}e^{i\frac{\sigma_{3}v_{\ell}x}{2}}\hat{G}_{\omega_{\ell}}\left(e^{iy_{\ell}k}\begin{bmatrix}
        \phi_{1,\ell}\left(k+\frac{v_{\ell}}{2}\right)\\
        \phi_{2,\ell}\left(k-\frac{v_{\ell}}{2}\right)
    \end{bmatrix}\right)(x-y_{\ell})
    \\ \nonumber
        &{+}\chi_{\left\{\frac{y_{1}+y_{2}}{2}<x\right\}}e^{i\frac{\sigma_{3}v_{1}x}{2}}\hat{G}_{\omega_{1}}\left(e^{iy_{1}k}\begin{bmatrix}
        \phi_{1}\left(k+\frac{v_{1}}{2}\right)\\
        \phi_{2}\left(k-\frac{v_{1}}{2}\right)
    \end{bmatrix}\right)(x-y_{1})\\
    \nonumber
        &{+}\chi_{\left\{x\leq \frac{y_{m}+y{m-1}}{2} \right\}}e^{i\frac{\sigma_{3}v_{m}x}{2}}\hat{G}_{\omega_{m}}\left(e^{iy_{m}k}\begin{bmatrix}
        \phi_{1,m}\left(k+\frac{v_{m}}{2}\right)\\
        \phi_{2,m}\left(k-\frac{v_{m}}{2}\right)
    \end{bmatrix}\right)(x-y_{m})\\
    \nonumber
    &{+}\chi_{\left\{x>\frac{y_{1}+y_{2}}{2} \right\}}e^{i\frac{\sigma_{3}v_{1}x}{2}}v_{d_{1}}(x-y_{1})\\ \nonumber
    &{+}\sum_{\ell=2}^{m-1}\chi_{\left\{\frac{y_{\ell+1}+y_{\ell}}{2}<x\leq \frac{y_{\ell}+y_{\ell-1}}{2} \right\}}e^{i\frac{\sigma_{3}v_{\ell}x}{2}}v_{d_{\ell}}(x-y_{\ell})\\ \nonumber
    &{+}\chi_{\left\{x\leq  \frac{y_{m}+y_{m-1}}{2}\right\}}e^{i\frac{\sigma_{3}v_{m}x}{2}}v_{d_{m}}(x-y_{m}).
\end{align}
Let $P_{\ell}$ be the following subsets of $\mathbb{R}$
\begin{equation*}
    P_{\ell}=
    \begin{cases}
     \left(\frac{y_{1}+y_{2}}{2},{+}\infty\right) \text{, if $\ell=1,$}\\
     \left(\frac{y_{\ell}+y_{\ell+1}}{2},\frac{y_{\ell}+y_{\ell-1}}{2}\right) \text{, if $1<\ell<m,$}\\
\left({-}\infty,\frac{y_{m}+y_{m-1}}{2}\right) \text{, if $\ell=m.$}
    \end{cases}
\end{equation*}
\par Furthermore, from the Definition \ref{s0def} of $\mathcal{S}(0),$ and using Lemma \ref{appFourier} and Theorem \ref{TT}, we can verify that if $\min v_{\ell}-v_{\ell+1}>0$ is large enough, then there exist constants $C>1,\,\beta>0$ satisfying
\begin{equation}\label{POOO}
   \norm{\mathcal{S}(0)(\vec{\phi})(x)-\sum_{\ell}^{m}\chi_{P_{\ell}}(x)\hat{G}_{\omega_{\ell}}\left(e^{iy_{\ell}k}
   \begin{bmatrix}
   \phi_{1,\ell}\left(k+\frac{v_{\ell}}{2}\right)\\
   \phi_{2,\ell}\left(k-\frac{v_{\ell}}{2}\right)
   \end{bmatrix}\right)}_{L^{2}_{x}(\mathbb{R})}\leq Ce^{{-}\beta(\min_{\ell}y_{\ell}-y_{\ell+1})}\norm{f}_{L^{2}_{x}(\mathbb{R})}, 
\end{equation}
always when $\min_{\ell}y_{\ell}-y_{\ell+1}>0$ is large enough. Consequently, identity \eqref{almostperfect} implies that \begin{equation}\label{perfectfinal}
    \norm{f(x)-\mathcal{S}(0)(\vec{\phi})(x)-\sum_{\ell}\chi_{P_{\ell}}(x)e^{i\frac{\sigma_{3}v_{\ell}x}{2}}v_{d_{\ell}}(x-y_{\ell})}_{L^{2}_{x}(\mathbb{R})}\leq Ce^{{-}\beta(\min_{\ell} y_{\ell}-y_{\ell+1})}\norm{f}_{L^{2}_{x}(\mathbb{R})}.
\end{equation}    
\par Therefore, since the $\Raa \mathcal{S}(0)$ is a closed subspace of $L^{2}_{x}(\mathbb{R},\mathbb{C}),$ and
\begin{equation*}
    W=\sppp\left\{\chi_{P_{\ell}}(x)e^{i\frac{\sigma_{3}v_{\ell}x}{2}}v_{d_{\ell}}(x-y_{\ell}) \text{, $v_{d_{\ell}}\in\Raa P_{d,\omega_{\ell}}$ and $\ell\in\{1,2,\,...,\,m\}$}\right\}
\end{equation*}
is finite dimensional, $H=\Raa \mathcal{S}(0)+W$ is a closed subspace of $L^{2}_{x}(\mathbb{R},\mathbb{C}).$
\par However, since \eqref{perfectfinal} is true for any $f\in L^{2}_{x}(\mathbb{R},\mathbb{C})$ and we have for any $f\in H^{\perp}$ that
\begin{equation*}
    \norm{f(x)-\mathcal{S}(0)(\vec{\phi})(x)-\sum_{\ell}\chi_{P_{\ell}}(x)e^{i\frac{\sigma_{3}v_{\ell}x}{2}}v_{d_{\ell}}(x-y_{\ell})}_{L^{2}_{x}(\mathbb{R})}\geq \norm{f}_{L^{2}_{x}(\mathbb{R})},
\end{equation*}
the subspace $H^{\perp}$ of $L^{2}_{x}(\mathbb{R},\mathbb{C})$ should be equal to $\{0\}$ when $\min_{\ell}y_{\ell}-y_{\ell+1}>1$ is sufficiently large. Therefore,
\begin{equation}\label{Poq11}
    \Raa \mathcal{S}(0)+W=L^{2}_{x}(\mathbb{R},\mathbb{C}^{2})
\end{equation}
when $\min_{\ell}y_{\ell}-y_{\ell+1}>1$ is very large.
\par Furthermore, each subspace $\Raa P_{d,\omega_{\ell}}\subset L^{2}(\mathbb{R},\mathbb{C}^{2})$ has a finite orthonormal basis $\{v_{d_{\ell},1},\, ...,\,v_{d_{\ell},n_{\ell}}\}$ of functions having exponential decay for any $\ell,$ and so there exists $\beta>0$ satisfying
\begin{equation*}
   \norm{\left[1-\chi_{P_{\ell}}(x)\right]v_{d_{\ell},n}(x-y_{\ell})}_{L^{2}_{x}(\mathbb{R})}\leq e^{{-}\beta \min_{\ell} (y_{\ell}-y_{\ell+1})} \text{, for all $\ell\in\{1,\,...,\,m\}.$}
\end{equation*}
\par In conclusion, using \eqref{poitt},
 we can verify similarly to the proof of \eqref{Poq11} that
\begin{equation*}
    \Raa \mathcal{S}(0)=L^{2}_{x}(\mathbb{R},\mathbb{C}^{2})+\sppp\left\{e^{i\frac{\sigma_{3}v_{\ell}x}{2}}v_{d_{\ell}}(x-y_{\ell}) \text{, $v_{d_{\ell}}\in\Raa P_{d,\omega_{\ell}}$ and $\ell\in\{1,2,\,...,\,m\}$}\right\},
\end{equation*}
when $\min_{\ell}y_{\ell}-y_{\ell+1}>1$ is very large, which is equivalent to the identity of Lemma \ref{dec}.
\end{proof}
\begin{proof}[Proof of Theorem \ref{princ}.]
It is enough to prove \eqref{princ11} when $t=0.$ Since $\Raa \mathcal{T}(0)$ is closed, we assume by contradiction that there exists a non-zero function $f\in L^{2}_{x}(\mathbb{R},\mathbb{C}^{2})$ satisfying
\begin{equation}\label{Contrahypo}
    f(x) \in \left[\Raa \mathcal{T}(0)\oplus\bigoplus_{n=1}^{m}\mathcal{G}_{\omega_{n}}\left(\Raa P_{d,\omega_{n}}\right)\right]^{\perp}. 
\end{equation}
But, Lemma \ref{dec} implies that there exist $\vec{\phi}$ in the domain of $\mathcal{S}(0)$ and functions $v_{d_{1}}\in\Raa P_{d,\omega_{1}},\,...,\,v_{d_{m}}\in P_{d,\omega_{m}}$ satisfying
\begin{equation*}
    f(x)=\mathcal{S}(\vec{\phi})(0,x)+\sum_{n=1}^{m}\mathcal{G}_{\omega_{n}}(v_{d_{\omega_{n}}})(x-y_{n}),
\end{equation*}
and Corollary \ref{c62} implies the existence of a $K_{m}>0$ satisfying
\begin{equation*}
    \norm{\mathcal{S}(\vec{\phi})(0,x)}_{L^{2}_{x}(\mathbb{R})}+\max_{n}\norm{v_{d_{\omega_{n}}}(x)}_{L^{2}_{x}(\mathbb{R})}\leq K_{m}\norm{f(x)}_{L^{2}_{x}(\mathbb{R})}
\end{equation*}
Therefore, Theorems \ref{tcont} and \ref{tdis} imply that there exists constants $\beta>0,\,C>1$ satisfying
\begin{multline*}
    \norm{f(x)-\mathcal{T}(\vec{\phi})(0,x)-\sum_{n=1}^{m}\mathcal{G}_{\omega_{n}}(v_{d_{\omega_{n}}}+r_{d_{n}})(x-y_{n})}_{L^{2}_{x}(\mathbb{R})}\\
    \leq C e^{{-}\beta \min_{\ell}(y_{\ell+1}-y_{\ell})}\left[\norm{\mathcal{S}(0)(\vec{\phi})(x)}_{L^{2}_{x}(\mathbb{R})}+\sum_{n=1}^{m}\norm{v_{d_{\omega_{n}}}}_{L^{2}_{x}(\mathbb{R})}\right].
\end{multline*}
Consequently, we deduce from Corollary \ref{c62} that there exists a constant $C_{m}>1$ satisfying
\begin{equation*}
    \norm{f(x)-\mathcal{T}(\vec{\phi})(0,x)-\sum_{n=1}^{m}\mathcal{G}_{\omega_{n}}(v_{d_{\omega_{n}}}+r_{d_{n}})(x-y_{n})}_{L^{2}_{x}(\mathbb{R})}
    \leq C_{m} e^{{-}\beta \min_{\ell}(y_{\ell+1}-y_{\ell})}\norm{f(x)}_{L^{2}_{x}(\mathbb{R})}, 
\end{equation*}
 when $\min_{\ell}y_{\ell}-y_{\ell+1}$ is large enough, but this contradicts the assumption \eqref{Contrahypo}.
 \par Therefore, any solution $\overrightarrow{\psi}$ of \eqref{p} has a unique representation of the form \eqref{princ11} when $\min_{\ell}y_{\ell}-y_{\ell+1}$ and $\min_{\ell}v_{\ell}-v_{\ell+1}$ are large enough.
 \par Furthermore, using Theorems \ref{tcont} and \ref{TT}, we can verify the inequalities \eqref{F0l1}, \eqref{F0l2} and the estimate \eqref{phi2kkdecay} for $n=0$ and $n=2.$  The proof of \eqref{phi2kkdecay} for $n=1$ is explained in Lemma \ref{h1coercc} of Section \ref{Dispsection}, it follows from the estimate \eqref{phi2kkdecay} for $n=0,\,n=2,$ and an argument of interpolation.
 \par In conclusion, Theorem \ref{princ} is true when $\min_{\ell}y_{\ell}-y_{\ell+1}$ is large enough.
\end{proof}

As a direct consequence of the asymptotic completeness,  Lemma \ref{kphi}, Theorem \ref{tcont}
 and Corollary \ref{c62}, we have the the asymptotic completeness in $H^1$.
\begin{corollary}
    \par There  exists  a constant $K>1$ satisfying
\begin{align*}
    \norm{\mathcal{T}(0)(\vec{\phi})(x)-\mathcal{S}(0)(\vec{\phi})(x)}_{H^{1}_{x}(\mathbb{R})}\leq & Ke^{{-}\min_{\ell}(y_{\ell}-y_{\ell+1})\beta}\norm{\mathcal{S}(0)(\vec{\phi})}_{L^{2}_{x}(\mathbb{R})}.     \end{align*}
Similarly to the Step $6$ of the proof of Lemma \ref{dec}, we can verify that 
\begin{equation*}
    H^{1}_{x}(\mathbb{R})=\Raa\{ \mathcal{T}(0)(\phi)\vert\,\max_{\ell}\norm{(1+\vert k\vert )(\phi_{1,\ell}(k),\phi_{2,\ell}(k))}_{L^{2}_{k}(\mathbb{R})}<{+}\infty\}\oplus \bigoplus_{l=1}^{m} \Raa P_{d,\omega_{\ell}}.
\end{equation*}
\end{corollary}

\section{Proof of the dispersive estimates}\label{Dispsection}
In this section, we will prove all the estimates of Theorem \ref{Decesti}. First, we need to consider the following proposition.
\begin{lemma}\label{interpoo}
There exists a $c>0$ satisfying following inequality holds for $\overrightarrow{\phi}\in L^{2}_{k}(\mathbb{R},\mathbb{C}^{2})$ belonging to the domain of $\hat{F}_{\omega}$
\begin{equation}\label{H1}
    \norm{\hat{F}_{\omega}(\overrightarrow{\phi})(x)}_{H^{1}_{x}(\mathbb{R})}\geq c\norm{(1+\vert k\vert)\overrightarrow{\phi}(k)}_{L^{2}_{k}(\mathbb{R})}
\end{equation}
\end{lemma}
\begin{proof}
Since $\sigma_{3}G^{*}_{\omega}\sigma_{3}\hat{F}_{\omega}=\mathrm{Id},$ it is enough to verify that the following map
\begin{equation*}
    T(\overrightarrow{u})(k)\coloneqq \left[1+\vert k\vert^{2}\right]^{\frac{1}{2}}G^{*}_{\omega}\left(\sigma_{3}\left[{-}\partial^{2}_{x}+1\right]^{{-}\frac{1}{2}}\overrightarrow{u}(x)\right) (k)
\end{equation*}
is a bounded map on $L^{2}.$
Next, we consider the following families of linear maps
\begin{equation*}
    T_{\theta}(\overrightarrow{u})(k)\coloneqq \left[1+\vert k\vert^{2}\right]^{\frac{\theta}{2}}G^{*}_{\omega}\left(\sigma_{3}\left[{-}\partial^{2}_{x}+1\right]^{{-}\frac{\theta}{2}}\overrightarrow{u}(x)\right)(k) \text{ for any $\theta\in\mathbb{C}$ satisfying $0\leq \Ree \theta\leq 2.$}
\end{equation*}
In particular for any two Schwartz functions $\overrightarrow{u}(x)\in \mathscr{S}_{x}(\mathbb{R},\mathbb{C}^{2})$ and $\overrightarrow{g}(k)\in \mathscr{S}_{x}(\mathbb{R},\mathbb{C}^{2}),$ we consider the following function
\begin{equation*}
    f(\theta)=\left\langle T_{\theta}(\overrightarrow{u})(k),\overrightarrow{g}(k) \right\rangle  \text{, for any $\theta\in\mathbb{C}$ satisfying $0\leq \Ree \theta\leq 2,$}
\end{equation*}
which is well-defined because both functions $\overrightarrow{u}$ and $\overrightarrow{g}$ are Schwartz.
\par Next, because of Lemma \ref{Asy1sol1} and the inverse formula the $\sigma_{3}G^{*}_{\omega}\sigma_{3}\hat{F}_{\omega}=\mathrm{Id},$ we have the following two inequalities
\begin{equation*}
    \norm{T_{2}(\overrightarrow{u})}_{L^{2}}\leq C_{2}\norm{\overrightarrow{u}}_{L^{2}},\, \norm{T_{0}(\overrightarrow{u})}_{L^{2}}\leq C_{0}\norm{\overrightarrow{u}}_{L^{2}},
\end{equation*}
such that $C_{0}$ and $C_{2}$ do not depend on $\overrightarrow{u}.$
In particular, we can verify from the elementary inequality $\left\vert\left[1+\vert k\vert^{2}\right]^{\frac{i}{2}}\right\vert=1$ and the Plancherel theorem that
\begin{equation}\label{intercond}
    \norm{T_{2+\beta i}(\overrightarrow{u})}_{L^{2}}\leq C_{2}\norm{\overrightarrow{u}}_{L^{2}},\, \norm{T_{0+\beta i}(\overrightarrow{u})}_{L^{2}}\leq C_{0}\norm{\overrightarrow{u}}_{L^{2}} \text{, for any $\beta\in\mathbb{R}.$}
\end{equation}
\par Furthermore, since both functions $\overrightarrow{u}$ and $\overrightarrow{g}$ are both Schwartz, and $G^{*}_{\omega}$ is a bounded map on $L^{2},$ we have that  
\begin{equation}\label{Analytic}
    f(\theta) \text{ is an analytic function on the strip $\{\theta\in\mathbb{C}\vert\,  1\leq \Ree \theta\leq 2\}.$}
\end{equation}
Therefore  $\vert f(i\beta) \vert\leq C_{0}\norm{\overrightarrow{u}}_{L^{2}}\norm{\overrightarrow{g}}_{L^{2}}$ and $\vert f(2+i\beta) \vert\leq C_{2}\norm{\overrightarrow{u}}_{L^{2}}\norm{\overrightarrow{g}}_{L^{2}},$ for any $\beta\in\mathbb{R}.$ 
\par Consequently, because of \eqref{intercond} and \eqref{Analytic}, we can verify using the three lines lemma that there exists a $C_{1}>0$ satisfying
\begin{equation}\label{ff}
    \vert f(1)\vert\leq C_{1}\norm{\overrightarrow{u}}_{L^{2}}\norm{\overrightarrow{g}}_{L^{2}}.
\end{equation}
Therefore, since the choice of $\overrightarrow{u}$ and $\overrightarrow{g}$ as Schwartz function was arbitrary, we can verify from the density of $\mathscr{S}(\mathbb{R},\mathbb{C}^{2})$ on $L^{2}(\mathbb{R},\mathbb{C}^{2})$ and \eqref{ff} that $T_{1}$ is a bounded operator on $L^{2}.$ 
In conclusion, \eqref{H1} is true for any $\overrightarrow{u}$ belonging to the domain of $\hat{F}_{\omega}$.
\end{proof}
Next, we consider the following proposition which is a direct  consequence of Theorem  \ref{princ} and  Theorem \ref{TT}.
\begin{corollary}\label{trivial22}
If all the hypotheses of Theorem \ref{princ} are true, then the map $\mathcal{T}(t)$ defined in Theorem \ref{tcont} is invertible. More precisely, for any $t\geq 0,$ the linear maps $inv(T),\,inv_{\ell,t}(T):L^{2}_{x}(\mathbb{R},\mathbb{C}^{2})\to L^{2}_{k}(\mathbb{R},\mathbb{C}^{2})$   defined by
\begin{align*}
inv_{t}(T)(\overrightarrow{f})(k)=&\begin{cases}
    \overrightarrow{\phi}(k) \text{, if $\mathcal{T}(\overrightarrow{\phi})(t,x)=f(x),$}\\
    0, \text{ if $f\in\left(\sigma_{3}\Raa \mathcal{T}(t)\right)^{\perp},$}
    \end{cases}
\\
inv_{\ell,t}(T)(\overrightarrow{f})(k)=&\begin{cases}
    \overrightarrow{\phi_\ell}(k) \text{, if $\mathcal{T}(\overrightarrow{\phi})(t,x)=f(x),$}\\
    0, \text{ if $f\in\left(\sigma_{3}\Raa \mathcal{T}(t)\right)^{\perp},$}
\end{cases}  \end{align*}
are well-defined, bounded and $inv(T)\circ \mathcal{T}(t)=\mathrm{Id}$ for any $t\geq 0.$  
\end{corollary}

\begin{remark}\label{relasst}
 In particular, Remark \ref{re000} and Theorem \ref{tcont} imply that there exists a constant $C>1$ satisfying for any $t\geq 0$ the estimate
 \begin{equation*}
 \norm{inv_{t}(T)}+\max_{\ell}\norm{inv_{\ell,t}(T)}\leq C.    
 \end{equation*}
\end{remark}
Corollary \ref{trivial22} will be useful in proving the following proposition about the $H^{1}$ norm of $\mathcal{T}(0)(\overrightarrow{\phi}).$ 
\begin{lemma}\label{h1coercc}
There exists a constant $c>0$ depending only on the potentials $V_{\ell}$ satisfying for any $t\geq 0$ the following inequality for any $\overrightarrow{\phi}$ belonging to the domain of $\mathcal{T}(t)$     \begin{equation*}
\norm{\mathcal{T}(\overrightarrow{\phi})(t,x)}_{H^{1}_{x}(\mathbb{R})}\geq c\max_{\ell}\norm{(1+\vert k\vert )\overrightarrow{\phi}_{\ell}(k)}_{L^{2}_{k}(\mathbb{R})}.    
\end{equation*}
\end{lemma}
Furthermore, a standard corollary of Lemmas \ref{kphi} and \ref{h1coercc} is the following proposition.
\begin{corollary}
 There exists a constant $C>0$ satisfying for any $t\geq \tau\geq 0$
 \begin{equation*}
    \norm{P_{c}(t)\mathcal{U}(t,\tau)\overrightarrow{\psi_{0}}}_{H^{1}_{x}(\mathbb{R})}\leq C \norm{\overrightarrow{\psi_{0}}}_{H^{1}_{x}(\mathbb{R})}.
 \end{equation*}
\end{corollary}
\begin{proof}[Proof of Lemma \ref{h1coercc}]
First, because of Remark \ref{relasst}, it is enough to prove Lemma \ref{h1coercc} when $t=0.$

\par Next, from Theorem \ref{princ}, there exist positive constants $c_{1}$ and $c_{2}$ satisfying
\begin{equation*}
\norm{\mathcal{T}(0)(\overrightarrow{\phi})}_{L^{2}_{x}(\mathbb{R})}\geq c_{1}\max_{\ell}\norm{\overrightarrow{\phi}_{\ell}(k)}_{L^{2}_{k}(\mathbb{R})},\, \norm{\mathcal{T}(0)(\overrightarrow{\phi})}_{H^{2}_{x}(\mathbb{R})}\geq c_{2}\max_{\ell}\norm{\left(1+\vert k\vert^{2}\right)\overrightarrow{\phi}_{\ell}(k)}_{L^{2}_{k}(\mathbb{R})}.   
\end{equation*}
In particular, using Corollary \ref{trivial22} and Theorem \ref{princ}, the two estimates above are equivalent to
\begin{align}\label{pt1}
\max_{\ell}\norm{inv_{\ell,0}(T)\left(\overrightarrow{f}(x)\right)(k)}_{L^{2}_{k}(\mathbb{R}) }\leq &C_{1}\norm{\overrightarrow{f}(x)}_{L^{2}_{k}(\mathbb{R})}  \\ \label{pt2}
\max_{\ell}\norm{(1+k^{2})inv_{\ell,0}(T)\left(\left[{-}\frac{d^{2}}{dx^{2}}+1\right]^{{-}1}\overrightarrow{f}(x)\right)(k)}_{L^{2}_{k}(\mathbb{R}) }\leq & C_{2}\norm{\overrightarrow{f}(x)}_{L^{2}_{k}(\mathbb{R})},  \end{align}
for some positive constants $C_{1},\,C_{2}>1$ and any $\overrightarrow{f}\in L^{2}_{x}(\mathbb{R},\mathbb{C}^{2}).$
\par Consequently, repeating the argument in the proof of Lemma \ref{interpoo}, we can deduce Lemma \ref{h1coercc} from the inequalities \eqref{pt1} and \eqref{pt2} by applying interpolation techniques. 
\end{proof}

Next, to simplify the notation in the next proposition, we consider the following definition.
\begin{definition}
 In the notation of Theorem \ref{princ}, we consider the following characteristic function.
 \begin{align*}
     \chi_{\ell}(\tau,x)=&\chi_{\left[\frac{y_{\ell}+y_{\ell+1}+\tau(v_{\ell}+v_{\ell+1})}{2},\frac{y_{\ell}+y_{\ell-1}+\tau(v_{\ell}+v_{\ell-1})}{2}\right]}(x) \text{, if $\ell\neq 1 $ and $\ell\neq m,$}\\
     \chi_{1}(\tau,x)=&\chi_{\left(\frac{y_{1}+y_{2}+\tau(v_{1}+v_{2})}{2},{+}\infty\right)}(x),\\
     \chi_{m}(\tau,x)=&\chi_{\left({-}\infty,\frac{y_{m}+y_{m-1}+\tau(v_{m}+v_{m-1})}{2}\right)}(x). 
 \end{align*}
\end{definition}
The proof of Theorem \ref{Decesti} will follow from Lemma \ref{interpoo} and the proposition below. 
\begin{theorem}\label{Decesti1}
 If $\min y_{\ell}-y_{\ell+1}>L$ and $\min v_{\ell}-v_{\ell+1}>M,$ the following estimates are true for constants $K>1,\,\beta>0$
 \begin{align}\nonumber
  \norm{\mathcal{U}(t,\tau)P_{c}(\tau)\overrightarrow{\psi_{0}}}_{L^{2}_{x}(\mathbb{R})}\leq & K \norm{P_{c}(\tau)\overrightarrow{\psi_{0}}}_{L^{2}_{x}(\mathbb{R})},\\ \label{P1}
   \norm{\mathcal{U}(t,\tau)P_{c}(\tau)\overrightarrow{\psi_{0}}}_{L^{\infty}_{x}(\mathbb{R})}\leq & \max_{\ell} \frac{K}{(t-\tau)^{\frac{1}{2}}}\norm{(1+\vert x-y_{\ell}-v_{\ell}\tau\vert )\chi_{\ell}(\tau,x)P_{c}(\tau)\overrightarrow{\psi_{0}}(x)}_{L^{2}_{x}(\mathbb{R})}. 
\end{align}
Moreover, we also have local decay estimates if (H4) from Theorem \ref{Decesti} holds
\begin{align}\label{Q1}
\norm{\mathcal{U}(t,\tau)P_{c}(\tau)\overrightarrow{\psi_{0}}}_{L^{\infty}_{x}(\mathbb{R})}\leq & \max_{\ell} \frac{K}{(t-\tau)^{\frac{1}{2}}}\left[\norm{P_{c}(\tau)\overrightarrow{\psi_{0}}(x)}_{L^{1}_{x}(\mathbb{R})}+e^ {{-}\beta \min_{\ell}(y_{\ell}-y_{\ell+1}+(v_{\ell}-v_{\ell+1})\tau)}\norm{P_{c}(\tau)\overrightarrow{\psi_{0}}(x)}_{L^{2}_{x}(\mathbb{R})}\right],\\ \label{Q2}
\norm{\frac{\mathcal{U}(t,\tau)P_{c}\overrightarrow{\psi_{0}}}{(1+\vert x-y_{\ell}-v_{\ell}t\vert)}}_{L^{\infty}_{x}(\mathbb{R})}\leq & \frac{K (1+\tau)}{(t-\tau)^{\frac{3}{2}}} \norm{P_{c}(\tau)\overrightarrow{\psi_{0}}(x)}_{L^{1}_{x}(\mathbb{R})}\\ \nonumber 
&{+}\frac{K}{(t-\tau)^{\frac{3}{2}}}\max_{\ell}\norm{(1+\vert x-y_{\ell}-v_{\ell}\tau\vert )\chi_{\ell}(\tau,x)P_{c}(\tau)\overrightarrow{\psi_{0}}(x)}_{L^{1}_{x}(\mathbb{R})}\\ \nonumber 
&{+}\frac{K(y_{1}-y_{m}+(v_{1}-v_{m})\tau)^{2}e^{{-}\beta\left[\min_{\ell} (v_{\ell}-v_{\ell+1})\tau+(y_{\ell}-y_{\ell+1})\right]}\norm{P_{c}(\tau)\overrightarrow{\psi}_{0}(x)}_{H^{2}}}{(t-\tau)^{\frac{3}{2}}}. 
 \end{align}
Furthermore, the projection $P_{c}$ satisfies the following identity for any $\overrightarrow{\psi_{0}}\in L^{2}_{x}(\mathbb{R},\mathbb{C}^{2})$
\begin{equation*}
    P_{c}(t)\mathcal{U}(t,\tau)\overrightarrow{\psi_{0}}= \mathcal{U}(t,\tau)P_{c}(\tau)\overrightarrow{\psi_{0}}=\mathcal{U}(t,\tau)P_{c}\overrightarrow{\psi_{0}}.
\end{equation*}
\end{theorem}
 \begin{proof}[Proof of Theorem \ref{Decesti}]
First, from Theorem \ref{princ}, we have that there exists a constant $K>1$ satisfying $\norm{P_{c}(t)\mathcal{U}(t,\tau)\overrightarrow{\psi_{0}}}_{L^{2}_{x}(\mathbb{R})}\leq K\norm{\overrightarrow{\psi_{0}}}_{L^{2}_{x}(\mathbb{R})}$ for any $\overrightarrow{\psi_{0}}(x)\in L^{2}_{x}(\mathbb{R},\mathbb{C}^{2}).$ Furthermore, we can verify from Theorem \ref{princ} and Lemmas \ref{kphi} and \ref{h1coercc} for any $t\geq\tau\geq 0$ that
\begin{equation*}
   \max_{n\in\{0,1\}}\norm{\frac{\mathcal{U}(t,\tau)P_{c}(\tau)\overrightarrow{\psi_{0}}}{(1+\vert x-y_{\ell}-v_{\ell}t\vert)^{n}}}_{L^{\infty}_{x}(\mathbb{R})}\leq C_{1}\norm{P_{c}(t)\mathcal{U}(t,\tau)\overrightarrow{\psi_{0}}}_{H^{1}_{x}(\mathbb{R})}\leq C_{2}\norm{P_{c}(\tau)\overrightarrow{\psi_{0}}}_{H^{1}_{x}(\mathbb{R})}, 
\end{equation*}
for some positive constants $C_{1},\,C_{2}>1.$ 
\par Consequently, we can verify from Theorem \ref{Decesti1} that there exists $K>1$ satisfying
\begin{equation*}
 \norm{\mathcal{U}(t,\tau)P_{c}(\tau)\overrightarrow{\psi_{0}}}_{L^{\infty}_{x}(\mathbb{R})}\leq K\min\left(\norm{P_{c}(\tau)\overrightarrow{\psi_{0}}}_{H^{1}_{x}(\mathbb{R})},\frac{\max_{\ell}\norm{(1+\vert x-y_{\ell}-v_{\ell}\tau\vert )\chi_{\ell}(\tau,x)P_{c}(\tau)\overrightarrow{\psi_{0}}(x)}_{L^{2}_{x}(\mathbb{R})}}{(t-\tau)^{\frac{1}{2}}}\right),    
\end{equation*}
which implies the inequality \eqref{decaystrit} of Theorem \ref{Decesti}.
\par Similarly, when (H4) holds, we can verify from Theorem \ref{Decesti1} that
\begin{multline}\label{lastwww}
   \norm{\frac{\mathcal{U}(t,\tau)P_{c}\overrightarrow{\psi_{0}}}{(1+\vert x-y_{\ell}-v_{\ell}t\vert)}}_{L^{\infty}_{x}(\mathbb{R})}\\
   \begin{aligned}
   \leq & K\min\left(\norm{P_{c}(\tau)\overrightarrow{\psi_{0}}(x)}_{H^{1}_{x}(\mathbb{R})},\frac{ (1+\tau)}{(t-\tau)^{\frac{3}{2}}} \max_{j\in\{1,2\}}\max_{\ell}\norm{P_{c}(\tau)\overrightarrow{\psi_{0}}(x)}_{L^{j}_{x}(\mathbb{R})}\right)\\ 
&{+}K\min\left(\norm{P_{c}(\tau)\overrightarrow{\psi_{0}}(x)}_{H^{1}_{x}(\mathbb{R})},\frac{1}{(t-\tau)^{\frac{3}{2}}}\max_{\ell}\norm{(1+\vert x-y_{\ell}-v_{\ell}\tau\vert )\chi_{\ell}(\tau,x)P_{c}(\tau)\overrightarrow{\psi_{0}}(x)}_{L^{1}_{x}(\mathbb{R})}\right)\\ 
&{+}K\min\left(\norm{P_{c}(\tau)\overrightarrow{\psi_{0}}(x)}_{H^{1}_{x}(\mathbb{R})},\frac{(y_{1}-y_{m}+(v_{1}-v_{m})\tau)^{2}e^{{-}\beta\left[\min_{\ell} (v_{\ell}-v_{\ell+1})\tau+(y_{\ell}-y_{\ell+1})\right]}\norm{P_{c}(\tau)\overrightarrow{\psi}_{0}(x)}_{H^{2}}}{(t-\tau)^{\frac{3}{2}}}\right)  
  \end{aligned}.
\end{multline}
Therefore, all the terms on the right-hand side of the inequality \eqref{lastwww} are less or equal than  
\begin{equation}\label{fin1}
    \frac{10K(y_{1}-y_{m}+\tau)}{(1+(t-\tau))^{\frac{1}{2}}(y_{1}-y_{m}+t)}\norm{P_{c}(\tau)\overrightarrow{\psi_{0}}}_{H^{1}_{x}(\mathbb{R})}, 
\end{equation}
when $0\leq \tau,\,0\leq t-\tau\leq 1$ and $y_{1}-y_{m}>1$ is large enough.
\par Moreover, we can deduce from Theorem \ref{Decesti1} that the inequality \eqref{weight} holds when
when $t-\tau\geq 1.$ In conclusion, the estimate \eqref{weight} is true for any real numbers$t\geq \tau\geq 0.$

\end{proof}
\begin{proof}[Proof of Theorem \ref{Decesti1} using Theorems \ref{princ}, \ref{TT}.]
 \par First, if $\overrightarrow{\psi_{0}}(x)\in \Raa P_{c}(\tau),$ then there exists a function $\vec{\phi}$ in the domain of $\mathcal{T}(\tau)$ satisfying
 \begin{equation*}
     \overrightarrow{\psi_{0}}(x)=\mathcal{T}(\vec{\phi})\left(\tau,x\right).
 \end{equation*}
 Next, using Lemma \ref{lem:decayonepotential},
 and Theorem \ref{princ}, we can verify the existence of constants $C_{\ell},\,K_{\ell}>1$ satisfying 
 \begin{multline*}
     \norm{\hat{G}_{\omega_{\ell}}\left(e^{{-}i(t-\tau)k^{2}\sigma_{3}}e^{i(v_{\ell}\tau+y_{\ell})k}\begin{bmatrix}
         e^{{-}i\tau (k+\frac{v_{\ell}}{2})^{2}}\phi_{1,\ell}(k+\frac{v_{\ell}}{2})\\
         e^{i\tau (k-\frac{v_{\ell}}{2})^{2}}\phi_{2,\ell}(k-\frac{v_{\ell}}{2})
     \end{bmatrix}\right)(x)}_{L^{2}_{x}(\mathbb{R})}\\
     \begin{aligned}
     \leq & C_{\ell} \norm{\hat{G}_{\omega_{\ell}}\left(e^{i(v_{\ell}\tau+y_{\ell})k}\begin{bmatrix}
         e^{{-}i\tau (k+\frac{v_{\ell}}{2})^{2}}\phi_{1,\ell}(k+\frac{v_{\ell}}{2})\\
         e^{i\tau (k-\frac{v_{\ell}}{2})^{2}}\phi_{2,\ell}(k-\frac{v_{\ell}}{2})
     \end{bmatrix}\right)}_{L^{2}_{x}(\mathbb{R})}
     \\
     \leq & K_{\ell}\norm{\overrightarrow{\psi_{0}}}_{L^{2}_{x}(\mathbb{R})},
    \end{aligned}
 \end{multline*}
where we used Theorems \ref{princ}, \ref{tcont} and Corollary \ref{ccc} of Section $4$ in the last inequality above. Therefore, using the Minkowski inequality, we can deduce from the estimate above with Theorem \ref{tcont} the inequality \eqref{l22}. 
\par It is enough to prove of second inequality is enough for the case where $\tau=0.$  First, from Lemma \ref{lem:decayonepotential}, Remark \ref{transition}, and the asymptotic behaviors of \eqref{asy1} and \eqref{asy3}, we can verify the following estimate
\begin{align*}
    \norm{\hat{G}_{\omega_{\ell}}\left(e^{{-}it(k^{2}+\omega_{\ell})\sigma_{3}}e^{iy_{\ell}k}
    \begin{bmatrix}
        \phi_{1,\ell}\left(k+\frac{v_{\ell}}{2}\right)\\
        \phi_{2,\ell}\left(k-\frac{v_{\ell}}{2}\right)
    \end{bmatrix}
    \right)(x)}_{L^{\infty}_{x}(\mathbb{R})}
    \leq & \frac{C_{\ell}}{t^{\frac{1}{2}}} \norm{\hat{G}_{\omega_{\ell}}\left(e^{iy_{\ell}k}
    \begin{bmatrix}
        \phi_{1,\ell}\left(k+\frac{v_{\ell}}{2}\right)\\
        \phi_{2,\ell}\left(k-\frac{v_{\ell}}{2}\right)
    \end{bmatrix}
    \right)(x)}_{L^{1}_{x}(\mathbb{R})}
    \\
    \leq & \max_{n\in\{0,1\}} \frac{C_{\ell,1}}{t^{\frac{1}{2}}} \norm{\left(
    \begin{bmatrix}
        \hat{\phi}_{1,\ell-n}\left(x\right)\\
        \hat{\phi}_{2,\ell-n}\left(x\right)
    \end{bmatrix}
    \right)}_{L^{1}_{x}(\mathbb{R})},
\end{align*}
when $\ell\neq 1,\,\ell\neq m$ and the positive constants $C_{\ell},\,C_{\ell,1}$ and $C_{\ell,2}$ depend only on $\ell,$ see also Definition \ref{s0def} and Remark \ref{transition}. Therefore, Theorem \ref{princ} implies that there exists a constant $K_{\ell}>1$ satisfying
\begin{align}\label{Gll1inft}
    \norm{\hat{G}_{\omega_{\ell}}\left(e^{{-}it(k^{2}+\omega_{\ell})\sigma_{3}}e^{iy_{\ell}k}
    \begin{bmatrix}
        \phi_{1,\ell}\left(k+\frac{v_{\ell}}{2}\right)\\
        \phi_{2,\ell}\left(k-\frac{v_{\ell}}{2}\right)
    \end{bmatrix}
    \right)(x)}_{L^{\infty}_{x}(\mathbb{R})}\leq &\frac{K_{\ell}}{t^{\frac{1}{2}}}\max_{n}\norm{\langle x+y_{n}\rangle\chi_{\left\{\frac{y_{n+1}+y_{n}}{2},\frac{y_{n-1}+y_{n}}{2}\right\}}(x)\overrightarrow{\psi_{0}}(x)}_{L^{2}_{x}(\mathbb{R})},
\end{align}
where $\langle x\rangle=(1+x^{2})^{\frac{1}{2}},$ for all $\ell$ satisfying $\ell \neq 1$ and $\ell\neq m.$ In particular, if $\ell\neq m$ and $\ell \neq 1,$ then, when hypothesis (H4) is true,  we can deduce from estimate \eqref{fl1teo2} of Theorem \ref{TT} that
\begin{equation}\label{Gl2infty}\norm{\hat{G}_{\omega_{\ell}}\left(e^{{-}it(k^{2}+\omega_{\ell})\sigma_{3}}e^{iy_{\ell}k}
    \begin{bmatrix}
        \phi_{1,\ell}\left(k+\frac{v_{\ell}}{2}\right)\\
        \phi_{2,\ell}\left(k-\frac{v_{\ell}}{2}\right)
    \end{bmatrix}
    \right)(x)}_{L^{\infty}_{x}(\mathbb{R})}\leq \frac{K_{\ell}}{t^{\frac{1}{2}}}\norm{\mathcal{S}(0)(\vec{\phi})}_{L^{1}_{x}(\mathbb{R})}.
\end{equation}
\par Next, for the case $\ell=1,$ we can verify that
the existence of constants $\beta>0$ and $C_{1},\,C_{1,1}>1$ satisfying \begin{align}\label{linfty1}
    \norm{\hat{G}_{\omega_{1}}\left(e^{{-}it(k^{2}+\omega_{\ell})\sigma_{3}}e^{iy_{1}k}
    \begin{bmatrix}
        \phi_{1}\left(k+\frac{v_{1}}{2}\right)\\
        \phi_{2}\left(k-\frac{v_{1}}{2}\right)
    \end{bmatrix}
    \right)(x)}_{L^{\infty}_{x}(\mathbb{R})}
    \leq & \frac{C_{1}}{t^{\frac{1}{2}}} \norm{\hat{G}_{\omega_{1}}\left(e^{iy_{1}k}
    \begin{bmatrix}
        \phi_{1}\left(k+\frac{v_{1}}{2}\right)\\
        \phi_{2}\left(k-\frac{v_{1}}{2}\right)
    \end{bmatrix}
    \right)(x)}_{L^{1}_{x}(\mathbb{R})}
    \\ \nonumber
    \leq &  \frac{C_{1,1}}{t^{\frac{1}{2}}} \norm{\left(
    \begin{bmatrix}
        \hat{\phi}_{1}\left(x\right)\\
        \hat{\phi}_{2}\left(x\right)
    \end{bmatrix}
    \right)}_{L^{1}_{x}(\mathbb{R})}\\ \nonumber &{+}\frac{C_{1,1}}{t^{\frac{1}{2}}}\norm{\chi_{\left(\frac{y_{2}-y_{1}}{2},{+}\infty\right)}(x)\hat{G}_{\omega_{1}}\left(e^{iy_{1}k}
    \begin{bmatrix}
        \phi_{1}\left(k+\frac{v_{1}}{2}\right)\\
        \phi_{2}\left(k-\frac{v_{1}}{2}\right)
    \end{bmatrix}
    \right)(x)}_{L^{1}_{x}(\mathbb{R})}\\ \nonumber
   &{+} \frac{C_{1,1}}{t^{\frac{1}{2}}}e^{{-}\beta(y_{1}-y_{2})}\norm{\overrightarrow{\psi_{0}}(x)}_{L^{2}_{x}(\mathbb{R})},
\end{align}
such that the last term of the inequality above is a consequence of Lemma \ref{LB3} with the asymptotic behavior \eqref{asy1} of $\mathcal{G}_{\omega}$ and Theorem \ref{princ}. The analysis of the $L^{\infty}$ norm of
\begin{equation*}
    \hat{G}_{\omega_{m}}\left(e^{iy_{m}k}
    \begin{bmatrix}
        \phi_{1,m}\left(k+\frac{v_{m}}{2}\right)\\
        \phi_{2,m}\left(k-\frac{v_{m}}{2}\right)
    \end{bmatrix}
    \right)(x)
\end{equation*}
is completely analogous. 

\par Consequently, using the definition of $\mathcal{T}(\vec{\phi})(0,x),$ we conclude
\begin{align*}
    \max_{\ell}\norm{\hat{G}_{\omega_{\ell}}\left(e^{{-}it(k^{2}+\omega_{\ell})\sigma_{3}}e^{iy_{\ell}k}
    \begin{bmatrix}
        \phi_{1,\ell}\left(k+\frac{v_{\ell}}{2}\right)\\
        \phi_{2,\ell}\left(k-\frac{v_{\ell}}{2}\right)
    \end{bmatrix}
    \right)(x)}_{L^{\infty}_{x}(\mathbb{R})}\leq & \frac{K}{t^{\frac{1}{2}}}\max_{n}\norm{(x+y_{n})\chi_{\left\{\frac{y_{n+1}+y_{n}}{2},\frac{y_{n-1}+y_{n}}{2}\right\}}(x)\overrightarrow{\psi_{0}}(x)}_{L^{2}_{x}(\mathbb{R})}\\&{+}\frac{K}{t^{\frac{1}{2}}}\norm{u_{0}}_{L^{2}_{x}(\mathbb{R})}.
\end{align*}
Furthermore, if the hypothesis (H4) is true, we can verify using estimates \eqref{Gl2infty}, \eqref{linfty1}, inequality \eqref{fl1teo2} of Theorem \ref{TT}, and Theorem \ref{tcont} that 
\begin{align*}
    \max_{\ell\in\{1,m\}}\norm{\hat{G}_{\omega_{\ell}}\left(e^{{-}it(k^{2}+\omega_{\ell})\sigma_{3}}e^{iy_{\ell}k}
    \begin{bmatrix}
        \phi_{1,\ell}\left(k+\frac{v_{\ell}}{2}\right)\\
        \phi_{2,\ell}\left(k-\frac{v_{\ell}}{2}\right)
    \end{bmatrix}
    \right)(x)}_{L^{\infty}_{x}(\mathbb{R})}\leq & \frac{K}{t^{\frac{1}{2}}}\max_{n}\left[\norm{\overrightarrow{\psi_{0}}}_{L^{1}_{x}(\mathbb{R})}\right].
\end{align*}
In conclusion, using the fact that $\overrightarrow{\psi_{0}}(x)=\mathcal{T}(\vec{\phi})(0,x),$ we can deduce from the estimate above, inequality \eqref{h1linft} of Theorem \ref{tcont}, and Lemma \ref{appFourier} that \eqref{P1} holds, and \eqref{Q1} is true under the assumption of (H4) . 

\par Next, it is enough to prove \eqref{Q2} when $\tau=0.$ Using Lemme \ref{lem:decayonepotential}, 
Lemma \ref{LB3} and Remark \ref{transition}, we can verify the existence of constants $K>1,\,\beta>0$ satisfying for any $1\leq \ell\leq m$ the following estimate
\begin{multline*}
\norm{\frac{1}{\langle x\rangle}\hat{G}_{\omega_{\ell}}\left( e^{{-}it(k^{2}+\omega_{\ell})\sigma_{3}}e^{iy_{\ell}k}
\begin{bmatrix}
        \phi_{1,\ell}\left(k+\frac{v_{\ell}}{2}\right)\\
        \phi_{2,\ell}\left(k-\frac{v_{\ell}}{2}\right)
    \end{bmatrix}\right)(x)}_{L^{\infty}_{x}(\mathbb{R})}\\
    \begin{aligned}    \leq & \frac{K}{t^{\frac{3}{2}}}\norm{\langle x\rangle\hat{G}_{\omega_{\ell}}\left( e^{iy_{\ell}k}
    \begin{bmatrix}
        \phi_{1,\ell}\left(k+\frac{v_{\ell}}{2}\right)\\
        \phi_{2,\ell}\left(k-\frac{v_{\ell}}{2}\right)
    \end{bmatrix}\right)(x)}_{L^{1}_{x}(\mathbb{R})}\\
    \leq & \frac{K}{t^{\frac{3}{2}}}\norm{\langle x\rangle\chi_{\left\{\frac{y_{\ell+1}-y_{\ell}}{2},\frac{y_{\ell-1}-y_{\ell}}{2}\right\}}(x)\hat{G}_{\omega_{\ell}}\left( e^{iy_{\ell}k}
    \begin{bmatrix}
        \phi_{1,\ell}\left(k+\frac{v_{\ell}}{2}\right)\\
        \phi_{2,\ell}\left(k-\frac{v_{\ell}}{2}\right)
    \end{bmatrix}\right)(x)}_{L^{1}_{x}(\mathbb{R})}\\
    &{+}\frac{K}{t^{\frac{3}{2}}}\max_{j\in\{0,1\}}\norm{\langle x\rangle F_{0}(x)\left(e^{iy_{\ell}k}\begin{bmatrix}
        \phi_{1,l-j}(k)\\
        \phi_{2,l-j}(k)
    \end{bmatrix}\right)}_{L^{1}_{x}(\mathbb{R})}\\
    &{+}\frac{Ke^{{-}\beta\min_{j\in\{{-}1,1\}}\vert y_{\ell}-y_{\ell-j} \vert }}{t^{\frac{3}{2}}}\norm{\hat{G}_{\omega_{\ell}}\left( e^{iy_{\ell}k}
    \begin{bmatrix}
        \phi_{1,\ell}\left(k+\frac{v_{\ell}}{2}\right)\\
        \phi_{2,\ell}\left(k-\frac{v_{\ell}}{2}\right)
    \end{bmatrix}\right)(x)}_{L^{2}_{x}(\mathbb{R})}
   \end{aligned},
\end{multline*}
where the last expression of the inequality above follows from Lemma  \ref{LB3}.
\par Similarly, we can verify that
\begin{multline*}
 \norm{\frac{1}{(1+\vert x\vert)}\hat{G}_{\omega_{1}}\left( e^{{-}it(k^{2}+\omega_{\ell})\sigma_{3}}e^{iy_{1}k}
    \begin{bmatrix}
        \phi_{1}\left(k+\frac{v_{1}}{2}\right)\\
        \phi_{2}\left(k-\frac{v_{1}}{2}\right)
    \end{bmatrix}\right)(x)}_{L^{\infty}_{x}(\mathbb{R})}\\
    \begin{aligned}
    \leq & \frac{K}{t^{\frac{3}{2}}}\norm{\langle x\rangle \chi_{\left\{\frac{y_{2}-y_{1}}{2},{+}\infty\right\}}(x)\hat{G}_{\omega_{1}}\left( e^{iy_{1}k}
    \begin{bmatrix}
        \phi_{1}\left(k+\frac{v_{1}}{2}\right)\\
        \phi_{2}\left(k-\frac{v_{1}}{2}\right)
    \end{bmatrix}\right)(x)}_{L^{1}_{x}(\mathbb{R})}\\
    &{+}\frac{K}{t^{\frac{3}{2}}}\norm{\langle x\rangle F_{0}(x)\left(e^{iy_{\ell}k}\begin{bmatrix}
        \phi_{1}(k)\\
        \phi_{2}(k)
    \end{bmatrix}\right)}_{L^{1}_{x}(\mathbb{R})}\\
    &{+}\frac{Ke^{{-}\beta (y_{1}-y_{2}) }}{t^{\frac{3}{2}}}\norm{\hat{G}_{\omega_{\ell}}\left( e^{iy_{1}k}
    \begin{bmatrix}
        \phi_{1}\left(k+\frac{v_{1}}{2}\right)\\
        \phi_{2}\left(k-\frac{v_{1}}{2}\right)
    \end{bmatrix}\right)(x)}_{L^{2}_{x}(\mathbb{R})}
   \end{aligned},
\end{multline*}
and
\begin{multline*}
   \norm{\frac{1}{(1+\vert x\vert)}\hat{G}_{\omega_{m}}\left( e^{{-}it(k^{2}+\omega_{\ell})\sigma_{3}}e^{iy_{m}k}
    \begin{bmatrix}
        \phi_{1}\left(k+\frac{v_{m}}{2}\right)\\
        \phi_{2}\left(k-\frac{v_{m}}{2}\right)
    \end{bmatrix}\right)(x)}_{L^{\infty}_{x}(\mathbb{R})}\\
    \begin{aligned}
    \leq & \frac{K}{t^{\frac{3}{2}}}\norm{\langle x\rangle \chi_{\left\{{-}\infty,\frac{y_{m-1}-y_{m}}{2}\right\}}(x)\hat{G}_{\omega_{m}}\left( e^{iy_{m}k}
    \begin{bmatrix}
        \phi_{1,m}\left(k+\frac{v_{m}}{2}\right)\\
        \phi_{2,m}\left(k-\frac{v_{m}}{2}\right)
    \end{bmatrix}\right)(x)}_{L^{1}_{x}(\mathbb{R})}\\
    &{+}\frac{K}{t^{\frac{3}{2}}}\norm{\langle x\rangle F_{0}(x)\left(e^{iy_{m}k}\begin{bmatrix}
        \phi_{1,m-1}(k)\\
        \phi_{2,m-1}(k)
    \end{bmatrix}\right)}_{L^{1}_{x}(\mathbb{R})}\\
    &{+}\frac{Ke^{{-}\beta (y_{m-2}-y_{m-1}) }}{t^{\frac{3}{2}}}\norm{\hat{G}_{\omega_{\ell}}\left( e^{iy_{m}k}
    \begin{bmatrix}
        \phi_{1,m-1}\left(k+\frac{v_{m-1}}{2}\right)\\
        \phi_{2,m-1}\left(k-\frac{v_{m-1}}{2}\right)
    \end{bmatrix}\right)(x)}_{L^{2}_{x}(\mathbb{R})}
   \end{aligned}.
\end{multline*}
Consequently, using the fact that $\overrightarrow{\psi_{0}}(x)=\mathcal{S}(\overrightarrow{\phi_{0}})(0,x)+r(0,x),$ for a function $r(0,x)$ satisfying the decay rates in Theorem \ref{tcont}, we can verify from triangular inequality and the estimates above that
\begin{align}\nonumber
    \max_{\ell}\norm{\frac{\chi_{\ell}(t,x)}{(1+\vert x-y_{\ell}-v_{\ell}t\vert)}\mathcal{U}(t,0)P_{c}\overrightarrow{\psi_{0}}}_{L^{\infty}_{x}(\mathbb{R})}\leq & \frac{K}{t^{\frac{3}{2}}}\max_{\ell}\norm{\langle x\rangle\chi_{\ell}(0,x+y_{\ell})\mathcal{S}(\overrightarrow{\phi_{0}})(0,x+y_{\ell})}_{L^{1}_{x}(\mathbb{R})} 
    \\ \nonumber &{+}\frac{K}{t^{\frac{3}{2}}}\max_{1\leq \ell\leq m-1}\max_{j,n\in\{0,1\}}\norm{\frac{d^{j}}{dk^{j}}\begin{bmatrix}
        e^{iy_{\ell+n}k}\phi_{1,\ell}(k)\\
        e^{iy_{\ell+n}k}\phi_{2,\ell}(k)
\end{bmatrix}}_{\mathcal{F}L^{1}}\\ \nonumber  &{+}
     \frac{K}{t^{\frac{3}{2}}}\max_{\ell}\norm{\langle x\rangle\chi_{\ell}(0,x+y_{\ell})r(0,x+y_{\ell})}_{L^{1}_{x}(\mathbb{R})}\\  \label{finaleqq}&{+}\frac{Ke^{-\beta (\min_{\ell}y_{\ell}-y_{\ell+1})}}{t^{\frac{3}{2}}}\norm{\overrightarrow{\psi_{0}}(x)}_{L^{2}_{x}(\mathbb{R})}.
\end{align}
In conclusion, if the hypothesis (H4) is true, we can use estimates \eqref{fl1teo1} and \eqref{fl1teo2} of Theorem \ref{TT} and Theorem \ref{tcont} in the estimate \eqref{finaleqq} to deduce that \eqref{Q2} is true.
\end{proof}

\section{Asymptotic solutions for charge transfer models}\label{asyinfinity}
In this section,  Theorem \ref{tcont} and Theorem \ref{tdis} will be proved, respectively, in \S \ref{solscont} and \S \ref{soldisc}. The proof of these results will follow as an application of the fixed point theorem in appropriate Banach spaces.
\begin{definition}
We define the Banach spaces $L^{2}_{\beta}(\mathbb{R},\mathbb{C}^{2})$ and $H^{1}_{\beta}(\mathbb{R},\mathbb{C}^{2})$ to be the sets of all functions $f\in L^\infty_t\left(\mathbb{R}_{\geq 0}, L^{2}_{x}\left(\mathbb{R},\mathbb{C}^{2}\right)\right)$ and  $f\in L^\infty_t\left(\mathbb{R}_{\geq 0}, H^1_{x}\left(\mathbb{R},\mathbb{C}^{2}\right)\right)$  satisfying for  $\beta>0$ the following inequalities respectively
\begin{align*}
\norm{f}_{L^{2}_{\beta}}\coloneqq\sup_{t\geq 0}e^{\beta t}\norm{f(t,x)}_{L^{2}_{x}(\mathbb{R})}<&{+}\infty \text{, if $f\in L^{2}_{\beta}(\mathbb{R},\mathbb{C}^{2}),$}\\ \norm{f}_{H^{1}_{\beta}}\coloneqq\sup_{t\geq 0}e^{\beta t}\norm{f(t,x)}_{H^{1}_{x}(\mathbb{R})}<&{+}\infty \text{, if $f\in H^{1}_{\beta}(\mathbb{R},\mathbb{C}^{2}).$}
\end{align*}   
\end{definition}
with the following lemma. 
\begin{lemma}\label{interactt}
For any real numbers $x_{2},x_{1}$, such that $\zeta=x_{2}-x_{1}>0$ and $\alpha,\,\beta,\,m>0$ with $\alpha\neq \beta$ the following bound holds:
\begin{equation*}
\int_{\mathbb{R}}\vert x-x_{1}\vert ^{m} e^{-\alpha(x-x_{1})_{+}}e^{-\beta(x_{2}-x)_{+}}\lesssim_{\alpha,\beta,m} \max\left(\left(1+\zeta^{m}\right)e^{-\alpha \zeta},e^{-\beta \zeta}\right),
\end{equation*}
For any $\alpha>0$, the following bound holds
\begin{equation*}
    \int_{\mathbb{R}}\vert x-x_{1}\vert^{m} e^{-\alpha(x-x_{1})_{+}}e^{-\alpha(x_{2}-x)_{+}}\lesssim_{\alpha}\left[1+\zeta^{m+1}\right] e^{-\alpha \zeta}.
\end{equation*}
\end{lemma}
\begin{proof}
    Elementary computations.
\end{proof}
More precisely, Theorems \ref{tcont} and \ref{tdis} will follow as a corollary of the following lemma.

\begin{lemma}\label{edl}
Let $\beta>0$ satisfying $\min\left(\beta,\min_{\ell}(v_{\ell}-v_{\ell+1})\right)>\max_{\ell}\max_{\lambda\in\sigma_{d}\mathcal{H}_{\omega_{\ell}}}\vert \I \lambda \vert,$ and $f(t,x)\in H^{1}_{\beta}$
such that $\overrightarrow{v}=(v_{1},\,v_{2},\,...,\,v_{m})\in \mathbb{R}^{m}$ satisfies
\begin{equation*}
    v_{\ell}-v_{\ell+1}>0 \text{, for all $\ell\geq 1.$}
\end{equation*}
There exist parameters $C(\overrightarrow{v},\beta),\,L(\overrightarrow{v},\beta)>1$ depending only on $\overrightarrow{v}$ and $\beta$ such that if 
\begin{equation*}
    \min_{\ell} y_{\ell}-y_{\ell+1}\geq L(\overrightarrow{v},\beta),
\end{equation*}
 then, for $\beta_{1}=\beta-\max_{\ell}\max_{\lambda\in\sigma_{d}\mathcal{H}_{\omega_{\ell}}}\vert \I \lambda \vert,$ there exists a unique solution $r(t,x)\in H^{1}_{\beta_{1}}$ being a   solution of the following partial differential equation
    \begin{equation}\label{rpde}
    i\partial_{t}r(t)+\sigma_{3}\partial^{2}_{x}r(t)-\sum_{\ell=1}^{m}V^{\sigma_{\ell}}_{\ell}(t)r(t)=f(t,x).   
 \end{equation}  
Furthermore, 
\begin{align}\label{e000}
\norm{r(t)}_{H^{1}_{\beta_{1}}}\leq & C(\overrightarrow{v},\beta)\norm{f(t)}_{H^{1}_{\beta}},
\end{align}
and if $f\in H^{k}_{\beta}(\mathbb{R},\mathbb{C}^{2})$ for some  $k\in\mathbb{N},$ then $r \in H^{k}_{\beta_{1}}(\mathbb{R},\mathbb{C}^{2})$ and
\begin{equation}\label{po}
\norm{r(t)}_{H^{k}_{\beta}}\leq  C_{k}(\overrightarrow{v},\beta)\norm{f(t)}_{H^{k}_{\beta}}.
\end{equation}

\end{lemma}
\begin{remark}
In particular, repeating the proof of Lemma \ref{edl}, we can verify that Lemma \ref{edl} still holds if we assume the following weak hypothesis for $v_{\infty}=\lim_{t\to{+}\infty}v(t)$ and a constant $d_{\infty}$ on the path $\sigma(t)=(y(t),v(t),\gamma(t))$
\begin{equation*}
    \max_{t\geq 0}
    (1+t)^{\epsilon}\vert y(t)-v_{\infty}t-d_{\infty} \vert+(1+t)^{2+\epsilon}\vert \dot v(t)\vert+(1+t)^{1+\epsilon}\vert \dot \gamma(t) \vert<{+}\infty.
\end{equation*}
\end{remark}
\begin{remark}\label{weightedremainderedecay}
 In particular, if $r(t)\in H^{3}_{x}(\mathbb{R})$ and $x^{2}r(t)\in H^{1}_{x}(\mathbb{R}),$ then since $r_{1}(t,x)=xr(t,x),\,r_{2}(t,x)=x^{2}r(t,x)$ and $r_{0}(t,x)=x\partial_{x}r(t,x)$ are solutions respectively of
 \begin{align*}
i\partial_{t}r_{1}(t,x)+\sigma_{3}\partial^{2}_{x}r_{1}(t,x)-\sum_{\ell=1}^{m}V^{\sigma_{\ell}}_{\ell}(t)r_{1}(t,x)=&2\sigma_{3}\partial_{x}r(t,x)+xf(t,x),\\
i\partial_{t}r_{2}(t,x)+\sigma_{3}\partial^{2}_{x}r_{2}(t,x)-\sum_{\ell=1}^{m}V^{\sigma_{\ell}}_{\ell}(t)r_{2}(t)=&2\sigma_{3}r(t,x)+4\sigma_{3}r_{1}(t,x)+x^{2}f(t,x),\\
i\partial_{t}r_{0}(t,x)+\sigma_{3}\partial^{2}_{x}r_{0}(t,x)-\sum_{\ell=1}^{m}V^{\sigma_{\ell}}_{\ell}(t)r_{0}(t)=&x\sum_{\ell=1}^{m}\partial_{x}V^{\sigma_{\ell}}_{\ell}(t)r(t,x)+x\partial_{x}f(t,x)+2\sigma_{3}\partial^{2}_{x}r(t,x), 
 \end{align*}
from which we can deduce from Lemma \ref{edl} that there exists $C(\overrightarrow{v},\beta)>1$ and $\beta>0$ satisfying
\begin{align*}
\norm{x r(t,x)}_{H^{1}_{\beta}(\mathbb{R})}\leq & C(\overrightarrow{v},\beta)\left[\norm{f(t,x)}_{H^{2}_{\beta}(\mathbb{R})}+\norm{xf(t,x)}_{H^{1}_{\beta}(\mathbb{R})}\right],\\
\norm{x^{2}r(t,x)}_{H^{1}_{\beta}(\mathbb{R})}\leq & C(\overrightarrow{v},\beta)\left[\norm{x^{2}f(t,x)}_{H^{1}_{\beta}}+\norm{xf(t,x)}_{H^{2}_{\beta}}+\norm{f(t,x)}_{H^{3}_{\beta}}\right].
\end{align*}
\end{remark}
\begin{proof}[Proof of Lemma \ref{edl}.]
We are going to verify that the inequality \eqref{e000} is true if $r(t)$ is a solution of \eqref{rpde} belonging to $H^{1}_{\beta}$. The proof of \eqref{po} is completely analogous to the proof of \eqref{e000} using induction using \eqref{po} for $k-1.$ The proof of existence and uniqueness of a solution of \eqref{rpde} $r(t)\in H^{1}_{\beta}$ follows using estimate \eqref{e000} and it is completely analogous to Step $2$ of the proof of Lemma $3.1$ from \cite{kinknew}.
\par First, we consider a smooth increasing cut-off function $\chi$ satisfying $0\leq \chi\leq 1,$ and
\begin{equation*}
 \chi(x)=
 \begin{cases}
 0 \text{, if $x\leq{-} \frac{1}{2}-\frac{1}{10000},$}\\
 1 \text{, if $x\geq  {-}\frac{1}{2},$}
 \end{cases}   
\end{equation*}
from which we define the following functions
\begin{align*}
    \chi_{1}(t,x)=&\chi\left(\frac{x}{y_{1}-y_{2}+(v_{1}-v_{2})t}\right),\\
    \chi_{2}(t,x)=&\chi\left(\frac{x}{y_{2}-y_{3}+(v_{2}-v_{3})t}\right)-\chi\left(\frac{x-(y_{1}-y_{2}+(v_{1}-v_{2})t)}{y_{1}-y_{2}+(v_{1}-v_{2})t}\right),\,...,\\
    \chi_{\ell}(t,x)=& \chi\left(\frac{x}{y_{\ell}-y_{\ell+1}+(v_{\ell}-v_{\ell+1})t}\right)-\chi\left(\frac{x-(y_{\ell-1}-y_{\ell}+(v_{\ell-1}-v_{\ell})t)}{y_{\ell-1}-y_{\ell}+(v_{\ell-1}-v_{\ell})t}\right),\,...,\\
    \chi_{m}(t,x)=&1-\chi\left(\frac{x-(y_{m-1}-y_{m}+(v_{m-1}-v_{m})t)}{y_{m-1}-y_{m}+(v_{m-1}-v_{m})t}\right).
\end{align*}
Next, for $\theta_{\ell}(t,x)=\frac{v_{\ell}x}{2}-\frac{v_{\ell}^{2}t}{4}+\gamma_{\ell}+\omega_{\ell}t,$ and setting
\begin{align}\label{v+}
   u_{+,\ell}(t,x):=&r_{1}(t,x+v_{\ell}t+y_{\ell})e^{{-}i\theta_{\ell}(t,x+v_{\ell}t+y_{\ell})}+r_{2}(t,x+v_{\ell}t+y_{\ell})e^{i\theta_{\ell}(t,x+v_{\ell}t+y_{\ell})},\\ \label{v-} u_{-,\ell}(t,x):=& r_{1}(t,x+v_{\ell}t+y_{\ell})e^{{-}i\theta_{\ell}(t,x+v_{\ell}t+y_{\ell})}-r_{2}(t,x+v_{\ell}t+y_{\ell})e^{i\theta_{\ell}(t,x+v_{\ell}t+y_{\ell})},
\end{align}
we can verify that for $f=(f_{1},f_{2})$ that  \eqref{rpde} is equivalent to the following differential system of equations
\begin{align}\label{edp+-}
   i\partial_{t}\begin{bmatrix}
       u_{+,\ell}(t,x)\\
       u_{-,\ell}(t,x)
   \end{bmatrix}-
   \begin{bmatrix}
    L_{-,\omega_{\ell}}u_{-,\ell}(t,x)\\
L_{+,\omega_{\ell}}u_{+,\ell}(t,x)
   \end{bmatrix}=&\sum_{j\neq \ell} V_{j}\left(x-(y_{j}-y_{\ell}+(v_{j}-v_{\ell})t)\right)\begin{bmatrix}
       u_{+,\ell}(t,x)\\
       u_{-,\ell}(t,x)
   \end{bmatrix} \nonumber 
   \\ \nonumber
   &{+} 
   \begin{bmatrix}
       \frac{1}{2} & \frac{1}{2}\\
       \frac{1}{2} & {-}\frac{1}{2}
   \end{bmatrix}
   \begin{bmatrix}
       f_{1}(t,x-y_\ell-v_\ell t)e^{{-}i\theta_{\ell}(t,x+y_{t})}\\
       f_{2}(t,x-y_\ell-v_\ell t)e^{i\theta_{\ell}(t,x+y_{t})},
   \end{bmatrix}
\end{align}
for a set of Schwartz functions $V_{j}(x)\in C^{\infty}(\mathbb{R},\mathbb{C}^{2\times 2})$ satisfying
\begin{equation*}
    \left\vert \frac{d^{n}}{dx^{n}}V_{j}(x)\right\vert\lesssim_{n} e^{{-}2\sqrt{\omega_{j}}\vert x\vert} \text{, for all $x\in\mathbb{R},$}
\end{equation*}
and the operators $L_{-,\omega_{\ell}}$ and $L_{+,\omega_{\ell}}$ are the self-adjoint operators defined by
\begin{align*}
    L_{+,\omega_{\ell}}\coloneqq &{-}\partial^{2}_{x}+\omega_{\ell}+U_{\ell}(x)+W_{\ell}(x),\\
    L_{-,\omega_{\ell}}\coloneqq &{-}\partial^{2}_{x}+\omega_{\ell}+U_{\ell}(x)-W_{\ell}(x).
\end{align*}
Next, we consider the localized Hamiltonians
\begin{align*}
 H_{\ell,n}\left(r\right)=&\frac{1}{4}\int_{\mathbb{R}}\chi_{\ell}(t,x)\chi_{n}(t,x)\left[\left\vert \partial_{x}u_{+,\ell}(t,x) \right\vert^{2}+\left\vert \partial_{x}u_{-,\ell}(t,x)\right\vert^{2}\right]\,dx\\
 &{+}\frac{\omega_{\ell}}{4}\int_{\mathbb{R}}\chi_{\ell}(t,x)\chi_{n}(t,x)\left[\left\vert u_{+}(t,x)\right\vert^{2}+\left\vert u_{-}(t,x) \right\vert^{2}\right]\,dx
 \\&{+}\frac{1}{4}\int_{\mathbb{R}}\chi_{\ell}(t,x)\chi_{n}(t,x)U_{\ell}(x)\left[\left\vert u_{+}(t,x)\right\vert^{2}+\left\vert u_{-}(t,x) \right\vert^{2}\right]\,dx\\
 &{+}\int_{\mathbb{R}}\chi_{\ell}(t,x)\chi_{n}(t,x)W_{\ell}(x)\left(\frac{\vert u_{+}(t,x) \vert^{2}}{4}-\frac{\vert u_{-}(t,x) \vert^{2}}{4}\right)\,dx.
\end{align*}
\par Furthermore, we have the following identity  
\begin{equation*}
\left\langle \mathcal{H}_{\omega_{\ell}}\left(\begin{bmatrix}
    r_{1}(x)\\
    r_{2}(x)
\end{bmatrix}\right),\sigma_{3} \begin{bmatrix}
    r_{1}(x)\\
    r_{2}(x)
\end{bmatrix}\right\rangle=\left\langle L_{+,\omega_{\ell}}\left(\frac{r_{1}+r_{2}}{2}\right),r_{1}+r_{2} \right\rangle+\left\langle L_{+,\omega_{\ell}}\left(\frac{r_{1}-r_{2}}{2}\right),r_{1}-r_{2} \right\rangle,
\end{equation*}
and the following coercive estimate for an $c>0$ and $K>1$
\begin{equation*}
    \left\langle \mathcal{H}_{\omega_{\ell}}(r),\sigma_{3}r \right\rangle\geq c\norm{r}_{H^{1}}^{2}-K\sum_{n=1}^{d_{\ell}}\left\vert\left\langle r,\sigma_{3}z_{n,\ell}(x)\right\rangle\right\vert^{2}, 
\end{equation*}
where $\{z_{\ell,1},\,...,\,z_{\ell, d_{\ell}}\}$ is an orthonormal basis of $P_{d,\omega_\ell}$ which is the range of the discrete spectrum projection of $\mathcal{H}_{\omega_{\ell}}.$ 
Therefore, we can verify the existence of $c_{0}>0$ and $K>1$ satisfying
\begin{align}\label{coercest}
    H_{\ell,\ell}(r)\geq& c_{0}\norm{\chi_{\ell}(t,x)r(t,x)}_{H^{1}}^{2}-K\sum_{n=1}^{d_{\ell}}\left\langle r(t),\sigma_{3}e^{i\sigma_{3}\theta_{\ell}(t,x)}z_{\ell,n}(x-v_{\ell}t-y_{\ell}) \right\rangle^{2}\\ \nonumber
    &{-}\frac{K}{\min_{n} y_{n-1}-y_{n}+(v_{n-1}-v_{n})t}\norm{r(t)}_{H^{1}}^{2},\\
    H_{\ell,n}(r)\geq & {-}K e^{{-}\beta \min_{\ell}(y_{\ell}-y_{\ell+1}+(v_{\ell}-v_{\ell+1}t))}\norm{r(t)}_{H^{1}}^{2} \text{, if $\ell\neq n,$}
\end{align}
where $\{z_{\ell,1},\,...,\, z_{\ell,d_{\ell}}\}$ is a basis of $ P_{d,\omega_{\ell}}.$ Therefore, we can verify the existence of $c_{1}>0$ and $K>1$ satisfying when $\min_{\ell} y_{\ell}-y_{\ell+1}$ is sufficiently large
\begin{equation}\label{coercfinal}
    \sum_{\ell,n}H_{\ell,n}(r)\geq c_{1}\norm{ r(t)}_{H^{1}}^{2}-K\sum_{\ell}\sum_{n=1}^{d_{\ell}}\left\langle r(t),\sigma_{3}e^{i\sigma_{3}\theta_{\ell}(t,x)}z_{\ell,n}(x-v_{\ell}t-y_{\ell}) \right\rangle^{2},
\end{equation}
such that $\{z_{\ell,1},\,...,\,z_{\ell,d_{\ell}}\}$ is a basis of $\Raa P_{d,\omega_{\ell}}.$
\par Next, using \eqref{rpde} and the definition of $u_{+,\ell},\,u_{-,\ell},$ we can verify that the following functional
\begin{equation*}
H(r)=\sum_{\ell,n=1}^{m}H_{\ell,n}(r)
\end{equation*}
satisfies for any $t\geq 0$
\begin{equation}\label{energyesti1}
    \left\vert\frac{d}{dt}H(r)(t)\right\vert=O\left(\frac{\norm{r(t)}_{H^{1}}^{2}}{t+\min y_{\ell}-y_{\ell+1}}+\norm{r(t)}_{H^{1}}\norm{f(t)}_{H^{1}}\right).
\end{equation}
More precisely, using the identities \eqref{v+} and \eqref{v-}, we can verify that
\begin{align*}
     \vert u_{+,\ell}(t,x)\vert^{2}+\vert u_{-,\ell}(t,x)\vert^{2}=& 2\vert r_{1}(t,x+v_{\ell}+y_{\ell}) \vert^{2}+2\vert r_{2}(t,x+v_{\ell}+y_{\ell}) \vert^{2},\\
     \vert \partial_{x}u_{+,\ell}(t,x)\vert^{2}+\vert \partial_{x}u_{-,\ell}(t,x)\vert^{2}=&2\left\vert \partial_{x}r_{1}(t,x+v_{\ell}t+y_{\ell})-\frac{iv_{\ell}}{2}r_{1}(t,x+v_{\ell}t+y_{\ell}) \right\vert^{2}\\&{+}2\left\vert \partial_{x}r_{2}(t,x+v_{\ell}t+y_{\ell})+\frac{iv_{\ell}}{2}r_{2}(t,x+v_{\ell}t+y_{\ell}) \right\vert^{2},
\end{align*}
from which we can deduce the following identity
\begin{multline}\label{mmm}
\begin{aligned}
    \sum_{\ell,n=1}^{m}H_{\ell,n}(r)=&\frac{1}{2}\int_{\mathbb{R}}\left\vert \partial_{x}r_{1}(t,x) \right\vert^{2}+\left\vert \partial_{x}r_{2}(t,x)\right\vert^{2}\,dx\\
 &{+}\sum_{\ell=1}^{n}\left(\frac{\omega_{\ell}}{2}+\frac{v_{\ell}^{2}}{8}\right)\int_{\mathbb{R}} \chi_{\ell}(t,x-v_{\ell}t-y_{\ell})\left[\left\vert r_{1}(t,x) \right\vert^{2}+\left\vert r_{2}(t,x)\right\vert^{2}\right]\,dx
 \\&{+}\frac{v_{\ell}}{2}\Ree i\int_{\mathbb{R}}\chi_{\ell}(t,x-v_{\ell}t-y_{\ell})\left[\partial_{x}r_{1}(t,x)\overline{r_{1}(t,x)}-\partial_{x}r_{2}(t,x)\overline{r_{2}(t,x)}\right]\,dx
 \\&{+}\frac{1}{2}\sum_{\ell=1}^{m}\int_{\mathbb{R}}\chi_{\ell}(t,x-v_{\ell}t-y_{\ell})U_{\ell}(x-v_{\ell}t-y_{\ell})\left[\left\vert r_{1}(t,x) \right\vert^{2}+\left\vert r_{2}(t,x) \right\vert^{2}\right]\,dx\\
 &{-}\Ree\sum_{\ell=1}^{m}\int_{\mathbb{R}}\chi_{\ell}(t,x-v_{\ell}t-y_{\ell})W_{\ell}\left(x-v_{\ell}t-y_{\ell}\right)e^{2i\theta_{\ell}(t,x)}r_{2}(t,x)\overline{r_{1}(t,x)}\,dx.
\end{aligned}
\end{multline}
Furthermore, using integration by parts and the  equation \eqref{rpde}, we can verify that  
\begin{multline}\label{dmot}
    \frac{d}{dt}\left[\frac{v_{\ell}}{2}\Ree i\int_{\mathbb{R}}\chi_{\ell}(t,x-v_{\ell}t-y_{\ell})\left[\partial_{x}r_{1}(t,x)\overline{r_{1}(t,x)}-\partial_{x}r_{2}(t,x)\overline{r_{2}(t,x)}\right]\,dx\right]\\
    \begin{aligned}
    = &\frac{v_{\ell}}{2}\int_{\mathbb{R}}\chi_{\ell}(t,x-v_{\ell}t-y_{\ell})U^{\prime}_{\ell}(x-v_{\ell}t-y_{\ell})\left[\vert r_{1}(t,x) \vert^{2}+\vert r_{2}(t,x) \vert^{2}\right]\,dx\\
    &{-}v_{\ell}\int_{\mathbb{R}}\chi_{\ell}(t,x-v_{\ell}t-y_{\ell})W^{\prime}_{\ell}(x-v_{\ell}t-y_{\ell})e^{2i\theta_{\ell}(t,x)}r_{2}(t,x)\overline{r_{1}(t,x)}\,dx\\
    &{-}v_{\ell}^{2}\Ree\int_{\mathbb{R}}i\chi_{\ell}(t,x-v_{\ell}t-y_{\ell})W_{\ell}\left(x-v_{\ell}t-y_{\ell}\right)e^{2i\theta_{\ell}(t,x)}r_{2}(t,x)\overline{r_{1}(t,x)}\,dx\\
&{+}O\left(\frac{\norm{r(t)}_{H^{1}}^{2}}{t+\min y_{\ell}-y_{\ell+1}}+\norm{r(t)}_{H^{1}}\norm{f(t)}_{H^{1}}\right),
    \end{aligned}
\end{multline}
and 
\begin{multline}\label{dmasst}
\frac{d}{dt}\frac{v_{\ell}^{2}}{8}\int_{\mathbb{R}}\chi_{\ell}(t,x-v_{\ell}t-y_{\ell})\left[\vert r_{1}(t,x)\vert^{2}+\vert r_{2}(t,x)\vert^{2}\right]\,dx\\
\begin{aligned}
    =&\Ree i\frac{v_{\ell}^{2}}{2}\int_{\mathbb{R}}\chi_{\ell}(t,x-v_{\ell}t-y_{\ell})W_{\ell}(x-v_{\ell}t-y_{\ell})\overline{r_{1}(t,x)}r_{2}(t,x)\,dx\\
&{+}O\left(\frac{\norm{r(t)}_{H^{1}}^{2}}{t+\min y_{\ell}-y_{\ell+1}}+\norm{r(t)}_{H^{1}}\norm{f(t)}_{H^{1}}\right).
\end{aligned}
\end{multline}
Therefore, estimate \eqref{energyesti1} follows from the time derivative of \eqref{mmm} using the  equation \eqref{rpde} and estimates \eqref{dmot}, \eqref{dmasst}.  
\par Next, for any two elements $\overrightarrow{z}_{1},\,\overrightarrow{z}_{2}$ of $\mathcal{H}_{\omega_{\ell}}$ satisfying
\begin{align*}
\mathcal{H}_{\omega_{\ell}}\left(\overrightarrow{z}_{1}\right)=0,\,\mathcal{H}_{\omega_{\ell}}\left(\overrightarrow{z}_{2}\right)=i\overrightarrow{z}_{1},
\end{align*}
we can verify using the  equation \eqref{rpde} and Lemma \ref{interactt} the existence of a positive constant $\beta_{0}$ satisfying for all $t\geq 0$ the following estimates
\begin{align*}
 \frac{d}{dt}\left \langle r(t),\sigma_{3}e^{i\sigma_{3}\theta_{\ell}(t,x)}\overrightarrow{z}_{1}\left(x-v_{\ell}t-y_{\ell}\right)\right\rangle= &O\left(\norm{f(t)}_{L^{2}}+\norm{r(t)}_{L^{2}}e^{{-}\beta_{0}(\min_{\ell} (v_{\ell}-v_{\ell+1}) t+\min_{\ell}(y_{\ell}-y_{\ell+1}))}\right),\\
 \frac{d}{dt}\left\langle r(t),\sigma_{3}e^{i\sigma_{3}\theta_{\ell}(t,x)}\overrightarrow{z}_{2}\left(x-v_{\ell}t-y_{\ell}\right)\right\rangle= &\left\langle r(t),\sigma_{3}e^{i\sigma_{3}\theta_{\ell}(t,x)}v_{1}\left(x-v_{\ell}t-y_{\ell}\right)\right\rangle\\&{+}O\left(\norm{f(t)}_{L^{2}}+\norm{r(t)}_{L^{2}}e^{{-}\beta_{0}(\min_{\ell} (v_{\ell}-v_{\ell+1}) t+\min_{\ell}(y_{\ell}-y_{\ell+1}))}\right).
\end{align*}
Consequently, using $\beta_{1}=\min\{\beta,\beta_{0}\min_{\ell}(v_{\ell}-v_{\ell+1})\},$ we can verify from the Fundamental Theorem of Calculus that if there exists $r(t)\in H^{1}_{\beta_{1}}$ and $t\geq 0,$ then
\begin{align}\label{discretp1}
    \left \langle r(t),\sigma_{3}e^{i\sigma_{3}\theta_{\ell}(t,x)}\overrightarrow{z}_{1}\left(x-v_{\ell}t-y_{\ell}\right)\right\rangle=& O\left(e^{{-}\beta t}\frac{\norm{f(t)}_{L^{2}_{\beta}}}{\beta}+e^{{-}(\beta_{1}+\beta_{0})t-\min_{\ell}(y_{\ell}-y_{\ell+1})}\frac{\norm{r(t)}_{L^{2}_{\beta}}}{\beta}\right),\\ \label{discretp2}
    \left \langle r(t),\sigma_{3}e^{i\sigma_{3}\theta_{\ell}(t,x)}\overrightarrow{z}_{2}\left(x-v_{\ell}t-y_{\ell}\right)\right\rangle=&O\left(e^{{-}\beta t}\frac{\norm{f(t)}_{L^{2}_{\beta}}}{\beta^{2}}+e^{{-}(\beta_{1}+\beta_{0})t-\min_{\ell}(y_{\ell}-y_{\ell+1})}\frac{\norm{r(t)}_{L^{2}_{\beta}}}{\beta^{2}}\right).
\end{align}
\par Next, let $\overrightarrow{z_{\lambda,\ell}}\in L^{2}_{x}(\mathbb{R},\mathbb{C}^{2})$ satisfying 
\begin{equation*}
\mathcal{H}_{\omega_{\ell}}(\overrightarrow{z_{\lambda,\ell}})=i\lambda \overrightarrow{z_{\lambda,\ell}} \text{\, for a $i\lambda \in \sigma_{d}(\mathcal{H}_{\omega_{\ell}}).$}
\end{equation*}
In particular, from the hypothesis \eqref{decV}, it can be verified that if $\lambda\in\mathbb{R}$ or $i\lambda \in\mathbb{R}$ that $\overrightarrow{z_{\lambda,\ell}}(x)$ satisfies
\begin{equation*}
    \left\vert \frac{d^{n}}{dx^{n}}\overrightarrow{z_{\lambda,\ell}}(x) \right\vert=O\left(e^{{-}\beta_{\lambda}x}\right) \text{, for all $n\in\mathbb{N}\cup\{0\},$}
\end{equation*}
for an explanation see for example \cite{schlag2}. Furthermore, the function
\begin{equation*}
    e^{\lambda t}\overrightarrow{z_{\lambda,\ell}}(x)
\end{equation*}
is a solution of 
$
    i\partial_{t}\overrightarrow{u}-\mathcal{H}_{\omega_{\ell}}\overrightarrow{u}=0.
$
Consequently, we can verify similarly to \eqref{discretp1} that

\begin{multline}\label{generaldisc} 
    \left \langle r(t),\sigma_{3}e^{i\sigma_{3}\theta_{\ell}(t,x)}e^{\lambda t}\overrightarrow{z_{\lambda,\ell}}\left(x-v_{\ell}t-y_{\ell}\right)\right\rangle\\
    =O\left(e^{(\vert \Ree\lambda\vert {-}\beta) t}\frac{\norm{f(t)}_{L^{2}_{\beta}}}{\left\vert \vert \Ree\lambda\vert{-}\beta\right\vert}+e^{[\vert \Ree\lambda\vert {-}\beta-\beta_{1}]t-c\min_{\ell}(y_{\ell}-y_{\ell+1})}\frac{\norm{r(t)}_{L^{2}_{\beta_{1}}}}{\left\vert \vert \Ree\lambda\vert{-}\beta\right\vert}\right)
\end{multline}
\par In conclusion, since $\kerrr (i\mathcal{H}_{\omega})^{2}=\kerrr (i\mathcal{H}_{\omega})^{n}$ for any $n\geq 2$ and $\kerrr (i\mathcal{H}_{\omega}-\lambda_\omega\mathrm{Id})=\kerrr (i\mathcal{H}_{\omega}-\lambda_\omega\mathrm{Id})^{n}$  for any $n\geq 1$ for $\lambda\neq 0,$ if we integrate \eqref{energyesti1} from $0$ through ${+}\infty$ using estimates \eqref{discretp1}, \eqref{discretp2} and \eqref{generaldisc}, we can verify when $\min_{\ell}v_{\ell}-v_{\ell+1}>0$ is large enough that there exists a constant $C(\beta_{1})>1$ satisfying
\begin{equation*}
    \norm{r(t)}_{H^{1}_{\beta_{1}}}\leq C(\beta_{1}) \norm{f(t)}_{H^{1}_{\beta}},
\end{equation*}
 when $r\in H^{1}_{\beta_{1}}.$ 
\end{proof}

\subsection{Proof of Theorem \ref{tdis}: Solutions converging to the bound states}\label{soldisc}
First, we recall $P_{d,\omega_{k}}$  is the  discrete spectrum projection of $\mathcal{H}_{\omega_{k}}.$ Let $u_{1},\,u_{2}$ be two different elements of $\kerrr \mathcal{H}_{\omega_{k}}^{2}$ satisfying
\begin{equation*}
    \mathcal{H}_{\omega_{k}}(u_{1})=0,\, \mathcal{H}_{\omega_{k}}(u_{2})=u_{1}.
\end{equation*}
Furthermore, because of the assumptions \eqref{decV} and $(H3),$ it is well-known that $u_{1}(x)$ and $u_{2}(x)$ have exponential decay on $\vert x\vert.$
\par Next, we are going to construct the following distinct solutions of \eqref{ldpe} for $t\geq 0$ 
\begin{align}\label{phix}
\nu_{1}(t,x)=&e^{i\sigma_{3}\left(\frac{v_{k}x}{2}-\frac{v_{k}^{2}t}{4}+\omega_{k}t\right)}u_{1}(x-v_{k}t-y_{k})+\overrightarrow{r}_{1}(t,x),\\ \label{iphi}
\nu_{2}(t,x)=&e^{i\sigma_{3}\left(\frac{v_{k}x}{2}-\frac{v_{k}^{2}t}{4}+\omega_{k}t\right)}u_{2}(x-v_{k}t-y_{k})-ite^{i\sigma_{3}\left(\frac{v_{k}x}{2}-\frac{v_{k}^{2}t}{4}+\omega_{k}t\right)}u_{1}(x-v_{k}t-y_{k})+\overrightarrow{r}_{2}(t,x),
\end{align}
where all the remainders $\vec{r}_{1},\, \vec{r}_{2}$ have the following decay for all $t\geq 0$
\begin{equation*}
   \max_{\ell\in\{1,2\}} \norm{\vec{r}_{\ell}(t)}_{H^{1}_{x}(\mathbb{R})}=O(e^{{-}\beta t}) \text{, for a $\beta>0.$}
\end{equation*}
In particular, it is well-known that the functions
\begin{align*}
e^{i\sigma_{3}\left(\frac{v_{k}x}{2}-\frac{v_{k}^{2}t}{4}+\omega_{k}t\right)}u_{1}(x-v_{k}t-y_{k}),\\ 
e^{i\sigma_{3}\left(\frac{v_{k}x}{2}-\frac{v_{k}^{2}t}{4}+\omega_{k}t\right)}u_{2}(x-v_{k}t-y_{k})-ite^{i\sigma_{3}\left(\frac{v_{k}x}{2}-\frac{v_{k}^{2}t}{4}+\omega_{k}t\right)}u_{1}(x-v_{k}t-y_{k}),
\end{align*}
 are solutions of the following equation with a single moving potential
 \begin{equation*}
    i\partial_{t}\overrightarrow{u}(t)+\sigma_{3}\partial^{2}_{x}\overrightarrow{u}(t)-V^{\sigma_{k}}_{k}\overrightarrow{u}(t)=0,   
 \end{equation*}
via Galilei transformations,   see \cite{perelmanasym} for example.
  \par Therefore, since all the functions $u_{1}$ and $u_{2}$  have exponential decay on $\vert x\vert,$ it is not difficult to verify using Lemma \ref{interactt} and the definition of the operators $V^{\sigma_{\ell}}_{\ell}$ that if $\nu_{j}$ is a   solution of \eqref{ldpe}, then $r_{j}$ is a   solution of a partial differential equation of the form
  \begin{equation*}
    i\partial_{t}r_{j}(t)+\sigma_{3}\partial^{2}_{x}r_{j}(t)-\sum_{\ell=1}^{m}V^{\sigma_{\ell}}_{\ell}r_{j}(t)=f_{j}(t,x), 
  \end{equation*}
such that each function $f_{j}$ has the following decay for constants $\beta>0,\,K>1$ both independent on the parameters $v_{\ell},\,y_{\ell}$ for all $1\leq \ell\leq m$ 
\begin{equation*}
    \norm{f_{j}(t,x)}_{H^{1}_{x}(\mathbb{R})}\leq K e^{{-}\beta(\min_{j}y_{j}-y_{j+1}+(v_{j}-v_{j+1})t)}.
\end{equation*}
\par More precisely, using Lemma \ref{interactt},
we can deduce the existence of a positive constant $\alpha>0$ depending only on the potentials $V_{\ell}(x)$ and on the parameters $\omega_{\ell}$ satisfying the following decay 
\begin{equation}
     \norm{f_{j}(t,x)}_{H^{1}_{x}(\mathbb{R})}\leq K e^{{-}\alpha(\min_{j}y_{j}-y_{j+1}+(v_{j}-v_{j+1})t)},
\end{equation}
when $t\geq 0.$
\par Consequently, Lemma \ref{edl} implies the existence of a unique function $r(t,x)$ satisfying Theorem \ref{tdis} when $\nu\in \kerrr \mathcal{H}_{\omega_{\ell}}^{2}.$
\par The proof of Theorem \ref{tdis} for the existence of a solution of the form
\begin{equation}\label{unstasol}
\overrightarrow{u}(t,x)=e^{i\sigma_{3}\left(\frac{v_{\ell}x}{2}-v_{\ell}^{2}t+\omega_{\ell}t\right)}e^{\lambda t} \overrightarrow{v_{\lambda,\ell}}(x-v_{\ell}t-y_{\ell})+r(t,x),    
\end{equation}
for $i\lambda\neq 0\in \sigma_{d}(\mathcal{H}_{\omega_{\ell}}),\,\mathcal{H}_{\omega_{\ell}}\overrightarrow{v_{\lambda,\ell}}(x)=i\lambda\overrightarrow{v_{\lambda,\ell}}(x)$ and $\norm{r(t)}_{H^{1}_{\beta}}<{+}\infty$ is completely similar. Moreover, since $\vert e^{\lambda t}\vert\leq e^{\vert \lambda \vert t}$ and $\overrightarrow{v_{\lambda,\ell}}$ is a Schwartz function with exponential decay, we can deduce using Lemma \ref{interactt} that if
$\min_{\ell} v_{\ell}-v_{\ell+1}\gg \vert \lambda\vert$ and $\overrightarrow{u}$ is a solution of \eqref{p}, then $r(t,x)$ satisfies a linear partial differential equation of the following kind
\begin{equation*}
    i\partial_{t}r(t)+\sigma_{3}\partial^{2}_{x}r(t)-\sum_{\ell=1}^{m}V^{\sigma_{\ell}}_{\ell}r(t)=f(t,x), 
\end{equation*}
such that $\norm{f(t,x)}_{H^{1}_{x}(\mathbb{R})}=O\left(e^{-{\alpha}[(\min_{\ell}v_{\ell}-v_{\ell+1}-\vert\lambda\vert)t+\min_{\ell}(y_{\ell}-y_{\ell+1})]}\right).$ Therefore, Lemma \ref{edl} implies the existence of a solution of the form \eqref{unstasol} satisfying Theorem \ref{tdis}.

\subsection{Proof of Theorem \ref{tcont}: Solutions in the scattering space}\label{solscont}

First, we recall the function $\mathcal{S}(\vec{\phi})(t,x)$ denoted by
{\footnotesize\begin{align*}
   \mathcal{S}(\vec{\phi})(t,x)= & \sum_{\ell=1}^{m}e^{i(\frac{v_{\ell}x}{2}-\frac{v_{\ell}^{2}t}{4}+\gamma_{\ell}+\omega_{\ell}t)\sigma_{3}}\hat{G}_{\omega_{\ell}}\left(
   e^{{-}it(k+\frac{v_{\ell}}{2})^{2}\sigma_{3}}e^{{-}i(\gamma_{\ell}+\omega_{\ell}t)\sigma_{3}} \begin{bmatrix}
       e^{i(y_{\ell}+v_{\ell}t)k}\phi_{1,\ell}\left(k+\frac{v_{\ell}}{2}\right)\\
       e^{i(y_{\ell}+v_{\ell}t)k}\phi_{2,\ell}\left(k-\frac{v_{\ell}}{2}\right)
    \end{bmatrix}\right)(x-y_{\ell}-v_{\ell}t)\\
    & {-}\frac{1}{\sqrt{2 \pi}}\int_{\mathbb{R}}e^{{-}it(k^{2}+\omega_{\ell})\sigma_{3}}
    \begin{bmatrix}
       \varphi_{1}(k)\\
       \varphi_{2}(k)
    \end{bmatrix} e^{ikx}\,dk,
\end{align*}}
where all the functions $\phi_{\ell}$ and $\varphi$ satisfy Definition \ref{s0def}. Moreover, for $V^{\sigma_{\ell}}_{\ell}$, see  \ref{H2}, each term 
\begin{equation*}
    A_{\ell}(t,x)=e^{i(\frac{v_{\ell}x}{2}-\frac{v_{\ell}^{2}t}{4}+\gamma_{\ell}+\omega_{\ell}t)\sigma_{3}}\hat{G}_{\omega_{\ell}}\left(
   e^{{-}it(k^{2}+\omega_{\ell})\sigma_{3}}e^{{-}i(\gamma_{\ell}+\omega_{\ell}t)\sigma_{3}} \begin{bmatrix}
       e^{iy_{\ell}k}\phi_{1,\ell}\left(k+\frac{v_{\ell}}{2}\right)\\
       e^{iy_{\ell}k}\phi_{2,\ell}\left(k-\frac{v_{\ell}}{2}\right)
    \end{bmatrix}\right)(x-y_{\ell}-v_{\ell}t)
\end{equation*}
is a solution of the following equation
\begin{equation}\label{pdedp}
   i\partial_{t}u+\sigma_{3}\partial^{2}_{x}u-V^{\sigma_{\ell}}_{\ell}(t)u=0. 
\end{equation}

\par Furthermore, using Lemma \ref{galil}, it is not difficult to verify for any $(g_{1},g_{2})\in L^{2}(\mathbb{R},\mathbb{C}^{2})$ that
\begin{equation*}
    \frac{e^{i\left(\frac{ v_{\ell}x}{2}-\frac{v_{\ell}^{2}t}{4}\right)\sigma_{3}}}{\sqrt{2\pi}}\int_{\mathbb{R}}e^{{-}it(k^{2}+\omega_{\ell})\sigma_{3}}e^{ik(x-y_{\ell}-v_{\ell}t)}\begin{bmatrix}
        e^{iy_{\ell}k}g_{1}\left(k+\frac{v_{\ell}}{2}\right)\\
        e^{iy_{\ell}k}g_{2}\left(k-\frac{v_{\ell}}{2}\right)
    \end{bmatrix}=\frac{1}{\sqrt{2\pi}}\int_{\mathbb{R}} e^{{-}it(k^{2}+\omega_{\ell})\sigma_{3}}e^{ikx}\begin{bmatrix}
    g_{1}(k)\\
    g_{2}(k)
    \end{bmatrix}\,dk.
\end{equation*}

\par Therefore, using Lemmas \ref{appFourier} and \ref{interactt},  and the definition of $(\varphi_{1},\varphi_{2})$ in Definition \ref{s0def}, we can verify for any $p,j \in\mathbb{N}$ the existence of constants $\beta_{0}>0,\,C_{p}>1$ satisfying the following estimate
\begin{multline}\label{intSS}
\norm{\langle x \rangle^p V^{n}_{\sigma_{n}}\left(\left[\sum_{\ell\neq n}A_{\ell}(t,x)\right]-\int_{\mathbb{R}}e^{{-}it(k^{2}+\omega_{\ell})\sigma_{3}}e^{ikx}
\begin{bmatrix}
\varphi_{1}(k)\\
\varphi_{2}(k)
\end{bmatrix}\,dk\right)}_{H^{j+1}_{x}}\\ 
\leq C_{p}\max_{\ell}\norm{(1+k^{2})^{\frac{j}{2}}(\phi_{1,\ell},\phi_{2,\ell})}_{L^{2}_{k}(\mathbb{R})} \max_{\ell}(y_{j}-y_{\ell}+(v_{\ell}-v_{\ell+1})t)^{p}e^{{-}\beta_{0} \left[\min_{n\neq l}\vert(y_{n}-y_{\ell})+(v_{n}-v_{\ell})t\vert\right]},
\end{multline}
when $\min_{l }y_{\ell}-y_{\ell+1}$ is large enough.
Furthermore, the estimate \eqref{intSS} implies that if $\mathcal{S}(\vec{\phi})(t,x)+r(t,x)$ is a   solution of \eqref{pdedp}, then the function $r(t)$ satisfies
\begin{equation}\label{finedpd}
    i\partial_{t}r+\sigma_{3}\partial^{2}_{x}r-\sum_{\ell}V^{\sigma_{\ell}}_{\ell}r=f(t),
\end{equation}
where $f(t,x)$ is a function satisfying
\begin{align*}
    \norm{f(t,x)}_{H^{1}_{x}(\mathbb{R})}\leq & C \max_{\ell}\norm{(\phi_{1,\ell},\phi_{2,\ell})}_{L^{2}_{k}(\mathbb{R})}e^{{-}\beta \left[\min_{\ell}(y_{\ell-1}-y_{\ell})+(v_{\ell-1}-v_{\ell})t\right]},\\
\norm{f(t,x)}_{H^{3}_{x}(\mathbb{R})}\leq & Ce^{{-}\beta \left[\min_{\ell}(y_{\ell-1}-y_{\ell})+(v_{\ell-1}-v_{\ell})t\right]}\max_{\ell}\norm{(1+k^{2})(\phi_{1,\ell},\phi_{2,\ell})}_{L^{2}_{k}(\mathbb{R})},\\
\norm{\langle x\rangle^2 f(t,x)}_{H^{1}_{x}(\mathbb{R})}\leq & C(y_{1}-y_{m}+(v_{1}-v_{m})t)^{2}e^{{-}\beta \left[\min_{\ell}(y_{\ell-1}-y_{\ell})+(v_{\ell-1}-v_{\ell})t\right]}\max_{\ell}\norm{(\phi_{1,\ell},\phi_{2,\ell})}_{L^{2}_{k}(\mathbb{R})},\\
\norm{\langle x\rangle
f(t,x)}_{H^{2}_{x}(\mathbb{R})}\leq & C(y_{1}-y_{m}+(v_{1}-v_{m})t)e^{{-}\beta \left[\min_{\ell}(y_{\ell-1}-y_{\ell})+(v_{\ell-1}-v_{\ell})t\right]}\max_{\ell}\norm{\left\langle k \right\rangle(\phi_{1,\ell},\phi_{2,\ell})}_{L^{2}_{k}(\mathbb{R})}.
\end{align*}
\par Consequently, using Lemma \ref{edl} in the partial differential equation \eqref{finedpd} and Theorem \ref{TT}, we deduce Theorem \ref{tcont}. The proof of estimate \eqref{tcontest} follows from inequality \eqref{intSS} and Remark \ref{weightedremainderedecay}.


\appendix

\section{Proof of Lemma \ref{tcopp}.}\label{app}
First, for any function $h=(h_{1},\,...,\,h_{2m-2})\in L^{2}_{k}\left(\mathbb{R},\mathbb{C}^{2m-2}\right),$ we consider the following linear system of equations for any $3\leq \ell\leq m-1$
\begin{align}\label{giant linear system 11}
    h_{1}=& g_{1}\left({-}k-\frac{v_{2}}{2}\right)-r_{2}(k)g_{2}\left(k-\frac{v_{2}}{2}\right)-s_{2}(k)g_{3}\left({-}k-\frac{v_{2}}{2}\right),\\  \nonumber
h_{2}=&g_{2}\left({-}k-\frac{v_{2}}{2}\right)-r_{1}\left({-}k-\frac{v_{2}}{2}+\frac{v_{1}}{2}\right)g_{1}\left(k+\frac{v_{2}}{2}-v_{1}\right),\\ \label{firstofall1}
h_{3}=&g_{3}\left({-}k-\frac{v_{3}}{2}\right)-r_{3}\left(k\right)g_{4}\left(k-\frac{v_{3}}{2}\right)-s_{3}(k)g_{5}\left({-}k-\frac{v_{3}}{2}\right)
,\\ \label{giant ind}
h_{4}=& g_{4}\left({-}k-\frac{v_{3}}{2}\right)-s_{2}\left({-}k-\frac{v_{3}}{2}+\frac{v_{2}}{2}\right)g_{2}\left({-}k-\frac{v_{3}}{2}\right)
\\ \nonumber
&{-}r_{2}\left({-}k+\frac{v_{2}}{2}-\frac{v_{3}}{2}\right)g_{3}\left(k+\frac{v_{3}}{2}-v_{2}\right),\, ...,\\
\label{2l-1}
h_{2\ell-1}=& g_{2\ell-1}\left({-}k-\frac{v_{\ell+1}}{2}\right)-r_{\ell+1}\left(k\right)g_{2\ell}\left(k-\frac{v_{\ell+1}}{2}\right)-s_{\ell+1}(k)g_{2\ell+1}\left({-}k-\frac{v_{\ell+1}}{2}\right), 
\\  \nonumber
h_{2\ell}=&g_{2\ell}\left({-}k-\frac{v_{\ell+1}}{2}\right)-s_{\ell}\left({-}k-\frac{v_{\ell+1}}{2}+\frac{v_{\ell}}{2}\right)g_{2\ell-2}\left({-}k-\frac{v_{\ell+1}}{2}\right)\\ \nonumber
&{-}r_{\ell}\left({-}k+\frac{v_{\ell}}{2}-\frac{v_{\ell+1}}{2}\right)g_{2\ell-1}\left(k+\frac{v_{\ell+1}}{2}-v_{\ell}\right),\,...,\\   \nonumber 
h_{2m-5}=& g_{2m-5}\left({-}k-\frac{v_{m-1}}{2}\right)-r_{m-1}(k)g_{2m-3}\left(k-\frac{v_{m-1}}{2}\right)\\ \nonumber &{-}s_{m-1}(k)g_{m-1,-,1}\left({-}k-\frac{v_{m-1}}{2}\right),\\  \nonumber
h_{2m-4}=&g_{2m-4}\left({-}k-\frac{v_{m-1}}{2}\right)-s_{m-2}\left({-}k-\frac{v_{m-1}}{2}+\frac{v_{m-2}}{2}\right)g_{2m-6}\left({-}k-\frac{v_{m-1}}{2}\right)\\
\nonumber
&{-}r_{m-2}\left({-}k+\frac{v_{m-2}-v_{m-1}}{2}\right)g_{2m-5}\left(k+\frac{v_{m-1}}{2}-v_{m-2}\right),\\ \label{2m-3}
h_{2m-3}=&g_{2m-3}\left({-}k-\frac{v_{m}}{2}\right)-r_{m}(k)g_{2m-2}\left(k-\frac{v_{m}}{2}\right),\\ \nonumber
h_{2m-2}=& g_{2m-2}\left({-}k-\frac{v_{m}}{2}\right)-s_{m-1}\left({-}k-\frac{v_{m}}{2}+\frac{v_{m-1}}{2}\right)g_{2m-4}\left({-}k-\frac{v_{m}}{2}\right)\\ \nonumber &{-}r_{m-1}\left({-}k+\frac{v_{m-1}}{2}-\frac{v_{m}}{2}\right)g_{2m-3}\left(k+\frac{v_{m}}{2}-v_{m-1}\right),
\end{align}
such that all the functions $r_{n},\,s_{n}$ satisfy \eqref{asyreftr}. We are going to prove by induction on $m$ that there exist $K_{m},\,c_{m}>0$ such that if $\min_{\ell} v_{\ell}-v_{\ell-1}>K_{m},$ then there exists a unique solution $g=(g_{1},g_{2},\,...,\,g_{2m-2})\in L^{2}(\mathbb{R},\mathbb{C} ^{2m-2})$ of the linear system above for any $h\in L^{2}(\mathbb{R},\mathbb{C} ^{2m-2}),$ and $g$ satisfies
\begin{equation}\label{maininequal}
    c_{m}\norm{g(k)}_{L^{2}_{k}(\mathbb{R},\mathbb{C}^{2m-2})}\leq \norm{h}_{L^{2}_{k}(\mathbb{R},\mathbb{C}^{2m-2})}.
\end{equation}
\par Moreover, we are going to prove the estimate \eqref{maininequal} in the next steps after we compute the following sum for each $\ell$
\begin{equation}\label{sl}
    S_{\ell}(k)=h_{2}(k)+\sum_{j=1}^{\ell}r_{2,j}\left({-}k-\frac{v_{2}}{2}+\frac{v_{1}}{2}\right)h_{2j-1}\left({-}k+v_{1}-\frac{v_{2}}{2}-\frac{v_{j+1}}{2}\right),
\end{equation}
such that the functions $r_{2,j}\in L^{\infty}$ satisfy \eqref{asyreftr} are going to be defined later.\\ 
\textbf{Step 1. (Case $m=3.$)} When $m=3,$ the linear system above consists of
\begin{align}\label{po1}
h_{1}(k)=&g_{1}\left({-}k-\frac{v_{2}}{2}\right)-r_{2}(k)g_{2}\left(k-\frac{v_{2}}{2}\right)-s_{2}(k)g_{3}\left({-}k-\frac{v_{2}}{2}\right),\\ 
h_{2}(k)=&g_{2}\left({-}k-\frac{v_{2}}{2}\right)-r_{1}\left({-}k-\frac{v_{2}}{2}+\frac{v_{1}}{2}\right)g_{1}\left(k+\frac{v_{2}}{2}-v_{1}\right),\\
h_{3}(k)=&g_{3}\left({-}k-\frac{v_{3}}{2}\right)-r_{3}(k)g_{4}\left(k-\frac{v_{3}}{2}\right),\\ \label{po2}
h_{4}(k)=&g_{4}\left({-}k-\frac{v_{3}}{2}\right)-s_{2}\left({-}k-\frac{v_{3}}{2}+\frac{v_{2}}{2}\right)g_{2}\left({-}k-\frac{v_{3}}{2}\right)\\ \nonumber
&{-}r_{2}\left({-}k+\frac{v_{2}}{2}-\frac{v_{3}}{2}\right)g_{3}\left(k+\frac{v_{3}}{2}-v_{2}\right).
\end{align}
Therefore, if $\min_{\ell} v_{\ell}-v_{\ell+1}>0$ is large enough, we can verify using the estimates \eqref{pep} that
\begin{multline}\label{h11}
h_{2}\left(k\right)+r_{1}\left({-}k-\frac{v_{2}}{2}+\frac{v_{1}}{2}\right)h_{1}\left({-}k-v_{2}+v_{1}\right)\\
=g_{2}\left({-}k-\frac{v_{2}}{2}\right)-r_{1,2}\left({-}k-\frac{v_{2}}{2}+\frac{v_{1}}{2}\right)g_{3}\left(k+\frac{v_{2}}{2}-v_{1}\right)+O\left(\min_{j\in\{1,2,3,4\},\,\ell\in\{1,2\}}\frac{\norm{g_{j}}_{L^{2}_{k}(\mathbb{R},\mathbb{C})}}{v_{\ell}-v_{\ell+1}}\right),
\end{multline}
for the function
\begin{equation*}
r_{1,2}\left({-}k-\frac{v_{2}}{2}+\frac{v_{1}}{2}\right)=r_{1}\left({-}k-\frac{v_{2}}{2}+\frac{v_{1}}{2}\right)s_{2}(k), \end{equation*}
and so the function $r_{1,2}\in L^{\infty}_{k}(\mathbb{R},\mathbb{C})$ satisfies \eqref{asyreftr}.
\par Similarly, we can verify that
\begin{multline}\label{h22}
h_{3}(k)+r_{3}(k)h_{4}\left({-}k\right)\\
=g_{3}\left({-}k-\frac{v_{3}}{2}\right)-r_{3,4}\left({-}k\right)g_{2}\left({-}k+\frac{v_{3}}{2}-v_{2}\right)+O\left(\min_{j\in\{1,2,3,4\},\,l\in\{1,2\}}\frac{\norm{g_{j}}_{L^{2}_{k}(\mathbb{R},\mathbb{C})}}{v_{\ell}-v_{\ell+1}}\right),
\end{multline}
such that the function $r_{3,4}\in L^{\infty}_{k}(\mathbb{R},\mathbb{C})$ satisfies  \eqref{asyreftr}.
\par Consequently, using estimates \eqref{h11}, \eqref{h22} and all the equations from \eqref{po1} through \eqref{po2},
we can verify similarly to the proof of Lemma \ref{Tp} the existence of a unique solutions $g_{1},\,g_{2},\,g_{3}$ and $g_{4}$ of $L^{2}(\mathbb{R},\mathbb{C}),$ and the existence of a constant $c>0$ such that if $\min_{\ell} v_{\ell}-v_{\ell+1}>0$ is sufficiently large, then
\begin{equation*}
    c\max_{j\in\{1,2,3,4\}}\norm{g_{j}}_{L^{2}_{k}(\mathbb{R},\mathbb{C})}\leq \norm{(h_{1},h_{2},h_{3},h_{4})}_{L^{2}_{k}(\mathbb{R},\mathbb{C}^{4})},
\end{equation*}
for any $\overrightarrow{h}\in L^{2}_{k}(\mathbb{R},\mathbb{C}^{4}).$
\par In conclusion, Lemma \ref{tcopp} is true when $m=3.$
\\
\textbf{Step 2. (Estimate of $S_{\ell}.$)}
 First, we are going to verify for any $\ell\leq m-2$ that there exist functions $r_{2,1},\, r_{2,2},\, ...,\,r_{2,m-1 }$ satisfying \eqref{pep} such that
\begin{multline}\label{slforr}
\begin{aligned}
S_{\ell}(k)=&g_{2}\left({-}k-\frac{v_{2}}{2}\right)-r_{2,\ell+1}\left({-}k-\frac{v_{2}}{2}+\frac{v_{1}}{2}\right)g_{2\ell+1}\left(k-v_{1}+\frac{v_{2}}{2}\right){+}O\left(\min_{n,m}\frac{\norm{g_{n}}_{L^{2}}}{v_{m}-v_{m+1}}\right),
\end{aligned}
\end{multline}
for any $\ell\leq m-2,$
where $S_{\ell}$ is defined at \eqref{sl}.
\par When $\ell=1,$ we consider
$
    r_{2,1}(k)=r_{1}\left(k\right),
$
and so, it is not difficult to verify using equation \eqref{giant linear system 11} and estimates \eqref{pep} that
\begin{multline}\label{s1kk}
 \begin{aligned}
    S_{1}(k)=&h_{2}\left(k\right)+r_{1}\left({-}k-\frac{v_{2}}{2}+\frac{v_{1}}{2}\right)h_{1}\left({-}k-v_{2}+v_{1}\right)
\\=&g_{2}\left({-}k-\frac{v_{2}}{2}\right){-}r_{2,2}\left({-}k-\frac{v_{2}}{2}+\frac{v_{1}}{2}\right)g_{3}\left(k+\frac{v_{2}}{2}-v_{1}\right){+}O\left(\min_{\ell,m}\frac{\norm{g_{\ell,\pm,1}}_{L^{2}}}{v_{m}-v_{m+1}}\right),
\end{aligned}
\end{multline}
where $r_{2,2}$ is defined for any $k\in\mathbb{R}$ by 
\begin{equation*}
r_{2,2}\left({-}k-\frac{v_{2}}{2}+\frac{v_{1}}{2}\right)=r_{1}\left({-}k-\frac{v_{2}}{2}+\frac{v_{1}}{2}\right)s_{2}\left({-}k+v_{1}-v_{2}\right),    
\end{equation*}
which clearly satisfies \eqref{asyreftr} because $r_{1}$ also satisfies \eqref{asyreftr}, and the decay of the remainder of
\eqref{s1kk} follows from \eqref{pep}.
Therefore, the estimate \eqref{slforr} is true for $\ell=1.$
\par Next, if \eqref{slforr} is true until $\ell=\ell_{0}$ for a $\ell_{0}<m-2,$ then
considering
\begin{multline*}
 \begin{aligned}
 S_{l_{0}+1}(k)=&S_{l_{0}}(k)+r_{2,\ell_{0}+1}\left({-}k-\frac{v_{2}}{2}+\frac{v_{1}}{2}\right)h_{2\ell_{0}+1}\left({-}k+v_{1}-\frac{v_{2}}{2}-\frac{v_{\ell_{0}+2}}{2}\right), 
\end{aligned}
\end{multline*}
it is not difficult to verify from \eqref{2l-1} and estimates \eqref{pep} that
\begin{multline*}
 \begin{aligned}
    S_{l_{0}+1}(k)=& g_{2}\left({-}k-\frac{v_{2}}{2}\right)-r_{2,\ell_{0}+2}\left({-}k-\frac{v_{2}}{2}+\frac{v_{1}}{2}\right)g_{2\ell_{0}+3}\left(k-v_{1}+\frac{v_{2}}{2}\right){+}O\left(\min_{n,m}\frac{\norm{g_{n}}_{L^{2}}}{v_{m}-v_{m+1}}\right) ,
 \end{aligned}   
\end{multline*}
for the function $r_{2,\ell_{0}+2}$ defined by
\begin{equation*}
r_{2,\ell_{0}+2}\left({-}k-\frac{v_{2}}{2}+\frac{v_{1}}{2}\right)=r_{2,\ell_{0}+1}\left({-}k-\frac{v_{2}}{2}+\frac{v_{1}}{2}\right)s_{\ell_{0}+2}\left({-}k+v_{1}-\frac{v_{2}}{2}-\frac{v_{\ell_{0}+2}}{2}\right),
 \end{equation*}
which also implies that 
$r_{2,\ell_{0}+2}$ satisfies \eqref{asyreftr} assuming that $r_{2,\ell_{0}+1}$ satisfies \eqref{asyreftr}. Therefore, \eqref{slforr} is true for any $\ell\leq m-2.$
\par Consequently, using \eqref{slforr} for $\ell=m-2,$ the fact that all the functions $r_{2,\ell}$ satisfy \eqref{asyreftr}, \eqref{pep}, and the definition of $h_{2m-3}$ at \eqref{2m-3}, we conclude that
\begin{equation}\label{sm-1}
  S_{m-1}(k)=g_{2}\left({-}k-\frac{v_{2}}{2}\right)+O\left(\min_{n,m}\frac{\norm{g_{n}}_{L^{2}}}{v_{m}-v_{m+1}}\right).
\end{equation}\\
\textbf{Step 3. (Conclusion of the proof of Lemma \ref{tcopp}.)} Next, we consider the following function
\begin{equation*}
    \mathcal{L}(h_{4},g_{2})(k)=h_{4}(k)+s_{2}\left({-}k-\frac{v_{3}}{2}+\frac{v_{2}}{2}\right)g_{2}\left({-}k-\frac{v_{3}}{2}\right).
\end{equation*}
In particular, the equation \eqref{giant ind} is equivalent to
\begin{equation}\label{l44}
    \mathcal{L}(h_{4},g_{2})(k)=g_{4}\left({-}k-\frac{v_{3}}{2}\right){-}r_{2}\left({-}k+\frac{v_{2}}{2}-\frac{v_{3}}{2}\right)g_{3}\left(k+\frac{v_{3}}{2}-v_{2}\right).
\end{equation}
\par Therefore, using \eqref{l44},  \eqref{firstofall1} and all the linear equations below \eqref{giant ind} of the linear system at the beginning of this section, assuming that Lemma \ref{tcopp} is true for $m-1,$ we can verify by induction for any $g_{2}\in L^{2}_{k}(\mathbb{R})$ that there exist unique functions 
\begin{equation*}
    g_{3}(k),\,g_{4}(k),\,g_{5}(k),\,g_{6}(k),\, ...,\,g_{2m-1}(k)\in L^{2}(\mathbb{R},\mathbb{C})
\end{equation*}
satisfying the linear equations \eqref{l44}, and all equations below \eqref{firstofall1} of the linear system starting at \eqref{giant linear system 11}.
Moreover, assuming that Lemma \ref{tcopp} is true for $m-1,$ we have that
\begin{align}\label{ffinal}
    \norm{g_{3}(k)}_{L^{2}_{k}(\mathbb{R})}+\norm{g_{4}}_{L^{2}_{k}(\mathbb{R})}+\max_{\ell\geq 5}\norm{g_{\ell}(k)}_{L^{2}_{k}(\mathbb{R})}\leq K\left[\norm{h(k)}_{L^{2}_{k}(\mathbb{R},\mathbb{C}^{2m-2})}+\norm{g_{2}(k)}_{L^{2}_{k}(\mathbb{R})}\right],
\end{align}
when $\min_{\ell} v_{\ell}-v_{\ell+1}>0$ is sufficiently large.
\par In conclusion, using the estimates \eqref{sl}, \eqref{sm-1} and \eqref{ffinal}, it is not difficult to verify that if the number $\min_{\ell} v_{\ell}-v_{\ell+1}>0$ is sufficiently large, then there exists a constant $K>1$ satisfying
\begin{align*}
    \max_{\ell\geq 1}\norm{g_{\ell}(k)}_{L^{2}_{k}(\mathbb{R})}\leq  K\norm{f(k)}_{L^{2}_{k}(\mathbb{R})}.
\end{align*}

\section{$\mathcal{F}L^{1}$ Estimates}\label{sec:appb}
First, we consider the following proposition.
\begin{lemma}\label{lB1}
Let $r_{1},\,r_{2}:\mathbb{R}\to\mathbb{C}$ be smooth functions that satisfy the decay estimate
\begin{equation*}
   \max_{l\in\{1,2\}} \left\vert\frac{d^{n}r_{\ell}(k)}{dk^{n}}\right\vert=O\left(\frac{1}{\left(1+\vert k\vert\right)^{n}}\right) \text{ for all $n\in\mathbb{N}.$}
\end{equation*}
There exists $C_{1}>1$ such that if $v_{1}-v_{2}>1$ and $y\in\mathbb{R},$ then
\begin{equation*}
    \norm{F_{0}(x)\left[r_{1}(k)r_{2}(k+v_{1}-v_{2})e^{iky}\psi(k)\right]}_{L^{1}_{x}(\mathbb{R})}\leq \frac{C}{1+v_{1}-v_{2}}\norm{F_{0}(x)\left[\psi(k)\right]}_{L^{1}_{x}(\mathbb{R})}.
\end{equation*}
\end{lemma}
\begin{proof}[Proof of Lemma \ref{lB1}.]
From Young Inequality, the assumption on the decay of $r_{1},$ and $r_{2}$ and the inequality $\norm{g}_{L^{1}_{k}(\mathbb{R})}\leq C \norm{(1+\vert k\vert)g(k)}_{L^{2}_{k}(\mathbb{R})}$ for a $C>1,$ we have for any $\ell\in\{0,1\}$ that
\begin{align*}
    \norm{x^{\ell}F_{0}(x)\left[r_{1}(k)r_{2}(k+v_{1}-v_{2})e^{iky}\psi(k)\right]}_{L^{1}_{x}(\mathbb{R})}\leq &\norm{F_{0}(x)\left[r_{1}(k)r_{2}(k+v_{1}-v_{2})\right]}\norm{F_{0}(x)\left[\psi(k)\right]}_{L^{1}_{x}(\mathbb{R})}\\
    \leq & C \norm{r_{1}(k)r_{2}(k+v_{1}-v_{2})}_{H^{1}_{k}(\mathbb{R})}\norm{F_{0}(x)\left[\psi(k)\right]}_{L^{1}_{x}(\mathbb{R})} \\
    \leq & \frac{C_{1}}{1+(v_{1}-v_{2})}\norm{F_{0}(x)\left[\psi(k)\right]}_{L^{1}_{x}(\mathbb{R})}, 
\end{align*}
and the last inequality above comes from the elementary estimate
\begin{equation*}
    \int_{\mathbb{R}}\frac{1}{\left(1+\vert k\vert\right)^{2}\left(1+\vert k-(v_{1}-v_{2})\vert^{2}\right)}\,dk\leq \frac{2}{1+\frac{( v_{1}-v_{2})^{2}}{4}}\int_{\mathbb{R}}\frac{1}{1+ k^{2}}\,dk.
\end{equation*}
\end{proof}
\begin{lemma}\label{LB2}
 If $r_{1}:\mathbb{R}\to\mathbb{C}$ is a smooth function that satisfying the decay estimate
\begin{equation*}
    \left\vert\frac{d^{n}r_{1}(k)}{dk^{n}}\right\vert=O\left(\frac{1}{\left(1+\vert k\vert\right)^{n+1}}\right) \text{ for all $n\in\mathbb{N},$}
\end{equation*}   
 then there exists $C>1$ such that for all $y\in\mathbb{R}$
 \begin{align*}
   \norm{F_{0}\left[r_{1}(k)e^{iky}G^{*}_{\omega_{\ell}}(f)(k)\right](x)}_{L^{1}_{x}(\mathbb{R})}\leq &C\norm{f(x)}_{L^{1}_{x}(\mathbb{R})}.
 \end{align*}
\end{lemma}
\begin{proof}[Proof of Lemma \ref{LB2}.]
First, from Lemma \ref{fordecay}, we observe that the operator $G^{*}_{\omega_{\ell}} $ satisfy for a constant $c>0$ the following estimate
\begin{equation*}
    \norm{G^{*}_{\omega_{\ell}}(f(x))}_{FL^{1}_{k}(\mathbb{R})}\leq c \norm{f(x)}_{L^{1}_{x}(\mathbb{R})}. 
\end{equation*}
Therefore, using Lemma \ref{fordecay}, we can verify that there exists $C>1$ satisfying
\begin{align*}
   \norm{F_{0}(x)\left[r_{1}(k)e^{iky}G^{*}_{\omega_{\ell}}(f)(k)\right]}_{L^{1}_{x}(\mathbb{R})}\leq & \norm{e^{iy_{1}k}r_{1}(k)}_{FL^{1}}\norm{G^{*}_{\omega_{\ell}}(f)(k)}_{FL^{1}}\\
   \leq & C\norm{r(k)}_{H^{1}_{k}(\mathbb{R})}\norm{f(x)}_{L^{1}_{x}(\mathbb{R})}.
\end{align*}
\end{proof}
\begin{lemma}\label{LB3}
Let $r:\mathbb{R}^{2}\to\mathbb{C}$ be a smooth function satisfying for a $\gamma(k)>c_{0}>0$ and $K>1$ the following estimates
\begin{align*}
    \left\vert \partial^{\ell}_{k}r(x,k) \right\vert\leq K &\frac{(1+\vert x\vert)^{\ell}e^{{-}\gamma\vert x\vert}}{1+\vert k\vert} \text{, for any $\ell\in\{0,1\}.$}
\end{align*}
There exists $C>1$ satisfying for any $f\in L^{2}_{x}(\mathbb{R})$ and any $N>1$
\begin{align*}
  \norm{(1+\vert x\vert)F_{0}(x)\left[\int_{{-}\infty}^{{-}N}r(k,x)f(x)\,dx\right]}_{L^{1}_{x}(\mathbb{R})}\leq & C N^{\frac{1}{2}}e^{{-}c_{0}N}\min\left(\norm{f}_{L^{2}_{x}(\mathbb{R})},\norm{f}_{L^{1}_{x}(\mathbb{R})}\right),\\
   \norm{(1+\vert x\vert)F_{0}(x)\left[\int_{N}^{{+}\infty}r(k,x)f(x)\,dx\right]}_{L^{1}_{x}(\mathbb{R})}\leq & C N^{\frac{1}{2}}e^{{-}c_{0} N}\min\left(\norm{f}_{L^{2}_{x}(\mathbb{R})},\norm{f}_{L^{1}_{x}(\mathbb{R})}\right).
\end{align*}

\end{lemma}
\begin{proof}
It is enough to prove the first inequality, the proof of the other is analogous. This follow from the existence of $C_{0},\,C>0$ satisfying the following estimates
\begin{align*}
    \norm{(1+\vert x\vert)F_{0}(x)\left[\int_{{-}\infty}^{{-}N}r(k,x)f(x)\,dx\right]}_{L^{1}_{x}(\mathbb{R})}\leq & \norm{\frac{1}{(1+\vert x\vert)}}_{L^{2}_{x}(\mathbb{R})}\norm{\int_{{-}\infty}^{{-}N}r(k,x)f(x)\,dx}_{H^{2}_{k}(\mathbb{R})}\\
    \leq &  \norm{\frac{1}{(1+\vert x\vert)}}_{L^{2}_{x}(\mathbb{R})}\int_{{-}\infty}^{{-}N}\norm{r(k,x)}_{H^{2}_{k}(\mathbb{R})}\vert f(x)\vert\,dx\\
    \leq & C_{0}\int_{{-}\infty}^{{-}N}(1+\vert x\vert)^{2}e^{{-}\gamma\vert x\vert}\vert f(x)\vert\,dx\\
    \leq & CNe^{{-}\gamma N}\min\left(\norm{f}_{L^{2}_{x}(\mathbb{R})},\norm{f}_{L^{1}_{x}(\mathbb{R})}\right)
\end{align*}as desired.
\end{proof}    
\begin{lemma}\label{fordecay}
Assume that the operator $\mathcal{H}_{\omega}$ satisfies all the assumptions of  \S \ref{subsub:assumption} and that $\omega_,\,{-}\omega$ are not resonance points of $\mathcal{H}_{\omega}.$
Then, there exists a constant $C>1$ satisfying the following estimates for any $f\in L^{1}_{x}(\mathbb{R},\mathbb{C}^{2}).$
\begin{equation}\label{estpp1}
\norm{F^{*}_\omega(k)(f)}_{\mathcal{F}L_{1}}\leq C\norm{f}_{L^{1}_{x}(\mathbb{R})}.
\end{equation}
Furthermore, if in addition, $xf(x)\in L^{1}_{x}(\mathbb{R},\mathbb{C}^{2}),$ then
\begin{equation}\label{estpp2}
    \norm{\partial_{k}F^{*}_\omega(k)(f)}_{\mathcal{F}L_{1}}\leq C\left[\norm{(1+\vert x\vert)f(x)}_{L^{1}_{x}(\mathbb{R})}\right].
\end{equation}
\end{lemma}
\begin{proof}
First, using Lemma $5.2$ of \cite{collotger}, we can find smooth cut-off functions $\chi_{+},\,\chi_{-}$ satisfying
\begin{itemize}
    \item $\chi_{+}(x)\in [0,1],$ and $\chi_{-}(x)=\chi_{+}({-}x)=1-\chi_{+}(x)$ for any $x\in\mathbb{R},$
    \item $\chi_{+}(x)=0$ on $[{-}\infty,{-}\frac{1}{2}]$ and $\chi_{+}(x)=1$ on $[\frac{1}{2},{+}\infty],$ 
\end{itemize}
and the following identity holds 
\begin{equation}\label{Flocalident}
    \mathcal{F}_\omega(x,k)=\chi_{+}(x)s(k)e^{ikx}+\chi_{-}(x)\left[e^{ikx}+r(k)e^{{-}ikx}\right]+R^{+}(x,k)e^{ikx}+R^{-}(x,k)e^{{-}ikx},
\end{equation}
such that the functions $R^{+}$ and $R^{-}$ are smooth satisfying the following decay estimates for a constant $\alpha>0$ and any $\ell,\,n\in\mathbb{N}$
\begin{equation}\label{pseddecay}
    \max_{\pm}\left\vert \frac{\partial^{\ell+n}}{\partial k^{\ell}\partial x^{n}}R^{\pm}(x,k) \right\vert\leq C_{\ell,n}\frac{e^{{-}\alpha \vert x\vert }}{(1+\vert k\vert)^{\ell+1}},
\end{equation}
where the parameters $C_{\ell,n}>1$ depend only on $\ell$ and $n.$
\par Next, since $1-s(k)$ and $r(k)$ belong to $H^{1}_{k}(\mathbb{R}),$ we can verify using the Fourier inversion formula that there exists a constant $C>1$ satisfying
\begin{align}\label{decap1}
    \norm{r(k)\int_{\mathbb{R}}e^{{-}ikx}f(x)\chi_{-}(x)\,dx}_{\mathcal{F}L_{1}}\leq \norm{\left[\hat{r}({-}\Diamond)*(\chi_{-}(\Diamond)f(\Diamond))\right](x)}_{L^{1}_{x}(\mathbb{R})}\leq & \norm{\hat{r}(x)}_{L^{1}_{x}(\mathbb{R})}\norm{\chi_{-}(x)f(x)}_{L^{1}_{x}(\mathbb{R})}\\ \nonumber
    \leq & C\norm{r(k)}_{H^{1}_{k}(\mathbb{R})}\norm{f(x)}_{L^{1}_{x}(\mathbb{R})}.
\end{align}
Similarly, we can verify the existence of a constant $C>1$ satisfying
\begin{equation}\label{decap2}
    \norm{(1-s(k))\int_{\mathbb{R}}e^{ikx}\chi_{+}(x)f(x)\,dx}_{\mathcal{F}L_{1}}\leq C\norm{1-s(k)}_{H^{1}_{k}(\mathbb{R})}\norm{f(x)}_{L^{1}_{x}(\mathbb{R})}.
\end{equation}
\par Furthermore, using Theorem $5.1$ from  \cite{psed} about the $L^2$ norm of pseudo-differential operators and estimates \eqref{pseddecay}, we can verify that
\begin{align}\label{decap3}
\max_{\pm}\norm{\int_{\mathbb{R}}R^{\pm}(x,k)e^{ikx}f(x)\,dx}_{\mathcal{F}L_{1}}\leq & C  \max_{\pm} \norm{\int_{\mathbb{R}}R^{\pm}(x,k)e^{ikx}f(x)\,dx}_{H^{1}_{k}}.
\end{align}
In particular, using the decay estimates of \eqref{Flocalident}, it is not difficult to verify that
\begin{equation*}
    \max_{\ell\in\{0,1\},\pm}\left \vert \frac{d^{\ell}}{dk^{\ell}}\int_{\mathbb{R}}R^{\pm}(x,k )e^{ikx}f(x) \right\vert\leq C\frac{\int_{\mathbb{R}}(1+\vert x\vert)^{\ell}e^{{-}\alpha\vert x\vert}\vert f(x)\vert \,dx}{1+\vert k\vert  }.
\end{equation*}
Therefore, since $\frac{1}{1+\vert k\vert}$ is in $L^{2}_{k}(\mathbb{R}),$ we obtain from \eqref{decap3} that
\begin{equation*}
\max_{\pm}\norm{\int_{\mathbb{R}}R^{\pm}(x,k)e^{ikx}f(x)\,dx}_{\mathcal{F}L_{1}}\leq K \norm{f}_{L^{1}_{x}(\mathbb{R})}.
\end{equation*}
\par Consequently, using the Fourier inversion formula, identity \eqref{Flocalident}, and inequalities \eqref{decap1}, \eqref{decap2}, \eqref{decap3},  we deduce from the triangle inequality that
\begin{equation*}
    \norm{\int_{\mathbb{R}}\mathcal{F}_\omega(x,k)f(x)\,dx}_{\mathcal{F}L_{1}}\leq C\norm{f(x)}_{L^{1}_{x}(\mathbb{R})},
\end{equation*}
from which we can deduce \eqref{estpp1}.
\par The proof of \eqref{estpp2} is similar. In particular, we differentiate the right-hand side of equation  \eqref{Flocalident} with respect to $k$, and 
estimate
\begin{equation*}
    \norm{\int_{\mathbb{R}}\partial_{k}\mathcal{F}_\omega(x,k)f(x)\,dx}_{\mathcal{F}L^{1}},
\end{equation*}
using the Fourier inversion formula, the decays of the functions $1-s(k),\,r(k),$ and Theorem $5.1$ from  \cite{psed} or \cite{CP1} for symbols in terms of Jost functions.
\end{proof}

\section{Proof of Theorem \ref{interpolation est.}}\label{A}
\par First, for $\epsilon\in (0,1)$ small enough, let $\chi :\mathbb{R}\to\mathbb{R}$ be a smooth cut-off function satisfying $ \chi(x)\in [0,1],$ and
\begin{equation*}
    \chi(x)=
    \begin{cases}
      1 \text{, if $x\leq {-}2\epsilon,$}\\
      0 \text{, if $x\geq {-}\epsilon.$}
    \end{cases}
\end{equation*}
Next, given a function $f\in L^{2}(\mathbb{R},\mathbb{C}^{2}),$ we we consider the following system of equations below having  unknowns: the functions $f_{\ell,\pm}\in L^{2}(\mathbb{R},\mathbb{C}^{2})$ for $\ell\in\{1,2\,...,\,m\},$ an element $\overrightarrow{\phi}$ belonging to the domain of $\mathcal{S}(0),$ and a finite set of functions $\overrightarrow{v_{d_{\ell}}}\in\Raa P_{d,\omega_{\ell}}.$ 
 \begin{align}\label{lG}
& e^{i\frac{\sigma_{3}v_{\ell}x}{2}}\hat{G}_{\omega_{\ell}}\left(e^{iy_{\ell}k}\begin{bmatrix}
    \phi_{1,\ell}\left(k+\frac{v_{\ell}}{2}\right)\\
     \phi_{2,\ell}\left(k-\frac{v_{\ell}}{2}\right)
 \end{bmatrix}\right)(x-y_{\ell})  \\
 = & \left[1-\chi\left(x-\frac{y_{\ell}+y_{\ell+1}}{2}\right)-\chi\left({-}x+\frac{y_{\ell}+y_{\ell-1}}{2}\right)\right]f(x)+ e^{i\frac{\sigma_{3}v_{\ell}x}{2}}\overrightarrow{v_{d_{\ell}}}(x-y_{\ell})\\ \nonumber
 & {+}\chi\left(x-\frac{y_{\ell}+y_{\ell+1}}{2}\right)f_{\ell,-}(x)+\chi\left({-}x+\frac{y_{\ell}+y_{\ell-1}}{2}\right)f_{\ell,+}(x),\\ \label{1G}
 & e^{i\frac{\sigma_{3}v_{1}x}{2}}\hat{G}_{\omega_{1}}\left(e^{iy_{1}k}\begin{bmatrix}
     \phi_{1,1}\left(k+\frac{v_{1}}{2}\right)\\
     \phi_{2,1}\left(k-\frac{v_{1}}{2}\right)
 \end{bmatrix}\right)(x-y_{1})\\    
 =&  \left[1-\chi\left(x-\frac{y_{1}+y_{2}}{2}\right)\right]f(x)+\chi\left(x-\frac{y_{1}+y_{2}}{2}\right)f_{1,-}(x)+e^{i\frac{\sigma_{3}v_{1}x}{2}}\overrightarrow{v_{d_{1}}}(x-y_{1}),\\ \label{mG}
& e^{i\frac{\sigma_{3}v_{m}x}{2}}\hat{G}_{\omega_{m}}\left(e^{iy_{m}k}\begin{bmatrix}
     \phi_{1,m}\left(k+\frac{v_{m}}{2}\right)\\
     \phi_{2,m}\left(k-\frac{v_{m}}{2}\right)
 \end{bmatrix}\right)(x-y_{m})\\
 = & \left[1-\chi\left({-}x+  \frac{(y_{m}+y_{m-1})}{2}\right)\right]f(x)+\chi\left({-}x+  \frac{(y_{m}+y_{m-1})}{2}\right)f_{m,+}(x){+}e^{i\frac{\sigma_{3}v_{m}x}{2}}\overrightarrow{v_{d_{m}}}(x-y_{m}),
 \end{align}
for any $\ell\in\{2,..., m-1\}.$
\par In particular, repeating the argument used in the proof of Lemma \ref{dec}, we can apply the distorted Fourier transforms $F^{*}_{\omega_{\ell}}$ and $G^{*}_{\omega_{\ell}}$ to the equations \eqref{lG}, \eqref{1G} and \eqref{mG} to obtain an approximate linear system on the variables
\begin{align}\label{unknown}
\begin{bmatrix}
    g_{1,-,1}(k)\\
    g_{1,-,2}(k)
\end{bmatrix}=&\int_{{-}\infty}^{0}f_{1,-}\left(x+\frac{y_{1}+y_{2}}{2}\right)\chi(x)e^{ikx}\,dx,\\ \nonumber  \begin{bmatrix}
    g_{\ell,-,1}(k)\\
    g_{\ell,-,2}(k)
\end{bmatrix}=&\int_{{-}\infty}^{0}f_{\ell,-}\left(x+\frac{y_{\ell}+y_{\ell+1}}{2}\right)\chi(x)e^{ikx}\,dx \text{, when $2\leq \ell\leq m-1,$}\\ \nonumber
\begin{bmatrix}
    g_{\ell,+,1}(k)\\
    g_{\ell,+,2}(k)
\end{bmatrix}=&\int_{0}^{{+}\infty}f_{\ell,+}\left(x+\frac{y_{\ell}+y_{\ell-1}}{2}\right)\chi({-}x)e^{ikx}\,dx \text{, when $2\leq \ell\leq m-1,$}\\ \nonumber \begin{bmatrix}
    g_{m,+,1}(k)\\
    g_{m,+,2}(k)
\end{bmatrix}=&\int_{0}^{{+}\infty}f_{m,+}\left(x+\frac{y_{m}+y_{m-1}}{2}\right)\chi({-}x)e^{ikx}\,dx.
\end{align}
Since $\supp \chi(x)\subset ({-}\infty,{-}\epsilon],$ and $\supp \chi({-}x)\subset [\epsilon,{+}\infty),$ it is not difficult to verify for any continuous function $f$ that
\begin{equation*}
    \int_{\mathbb{R}}\chi(x)f(x)\,dx=\int_{{-}\infty}^{0}\chi(x)f(x)\,dx,\,\, \int_{\mathbb{R}}\chi({-}x)f(x)\,dx=\int_{0}^{{+}\infty}\chi({-}x)f(x)\,dx.
\end{equation*}
\par Furthermore, applying Lemma \ref{Tp} in the approximate linear system satisfied by $g_{\ell,\pm,1}$ and $g_{\ell,\pm,2},$ we can verify similarly to the proof fo Theorem \ref{TT} that the following estimates hold.
\begin{align}\label{esttt1}
    \max_{\pm,\ell}\norm{\chi\left(x-\frac{y_{\ell\mp 1}+y_{\ell}}{2}\right)\langle x \rangle^{j} f_{\ell,\pm}\left(x\right)}_{L^{2}_{x}(\mathbb{R})}\leq &C_{v,j}\max_{\ell}\norm{\chi_{\ell}(x)\langle x-y_{\ell}\rangle^{j} f}_{L^{2}_{x}(\mathbb{R})},\\ \nonumber \max_{\ell,\pm}\norm{\chi\left(x-\frac{y_{\ell\mp 1}+y_{\ell}}{2}\right)f_{\ell,\pm}\left(x\right)}_{L^{1}_{x}(\mathbb{R})}\cong \max_{\ell,\pm}\norm{g_{\ell,\pm}(k)}_{\mathcal{F}L^{1}} \leq & C_{v}\norm{f}_{L^{1}_{x}(\mathbb{R})},\\
    \max_{\ell}\norm{\langle k\rangle^{2}g_{\ell,\pm}(k)}_{L^{2}_{k}(\mathbb{R})}\leq &C_{v}\norm{f}_{H^{2}_{x}(\mathbb{R})},
\end{align}    
and
\begin{align}\label{D0}
    \max_{\ell,n\in\{0,1,2\}}\norm{\frac{d^{n}}{dk^{n}}g_{\ell,\pm}(k)}_{L^{2}_{k}(\mathbb{R})}\leq   C\max_{n\in\{0,1,2\}} (y_{1}-y_{m})^{n}\max\norm{\langle x-y_{\ell}\rangle \chi_{\ell}^{\perp}(x)f(x)}_{H^{2-n}_{k}(\mathbb{R})}, 
\end{align}
where 
\begin{equation*}
    \chi_{\ell}^{\perp}(x)=1-\chi\left(x-\frac{y_{\ell}+y_{\ell+1}}{2}\right)-\chi\left({-}x+\frac{y_{\ell}+y_{\ell-1}}{2}\right),
\end{equation*} 
and the equivalence relation above is a consequence of the inverse Fourier transform identity. 
\par Furthermore, we have the following proposition.
\begin{lemma}\label{lemmafordecay}
Assume that the operator $\mathcal{H}_{\omega}$ satisfies the assumptions $(H1)-(H3),$ and that $\omega_,\,{-}\omega$ are not resonance points of $\mathcal{H}_{\omega}.$
Then, there exists a constant $C>1$ satisfying the following estimates for any $f\in W^{1,1}_{x}(\mathbb{R},\mathbb{C}^{2})\cap H^{1}_{x}(\mathbb{R},\mathbb{C}^{2}).$
\begin{equation}\label{estpp1}
\norm{kF^{*}_\omega(k)(f)}_{\mathcal{F}L_{1}}\leq C\norm{f}_{W^{1,1}_{x}(\mathbb{R})}.
\end{equation}
Furthermore, if $xf(x)\in W^{1,1}_{x}(\mathbb{R},\mathbb{C}^{2}),$ then
\begin{equation}\label{estpp2}
\norm{k\partial_{k}F^{*}_\omega(k)(f)}_{\mathcal{F}L_{1}}\leq C\left[\norm{(1+\vert x\vert)\partial_{x}f(x)}_{L^{1}_{x}(\mathbb{R})}+\norm{(1+\vert x\vert  )f}_{L^{1}_{x}(\mathbb{R})}\right].
\end{equation}
\end{lemma}
\begin{proof}
 The proof is completely similar to the proof of Lemma $B4$.    
\end{proof}
Furthermore, using Lemma $B.3,$ we can apply \eqref{estpp1} to the estimate \eqref{princprinc} and obtain
\begin{equation*}
  \max_{\ell}\norm{k\phi_{\ell}(k)}_{\mathcal{F}L^{1}}\leq C(v)\norm{T(0)(\vec{\phi})}_{W^{1,1}_{x}(\mathbb{R})}.  
\end{equation*}

Consequently, using the definition of the distorted Fourier basis $\mathcal{F},\,\mathcal{G},$ we can verify the following proposition.
\begin{proposition}\label{dxlinfty}
If $t>>0,$ then
\begin{multline*}
    \norm{\frac{\partial}{\partial x}\left[e^{i\frac{v_{\ell}x\sigma_{3}}{2}}\hat{G}_{\omega_{\ell}}\left(e^{{-}it(k^{2}+\omega_{\ell})\sigma_{3}}\phi_{\ell}\left(k+\sigma_{3}\frac{v_{\ell}}{2}\right)\right)(x)\right]}_{L^{\infty}_{x}(\mathbb{R})}\\
    \leq \frac{C(v)}{t^{\frac{1}{2}}}\left[\max_{j\in\{\ell,\ell-1\}}\norm{\phi_{j}(k)}_{\mathcal{F}L^{1}}+\norm{k\phi_{j}(k)}_{\mathcal{F}L^{1}}\right].
\end{multline*}
\end{proposition}
\begin{proof}
Since it is well-known that
\begin{equation*}
\norm{e^{i\frac{v_{\ell}x\sigma_{3}}{2}}\hat{G}_{\omega_{\ell}}\left(e^{{-}it(k^{2}+\omega_{\ell})\sigma_{3}}\phi_{\ell}\left(k+\sigma_{3}\frac{v_{\ell}}{2}\right)\right)(x)}_{L^{\infty}_{x}(\mathbb{R})}\leq \frac{C}{t^{\frac{1}{2}}}\max_{j\in\{\ell,\ell-1\}}\norm{\phi_{j}(k)}_{FL^{1}},  
\end{equation*}
it is enough to verify that 
\begin{multline*}
    \norm{\frac{\partial}{\partial x}\left[\hat{G}_{\omega_{\ell}}\left(e^{{-}it(k^{2}+\omega_{\ell})\sigma_{3}}\phi_{\ell}\left(k+\sigma_{3}\frac{v_{\ell}}{2}\right)\right)(x)\right]}_{L^{\infty}_{x}(\mathbb{R})}\\
    \leq \frac{C(v)}{t^{\frac{1}{2}}}\left[\max_{j\in\{\ell,\ell-1\}}\norm{\phi_{j}(k)}_{\mathcal{F}L^{1}}+\norm{k\phi_{j}(k)}_{\mathcal{F}L^{1}}\right].
\end{multline*}
From the definition of $\hat{G}_{\omega_{\ell}},$ we have that
\begin{equation*}
\hat{G}_{\omega_{\ell}}\left(e^{{-}it(k^ {2}+\omega_{\ell})\sigma_{3}}\begin{bmatrix}\varphi(k)\\
0\end{bmatrix}\right)(x)=\int_{\mathbb{R}}e^ {{-}it(k^{2}+\omega_{\ell})}\frac{\mathcal{G}_{\omega_{\ell}}(x,{-}k)}{s_{\ell}({-}k)}\varphi(k)\,dk.
\end{equation*}
In particular, using the fact that $\frac{k}{s_{\ell}(k)}\neq 0$ is smooth on $\mathbb{R},$ we can verify that
\begin{align*}
   \frac{d}{dx}\int_{\mathbb{R}}e^ {{-}it(k^{2}+\omega_{\ell})}\frac{\mathcal{G}_{\omega_{\ell}}(x,{-}k)}{s_{\ell}({-}k)}\varphi(k)\,dk=& \int_{\mathbb{R}}e^ {{-}it(k^{2}+\omega_{\ell})}\frac{\partial}{\partial x}\left[\frac{\mathcal{G}_{\omega_{\ell}}(x,{-}k)}{k}\right]\frac{k}{s_{\ell}({-}k)}\varphi(k)\,dk.
\end{align*}
Next, let $\chi_{L}$ be a $C^{\infty}$ cut-off function supported in $[{-}2L,2L]$ satisfying $\chi_{L}(k)=1$ if $\vert k\vert\leq L.$ We can verify that
\begin{multline*}
 \norm{\int_{\mathbb{R}}e^ {{-}it(k^{2}+\omega_{\ell})}\chi_{L}(k)\frac{\partial}{\partial x}\left[\frac{\mathcal{G}_{\omega_{\ell}}(x,{-}k)}{k}\right]\frac{k}{s_{\ell}({-}k)}\varphi(k)\,dk}_{L^{\infty}_{x}([{-}L,L])}\\
 \begin{aligned}
 \leq & C\max_{z\in[{-}L,L]}\norm{\int_{\mathbb{R}}e^ {{-}it(k^{2}+\omega_{\ell})}e^{ik(x-z)}\chi_{L}(k)\frac{\partial}{\partial z}\left[\frac{\mathcal{G}_{\omega_{\ell}}(z,{-}k)}{k}\right]\frac{k}{s_{\ell}({-}k)}\varphi(k)\,dk}_{L^{\infty}_{x}([{-}L,L])}\\
 \leq & \max_{z\in[L,{-}L]}\frac{C}{t^{\frac{1}{2}}}\norm{\int_{\mathbb{R}}e^{ik(x-z)}\chi_{L}(k)\frac{\partial}{\partial z}\left[\frac{\mathcal{G}_{\omega_{\ell}}(z,{-}k)}{k}\right]\frac{k}{s_{\ell}({-}k)}\varphi(k)\,dk}_{L^{1}_{x}(\mathbb{R})}\\
 \leq & \max_{z\in[{-}L,L]}\frac{C}{t^{\frac{1}{2}}}\norm{\varphi(k)}_{\mathcal{F}L^{1}}\norm{\chi_{L}(k)\frac{\partial}{\partial z}\left[\frac{\mathcal{G}_{\omega_{\ell}}(z,{-}k)}{k}\right]\frac{k}{s_{\ell}({-}k)}}_{\mathcal{F}L^{1}}\\
 \leq & \max_{z\in[L,{-}L]}\frac{C_{1}}{t^{\frac{1}{2}}}\norm{\varphi}_{\mathcal{F}L^{1}}\norm{\chi_{L}(k)\frac{\partial}{\partial z}\left[\frac{\mathcal{G}_{\omega_{\ell}}(z,{-}k)}{k}\right]\frac{k}{s_{\ell}({-}k)}}_{H^{1}_{k}}\leq \frac{K}{t^{\frac{1}{2}}}\norm{\varphi}_{\mathcal{F}L^{1}},
\end{aligned}
\end{multline*}
where the facts that $s(k)k^{{-}1}$ being smooth, $\partial^{n}_{k^{n}}\mathcal{G}_{\omega_{\ell}}(x,0)=0,$ and $\mathcal{G}_{\omega_{\ell}}(x,k)$ being a smooth function were used for the proof of the last inequality above. 
\par If $\vert x\vert>L,$ we can use the asymptotic properties of $\mathcal{G}_{\omega_{\ell}}$ to verify that 
\begin{equation*}
\norm{\int_{\mathbb{R}}\frac{\partial}{\partial x}\frac{\mathcal{G}_{\omega_{\ell}}(x,{-}k)\varphi(k)}{s_{\ell}({-}k)}\,dk}_{L^{1}_{x}(\vert x\vert\geq L)}\leq \max\left(\norm{\varphi(k)}_{\mathcal{F}L^{1}},\norm{k\varphi(k)}_{\mathcal{F}L^{1}},\norm{k\phi(k)}_{\mathcal{F}L^{1}},\norm{\phi(k)}_{\mathcal{F}L^{1}}\right),
\end{equation*}
where 
\begin{equation*}
    \phi(k)=\frac{\varphi(k)-r_{\ell}(k)\varphi({-}k)}{s_{\ell}(k)},
\end{equation*}
see item $b)$ of Definition $1.5.$
\par Next, since 
\begin{equation*}
    \frac{d^{l}}{dk^{l}}\left[1-\frac{1}{s({-}k)}\right]=O\left(\frac{1}{(1+\vert k\vert)^{\ell+1}}\right) \text{, when $k\in \supp$ of $1-\chi_{L},$}
\end{equation*}
we can verify similarly to the estimate of
\begin{equation*}
    \norm{\int_{\mathbb{R}}e^ {{-}it(k^{2}+\omega_{\ell})}\chi_{L}(k)\frac{\partial}{\partial x}\left[\frac{\mathcal{G}_{\omega_{\ell}}(x,{-}k)}{k}\right]\frac{k}{s_{\ell}({-}k)}\varphi(k)\,dk}_{L^{\infty}_{x}([{-}L,L])}
\end{equation*}
using that $\partial^{n}_{k^{n}}\partial_{x}\mathcal{G}_{\omega_{\ell}}(x,k)\in L^{\infty}_{x}(\mathbb{R})$ that
\begin{multline*}
\norm{\int_{\mathbb{R}}e^ {{-}it(k^{2}+\omega_{\ell})}[1-\chi_{L}(k)]\frac{\partial}{\partial x}\left[\frac{\mathcal{G}_{\omega_{\ell}}(x,{-}k)}{k}\right]\frac{k}{s_{\ell}({-}k)}\varphi(k)\,dk}_{L^{\infty}_{x}([{-}L,L])}\\
    \begin{aligned}
      \leq &\frac{C}{t^{\frac{1}{2}}}\max_{z\in[{-}L,L]}\norm{[1-\chi_{L}(k)]e^{ikz}\left[\frac{\partial_{z}\mathcal{G}_{\omega_{\ell}}(z,{-}k)}{s_{\ell}({-}k)k}\right]}_{H^{1}_{k}}[\norm{k\varphi(k)}_{\mathcal{F}L^{1}}+\norm{\varphi(k)}_{\mathcal{F}L^{1}}]\\
      \lesssim &\frac{1} {t^{\frac{1}{2}}}[\norm{k\varphi(k)}_{\mathcal{F}L^{1}}+\norm{\varphi(k)}_{\mathcal{F}L^{1}}].
    \end{aligned}
\end{multline*}
In conclusion, we deduce from triangular inequality that the statement of Proposition \ref{dxlinfty} is true.
\end{proof}
 
\par In particular, the triangular inequality implies that 
\begin{equation*}
\norm{\frac{\partial}{\partial x}\mathcal{S}(t)(\vec{\phi})}_{L^{\infty}}\leq \frac{C(v)}{t^{\frac{1}{2}}}\left[\max_{\ell}\norm{\vec{\phi}_{\ell}(k)}_{\mathcal{F}L^{1}}+\norm{k\vec{\phi}_{\ell}(k)}_{\mathcal{F}L^{1}}\right].
\end{equation*}
Next, using Theorem $1.8$ and Lemma $7.3,$ we can obtain the following.
\begin{align*}
    \norm{\frac{\partial}{\partial x}\left[\mathcal{T}(t)(\vec{\phi})-\mathcal{S}(t)(\vec{\phi})\right]}_{L^{\infty}_{x}}\leq C & \norm{\frac{\partial}{\partial x}\left[\mathcal{T}(t)(\vec{\phi})-\mathcal{S}(t)(\vec{\phi})\right]}_{H^{1}_{x}}
    \\ \leq & K(v) e^{{-}\beta \min_{\ell}(v_{\ell}-v_{\ell+1})t+(y_{\ell}-y_{\ell+1})}\norm{\mathcal{S}(t)(\vec{\phi})}_{H^{1}_{x}(\mathbb{R})}\\
    \sim & e^{{-}\beta \min_{\ell}(v_{\ell}-v_{\ell+1})t+(y_{\ell}-y_{\ell+1})}\norm{\mathcal{S}(0)(\vec{\phi})}_{H^{1}_{x}(\mathbb{R})}\\
    \sim & e^{{-}\beta \min_{\ell}(v_{\ell}-v_{\ell+1})t+(y_{\ell}-y_{\ell+1})}\norm{\mathcal{T}(0)(\vec{\phi})}_{H^{1}_{x}(\mathbb{R})}. 
\end{align*}

Consequently, using \eqref{princprinc}, we have the following proposition.
\begin{proposition}
The following estimate is true.
\begin{equation*}
    \norm{\frac{\partial}{\partial x}\mathcal{T}(t)(\vec{\phi})}_{L^{\infty}_{x}}\leq \left[\frac{C(v)}{t^{\frac{1}{2}}}\norm{\mathcal{T}(0)(\vec{\phi})}_{W^{1,1}_{x}(\mathbb{R})}+C(v)e^{{-}\beta [\min_{\ell}(v_{\ell}-v_{\ell+1})t+(y_{\ell}-y_{\ell+1})]}\norm{\mathcal{T}(0)(\vec{\phi})}_{H^{1}_{x}(\mathbb{R})}\right].
\end{equation*}
\end{proposition}
\par In particular, if $f\in\Raa \mathcal{T}(0),$ we can verify from equations \eqref{lG}, \eqref{mG} and \eqref{1G} that
\begin{equation*}
  \max_{\ell} \norm{\langle x\rangle \vec{v_{d_{\ell}}}(x)}_{L^{1}_{x}(\mathbb{R})}\cong \max_{\ell}\norm{\vec{v_{d_{\ell}}}(x)}_{L^{2}_{x}(\mathbb{R})}=O\left(e^{{-}\beta \min_{\ell}(y_{\ell}-y_{\ell+1})}\max_{\ell,\pm}\left[\norm{f_{\ell,\pm}}_{L^{2}_{x}(\mathbb{R})}\right]\right).
\end{equation*}
Consequently, \eqref{lG} and \eqref{esttt1} imply that
\begin{multline}\label{gwe}
 \max_{\ell} \norm{\langle x\rangle \hat{G}_{\ell}\left(e^{iy_{\ell}k}\vec{\phi_{\ell}}\left(k+\sigma_{3}\frac{v_{\ell}}{2}\right)\right)(x)}_{L^{1}_{x}(\mathbb{R})}\\
 \leq C\left[\max_{\ell}\norm{\langle x-y_{\ell}\rangle\chi_{\ell}^{\perp}(x)f(x)}_{L^{1}_{x}(\mathbb{R})}+\max_{\ell}\norm{\langle x\rangle f_{\ell,\pm}(x+y_{\ell})\chi\left(\mp x\mp\frac{y_{\ell}-y_{\ell\mp 1}}{2}\right)}_{L^{1}_{x}(\mathbb{R})}\right]\\
 +Ce^{{-}\beta\min_{\ell}(y_{\ell}-y_{\ell+1})}\norm{f}_{L^{2}_{x}(\mathbb{R})}.
\end{multline}
Therefore, the estimate of $\norm{\langle x\rangle \hat{G}_{\ell}\left(e^{iy_{\ell}k}\vec{\phi_{\ell}}\left(k+\sigma_{3}\frac{v_{\ell}}{2}\right)\right)(x)}_{L^{1}_{x}(\mathbb{R})}$ will follow from the estimate of 
\begin{equation*}
    \norm{\langle x\rangle f_{\ell,\pm}(x+y_{\ell})\chi\left(\mp x\mp\frac{y_{\ell}-y_{\ell\mp 1}}{2}\right)}_{L^{1}_{x}(\mathbb{R})},
\end{equation*}
which will be obtained from the linear system solved by \eqref{unknown}.
\begin{lemma}\label{l2}
If $f=\mathcal{T}(0)(\vec{\psi})(x),$ $\vec{\phi}$ satisfies all the equations in \eqref{unknown}, then there exist uniform constants $C>1$ satisfying
\begin{align*}
 \max_{\ell}\norm{\langle k \rangle^{2} \vec{\phi}_{\ell}(k)}\leq & C\norm{f(x)}_{H^{2}_{x}(\mathbb{R})},\\
    \max_{\ell}\norm{\vec{\phi_{\ell}}-\vec{\psi_{\ell}}}_{H^{2}_{k}(\mathbb{R})}\leq & C e^{{-}\beta \min_{\ell}(y_{\ell}-y_{\ell+1})}\norm{f(x)}_{H^{2}_{x}(\mathbb{R})},\\
     \max_{\ell}\norm{k(\vec{\phi_{\ell}}-\vec{\psi_{\ell}})}_{H^{2}_{k}(\mathbb{R})}\leq & C e^{{-}\beta \min_{\ell}(y_{\ell}-y_{\ell+1})}\norm{f(x)}_{H^{2}_{x}(\mathbb{R})},\\
     \norm{\langle x-y_{\ell} \rangle ^ {2}\chi_{\ell}^{\perp}(x)\mathcal{S}(0)\left[\vec{\phi}-\vec{\psi}\right]}_{H^{1}_{x}(\mathbb{R})}\leq & C(y_{1}-y_{m})^{2}e^{{-}\beta\min_{\ell}(y_{\ell}-y_{\ell+1})}\norm{\vec{f}(x)}_{H^{2}_{x}(\mathbb{R}
     )}.
\end{align*}
\end{lemma}
\begin{proof}
The proof of the inequality
\begin{equation*}
   \max_{\ell}\norm{\langle k \rangle^{2} \vec{\phi}_{\ell}(k)}\leq  C\norm{f(x)}_{H^{2}_{x}(\mathbb{R})} 
\end{equation*}
 is completely similar to the proof of Corollary \ref{coo}.
 \par Next, from the definition of $\vec{\phi}$ and $\mathcal{T},$ we have that
\begin{multline}\label{sdphi}
\sum_{\ell}\chi_{\left[\frac{y_{\ell}+y_{\ell+1}}{2},\frac{y_{\ell}+y_{\ell-1}}{2}\right]}(x)\frac{\partial^{j}}{\partial x^{j}}\left[\langle x-y_{\ell}\rangle^{n} e^{i\frac{v_{\ell}x\sigma_{3}}{2}}\hat{G}_{\omega_{\ell}}\left(e^{iy_{\ell}k}\begin{bmatrix}
      [\phi_{1,\ell}-\psi_{1,\ell}]\left(k+\frac{v_{\ell}}{2}\right)\\
      [\phi_{2,\ell}-\psi_{2,\ell}]\left(k-\frac{v_{\ell}}{2}\right)
  \end{bmatrix}\right)(x-y_{\ell})\right]\\
=\sum_{\ell}\chi_{\left[\frac{y_{\ell}+y_{\ell+1}}{2},\frac{y_{\ell}+y_{\ell-1}}{2}\right]}(x)\frac{\partial^{j}}{\partial x^{j}}\left[\langle x-y_{\ell}\rangle^{n}r(t,x)\right]\\{+}O_{L^{2}}\left(e^{{-}\beta \min_{\ell}(y_{\ell}-y_{\ell+1})}\norm{f}_{L^{2}_{x}(\mathbb{R})}\right),
\end{multline}
for any $n\in\{0,1,2\},\,j\in\{0,1\}$
such that $r(t,x)=\mathcal{T}(0)(\vec{\psi})-\mathcal{S}(0)(\vec{\psi}).$
\par In particular, Remark $7.5$ and the proof of Theorem $1.8$ imply that the first two inequalities of the Lemma above can be obtained from \eqref{sdphi}.
\par Next, using the definition of $\mathcal{S}(0),$ we can verify for any $j\in\{0,1\}$ that
\begin{multline*}
\left\langle x-y_{\ell} \right\rangle^{2}\chi_{\ell}^{\perp}(x)\frac{\partial^{j}}{\partial x^{j}}\mathcal{S}(0)(\vec{\phi}-\vec{\psi})(x)=\chi_{\ell}^{\perp}(x)\langle x-y_{\ell}\rangle^{2}\frac{\partial^{j}}{\partial x^ {j}}\left[e^{i\frac{\sigma_{3}v_{\ell}x}{2}}\hat{G}_{\omega_{\ell}}\left(\vec{\phi}_{\ell}-\vec{\psi}_{\ell}\right)(x-y_{\ell})\right]\\{+}O_{L^{2}}\left((y_{1}-y_{m})^{2}e^{{-}\beta\min_{\ell}(y_{\ell}-y_{\ell+1})}\max_{\ell}\left[\norm{\vec{\phi}_{\ell}}_{L^ {2}}+\norm{\vec{\psi}_{\ell}}_{L^{2}}\right]\right). \end{multline*}
We also remark that $\supp \chi_{\ell}^{\perp}\subset \supp \chi_{\left[\frac{y_{\ell}+y_{\ell+1}}{2}-2\epsilon,\frac{y_{\ell}+y_{\ell-1}}{2}+2\epsilon\right]},$ from the definition of the smooth cut-off function $\chi_{\ell}.$ 
In conclusion, using estimate \eqref{tcontest} from Theorem \ref{tcont}, the inequality above, estimate \eqref{sdphi}, and triangular inequality, we obtain the last inequality of the proposition. 
\end{proof}
\par Therefore, we can deduce from the above lemma and triangular inequality that if $f=\mathcal{T}(0)(\vec{\psi}),$ then
\begin{multline}\label{eddd1}
    \norm{\langle x \rangle \hat{G}_{\omega_{\ell}}\left(e^{iy_{\ell}k}\vec{\psi}_{\ell}\left(k+\sigma_{3}\frac{v_{\ell}}{2}\right)\right)(x) }_{L^{1}_{x}(\mathbb{R})}\\ \leq C\left[\max_{\ell}\norm{\langle x-y_{\ell}\rangle\chi_{\ell}^{\perp}(x)f(x)}_{L^{1}_{x}(\mathbb{R})}+\max_{\ell}\norm{\langle x\rangle f_{\ell,\pm}(x+y_{\ell})\chi\left(\mp x\pm\frac{y_{\ell}-y_{\ell\mp 1}}{2}\right)}_{L^{1}_{x}(\mathbb{R})}\right]\\
 +C(v)e^{{-}\beta\min_{\ell}(y_{\ell}-y_{\ell+1})}\norm{f}_{H^{2}_{x}(\mathbb{R})}.
\end{multline}
Next, similarly to the proof of Theorem \ref{TT}, we can verify that
\begin{multline*}
     \norm{\langle x\rangle f_{\ell,\pm}(x+y_{\ell})\chi\left(\mp x\pm\frac{y_{\ell}-y_{\ell\mp 1}}{2}\right)}_{L^{1}_{x}(\mathbb{R})}\leq C\max_{j\in\{0,1\},\ell}(y_{1}-y_{m})^{1-j}\norm{\langle x-y_{\ell}\rangle^{j}\chi_{\ell}^{\perp}(x)f(x)}_{L^{1}_{x}(\mathbb{R})}\\{+}Ce^ {{-}\beta\min(y_{\ell}-y_{\ell+1})}\max\left(\norm{f}_{L^{1}_{x}(\mathbb{R})},\norm{f}_{L^{2}_{x}(\mathbb{R})}\right),
\end{multline*}
from which we deduce that
\begin{align*}
    \norm{\langle x \rangle \hat{G}_{\omega_{\ell}}\left(e^{iy_{\ell}k}\vec{\psi}_{\ell}\left(k+\sigma_{3}\frac{v_{\ell}}{2}\right)\right)(x) }_{L^{1}_{x}(\mathbb{R})} \leq & C(v)\left[\max_{\ell}\norm{\langle x-y_{\ell}\rangle\chi_{\ell}^{\perp}(x)f(x)}_{L^{1}_{x}(\mathbb{R})}
 +\norm{f}_{L^{1}_{x}(\mathbb{R})}(y_{1}-y_{m})\right]\\
 &{+}e^{{-}\beta\min_{\ell}(y_{\ell}-y_{\ell+1})}\norm{f}_{H^{2}_{x}(\mathbb{R})}.
\end{align*}
Next, using the inverse formula of the Fourier transform and \eqref{unknown}, we can verify the following estimate
\begin{equation*}
    c\norm{k\partial_{k}g_{\ell,\pm}(k)}_{\mathcal{F}L^{1}} \leq \norm{\frac{d}{dx}\left[x f_{\ell,\pm}\left(x+\frac{y_{\ell}+y_{\ell\mp 1}}{2}\right)\chi\left(\mp x\right)\right]}_{L^{1}_{x}(\mathbb{R})}\leq C \norm{k\partial_{k}g_{\ell,\pm}(k)}_{\mathcal{F}L^{1}}.
\end{equation*}
\par Furthermore, using estimates \eqref{esttt1} and Lemma \ref{fordecay}, we can verify that
\begin{align*}
   \max_{\ell}\norm{\langle k\rangle \partial_{k}g_{\ell,\pm}(k)}_{FL^{1}}\leq & C(v) \max_{\ell}\left[\norm{\frac{d}{dx}\left[\langle x-y_{\ell}\rangle\chi_{\ell}^{\perp}(x)f(x)\right]}_{L^{1}_{x}(\mathbb{R})}+\norm{\langle x-y_{\ell}\rangle\chi_{\ell}^{\perp}(x)f(x)}_{L^{1}_{x}(\mathbb{R})}\right]
 \\&{+}C(v)[y_{1}-y_{m}]\norm{f}_{L^{1}_{x}(\mathbb{R})}+Ce^ {{-}\beta\min_{\ell}(y_{\ell}-y_{\ell+1})}\norm{f}_{L^{2}_{x}(\mathbb{R})}\\
 &{+}C(v)\max_{\ell}\norm{\left\langle \frac{d}{dx}  \right\rangle\left[\langle x \rangle\hat{G}_{\omega_{\ell}}\left(e^{iy_{\ell}k}\vec{\phi}_{\ell}\left(k+\sigma_{3}\frac{v_{\ell}}{2}\right)\right)(x)\right]}_{L^{1}_{x}(\mathbb{R})}. 
\end{align*}

\par Moreover, using the inverse Fourier formula and the identities
\begin{align*}
    \langle k\rangle \partial_{k}g_{\ell,-}(k)=&c\int_{{-}\infty}^{{+}\infty} \langle \frac{d}{dx}\rangle \left[x f_{\ell,-}\left(x+\frac{y_{\ell}+y_{\ell+1}}{2}\right)\chi(x)\right]e^{ikx}\,dx,\\
    \langle k\rangle \partial_{k}g_{\ell,+}(k)=&c\int_{{-}\infty}^{{+}\infty} \langle \frac{d}{dx} \rangle\left[ x f_{\ell,+}\left(x+\frac{y_{\ell}+y_{\ell-1}}{2}\right)\chi({-}x)\right]e^{ikx}\,dx,  
\end{align*}
we deduce that for some constant $c\in\mathbb{C}$ that
\begin{align}\label{F01}
   F_{0}\left(\langle k\rangle \partial_{k}g_{\ell,-}(k)\right)(x)=2\pi c \langle \frac{d}{dx}  \rangle\left[x f_{\ell,-}\left(x+\frac{y_{\ell}+y_{\ell+1}}{2}\right)\chi(x)\right],\\ \label{F02}
   F_{0}\left(\langle k\rangle \partial_{k}g_{\ell,+}(k)\right)(x)=2c\pi \langle \frac{d}{dx} \rangle\left[ xf_{\ell,+}\left(x+\frac{y_{\ell}+y_{\ell-1}}{2}\right)\chi({-}x)\right].
\end{align}
\begin{proposition}
If $p\in (1,2),$ there exists $C(p)>0$ satisfying the following inequality for all $t>0$ and $p^{*}=\frac{p}{p-1}$
\begin{equation*}
\norm{\frac{\hat{G}_{\omega_{\ell}}(e^{{-}itk^{2}\sigma_{3}}\vec{\phi})(x)}{\langle x\rangle}}_{W^{1,p^{*}}_{x}(\mathbb{R})}\leq \frac{C(p)}{t^{\frac{3}{2}(\frac{1}{p}-\frac{1}{p^{*}})}}\max_{n\in\{0,1\}}\norm{\langle x \rangle^{\frac{1}{p}-\frac{1}{p^{*}}}\partial^{n}_{x}\hat{G}_{\omega_{\ell}}(\vec{\phi})(x)}_{L^{p}_{x}(\mathbb{R})}.
\end{equation*}
In particular,
\begin{align*}
\norm{\frac{\hat{G}_{\omega_{\ell}}(e^{{-}itk^{2}\sigma_{3}}\vec{\phi})(x)}{\langle x\rangle}}_{W^{1,p^{*}}_{x}(\mathbb{R})}\leq & \frac{C(p)}{t^{\frac{3}{2}(\frac{1}{p}-\frac{1}{p^{*}})}}\norm{\langle x \rangle\partial_{x}\hat{G}_{\omega_{\ell}}(\vec{\phi})(x)}_{L^{1}_{x}(\mathbb{R})}^{\frac{2-p}{p}} \norm{\partial_{x}\hat{G}_{\omega_{\ell}}(\vec{\phi})(x)}_{L^{2}_{x}(\mathbb{R})}^{\frac{2(p-1)}{p}}\\
&{+}\frac{C(p)}{t^{\frac{3}{2}(\frac{1}{p}-\frac{1}{p^{*}})}}\norm{\langle x \rangle\hat{G}_{\omega_{\ell}}(\vec{\phi})(x)}_{L^{1}_{x}(\mathbb{R})}^{\frac{2-p}{p}}\norm{\hat{G}_{\omega_{\ell}}(\vec{\phi})(x)}_{L^{2}_{x}(\mathbb{R})}^{\frac{2(p-1)}{p}}.  
\end{align*}
\end{proposition}
\begin{proof}
    See the proof of Corollary $8.3$ in \cite{KriegerSchlag}. the second inequality follows from the first inequality and Holder inequality.
\end{proof}
\subsection{Estimate of $\norm{\langle x \rangle\frac{d}{dx}\hat{G}_{\omega_{\ell}}\left(\vec{\phi}_{\ell}\right)(x)}_{L^{1}}$}
From the equations \eqref{1G}, \eqref{lG} and \eqref{mG}, to estimate $\norm{\langle x \rangle\frac{d}{dx}\hat{G}_{\omega_{\ell}}\left(\vec{\phi}_{\ell}\right)(x)}_{L^{1}_{x}(\mathbb{R})},$ it is enough to estimate the terms
\begin{equation*}
    \norm{\left\langle \frac{d}{dx} \right\rangle x \chi\left(x+\frac{y_{\ell}-y_{\ell+1}}{2}\right)f_{\ell,-}(x+y_{\ell})}_{L^{1}_{x}(\mathbb{R})}+ \norm{\left\langle \frac{d}{dx} \right\rangle x \chi\left({-}x+\frac{y_{\ell-1}-y_{\ell}}{2}\right)f_{\ell,+}(x+y_{\ell})}_{L^{1}_{x}(\mathbb{R})}. 
\end{equation*}
In particular, using Minkowski inequality, we obtain that
\begin{align*}
    \norm{\left\langle \frac{d}{dx} \right\rangle x \chi\left(x+\frac{y_{\ell}-y_{\ell+1}}{2}\right)f_{\ell,-}(x+y_{\ell})}_{L^{1}_{x}(\mathbb{R})}\leq & \norm{\left\langle \frac{d}{dx} \right\rangle x \chi\left(x\right)f_{\ell,-}\left(x+\frac{y_{\ell}+y_{\ell+1}}{2}\right)}_{L^{1}_{x}(\mathbb{R})}\\
    &{+}\frac{y_{\ell}-y_{\ell+1}}{2}\norm{\left\langle \frac{d}{dx} \right\rangle \chi\left(x\right)f_{\ell,-}\left(x+\frac{y_{\ell}+y_{\ell+1}}{2}\right)}_{L^{1}_{x}(\mathbb{R})}
\end{align*}
Consequently, from \eqref{F01} and \eqref{F02}, we have
\begin{equation*}
    \norm{\left\langle \frac{d}{dx} \right\rangle x \chi\left(x+\frac{y_{\ell}-y_{\ell+1}}{2}\right)f_{\ell,-}(x+y_{\ell})}_{L^{1}_{x}(\mathbb{R})}\leq C (y_{1}-y_{m}) \norm{\langle k\rangle g_{\ell,-}(k)}_{FL^{1}}+\norm{\langle k\rangle \partial_{k}g_{\ell,-}(k)}_{FL^{1}}. 
\end{equation*}
\par Furthermore, using \eqref{lG}, \eqref{1G}, \eqref{mG}, we can verify similarly to the argument used in the steps $5$ and $6$ of the proof of Lemma \ref{dec} in Section $5$ that the functions $\vec{g}_{\ell}$ solve a linear system from Lemma \ref{tcopp} for 
\begin{align*}
  T_{\ell}(f)(k)=&c_{1}e^{i\sigma_{3}c_{4}}F^{*}_{\omega_{\ell}}\left(\chi^ {\perp}_{\ell}(x)e^{\frac{{-}i\sigma_{3}v_{\ell}(x+y_{\ell})}{2}}f(x+y_{\ell})\right)(k)\\&{+}c_{2}e^ {i\sigma_{3}c_{5}}G^{*}_{\omega_{\ell+1}}\left(\chi^ {\perp}_{\ell+1}(x)e^{\frac{{-}i\sigma_{3}v_{\ell+1}(x+y_{\ell+1})}{2}}f(x+y_{\ell+1})\right)(k+v_{*}).   
\end{align*}
Therefore, using Lemma \ref{lemmafordecay}, we can conclude that
\begin{align}\label{princprinc}
  \max_{j\in\{0,1\},\ell\in[2m-2]}\norm{\langle k\rangle^{j} T_{\ell}(f)(k)}_{\mathcal{F}L^{1}}\leq &C(v)\norm{f}_{W^{1,1}_{x}(\mathbb{R})},\\
 \max_{j\in\{0,1\},\ell\in[2m-2]}\norm{\langle k\rangle^{j} \partial_{k}T_{\ell}(f)(k)}_{\mathcal{F}L^{1}}\leq &C(v)\left[(y_{1}-y_{m})\norm{f}_{W^{1,1}_{x}(\mathbb{R})}+\max_{\ell}\norm{(1+\vert x-y_{\ell}\vert)\langle \partial_{x}\rangle \chi_{\ell}^{\perp}(x)f(x)}_{L^{1}_{x}(\mathbb{R})}\right].  
\end{align}
In conclusion, similarly to the proof of \eqref{eddd1}, we obtain from \eqref{1G}, \eqref{lG}, \eqref{mG} and the Fourier identities \eqref{F01}, \eqref{F02} that
\begin{align}\label{xxG}
\norm{\langle x\rangle \langle \frac{d}{dx}\rangle \hat{G}_{\omega_{\ell}}\left(e^{iy_{\ell}k}\begin{bmatrix}
     \phi_{1,\ell}\left(k+\frac{v_{\ell}}{2}\right)\\
     \phi_{2,\ell}\left(k-\frac{v_{\ell}}{2}\right)
 \end{bmatrix}\right)(x)}_{L^{1}_{x}(\mathbb{R})}\leq & C(v)\max_{\ell}\norm{(1+\vert x-y_{\ell}\vert)\langle \partial_{x}\rangle \chi_{\ell}^{\perp}(x)f(x)}_{L^{1}_{x}(\mathbb{R})}\\ \nonumber
 &{+}C(v)(y_{1}-y_{m})[\norm{f}_{W^{1,1}_{x}(\mathbb{R})}+e^{{-}\beta \min_{\ell}(y_{\ell}-y_{\ell+1})}\norm{f}_{H^{1}_{x}(\mathbb{R})}].
\end{align}
\par Finally, we can finish the proof of Theorem \ref{interpolation est.}
\begin{proof}[Proof of Theorem \ref{interpolation est.}]
First, to simplify our notation, we consider
\begin{equation*}
    \mathcal{T}(t)(\vec{\psi})=\mathcal{U}(t,s)P_{c}(s)(\vec{f}(x)).
\end{equation*}
Theorems \ref{tcont} and \ref{princ} imply that
\begin{multline}\label{po1}
    \max_{\ell}\norm{\chi_{\ell}^ {\perp}(s,x)\langle x-v_{\ell}s-y_{\ell} \rangle^{2} \left(\mathcal{T}(s)(\vec{\psi})-\mathcal{S}(s)(\vec{\psi})\right)}_{H^{1}_{x}(\mathbb{R})}\\
    \leq C(v)(y_{1}-y_{m}+(v_{1}-v_{m})s)^{2}e^{{-}\beta(\min_{\ell}(v_{\ell}-v_{\ell+1})s+y_{\ell}-y_{\ell+1})}\norm{\mathcal{T}(s)(\vec{\psi})}_{H^{2}_{x}(\mathbb{R})},
\end{multline}
for some constant $C(v)>1$ depending on the speeds $v_{1},\,...,\,v_{m}.$ 
\par Consequently, using \eqref{po1} and the last inequality of Lemma \eqref{l2} for $t_{0}=s$ in the place of $0,$ we obtain from triangular inequality that
\begin{multline}\label{tri}
  \max_{\ell}\norm{\chi_{\ell}^{\perp}(s,x)\langle x-v_{\ell}s-y_{\ell} \rangle^{2}[\mathcal{T}(s)(\vec{\phi}-\vec{\psi})]}_{H^{1}_{x}(\mathbb{R})} \\
  \leq C(v)(y_{1}-y_{m}+(v_{1}-v_{m})s)^{2}e^{{-}\beta(\min_{\ell}(v_{\ell}-v_{\ell+1})s+y_{\ell}-y_{\ell+1})}\norm{\mathcal{T}(s)(\vec{\psi})}_{H^{2}_{x}(\mathbb{R})}.
\end{multline}
In the proof of estimate \eqref{tri}, we used the last estimate of \eqref{esttt1} which implies that 
\begin{equation*}
    \norm{\mathcal{T}(s)(\vec{\phi})}_{H^{2}_{x}(\mathbb{R})}\leq C(v)\norm{\vec{f}(x)}_{L^{2}_{x}(\mathbb{R})}.
\end{equation*}
\par In particular, since $\max_{\ell}\norm{\chi_{\ell}^{\perp}(s,x)\langle x-v_{\ell}s-y_{\ell} \rangle f(x)}_{W^{1,1}_{x}(\mathbb{R})}$ is always larger or equivalent to
\begin{equation*}
\norm{f(x)}_{W^{1,1}}+\max_{\ell}\norm{(1+\vert x-y_{\ell}-v_{\ell}s\vert)\langle \partial_{x}\rangle[ \chi_{\ell}^{\perp}(x)\vec{f}(x)]}_{L^{1}_{x}(\mathbb{R})},
\end{equation*}
we can deduce Theorem \ref{interpolation est.} from \eqref{xxG}, \eqref{po1}, \eqref{tri} and triangular inequality.
\end{proof}

\section*{Statements and Declarations}

\subsection*{Competing Interests}
The authors have no conflicts to disclose.

\subsection*{Acknowledgement(s)}

The first author was partially supported by NSF grant DMS-2350301 and by Simons foundation MP-TSM00002258.
\subsection*{Availability of data and materials}
    The authors can confirm that all relevant data are included in this article and its supplementary information files.

\bibliographystyle{plain}
\bibliography{ref}

\bigskip
\end{document}